\documentclass[a4paper]{amsart}
\usepackage{amsmath,amsthm,amssymb,latexsym,epic,bbm,comment,mathrsfs}
\usepackage{graphicx,enumerate,stmaryrd,color,tikz,mathdots}
\usepackage[all,2cell]{xy}
\xyoption{2cell}
\usetikzlibrary{arrows}

\newtheorem{theorem}{Theorem}

\newtheorem{proposition}[theorem]{Proposition}

\newtheorem{remark}[theorem]{Remark}
\usepackage[all]{xy}
\usepackage[active]{srcltx}
\usepackage[parfill]{parskip}
\usepackage{enumerate}

\renewcommand{\familydefault}{cmss}
\newcommand{\mone}{{\,\text{-}1}}
\newcommand{\mtwo}{{\,\text{-}2}}

\newcommand{\tto}{\twoheadrightarrow}
\font\sc=rsfs10
\newcommand{\cC}{\sc\mbox{C}\hspace{1.0pt}}
\newcommand{\cG}{\sc\mbox{G}\hspace{1.0pt}}
\newcommand{\cM}{\sc\mbox{M}\hspace{1.0pt}}
\newcommand{\cR}{\sc\mbox{R}\hspace{1.0pt}}
\newcommand{\cI}{\sc\mbox{I}\hspace{1.0pt}}
\newcommand{\cJ}{\sc\mbox{J}\hspace{1.0pt}}
\newcommand{\cS}{\sc\mbox{S}\hspace{1.0pt}}
\newcommand{\cT}{\sc\mbox{T}\hspace{1.0pt}}
\newcommand{\cD}{\sc\mbox{D}\hspace{1.0pt}}
\newcommand{\cL}{\sc\mbox{L}\hspace{1.0pt}}
\newcommand{\cP}{\sc\mbox{P}\hspace{1.0pt}}
\newcommand{\cA}{\sc\mbox{A}\hspace{1.0pt}}
\newcommand{\cB}{\sc\mbox{B}\hspace{1.0pt}}
\newcommand{\cU}{\sc\mbox{U}\hspace{1.0pt}}
\newcommand{\cH}{\sc\mbox{H}\hspace{1.0pt}}
\font\scc=rsfs7
\newcommand{\ccC}{\scc\mbox{C}\hspace{1.0pt}}
\newcommand{\ccP}{\scc\mbox{P}\hspace{1.0pt}}
\newcommand{\ccA}{\scc\mbox{A}\hspace{1.0pt}}
\newcommand{\ccJ}{\scc\mbox{J}\hspace{1.0pt}}
\newcommand{\ccS}{\scc\mbox{S}\hspace{1.0pt}}
\newcommand{\ccG}{\scc\mbox{G}\hspace{1.0pt}}
\newcommand{\ccH}{\scc\mbox{H}\hspace{1.0pt}}
\newcommand{\Hom}{\operatorname{Hom}}

\begin{document}
\title[Infinite rank module categories]
{Combinatorics of infinite rank module categories over\\
finite dimensional $\mathfrak{sl}_3$-modules
in Lie-algebraic context}

\author[V.~Mazorchuk and X.~Zhu]
{Volodymyr Mazorchuk and Xiaoyu Zhu}

\begin{abstract}
We determine the combinatorics of transitive module categories 
over the monoidal category of finite dimensional 
$\mathfrak{sl}_3$-modules which arise when acting by the latter
monoidal category on arbitrary simple $\mathfrak{sl}_3$-modules.
This gives us a family of eight graphs which can be viewed
as $\mathfrak{sl}_3$-generalizations of the classical infinite
Dynkin diagrams.
\end{abstract}

\maketitle

\section{Introduction and description of the results}\label{s1}

\subsection{Infinite Dynkin diagrams}\label{s1.1}

Classical Dynkin diagrams are combinatorial objects that arise in a variety
of different contexts, for example, in the classification of finite
root systems, semi-simple Lie algebras or representation finite 
hereditary algebras. There are various generalizations of 
Dynkin diagrams, for instance, affine Dynkin diagrams which appear,
among other places, in the McKay correspondence and in spectral 
theory of finite graphs. Infinite generalization of 
Dynkin diagrams also appear in various contexts. 
For example, the following list of infinite 
Dynkin diagrams appears in \cite{HPR,HPR2}:

\resizebox{6cm}{!}{
$
\xymatrix@R=4mm{
A_\infty:&&\bullet\ar@{-}[r]&\bullet\ar@{-}[r]&
\bullet\ar@{-}[r]&\dots
}\qquad\qquad\qquad
$
}
\resizebox{6cm}{!}{
$
\xymatrix@R=4mm{
A_\infty^\infty:&&\dots\ar@{-}[r]&\bullet\ar@{-}[r]&\bullet\ar@{-}[r]&
\bullet\ar@{-}[r]&\dots
}\qquad\qquad\qquad
$
}

\resizebox{6cm}{!}{
$
\xymatrix@R=4mm{
B_\infty:&&\bullet\ar@/_2mm/@{-}[r]&\bullet\ar@{-}[r]\ar@/_2mm/[l]&
\bullet\ar@{-}[r]&\dots
}\qquad\qquad\qquad
$
}
\resizebox{6cm}{!}{
$
\xymatrix@R=4mm{
C_\infty:&&\bullet\ar@/_2mm/@{-}[r]\ar@/^2mm/[r]&\bullet\ar@{-}[r]&
\bullet\ar@{-}[r]&\dots
}\qquad\qquad\qquad
$
}

\resizebox{6cm}{!}{
$
\xymatrix@R=4mm{
D_\infty:&&\bullet\ar@{-}[r]&\bullet\ar@{-}[r]&
\bullet\ar@{-}[r]&\dots\\
&&&\bullet\ar@{-}[u]&
}\qquad\qquad\qquad
$
}
\resizebox{6cm}{!}{
$
\xymatrix@R=4mm{
T_\infty:&&\bullet\ar@{-}@(ul,dl)[]\ar@{-}[r]&\bullet\ar@{-}[r]&
\bullet\ar@{-}[r]&\dots
}\qquad\qquad\qquad
$
}

In our previous paper \cite{MZ}, we investigated combinatorics of 
transitive module categories over the monoidal category of finite dimensional 
$\mathfrak{sl}_2$-modules which arise when acting by the latter
monoidal category on arbitrary simple $\mathfrak{sl}_2$-modules.
It turns out that this combinatorics is always described
by an infinite Dynkin diagram. One could loosely say that infinite Dynkin 
diagrams capture $\mathfrak{sl}_2$-symmetries in the Lie algebraic context.

\subsection{$\mathfrak{sl}_3$-generalization}\label{s1.2}

In the present paper, we study $\mathfrak{sl}_3$-generalizations of the
results of \cite{MZ} mentioned above. We consider the monoidal category
$\mathscr{C}$ of finite dimensional $\mathfrak{sl}_3$-modules and look
at the $\mathscr{C}$-module categories of the form
$\mathrm{add}(\mathscr{C}\cdot L)$, where $L$ is a simple
$\mathfrak{sl}_3$-module. A major part of the paper is dedicated to the
case when $L$ is a simple highest weight module or a Whittaker module.
However, our main result, Theorem~\ref{thm-main}, deals with the case
of an arbitrary simple $L$.

Given a ``nice'' $\mathscr{C}$-module category $\mathbf{M}$,
for example, an additive and Krull-Schmidt category on which 
$\mathscr{C}$ acts via additive functors, the combinatorial
shadow of the action of an object of $\mathscr{C}$ can be
collected in a non-negative integer matrix which bookkeeps
the multiplicities of the action in the natural basis of
the split Grothendieck group consisting of all indecomposable objects.
The category $\mathscr{C}$ is generated, as a monoidal category,
by one object: the natural $\mathfrak{sl}_3$-module. 
Therefore, the underlying combinatorics of the action is,
to some extent, determined by one such matrix or, equivalently, 
by a certain oriented graph, call it  $\Gamma$.  

The $\mathscr{C}$-module category $\mathrm{add}(\mathscr{C}\cdot L)$ 
admits a ``composition series'' with ``simple'' subquotients. 
Combinatorially, this simplicity means that the corresponding graph 
$\Gamma$ is strongly connected. In this paper we determine
all possible $\Gamma$ which can appear, for arbitrary $L$.
It turns out that the list is finite (however, all graphs in this
list are infinite). In fact, the list consists
of exactly eight graphs, see Figure~\ref{fig16}.
It is natural to view these eight graphs as 
$\mathfrak{sl}_3$-generalizations of the infinite Dynkin diagrams
(if we interpret the latter as $\mathfrak{sl}_2$-related objects).

\subsection{Content of the paper}\label{s1.3}

Our main result, Theorem~\ref{thm-main}, asserts exactly that
the combinatorics of a ``simple'' subquotient of a 
$\mathscr{C}$-module category of the form 
$\mathrm{add}(\mathscr{C}\cdot L)$ is governed by one of
the eight graphs listed in Figure~\ref{fig16}.

However, a major part of the paper is devoted to determination of
these eight graphs, for very particular simple modules $L$.
Seven of these eight graphs come from the case when $L$
is a simple highest weight module, that is, an object of the 
BGG category $\mathcal{O}$ associated to $\mathfrak{sl}_3$, 
cf. \cite{BGG}. The remaining graph appears in the case when $L$
is a non-degenerate Whittaker module, cf. \cite{Ko}. After this,
fairly technical part, it is not too difficult to show, using
the standard techniques of Harish-Chandra bimodules, that 
all other cases of $L$ are similar to either category
$\mathcal{O}$ or to  Whittaker modules (in fact, there is even
an equivalence of the corresponding $\mathscr{C}$-module categories,
generalizing the approach of \cite{MiSo}).

For all graphs, we determine the Perron-Frobenius eigenvector
for the Perron-Frobenius eigenvalue $3$ of the corresponding
matrix and discuss the interpretation of the coefficients of
this eigenvector in Lie-theoretic context. The coefficients
are usually naturally related to either the dimension or
the Gelfand-Kirillov dimension or the Bernstein coefficient of
the modules involved.

\subsection{Structure of the paper}\label{s1.4}

Section~\ref{s2} introduces all necessary preliminaries and notation
related to the Lie algebra $\mathfrak{sl}_3$. Weight pictures play
an important role in helping to understand the combinatorics we want to
describe. Therefore, in Section~\ref{s2}, we, in particular, provide
pictorial representations of all our main ingredients.

In Section~\ref{s3} we collected all necessary preliminaries 
related to the monoidal category $\mathscr{C}$ and module categories
over it. We also comment on the relation of our problems to 
the classical Perron-Frobenius theory.

Sections~\ref{s4}, \ref{s5} and \ref{s6} contain the special cases
preparation work for the main theorem. Sections~\ref{s4} looks at the
regular $\mathscr{C}$-module category. This is a very special case
in which one can derive many more properties compared to the general
case. In particular, here we prove Theorem~\ref{thm-s4.3-1} which asserts
that a simple $\mathscr{C}$-module category with the same combinatorics
as the regular $\mathscr{C}$-module category is, in fact, equivalent to
the regular $\mathscr{C}$-module category. Section~\ref{s5} studies
combinatorics of 
$\mathscr{C}$-module categories in the context of BGG category
$\mathcal{O}$. Section~\ref{s6} studies combinatorics of 
$\mathscr{C}$-module categories in the context of 
Whittaker modules.

Our main result, Theorem~\ref{thm-main}, is formulated and
proved in Section~\ref{s7}.

\subsection*{Acknowledgments} The first author is partially
supported by the Swedish Research Council.
The second author is partially supported by the
Zhejiang Provincial Natural Science Foundation of China, 
Grant No. QN25A010010.

\section{The Lie algebra $\mathfrak{sl}_3$}\label{s2}

\subsection{Setup}\label{s2.1}

In this paper, we work over the field $\mathbb{C}$ of complex numbers.
Consider the Lie algebra $\mathfrak{g}:=\mathfrak{sl}_3=\mathfrak{sl}_3(\mathbb{C})$
of all traceless $3\times 3$ complex matrices. It has
the standard triangular decomposition
\begin{displaymath}
\mathfrak{sl}_3=\mathfrak{n}_-\oplus \mathfrak{h}\oplus \mathfrak{n}_+. 
\end{displaymath}
Here $\mathfrak{h}$ is the Cartan subalgebra of all traceless diagonal
matrices, $\mathfrak{n}_+$ is the subalgebra of all strictly upper triangular
matrices and $\mathfrak{n}_-$ is the subalgebra of all strictly lower 
triangular matrices.

For $i,j\in\{1,2,3\}$, we denote by $e_{i,j}$ the corresponding matrix unit.
Then $\{e_{1,2},e_{1,3},e_{2,3}\}$ is a basis in $\mathfrak{n}_+$,
$\{e_{2,1},e_{3,1},e_{3,2}\}$ is a basis in $\mathfrak{n}_-$,
and $\{h_1,h_2\}$, where $h_1:=e_{1,1}-e_{2,2}$ and $h_2:=e_{2,2}-e_{3,3}$,
is a basis in $\mathfrak{h}$.

For a Lie algebra $\mathfrak{a}$, we denote by $U(\mathfrak{a})$
the universal enveloping algebra of $\mathfrak{a}$.

\subsection{Weights and roots}\label{s2.2}

Consider the dual space $\mathfrak{h}^*:=
\mathrm{Hom}_\mathbb{C}(\mathfrak{h},\mathbb{C})$. As usual, we will call the
elements of $\mathfrak{h}^*$ {\em weights}. An element
$\lambda\in \mathfrak{h}^*$ will be identified with
$(\lambda(h_1),\lambda(h_2))\in\mathbb{C}^2$.
We set $\lambda_1:=\lambda(h_1)$ and $\lambda_2:=\lambda(h_2)$.

Let $\mathbf{R}\subset \mathfrak{h}^*$ be the root system of
$\mathfrak{sl}_3$. Then the above triangular decomposition 
induces a choice of a basis  $\pi$ in $\mathbf{R}$ and the corresponding
decomposition of $\mathbf{R}$ into positive and negative roots:
$\mathbf{R}=\mathbf{R}_+\coprod\mathbf{R}_-$. We denote by
$\alpha$ the root corresponding to $e_{12}$ and by
$\beta$ the root corresponding to $e_{23}$. Then 
$\pi=\{\alpha,\beta\}$, $\mathbf{R}_+=\{\alpha,\beta,\alpha+\beta\}$
and $\mathbf{R}_-=\{-\alpha,-\beta,-\alpha-\beta\}$.

We denote by $\Theta$ the set $\mathbb{Z}\pi$, and by
$\Lambda$ the set of all {\em integral weights}, that is, 
the set of all $\lambda\in \mathfrak{h}^*$ such that 
both $\lambda(h_1)$ and $\lambda(h_2)$ are integers.
Note that $\Theta$ is a subgroup of $\Lambda$ and
$\Lambda/\Theta$ is a cyclic group of order $3$.
The sets $\Lambda$, $\Theta$ and $\mathbf{R}$ are
major players in our story. They are visualized 
in Figure~\ref{fig1}.

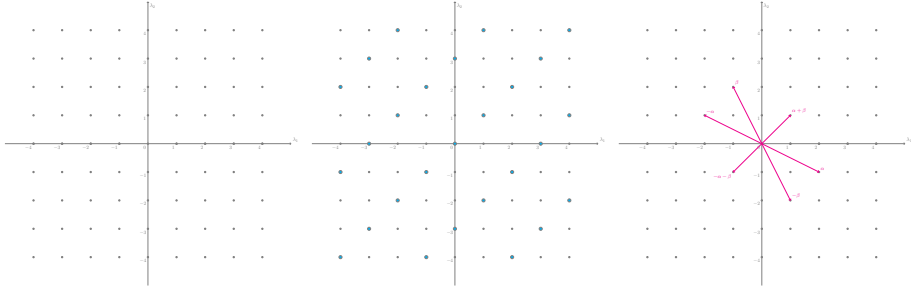
\begin{figure}
\resizebox{12cm}{!}{
\begin{tikzpicture}
\draw[gray, thin,  ->] (0,10) -- (20,10) node[anchor=south west] {\large$\lambda_1$};
\draw[gray, thin,  ->] (10,00) -- (10,20) node[anchor=north west] {\large$\lambda_2$};
\draw[gray,fill=gray] (10,10) circle (.3ex) node[anchor=north east] {\color{gray}$0$};
\draw[gray,fill=gray] (12,10) circle (.3ex) node[anchor=north east] {\color{gray}$1$};
\draw[gray,fill=gray] (14,10) circle (.3ex) node[anchor=north east] {\color{gray}$2$};
\draw[gray,fill=gray] (16,10) circle (.3ex) node[anchor=north east] {\color{gray}$3$};
\draw[gray,fill=gray] (18,10) circle (.3ex) node[anchor=north east] {\color{gray}$4$};
\draw[gray,fill=gray] (8,10) circle (.3ex) node[anchor=north east] {\color{gray}$-1$};
\draw[gray,fill=gray] (6,10) circle (.3ex) node[anchor=north east] {\color{gray}$-2$};
\draw[gray,fill=gray] (4,10) circle (.3ex) node[anchor=north east] {\color{gray}$-3$};
\draw[gray,fill=gray] (2,10) circle (.3ex) node[anchor=north east] {\color{gray}$-4$};
\draw[gray,fill=gray] (10,12) circle (.3ex) node[anchor=north east] {\color{gray}$1$};
\draw[gray,fill=gray] (10,14) circle (.3ex) node[anchor=north east] {\color{gray}$2$};
\draw[gray,fill=gray] (10,16) circle (.3ex) node[anchor=north east] {\color{gray}$3$};
\draw[gray,fill=gray] (10,18) circle (.3ex) node[anchor=north east] {\color{gray}$4$};
\draw[gray,fill=gray] (10,8) circle (.3ex) node[anchor=north east] {\color{gray}$-1$};
\draw[gray,fill=gray] (10,6) circle (.3ex) node[anchor=north east] {\color{gray}$-2$};
\draw[gray,fill=gray] (10,4) circle (.3ex) node[anchor=north east] {\color{gray}$-3$};
\draw[gray,fill=gray] (10,2) circle (.3ex) node[anchor=north east] {\color{gray}$-4$};
\draw[gray,fill=gray] (2,2) circle (.3ex);
\draw[gray,fill=gray] (2,4) circle (.3ex);
\draw[gray,fill=gray] (2,6) circle (.3ex);
\draw[gray,fill=gray] (2,8) circle (.3ex);
\draw[gray,fill=gray] (2,10) circle (.3ex);
\draw[gray,fill=gray] (2,12) circle (.3ex);
\draw[gray,fill=gray] (2,14) circle (.3ex);
\draw[gray,fill=gray] (2,16) circle (.3ex);
\draw[gray,fill=gray] (2,18) circle (.3ex);
\draw[gray,fill=gray] (4,2) circle (.3ex);
\draw[gray,fill=gray] (4,4) circle (.3ex);
\draw[gray,fill=gray] (4,6) circle (.3ex);
\draw[gray,fill=gray] (4,8) circle (.3ex);
\draw[gray,fill=gray] (4,10) circle (.3ex);
\draw[gray,fill=gray] (4,12) circle (.3ex);
\draw[gray,fill=gray] (4,14) circle (.3ex);
\draw[gray,fill=gray] (4,16) circle (.3ex);
\draw[gray,fill=gray] (4,18) circle (.3ex);
\draw[gray,fill=gray] (6,2) circle (.3ex);
\draw[gray,fill=gray] (6,4) circle (.3ex);
\draw[gray,fill=gray] (6,6) circle (.3ex);
\draw[gray,fill=gray] (6,8) circle (.3ex);
\draw[gray,fill=gray] (6,10) circle (.3ex);
\draw[gray,fill=gray] (6,12) circle (.3ex);
\draw[gray,fill=gray] (6,14) circle (.3ex);
\draw[gray,fill=gray] (6,16) circle (.3ex);
\draw[gray,fill=gray] (6,18) circle (.3ex);
\draw[gray,fill=gray] (8,2) circle (.3ex);
\draw[gray,fill=gray] (8,4) circle (.3ex);
\draw[gray,fill=gray] (8,6) circle (.3ex);
\draw[gray,fill=gray] (8,8) circle (.3ex);
\draw[gray,fill=gray] (8,10) circle (.3ex);
\draw[gray,fill=gray] (8,12) circle (.3ex);
\draw[gray,fill=gray] (8,14) circle (.3ex);
\draw[gray,fill=gray] (8,16) circle (.3ex);
\draw[gray,fill=gray] (8,18) circle (.3ex);
\draw[gray,fill=gray] (12,2) circle (.3ex);
\draw[gray,fill=gray] (12,4) circle (.3ex);
\draw[gray,fill=gray] (12,6) circle (.3ex);
\draw[gray,fill=gray] (12,8) circle (.3ex);
\draw[gray,fill=gray] (12,10) circle (.3ex);
\draw[gray,fill=gray] (12,12) circle (.3ex);
\draw[gray,fill=gray] (12,14) circle (.3ex);
\draw[gray,fill=gray] (12,16) circle (.3ex);
\draw[gray,fill=gray] (12,18) circle (.3ex);
\draw[gray,fill=gray] (14,2) circle (.3ex);
\draw[gray,fill=gray] (14,4) circle (.3ex);
\draw[gray,fill=gray] (14,6) circle (.3ex);
\draw[gray,fill=gray] (14,8) circle (.3ex);
\draw[gray,fill=gray] (14,10) circle (.3ex);
\draw[gray,fill=gray] (14,12) circle (.3ex);
\draw[gray,fill=gray] (14,14) circle (.3ex);
\draw[gray,fill=gray] (14,16) circle (.3ex);
\draw[gray,fill=gray] (14,18) circle (.3ex);
\draw[gray,fill=gray] (16,2) circle (.3ex);
\draw[gray,fill=gray] (16,4) circle (.3ex);
\draw[gray,fill=gray] (16,6) circle (.3ex);
\draw[gray,fill=gray] (16,8) circle (.3ex);
\draw[gray,fill=gray] (16,10) circle (.3ex);
\draw[gray,fill=gray] (16,12) circle (.3ex);
\draw[gray,fill=gray] (16,14) circle (.3ex);
\draw[gray,fill=gray] (16,16) circle (.3ex);
\draw[gray,fill=gray] (16,18) circle (.3ex);
\draw[gray,fill=gray] (18,2) circle (.3ex);
\draw[gray,fill=gray] (18,4) circle (.3ex);
\draw[gray,fill=gray] (18,6) circle (.3ex);
\draw[gray,fill=gray] (18,8) circle (.3ex);
\draw[gray,fill=gray] (18,10) circle (.3ex);
\draw[gray,fill=gray] (18,12) circle (.3ex);
\draw[gray,fill=gray] (18,14) circle (.3ex);
\draw[gray,fill=gray] (18,16) circle (.3ex);
\draw[gray,fill=gray] (18,18) circle (.3ex);
\end{tikzpicture}
\qquad
\begin{tikzpicture}
\draw[gray, thin,  ->] (0,10) -- (20,10) node[anchor=south west] {\large$\lambda_1$};
\draw[gray, thin,  ->] (10,00) -- (10,20) node[anchor=north west] {\large$\lambda_2$};
\draw[gray,fill=cyan] (10,10) circle (.7ex) node[anchor=north east] {\color{gray}$0$};
\draw[gray,fill=gray] (12,10) circle (.3ex) node[anchor=north east] {\color{gray}$1$};
\draw[gray,fill=gray] (14,10) circle (.3ex) node[anchor=north east] {\color{gray}$2$};
\draw[gray,fill=cyan] (16,10) circle (.7ex) node[anchor=north east] {\color{gray}$3$};
\draw[gray,fill=gray] (18,10) circle (.3ex) node[anchor=north east] {\color{gray}$4$};
\draw[gray,fill=gray] (8,10) circle (.3ex) node[anchor=north east] {\color{gray}$-1$};
\draw[gray,fill=gray] (6,10) circle (.3ex) node[anchor=north east] {\color{gray}$-2$};
\draw[gray,fill=cyan] (4,10) circle (.7ex) node[anchor=north east] {\color{gray}$-3$};
\draw[gray,fill=gray] (2,10) circle (.3ex) node[anchor=north east] {\color{gray}$-4$};
\draw[gray,fill=gray] (10,12) circle (.3ex) node[anchor=north east] {\color{gray}$1$};
\draw[gray,fill=gray] (10,14) circle (.3ex) node[anchor=north east] {\color{gray}$2$};
\draw[gray,fill=cyan] (10,16) circle (.7ex) node[anchor=north east] {\color{gray}$3$};
\draw[gray,fill=gray] (10,18) circle (.3ex) node[anchor=north east] {\color{gray}$4$};
\draw[gray,fill=gray] (10,8) circle (.3ex) node[anchor=north east] {\color{gray}$-1$};
\draw[gray,fill=gray] (10,6) circle (.3ex) node[anchor=north east] {\color{gray}$-2$};
\draw[gray,fill=cyan] (10,4) circle (.7ex) node[anchor=north east] {\color{gray}$-3$};
\draw[gray,fill=gray] (10,2) circle (.3ex) node[anchor=north east] {\color{gray}$-4$};
\draw[gray,fill=cyan] (2,2) circle (.7ex);
\draw[gray,fill=gray] (2,4) circle (.3ex);
\draw[gray,fill=gray] (2,6) circle (.3ex);
\draw[gray,fill=cyan] (2,8) circle (.7ex);
\draw[gray,fill=gray] (2,12) circle (.3ex);
\draw[gray,fill=cyan] (2,14) circle (.7ex);
\draw[gray,fill=gray] (2,16) circle (.3ex);
\draw[gray,fill=gray] (2,18) circle (.3ex);
\draw[gray,fill=gray] (4,2) circle (.3ex);
\draw[gray,fill=cyan] (4,4) circle (.7ex);
\draw[gray,fill=gray] (4,6) circle (.3ex);
\draw[gray,fill=gray] (4,8) circle (.3ex);
\draw[gray,fill=gray] (4,12) circle (.3ex);
\draw[gray,fill=gray] (4,14) circle (.3ex);
\draw[gray,fill=cyan] (4,16) circle (.7ex);
\draw[gray,fill=gray] (4,18) circle (.3ex);
\draw[gray,fill=gray] (6,2) circle (.3ex);
\draw[gray,fill=gray] (6,4) circle (.3ex);
\draw[gray,fill=cyan] (6,6) circle (.7ex);
\draw[gray,fill=gray] (6,8) circle (.3ex);
\draw[gray,fill=cyan] (6,12) circle (.7ex);
\draw[gray,fill=gray] (6,14) circle (.3ex);
\draw[gray,fill=gray] (6,16) circle (.3ex);
\draw[gray,fill=cyan] (6,18) circle (.7ex);
\draw[gray,fill=cyan] (8,2) circle (.7ex);
\draw[gray,fill=gray] (8,4) circle (.3ex);
\draw[gray,fill=gray] (8,6) circle (.3ex);
\draw[gray,fill=cyan] (8,8) circle (.7ex);
\draw[gray,fill=gray] (8,12) circle (.3ex);
\draw[gray,fill=cyan] (8,14) circle (.7ex);
\draw[gray,fill=gray] (8,16) circle (.3ex);
\draw[gray,fill=gray] (8,18) circle (.3ex);
\draw[gray,fill=gray] (12,2) circle (.3ex);
\draw[gray,fill=gray] (12,4) circle (.3ex);
\draw[gray,fill=cyan] (12,6) circle (.7ex);
\draw[gray,fill=gray] (12,8) circle (.3ex);
\draw[gray,fill=cyan] (12,12) circle (.7ex);
\draw[gray,fill=gray] (12,14) circle (.3ex);
\draw[gray,fill=gray] (12,16) circle (.3ex);
\draw[gray,fill=cyan] (12,18) circle (.7ex);
\draw[gray,fill=cyan] (14,2) circle (.7ex);
\draw[gray,fill=gray] (14,4) circle (.3ex);
\draw[gray,fill=gray] (14,6) circle (.3ex);
\draw[gray,fill=cyan] (14,8) circle (.7ex);
\draw[gray,fill=gray] (14,12) circle (.3ex);
\draw[gray,fill=cyan] (14,14) circle (.7ex);
\draw[gray,fill=gray] (14,16) circle (.3ex);
\draw[gray,fill=gray] (14,18) circle (.3ex);
\draw[gray,fill=gray] (16,2) circle (.3ex);
\draw[gray,fill=cyan] (16,4) circle (.7ex);
\draw[gray,fill=gray] (16,6) circle (.3ex);
\draw[gray,fill=gray] (16,8) circle (.3ex);
\draw[gray,fill=gray] (16,12) circle (.3ex);
\draw[gray,fill=gray] (16,14) circle (.3ex);
\draw[gray,fill=cyan] (16,16) circle (.7ex);
\draw[gray,fill=gray] (16,18) circle (.3ex);
\draw[gray,fill=gray] (18,2) circle (.3ex);
\draw[gray,fill=gray] (18,4) circle (.3ex);
\draw[gray,fill=cyan] (18,6) circle (.7ex);
\draw[gray,fill=gray] (18,8) circle (.3ex);
\draw[gray,fill=cyan] (18,12) circle (.7ex);
\draw[gray,fill=gray] (18,14) circle (.3ex);
\draw[gray,fill=gray] (18,16) circle (.3ex);
\draw[gray,fill=cyan] (18,18) circle (.7ex);
\end{tikzpicture}
\qquad
\begin{tikzpicture}
\draw[gray, thin,  ->] (0,10) -- (20,10) node[anchor=south west] {\large$\lambda_1$};
\draw[gray, thin,  ->] (10,00) -- (10,20) node[anchor=north west] {\large$\lambda_2$};
\draw[gray,fill=gray] (10,10) circle (.3ex) node[anchor=north east] {\color{gray}$0$};
\draw[gray,fill=gray] (12,10) circle (.3ex) node[anchor=north east] {\color{gray}$1$};
\draw[gray,fill=gray] (14,10) circle (.3ex) node[anchor=north east] {\color{gray}$2$};
\draw[gray,fill=gray] (16,10) circle (.3ex) node[anchor=north east] {\color{gray}$3$};
\draw[gray,fill=gray] (18,10) circle (.3ex) node[anchor=north east] {\color{gray}$4$};
\draw[gray,fill=gray] (8,10) circle (.3ex) node[anchor=north east] {\color{gray}$-1$};
\draw[gray,fill=gray] (6,10) circle (.3ex) node[anchor=north east] {\color{gray}$-2$};
\draw[gray,fill=gray] (4,10) circle (.3ex) node[anchor=north east] {\color{gray}$-3$};
\draw[gray,fill=gray] (2,10) circle (.3ex) node[anchor=north east] {\color{gray}$-4$};
\draw[gray,fill=gray] (10,12) circle (.3ex) node[anchor=north east] {\color{gray}$1$};
\draw[gray,fill=gray] (10,14) circle (.3ex) node[anchor=north east] {\color{gray}$2$};
\draw[gray,fill=gray] (10,16) circle (.3ex) node[anchor=north east] {\color{gray}$3$};
\draw[gray,fill=gray] (10,18) circle (.3ex) node[anchor=north east] {\color{gray}$4$};
\draw[gray,fill=gray] (10,8) circle (.3ex) node[anchor=north east] {\color{gray}$-1$};
\draw[gray,fill=gray] (10,6) circle (.3ex) node[anchor=north east] {\color{gray}$-2$};
\draw[gray,fill=gray] (10,4) circle (.3ex) node[anchor=north east] {\color{gray}$-3$};
\draw[gray,fill=gray] (10,2) circle (.3ex) node[anchor=north east] {\color{gray}$-4$};
\draw[gray,fill=gray] (2,2) circle (.3ex);
\draw[gray,fill=gray] (2,4) circle (.3ex);
\draw[gray,fill=gray] (2,6) circle (.3ex);
\draw[gray,fill=gray] (2,8) circle (.3ex);
\draw[gray,fill=gray] (2,10) circle (.3ex);
\draw[gray,fill=gray] (2,12) circle (.3ex);
\draw[gray,fill=gray] (2,14) circle (.3ex);
\draw[gray,fill=gray] (2,16) circle (.3ex);
\draw[gray,fill=gray] (2,18) circle (.3ex);
\draw[gray,fill=gray] (4,2) circle (.3ex);
\draw[gray,fill=gray] (4,4) circle (.3ex);
\draw[gray,fill=gray] (4,6) circle (.3ex);
\draw[gray,fill=gray] (4,8) circle (.3ex);
\draw[gray,fill=gray] (4,10) circle (.3ex);
\draw[gray,fill=gray] (4,12) circle (.3ex);
\draw[gray,fill=gray] (4,14) circle (.3ex);
\draw[gray,fill=gray] (4,16) circle (.3ex);
\draw[gray,fill=gray] (4,18) circle (.3ex);
\draw[gray,fill=gray] (6,2) circle (.3ex);
\draw[gray,fill=gray] (6,4) circle (.3ex);
\draw[gray,fill=gray] (6,6) circle (.3ex);
\draw[gray,fill=gray] (6,8) circle (.3ex);
\draw[gray,fill=gray] (6,10) circle (.3ex);
\draw[gray,fill=gray] (6,12) circle (.3ex);
\draw[gray,fill=gray] (6,14) circle (.3ex);
\draw[gray,fill=gray] (6,16) circle (.3ex);
\draw[gray,fill=gray] (6,18) circle (.3ex);
\draw[gray,fill=gray] (8,2) circle (.3ex);
\draw[gray,fill=gray] (8,4) circle (.3ex);
\draw[gray,fill=gray] (8,6) circle (.3ex);
\draw[gray,fill=gray] (8,8) circle (.3ex);
\draw[gray,fill=gray] (8,10) circle (.3ex);
\draw[gray,fill=gray] (8,12) circle (.3ex);
\draw[gray,fill=gray] (8,14) circle (.3ex);
\draw[gray,fill=gray] (8,16) circle (.3ex);
\draw[gray,fill=gray] (8,18) circle (.3ex);
\draw[gray,fill=gray] (12,2) circle (.3ex);
\draw[gray,fill=gray] (12,4) circle (.3ex);
\draw[gray,fill=gray] (12,6) circle (.3ex);
\draw[gray,fill=gray] (12,8) circle (.3ex);
\draw[gray,fill=gray] (12,10) circle (.3ex);
\draw[gray,fill=gray] (12,12) circle (.3ex);
\draw[gray,fill=gray] (12,14) circle (.3ex);
\draw[gray,fill=gray] (12,16) circle (.3ex);
\draw[gray,fill=gray] (12,18) circle (.3ex);
\draw[gray,fill=gray] (14,2) circle (.3ex);
\draw[gray,fill=gray] (14,4) circle (.3ex);
\draw[gray,fill=gray] (14,6) circle (.3ex);
\draw[gray,fill=gray] (14,8) circle (.3ex);
\draw[gray,fill=gray] (14,10) circle (.3ex);
\draw[gray,fill=gray] (14,12) circle (.3ex);
\draw[gray,fill=gray] (14,14) circle (.3ex);
\draw[gray,fill=gray] (14,16) circle (.3ex);
\draw[gray,fill=gray] (14,18) circle (.3ex);
\draw[gray,fill=gray] (16,2) circle (.3ex);
\draw[gray,fill=gray] (16,4) circle (.3ex);
\draw[gray,fill=gray] (16,6) circle (.3ex);
\draw[gray,fill=gray] (16,8) circle (.3ex);
\draw[gray,fill=gray] (16,10) circle (.3ex);
\draw[gray,fill=gray] (16,12) circle (.3ex);
\draw[gray,fill=gray] (16,14) circle (.3ex);
\draw[gray,fill=gray] (16,16) circle (.3ex);
\draw[gray,fill=gray] (16,18) circle (.3ex);
\draw[gray,fill=gray] (18,2) circle (.3ex);
\draw[gray,fill=gray] (18,4) circle (.3ex);
\draw[gray,fill=gray] (18,6) circle (.3ex);
\draw[gray,fill=gray] (18,8) circle (.3ex);
\draw[gray,fill=gray] (18,10) circle (.3ex);
\draw[gray,fill=gray] (18,12) circle (.3ex);
\draw[gray,fill=gray] (18,14) circle (.3ex);
\draw[gray,fill=gray] (18,16) circle (.3ex);
\draw[gray,fill=gray] (18,18) circle (.3ex);
\draw[gray, thick, magenta,  ->] (10,10) -- (14,8) node[anchor=south west] {\large\color{magenta}$\alpha$};
\draw[gray, thick, magenta,  ->] (10,10) -- (8,14) node[anchor=south west] {\large\color{magenta}$\beta$};
\draw[gray, thick, magenta,  ->] (10,10) -- (12,12) node[anchor=south west] {\large\color{magenta}$\alpha+\beta$};
\draw[gray, thick, magenta,  ->] (10,10) -- (6,12) node[anchor=south west] {\large\color{magenta}$-\alpha$};
\draw[gray, thick, magenta,  ->] (10,10) -- (12,6) node[anchor=south west] {\large\color{magenta}$-\beta$};
\draw[gray, thick, magenta,  ->] (10,10) -- (8,8) node[anchor=north east] {\large\color{magenta}$-\alpha-\beta$};
\end{tikzpicture}
}
\caption{The sets {\color{gray}$\Lambda$}, {\color{cyan}$\Theta$} 
and {\color{magenta}$\mathbf{R}$}}\label{fig1}
\end{figure}

\subsection{Highest weight modules}\label{s2.3}

For each $\lambda\in\mathfrak{h}^*$, we have the 
corresponding {\em Verma module} $\Delta(\lambda)$
with highest weight $\lambda$, see
\cite{Di,Hu}. We denote by $L(\lambda)$ the unique simple top 
of $\Delta(\lambda)$. We refer to \cite{Di,Hu} for all 
details on  standard properties of these modules.

Let $Z(\mathfrak{g})$ be the center of $U(\mathfrak{g})$.
Then each element $z\in Z(\mathfrak{g})$ acts on $\Delta(\lambda)$
(and hence also on $L(\lambda)$)
as a scalar, which we denote by $\chi_{{}_\lambda}(z)$. The map
$\chi_{{}_\lambda}:Z(\mathfrak{g})\to \mathbb{C}$ is an algebra
homomorphism and is called the {\em central character} 
(of $\Delta(\lambda)$). Let $W$ be the Weyl group of 
the pair $(\mathfrak{g},\mathfrak{h})$. Set
\begin{displaymath}
\rho:=\alpha+\beta=\frac{1}{2}\sum_{\gamma\in\mathbf{R}_+}\gamma.
\end{displaymath}
Then we can define the {\em dot-action} of $W$ on 
$\mathfrak{h}^*$ from the usual action of $W$
on $\mathfrak{h}^*$ as follows:
$w\cdot \lambda:=w(\lambda+\rho)-\rho$.
Now, for $\lambda,\mu\in\mathfrak{h}^*$, we have 
$\chi_{{}_\lambda}=\chi_{{}_\mu}$ if and only if 
$\lambda\in W\cdot \mu$.

Note that $W$ is isomorphic to the symmetric group $S_3$.
If we denote by $s$ the simple reflection with respect to $\alpha$
and by $r$ the simple reflection with respect to $\beta$,
then $W=\{e,r,s,rs,sr,w_0\}$, where $w_0=srs=rsr$.

\subsection{Finite dimensional $\mathfrak{sl}_3$-modules}\label{s2.4}

Our main protagonist in this paper is the semi-simple mo\-no\-i\-dal category
$\mathscr{C}$ of all finite dimensional $\mathfrak{sl}_3$-modules.
The monoidal structure is given by tensoring over $\mathbb{C}$
(which has the natural structure of an $\mathfrak{sl}_3$-module
using the standard comultiplication on $U(\mathfrak{sl}_3)$
which sends $g\in \mathfrak{sl}_3$ to $g\otimes\mathrm{id}+\mathrm{id}\otimes g$).
The monoidal category $\mathscr{C}$ is rigid and symmetric.
The dual of an object $V$ is the object 
$V^*=\mathrm{Hom}_\mathbb{C}(V,\mathbb{C})$, on which the 
action of $\mathfrak{sl}_3$ is defined using the canonical 
anti-involution $g\mapsto -g$, for $g\in \mathfrak{sl}_3$.

The simple objects in $\mathscr{C}$, up to isomorphism, are exactly
the modules $L(\lambda)$, where $\lambda\in\mathfrak{h}^*$ is such
that both $\lambda_1$ and $\lambda_2$ are non-negative integers.
As an idempotent split additive 
monoidal category, the category $\mathscr{C}$ is generated
by the object $L((1,0))$. The latter object is the {\em natural}
$3$-dimensional $\mathfrak{sl}_3$-module $\mathbb{C}^3$ on which
$\mathfrak{sl}_3$ acts by matrix multiplication. To simplify
notation, we denote the object $L((1,0))$ by $\mathrm{F}$
and the object $L((0,1))$ by $\mathrm{G}$. Note that
$\mathrm{F}$ and $\mathrm{G}$ are biadjoint.

\subsection{Weights of finite dimensional $\mathfrak{sl}_3$-modules}\label{s2.5}

For a $\mathfrak{g}$-module $M$ and $\lambda\in\mathfrak{h}^*$,
consider the space
\begin{displaymath}
M_\lambda=\{v\in M\,:\, h(v)=\lambda(h)v,\,\,\text{ for all }h\in\mathfrak{h}\}. 
\end{displaymath}
The space $M_\lambda$ is usually called the {\em weight space}
of $M$ corresponding to the weight $\lambda$. 
A $\mathfrak{g}$-module $M$ is called a {\em weight module} provided
that it is the direct sum of its weight spaces. For a weight
$\mathfrak{g}$-module $M$, the set of all $\lambda\in\mathfrak{h}^*$
such that $M_\lambda\neq 0$ is called the {\em support} of $M$
and denoted $\mathrm{supp}(M)$.

All finite dimensional $\mathfrak{g}$-modules are weight modules. Moreover,
the support of each finite dimensional $\mathfrak{g}$-module is
$W$-invariant. Furthermore, the set $\Lambda$ coincides with the 
union of the supports of all finite dimensional $\mathfrak{g}$-module.

The algebra $U(\mathfrak{g})$ is a $\mathfrak{g}$-module with respect to the
adjoint action. Moreover, the latter action is locally finite.
Furthermore, the set $\Theta$ coincides with the 
union of the supports of all finite dimensional $\mathfrak{g}$-modules
which appear as summands of $U(\mathfrak{g})$, considered as the
adjoint $\mathfrak{g}$-module. Note that, a simple finite dimensional
$\mathfrak{g}$-module $L(\lambda)$ appears as a summand of $U(\mathfrak{g})$, 
considered as an adjoint $\mathfrak{g}$-module, if and only if
$L(\lambda)_0\neq 0$.

In Figure~\ref{fig2}, we depicted the supports of the simple 
$\mathfrak{g}$-modules $L((1,0))$, its dual $L((0,1))=L((1,0))^*$
and the adjoint module $\mathfrak{g}=L((1,1))$. Note that, for
the latter module, the weight space corresponding to $0$
has dimension $2$.

\begin{figure}
\resizebox{10cm}{!}{
\begin{tikzpicture}
\draw[gray, thin,  ->] (0,10) -- (20,10) node[anchor=south west] {\large$\lambda_1$};
\draw[gray, thin,  ->] (10,00) -- (10,20) node[anchor=north west] {\large$\lambda_2$};
\draw[gray,fill=gray] (10,10) circle (.3ex) node[anchor=north east] {\color{gray}$0$};
\draw[gray,fill=violet] (12,10) circle (1.3ex) node[anchor=north east] {\color{gray}$1$};
\draw[gray,fill=gray] (14,10) circle (.3ex) node[anchor=north east] {\color{gray}$2$};
\draw[gray,fill=gray] (16,10) circle (.3ex) node[anchor=north east] {\color{gray}$3$};
\draw[gray,fill=gray] (18,10) circle (.3ex) node[anchor=north east] {\color{gray}$4$};
\draw[gray,fill=gray] (8,10) circle (.3ex) node[anchor=north east] {\color{gray}$-1$};
\draw[gray,fill=gray] (6,10) circle (.3ex) node[anchor=north east] {\color{gray}$-2$};
\draw[gray,fill=gray] (4,10) circle (.3ex) node[anchor=north east] {\color{gray}$-3$};
\draw[gray,fill=gray] (2,10) circle (.3ex) node[anchor=north east] {\color{gray}$-4$};
\draw[gray,fill=gray] (10,12) circle (.3ex) node[anchor=north east] {\color{gray}$1$};
\draw[gray,fill=gray] (10,14) circle (.3ex) node[anchor=north east] {\color{gray}$2$};
\draw[gray,fill=gray] (10,16) circle (.3ex) node[anchor=north east] {\color{gray}$3$};
\draw[gray,fill=gray] (10,18) circle (.3ex) node[anchor=north east] {\color{gray}$4$};
\draw[gray,fill=violet] (10,8) circle (1.3ex) node[anchor=north east] {\color{gray}$-1$};
\draw[gray,fill=gray] (10,6) circle (.3ex) node[anchor=north east] {\color{gray}$-2$};
\draw[gray,fill=gray] (10,4) circle (.3ex) node[anchor=north east] {\color{gray}$-3$};
\draw[gray,fill=gray] (10,2) circle (.3ex) node[anchor=north east] {\color{gray}$-4$};
\draw[gray,fill=gray] (2,2) circle (.3ex);
\draw[gray,fill=gray] (2,4) circle (.3ex);
\draw[gray,fill=gray] (2,6) circle (.3ex);
\draw[gray,fill=gray] (2,8) circle (.3ex);
\draw[gray,fill=gray] (2,10) circle (.3ex);
\draw[gray,fill=gray] (2,12) circle (.3ex);
\draw[gray,fill=gray] (2,14) circle (.3ex);
\draw[gray,fill=gray] (2,16) circle (.3ex);
\draw[gray,fill=gray] (2,18) circle (.3ex);
\draw[gray,fill=gray] (4,2) circle (.3ex);
\draw[gray,fill=gray] (4,4) circle (.3ex);
\draw[gray,fill=gray] (4,6) circle (.3ex);
\draw[gray,fill=gray] (4,8) circle (.3ex);
\draw[gray,fill=gray] (4,10) circle (.3ex);
\draw[gray,fill=gray] (4,12) circle (.3ex);
\draw[gray,fill=gray] (4,14) circle (.3ex);
\draw[gray,fill=gray] (4,16) circle (.3ex);
\draw[gray,fill=gray] (4,18) circle (.3ex);
\draw[gray,fill=gray] (6,2) circle (.3ex);
\draw[gray,fill=gray] (6,4) circle (.3ex);
\draw[gray,fill=gray] (6,6) circle (.3ex);
\draw[gray,fill=gray] (6,8) circle (.3ex);
\draw[gray,fill=gray] (6,10) circle (.3ex);
\draw[gray,fill=gray] (6,12) circle (.3ex);
\draw[gray,fill=gray] (6,14) circle (.3ex);
\draw[gray,fill=gray] (6,16) circle (.3ex);
\draw[gray,fill=gray] (6,18) circle (.3ex);
\draw[gray,fill=gray] (8,2) circle (.3ex);
\draw[gray,fill=gray] (8,4) circle (.3ex);
\draw[gray,fill=gray] (8,6) circle (.3ex);
\draw[gray,fill=gray] (8,8) circle (.3ex);
\draw[gray,fill=gray] (8,10) circle (.3ex);
\draw[gray,fill=violet] (8,12) circle (1.3ex);
\draw[gray,fill=gray] (8,14) circle (.3ex);
\draw[gray,fill=gray] (8,16) circle (.3ex);
\draw[gray,fill=gray] (8,18) circle (.3ex);
\draw[gray,fill=gray] (12,2) circle (.3ex);
\draw[gray,fill=gray] (12,4) circle (.3ex);
\draw[gray,fill=gray] (12,6) circle (.3ex);
\draw[gray,fill=gray] (12,8) circle (.3ex);
\draw[gray,fill=gray] (12,10) circle (.3ex);
\draw[gray,fill=gray] (12,12) circle (.3ex);
\draw[gray,fill=gray] (12,14) circle (.3ex);
\draw[gray,fill=gray] (12,16) circle (.3ex);
\draw[gray,fill=gray] (12,18) circle (.3ex);
\draw[gray,fill=gray] (14,2) circle (.3ex);
\draw[gray,fill=gray] (14,4) circle (.3ex);
\draw[gray,fill=gray] (14,6) circle (.3ex);
\draw[gray,fill=gray] (14,8) circle (.3ex);
\draw[gray,fill=gray] (14,10) circle (.3ex);
\draw[gray,fill=gray] (14,12) circle (.3ex);
\draw[gray,fill=gray] (14,14) circle (.3ex);
\draw[gray,fill=gray] (14,16) circle (.3ex);
\draw[gray,fill=gray] (14,18) circle (.3ex);
\draw[gray,fill=gray] (16,2) circle (.3ex);
\draw[gray,fill=gray] (16,4) circle (.3ex);
\draw[gray,fill=gray] (16,6) circle (.3ex);
\draw[gray,fill=gray] (16,8) circle (.3ex);
\draw[gray,fill=gray] (16,10) circle (.3ex);
\draw[gray,fill=gray] (16,12) circle (.3ex);
\draw[gray,fill=gray] (16,14) circle (.3ex);
\draw[gray,fill=gray] (16,16) circle (.3ex);
\draw[gray,fill=gray] (16,18) circle (.3ex);
\draw[gray,fill=gray] (18,2) circle (.3ex);
\draw[gray,fill=gray] (18,4) circle (.3ex);
\draw[gray,fill=gray] (18,6) circle (.3ex);
\draw[gray,fill=gray] (18,8) circle (.3ex);
\draw[gray,fill=gray] (18,10) circle (.3ex);
\draw[gray,fill=gray] (18,12) circle (.3ex);
\draw[gray,fill=gray] (18,14) circle (.3ex);
\draw[gray,fill=gray] (18,16) circle (.3ex);
\draw[gray,fill=gray] (18,18) circle (.3ex);
\end{tikzpicture}
\qquad
\begin{tikzpicture}
\draw[gray, thin,  ->] (0,10) -- (20,10) node[anchor=south west] {\large$\lambda_1$};
\draw[gray, thin,  ->] (10,00) -- (10,20) node[anchor=north west] {\large$\lambda_2$};
\draw[gray,fill=gray] (10,10) circle (.3ex) node[anchor=north east] {\color{gray}$0$};
\draw[gray,fill=gray] (12,10) circle (.3ex) node[anchor=north east] {\color{gray}$1$};
\draw[gray,fill=gray] (14,10) circle (.3ex) node[anchor=north east] {\color{gray}$2$};
\draw[gray,fill=gray] (16,10) circle (.3ex) node[anchor=north east] {\color{gray}$3$};
\draw[gray,fill=gray] (18,10) circle (.3ex) node[anchor=north east] {\color{gray}$4$};
\draw[gray,fill=cyan] (8,10) circle (1.3ex) node[anchor=north east] {\color{gray}$-1$};
\draw[gray,fill=gray] (6,10) circle (.3ex) node[anchor=north east] {\color{gray}$-2$};
\draw[gray,fill=gray] (4,10) circle (.3ex) node[anchor=north east] {\color{gray}$-3$};
\draw[gray,fill=gray] (2,10) circle (.3ex) node[anchor=north east] {\color{gray}$-4$};
\draw[gray,fill=cyan] (10,12) circle (1.3ex) node[anchor=north east] {\color{gray}$1$};
\draw[gray,fill=gray] (10,14) circle (.3ex) node[anchor=north east] {\color{gray}$2$};
\draw[gray,fill=gray] (10,16) circle (.3ex) node[anchor=north east] {\color{gray}$3$};
\draw[gray,fill=gray] (10,18) circle (.3ex) node[anchor=north east] {\color{gray}$4$};
\draw[gray,fill=gray] (10,8) circle (.3ex) node[anchor=north east] {\color{gray}$-1$};
\draw[gray,fill=gray] (10,6) circle (.3ex) node[anchor=north east] {\color{gray}$-2$};
\draw[gray,fill=gray] (10,4) circle (.3ex) node[anchor=north east] {\color{gray}$-3$};
\draw[gray,fill=gray] (10,2) circle (.3ex) node[anchor=north east] {\color{gray}$-4$};
\draw[gray,fill=gray] (2,2) circle (.3ex);
\draw[gray,fill=gray] (2,4) circle (.3ex);
\draw[gray,fill=gray] (2,6) circle (.3ex);
\draw[gray,fill=gray] (2,8) circle (.3ex);
\draw[gray,fill=gray] (2,10) circle (.3ex);
\draw[gray,fill=gray] (2,12) circle (.3ex);
\draw[gray,fill=gray] (2,14) circle (.3ex);
\draw[gray,fill=gray] (2,16) circle (.3ex);
\draw[gray,fill=gray] (2,18) circle (.3ex);
\draw[gray,fill=gray] (4,2) circle (.3ex);
\draw[gray,fill=gray] (4,4) circle (.3ex);
\draw[gray,fill=gray] (4,6) circle (.3ex);
\draw[gray,fill=gray] (4,8) circle (.3ex);
\draw[gray,fill=gray] (4,10) circle (.3ex);
\draw[gray,fill=gray] (4,12) circle (.3ex);
\draw[gray,fill=gray] (4,14) circle (.3ex);
\draw[gray,fill=gray] (4,16) circle (.3ex);
\draw[gray,fill=gray] (4,18) circle (.3ex);
\draw[gray,fill=gray] (6,2) circle (.3ex);
\draw[gray,fill=gray] (6,4) circle (.3ex);
\draw[gray,fill=gray] (6,6) circle (.3ex);
\draw[gray,fill=gray] (6,8) circle (.3ex);
\draw[gray,fill=gray] (6,10) circle (.3ex);
\draw[gray,fill=gray] (6,12) circle (.3ex);
\draw[gray,fill=gray] (6,14) circle (.3ex);
\draw[gray,fill=gray] (6,16) circle (.3ex);
\draw[gray,fill=gray] (6,18) circle (.3ex);
\draw[gray,fill=gray] (8,2) circle (.3ex);
\draw[gray,fill=gray] (8,4) circle (.3ex);
\draw[gray,fill=gray] (8,6) circle (.3ex);
\draw[gray,fill=gray] (8,8) circle (.3ex);
\draw[gray,fill=gray] (8,10) circle (.3ex);
\draw[gray,fill=gray] (8,12) circle (.3ex);
\draw[gray,fill=gray] (8,14) circle (.3ex);
\draw[gray,fill=gray] (8,16) circle (.3ex);
\draw[gray,fill=gray] (8,18) circle (.3ex);
\draw[gray,fill=gray] (12,2) circle (.3ex);
\draw[gray,fill=gray] (12,4) circle (.3ex);
\draw[gray,fill=gray] (12,6) circle (.3ex);
\draw[gray,fill=cyan] (12,8) circle (1.3ex);
\draw[gray,fill=gray] (12,10) circle (.3ex);
\draw[gray,fill=gray] (12,12) circle (.3ex);
\draw[gray,fill=gray] (12,14) circle (.3ex);
\draw[gray,fill=gray] (12,16) circle (.3ex);
\draw[gray,fill=gray] (12,18) circle (.3ex);
\draw[gray,fill=gray] (14,2) circle (.3ex);
\draw[gray,fill=gray] (14,4) circle (.3ex);
\draw[gray,fill=gray] (14,6) circle (.3ex);
\draw[gray,fill=gray] (14,8) circle (.3ex);
\draw[gray,fill=gray] (14,10) circle (.3ex);
\draw[gray,fill=gray] (14,12) circle (.3ex);
\draw[gray,fill=gray] (14,14) circle (.3ex);
\draw[gray,fill=gray] (14,16) circle (.3ex);
\draw[gray,fill=gray] (14,18) circle (.3ex);
\draw[gray,fill=gray] (16,2) circle (.3ex);
\draw[gray,fill=gray] (16,4) circle (.3ex);
\draw[gray,fill=gray] (16,6) circle (.3ex);
\draw[gray,fill=gray] (16,8) circle (.3ex);
\draw[gray,fill=gray] (16,10) circle (.3ex);
\draw[gray,fill=gray] (16,12) circle (.3ex);
\draw[gray,fill=gray] (16,14) circle (.3ex);
\draw[gray,fill=gray] (16,16) circle (.3ex);
\draw[gray,fill=gray] (16,18) circle (.3ex);
\draw[gray,fill=gray] (18,2) circle (.3ex);
\draw[gray,fill=gray] (18,4) circle (.3ex);
\draw[gray,fill=gray] (18,6) circle (.3ex);
\draw[gray,fill=gray] (18,8) circle (.3ex);
\draw[gray,fill=gray] (18,10) circle (.3ex);
\draw[gray,fill=gray] (18,12) circle (.3ex);
\draw[gray,fill=gray] (18,14) circle (.3ex);
\draw[gray,fill=gray] (18,16) circle (.3ex);
\draw[gray,fill=gray] (18,18) circle (.3ex);
\end{tikzpicture}
\qquad
\begin{tikzpicture}
\draw[gray, thin,  ->] (0,10) -- (20,10) node[anchor=south west] {\large$\lambda_1$};
\draw[gray, thin,  ->] (10,00) -- (10,20) node[anchor=north west] {\large$\lambda_2$};
\draw[gray,fill=magenta] (10,10) circle (2.3ex) node[anchor=north east] {\color{gray}$0$};
\draw[gray,fill=gray] (12,10) circle (.3ex) node[anchor=north east] {\color{gray}$1$};
\draw[gray,fill=gray] (14,10) circle (.3ex) node[anchor=north east] {\color{gray}$2$};
\draw[gray,fill=gray] (16,10) circle (.3ex) node[anchor=north east] {\color{gray}$3$};
\draw[gray,fill=gray] (18,10) circle (.3ex) node[anchor=north east] {\color{gray}$4$};
\draw[gray,fill=gray] (8,10) circle (.3ex) node[anchor=north east] {\color{gray}$-1$};
\draw[gray,fill=gray] (6,10) circle (.3ex) node[anchor=north east] {\color{gray}$-2$};
\draw[gray,fill=gray] (4,10) circle (.3ex) node[anchor=north east] {\color{gray}$-3$};
\draw[gray,fill=gray] (2,10) circle (.3ex) node[anchor=north east] {\color{gray}$-4$};
\draw[gray,fill=gray] (10,12) circle (.3ex) node[anchor=north east] {\color{gray}$1$};
\draw[gray,fill=gray] (10,14) circle (.3ex) node[anchor=north east] {\color{gray}$2$};
\draw[gray,fill=gray] (10,16) circle (.3ex) node[anchor=north east] {\color{gray}$3$};
\draw[gray,fill=gray] (10,18) circle (.3ex) node[anchor=north east] {\color{gray}$4$};
\draw[gray,fill=gray] (10,8) circle (.3ex) node[anchor=north east] {\color{gray}$-1$};
\draw[gray,fill=gray] (10,6) circle (.3ex) node[anchor=north east] {\color{gray}$-2$};
\draw[gray,fill=gray] (10,4) circle (.3ex) node[anchor=north east] {\color{gray}$-3$};
\draw[gray,fill=gray] (10,2) circle (.3ex) node[anchor=north east] {\color{gray}$-4$};
\draw[gray,fill=gray] (2,2) circle (.3ex);
\draw[gray,fill=gray] (2,4) circle (.3ex);
\draw[gray,fill=gray] (2,6) circle (.3ex);
\draw[gray,fill=gray] (2,8) circle (.3ex);
\draw[gray,fill=gray] (2,10) circle (.3ex);
\draw[gray,fill=gray] (2,12) circle (.3ex);
\draw[gray,fill=gray] (2,14) circle (.3ex);
\draw[gray,fill=gray] (2,16) circle (.3ex);
\draw[gray,fill=gray] (2,18) circle (.3ex);
\draw[gray,fill=gray] (4,2) circle (.3ex);
\draw[gray,fill=gray] (4,4) circle (.3ex);
\draw[gray,fill=gray] (4,6) circle (.3ex);
\draw[gray,fill=gray] (4,8) circle (.3ex);
\draw[gray,fill=gray] (4,10) circle (.3ex);
\draw[gray,fill=gray] (4,12) circle (.3ex);
\draw[gray,fill=gray] (4,14) circle (.3ex);
\draw[gray,fill=gray] (4,16) circle (.3ex);
\draw[gray,fill=gray] (4,18) circle (.3ex);
\draw[gray,fill=gray] (6,2) circle (.3ex);
\draw[gray,fill=gray] (6,4) circle (.3ex);
\draw[gray,fill=gray] (6,6) circle (.3ex);
\draw[gray,fill=gray] (6,8) circle (.3ex);
\draw[gray,fill=gray] (6,10) circle (.3ex);
\draw[gray,fill=magenta] (6,12) circle (1.3ex);
\draw[gray,fill=gray] (6,14) circle (.3ex);
\draw[gray,fill=gray] (6,16) circle (.3ex);
\draw[gray,fill=gray] (6,18) circle (.3ex);
\draw[gray,fill=gray] (8,2) circle (.3ex);
\draw[gray,fill=gray] (8,4) circle (.3ex);
\draw[gray,fill=gray] (8,6) circle (.3ex);
\draw[gray,fill=magenta] (8,8) circle (1.3ex);
\draw[gray,fill=gray] (8,10) circle (.3ex);
\draw[gray,fill=gray] (8,12) circle (.3ex);
\draw[gray,fill=magenta] (8,14) circle (1.3ex);
\draw[gray,fill=gray] (8,16) circle (.3ex);
\draw[gray,fill=gray] (8,18) circle (.3ex);
\draw[gray,fill=gray] (12,2) circle (.3ex);
\draw[gray,fill=gray] (12,4) circle (.3ex);
\draw[gray,fill=magenta] (12,6) circle (1.3ex);
\draw[gray,fill=gray] (12,8) circle (.3ex);
\draw[gray,fill=gray] (12,10) circle (.3ex);
\draw[gray,fill=magenta] (12,12) circle (1.3ex);
\draw[gray,fill=gray] (12,14) circle (.3ex);
\draw[gray,fill=gray] (12,16) circle (.3ex);
\draw[gray,fill=gray] (12,18) circle (.3ex);
\draw[gray,fill=gray] (14,2) circle (.3ex);
\draw[gray,fill=gray] (14,4) circle (.3ex);
\draw[gray,fill=gray] (14,6) circle (.3ex);
\draw[gray,fill=magenta] (14,8) circle (1.3ex);
\draw[gray,fill=gray] (14,10) circle (.3ex);
\draw[gray,fill=gray] (14,12) circle (.3ex);
\draw[gray,fill=gray] (14,14) circle (.3ex);
\draw[gray,fill=gray] (14,16) circle (.3ex);
\draw[gray,fill=gray] (14,18) circle (.3ex);
\draw[gray,fill=gray] (16,2) circle (.3ex);
\draw[gray,fill=gray] (16,4) circle (.3ex);
\draw[gray,fill=gray] (16,6) circle (.3ex);
\draw[gray,fill=gray] (16,8) circle (.3ex);
\draw[gray,fill=gray] (16,10) circle (.3ex);
\draw[gray,fill=gray] (16,12) circle (.3ex);
\draw[gray,fill=gray] (16,14) circle (.3ex);
\draw[gray,fill=gray] (16,16) circle (.3ex);
\draw[gray,fill=gray] (16,18) circle (.3ex);
\draw[gray,fill=gray] (18,2) circle (.3ex);
\draw[gray,fill=gray] (18,4) circle (.3ex);
\draw[gray,fill=gray] (18,6) circle (.3ex);
\draw[gray,fill=gray] (18,8) circle (.3ex);
\draw[gray,fill=gray] (18,10) circle (.3ex);
\draw[gray,fill=gray] (18,12) circle (.3ex);
\draw[gray,fill=gray] (18,14) circle (.3ex);
\draw[gray,fill=gray] (18,16) circle (.3ex);
\draw[gray,fill=gray] (18,18) circle (.3ex);
\end{tikzpicture}
\qquad
}
\caption{The supports of  {\color{violet}$L((1,0))$}, {\color{cyan}$L((0,1))$} 
and {\color{magenta}$L((1,1))$}}\label{fig2}
\end{figure}
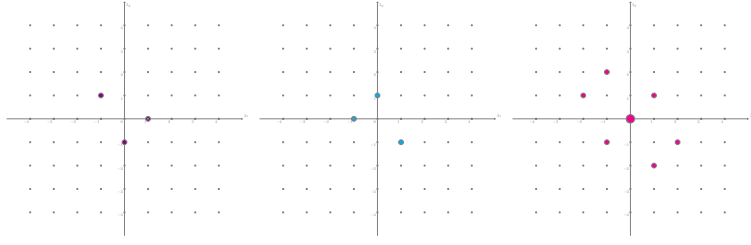

\subsection{Gelfand-Kirillov dimension and Bernstein number}\label{s2.6}

Each finitely generated $U(\mathfrak{g})$-module $M$ has finite
{\em Gelfand-Kirillov dimension}, denoted $\mathrm{GKdim}(M)$.
This dimension characterizes the rate of (polynomial) growth
of $M$ with respect to some choice of a filtration of $U(\mathfrak{g})$.
In fact, it is well-known (see \cite{Ja,KrLe}) that $\mathrm{GKdim}(M)$
is independent of the choice of such filtration.

We denote by $\mathbf{b}(M)$ the leading term of the polynomial that
describes the growth of $M$, this is known as the {\em Bernstein number}.
Unlike $\mathrm{GKdim}(M)$, the number $\mathbf{b}(M)$ might depend on
the choice of a filtration of $U(\mathfrak{g})$. Therefore we fix
some choice of such filtration until the end of the paper
(for example, we can take the filtration by monomial degree).

For any finite-dimensional $U(\mathfrak{g})$-module $V$, we have:
\begin{itemize}
\item $\mathrm{GKdim}(V)=0$ and $\mathbf{b}(V)=\dim(V)$;
\item $\mathrm{GKdim}(M\otimes_{\mathbb{C}} V)=\mathrm{GKdim}(M)$;
\item $\mathbf{b}(M\otimes_{\mathbb{C}} V)=\dim(V)\cdot \mathbf{b}(M)$.
\end{itemize}

\subsection{Projective functors}\label{s2.7}

Denote by $\mathcal{Z}$ the full subcategory of $\mathfrak{g}\text{-}\mathrm{mod}$
consisting of all objects, the action of $Z(\mathfrak{g})$ on which is 
locally finite. The category $\mathcal{Z}$ decomposes into a direct sum of 
full subcategories $\mathcal{Z}_{{}_\chi}$, indexed by central characters 
$\chi:Z(\mathfrak{g})\to\mathbb{C}$. The category $\mathcal{Z}_{{}_\chi}$
consists of all modules in $\mathcal{Z}$ on which the kernel of 
$\chi$ acts locally nilpotently. The category $\mathcal{Z}$ is stable
under the action of $\mathscr{C}$ on $\mathfrak{g}\text{-}\mathrm{mod}$.

A {\em projective functor} is an endofunctor of $\mathcal{Z}$
which is isomorphic to a direct summand of tensoring with some
finite dimensional $\mathfrak{g}$-modules. Indecomposable projective
functors are classified in \cite{BG}. They are in bijection with 
dot-orbits of $W$ on pairs $(\lambda,\mu)$, where $\lambda,\mu\in\mathfrak{h}^*$
are such that $\lambda-\mu\in\Lambda$. Each such orbit contains
at least one pair $(\lambda,\mu)$ such that 
\begin{itemize}
\item $\lambda$ is dominant with respect to its integral Weyl group;
\item $\mu$ is anti-dominant with respect to the dot-stabilizer of $\lambda$.
\end{itemize}
Such a pair  $(\lambda,\mu)$ is called {\em proper}. We will denote the 
indecomposable projective functor corresponding to $(\lambda,\mu)$
by $\theta_{\lambda,\mu}$.

\section{Locally finitary  $\mathscr{C}$-module categories}\label{s3}

\subsection{Basics}\label{s3.1}

We refer to \cite{EGNO} for basics on representations of 
monoidal categories.

Let $\mathcal{M}$ be a $\mathscr{C}$-module category. We will say that 
$\mathcal{M}$ is {\em locally finitary} provided that
\begin{itemize}
\item $\mathcal{M}$ is $\mathbb{C}$-linear, additive,
idempotent split and Krull-Schmidt;
\item $\mathcal{M}$ has at most countably many indecomposable
objects, up to isomorphism;
\item $\dim \mathcal{M}(X,Y)<\infty$, for all $X,Y\in \mathcal{M}$.
\end{itemize}
A typical example is $\mathscr{C}$ itself, considered
as a $\mathscr{C}$-module category with respect to the
left regular action.

A locally finitary $\mathscr{C}$-module category $\mathcal{M}$
is called {\em admissible} provided that it has weak kernels,
see \cite{Fr}. For example, any semi-simple 
$\mathscr{C}$-module category $\mathcal{M}$ is admissible.

\subsection{Module categories of interest}\label{s3.2}

The category $\mathfrak{sl}_3\text{-}\mathrm{Mod}$
of all $\mathfrak{sl}_3$-modules is a $\mathscr{C}$-module category
in the obvious way. The category $\mathfrak{sl}_3\text{-}\mathrm{mod}$
of all finitely generated $\mathfrak{sl}_3$-modules is a 
$\mathscr{C}$-module subcategory of $\mathfrak{sl}_3\text{-}\mathrm{Mod}$.

We will be interested in $\mathscr{C}$-module subcategories 
of $\mathfrak{sl}_3\text{-}\mathrm{mod}$ of the form
$\mathrm{add}(\mathscr{C}\cdot L)$, where $L$
is a simple $\mathfrak{sl}_3$-module (not necessarily finite dimensional).
These categories are defined as follows:

Given a simple $\mathfrak{sl}_3$-module $L$ and a finite
dimensional $\mathfrak{sl}_3$-module $V$, the module 
$V\otimes_{\mathbb{C}}L$ has a finite dimensional endomorphism
algebra (see \cite{MMM}) and hence decomposes into a finite direct sum of 
indecomposable modules. The category $\mathrm{add}(\mathscr{C}\cdot L)$
is the full subcategory of $\mathfrak{sl}_3\text{-}\mathrm{mod}$
which consists of all objects that are isomorphic to direct sums
of modules appearing as summands in $V\otimes_{\mathbb{C}}L$,
where $V$ can vary inside $\mathscr{C}$. Consequently, any 
$\mathrm{add}(\mathscr{C}\cdot L)$ is a locally finitary
$\mathscr{C}$-module category.

\subsection{Action matrices and graphs}\label{s3.3}

Let $\mathcal{M}$ be a locally finitary $\mathscr{C}$-module category
and $\mathrm{Ind}(\mathcal{M})=\{N_i\,:\,i\in \mathtt{I}\}$ be a
complete and irredundant
set of representatives of isomorphism classes of indecomposable
objects in $\mathcal{M}$. Then, to each $\theta\in\mathscr{C}$,
we can associate the $\mathtt{I}\times \mathtt{I}$ matrix $[\theta]$ 
with coefficients $m_{i,j}$, where $i,j\in \mathtt{I}$, defined
as follows: $m_{i,j}$ is the multiplicity of $X_i$ as a summand
of $\theta(X_j)$. 

Note that all $m_{i,j}$ are non-negative integers. Therefore, we 
can alternatively encode $[\theta]$ as an oriented graph,
which we call $\Gamma_\theta$, whose vertex set is $\mathtt{I}$
and which has exactly $m_{i,j}$ oriented edges from 
$j$ to $i$. 

Since the monoidal category $\mathscr{C}$ is generated by
$\mathrm{F}$, in many cases, almost all essential information 
about all $[\theta]$ and all $\Gamma_\theta$ is contained already 
in $[\mathrm{F}]$ and $\Gamma_\mathrm{F}$. For example, 
if $\mathcal{M}$ is semi-simple, then $[\mathrm{F}^*]=[\mathrm{F}]^t$,
see e.g. \cite[Lemma~8]{AM}. In particular, the graph
$\Gamma_{\mathrm{F}^*}$ is the opposite of $\Gamma_\mathrm{F}$.
Note that $[\mathrm{F}^*]$
and $[\mathrm{F}]$ commute, which puts very strong restrictions on
$[\mathrm{F}]$.

We will make the connection between $[\mathrm{F}]$
and an arbitrary $[\theta]$ more precise
in the next subsection, in particular, we will describe how
any $[\theta]$ can be expressed as a polynomial in $[\mathrm{F}]$
and $[\mathrm{F}^*]$. 

We will call
$[\mathrm{F}]$ the {\em matrix of $\mathcal{M}$},
$[\mathrm{F}^*]$ the {\em dual matrix of $\mathcal{M}$},
$\Gamma_\mathrm{F}$ the {\em graph of $\mathcal{M}$}
and $\Gamma_{\mathrm{F}^*}$ the {\em dual graph of $\mathcal{M}$}.

\subsection{Grothendieck groups}\label{s3.4}

Consider the Grothendieck ring $\mathbf{Gr}(\mathscr{C})$
of $\mathscr{C}$. It has the natural basis given by 
$[L((i,j))]$, where $i,j\in\mathbb{Z}_{\geq 0}$, with the
structure constants with respect to this basis given by 
the Clebsch-Gordan coefficients $\gamma_{(i,j),(i',j')}^{(k,l)}$
defined via
\begin{displaymath}
L((i,j))\otimes_\mathbb{C}L((i',j'))\cong
\bigoplus_{(k,l)}L((k,l))^{\oplus \gamma_{(i,j),(i',j')}^{(k,l)}}.
\end{displaymath}
Since $\mathfrak{h}$-eigenvalues behave additively with respect to
tensor products, one obtains both the equality
$\gamma_{(i,j),(i',j')}^{(i+i',j+j')}=1$
and the fact that $\gamma_{(i,j),(i',j')}^{(k,l)}\neq 0$ implies the
inequality $(k,l)\leq (i+i',j+j')$ in the sense that 
$(i+i'-k,j+j'-l)\in\mathbb{Z}_{\geq 0}\pi$. Consequently,
$\mathbf{Gr}(\mathscr{C})$ is isomorphic to the 
polynomial algebra $\mathbb{Z}[x,y]$, by sending $[L(1,0)]$ to $x$
and $[L(0,1)]$ to $y$.

Following \cite[Subsection~2.2]{MMMT}, for
$(i,j)\in \mathbb{Z}$, define polynomials
$\mathtt{U}_{i,j}(x,y)$ recursively as follows:
\begin{itemize}
\item $\mathtt{U}_{i,j}(x,y)=0$ provided that $(i,j)\not\in \mathbb{Z}_{\geq 0}^2$;
\item $\mathtt{U}_{0,0}(x,y)=1$;
\item $X\mathtt{U}_{i,j}(x,y)=\mathtt{U}_{i+1,j}(x,y)+
\mathtt{U}_{i-1,j+1}(x,y)+\mathtt{U}_{i,j-1}(x,y)$;
\item $Y\mathtt{U}_{i,j}(x,y)=\mathtt{U}_{i,j+1}(x,y)+
\mathtt{U}_{i+1,j-1}(x,y)+\mathtt{U}_{i-1,j}(x,y)$.
\end{itemize}
Here are some small examples:
\begin{displaymath}
\mathtt{U}_{1,0}(x,y)=x,\,\, \mathtt{U}_{0,1}(x,y)=y,\,\,
\mathtt{U}_{2,0}(x,y)=x^2-y, \,\, \mathtt{U}_{1,1}(x,y)=xy-1. 
\end{displaymath}
Note that directly from the definitions it follows
that $\mathtt{U}_{i,j}(x,y)=\mathtt{U}_{j,i}(y,x)$.

The polynomials $\{\mathtt{U}_{i,j}(x,y)\,:\, (i,j)\in \mathbb{Z}_{\geq 0}\}$
form a basis of $\mathbb{Z}[x,y]$, moreover, under the above isomorphism
between $\mathbf{Gr}(\mathscr{C})$ and $\mathbb{Z}[x,y]$, the polynomial
$\mathtt{U}_{i,j}(x,y)$ corresponds to $[L((i,j))]$. This is due to the fact 
that the recursive definition of $\mathtt{U}_{i,j}$ simply mimics the
Clebsch-Gordan rule for tensoring with $L((1,0))$ and $L((0,1))$.
It could be helpful to compare the recursive steps with the supports of
$L((1,0))$ and $L((0,1))$ given in Figure~\ref{fig3}.

If $\mathcal{M}$ is a locally finitary $\mathscr{C}$-module, then the
split Grothendieck group $\mathbf{Gr}(\mathcal{M})$ is, naturally,
a $\mathbf{Gr}(\mathscr{C})$-module. The group $\mathbf{Gr}(\mathcal{M})$
is a free abelian group with the standard basis given by the
isomorphism classes of the indecomposable objects in $\mathcal{M}$.
With respect to this basis, the matrix of the linear operator 
corresponding to the action of $\theta\in \mathscr{C}$ 
is exactly the action matrix $[\theta]$.
Consequently, if $\mathcal{M}$ is a locally finitary $\mathscr{C}$-module
category, then, for any $(i,j)\in \mathbb{Z}_{\geq 0}$, we have
\begin{displaymath}
[L((i,j))]= \mathtt{U}_{i,j}([L((1,0))],[L((0,1))]).
\end{displaymath}
In particular, the knowledge of $[L((1,0))]$ and $[L((0,1))]$
determines all the remaining $[L((i,j))]$ uniquely.

\subsection{Transitive module categories}\label{s3.5}

A locally finitary $\mathscr{C}$-module category $\mathcal{M}$
is called {\em transitive} provided that, for all indecomposable objects
$X$ and $Y$ in $\mathcal{M}$, there is $\theta\in \mathscr{C}$
such that $Y$ is isomorphic to a summand of $\theta(X)$,
see \cite{MM5}. Alternatively, $\mathcal{M}$ is transitive 
provided that $\Gamma_\mathrm{F}$ is strongly connected.

A locally finitary $\mathscr{C}$-module category $\mathcal{M}$
is called {\em simple} provided that $\mathcal{M}$ is non-zero
and does not have any non-trivial $\mathscr{C}$-stable ideals.
Every simple $\mathscr{C}$-module category is automatically 
transitive. Every locally finitary  transitive $\mathscr{C}$-module category
has a unique simple transitive quotient. More generally,
every locally finitary $\mathscr{C}$-module category 
admits a (possibly infinite) filtration with transitive 
subquotients. The multiset of the corresponding simple 
transitive subquotients is independent of the choice of
the filtration.

\subsection{Abelianization}\label{s3.6}

Given a locally finitary $\mathscr{C}$-module category $\mathcal{M}$,
the abelianization $\overline{\mathcal{M}}$ of ${\mathcal{M}}$ is defined
as the following category of diagrams over ${\mathcal{M}}$, see \cite{MM1}:
\begin{itemize}
\item the objects of $\overline{\mathcal{M}}$ are diagrams
$\alpha:X\to Y$ over ${\mathcal{M}}$;
\item the morphisms in $\overline{\mathcal{M}}$ from 
$\alpha:X\to Y$ to $\alpha':X'\to Y'$ are given as the
quotient of the space of all solid commutative diagrams over ${\mathcal{M}}$
as below, modulo the subspace generated by all diagrams for
which the right vertical map factors through some dotted morphism:
\begin{displaymath}
\xymatrix{
X\ar[rr]^{\alpha}\ar[d]&&Y\ar@{.>}[lld]\ar[d]\\
X'\ar[rr]^{\alpha'}&&Y'
}
\end{displaymath}
\end{itemize}
The category $\overline{\mathcal{M}}$ inherits from ${\mathcal{M}}$
a structure of a $\mathscr{C}$-module category. Furthermore,
$\overline{\mathcal{M}}$ is abelian provided that ${\mathcal{M}}$
is admissible. In the latter case,  $\overline{\mathcal{M}}$
has enough projective objects  and ${\mathcal{M}}$ is equivalent,
via the embedding $X\mapsto (0\to X)$, to the category of projective
objects in $\overline{\mathcal{M}}$. Moreover, for an indecomposable
$X\in {\mathcal{M}}$, mapping the object $0\to X$ to its simple top gives rise
to a bijection between the indecomposable projective and simple 
objects in $\overline{\mathcal{M}}$.

For $\theta\in \mathscr{C}$ and two simple objects 
$L$ and $N$ in $\overline{\mathcal{M}}$, we can define 
$a_{N,L}^\theta$ as the composition multiplicity of $N$ in $\theta L$,
which coincides with $\dim \overline{\mathcal{M}}(P_N,\theta L)$,
where $P_N$ is the projective cover of $N$. These can be assembled
into the matrix $\llbracket \theta\rrbracket$. By
\cite[Lemma~8]{AM}, the matrix $\llbracket \theta\rrbracket$ is
transposed to the matrix $[\theta^*]$.

\subsection{Positive eigenvectors}\label{s3.7}

Let $L$ be a simple $\mathfrak{sl}_3$-module. Consider the
$\mathscr{C}$-module category $\mathrm{add}(\mathscr{C}\cdot L)$.
This is an admissible $\mathscr{C}$-module category.
Let us assume that $\{X_i\,:\,i\in\mathtt{I}\}$ is a complete and
irredundant list of indecomposable objects in 
$\mathrm{add}(\mathscr{C}\cdot L)$. Then, for every $i\in\mathtt{I}$,
we have (see \cite{MMM}):
\begin{displaymath}
\mathrm{GKdim}(X_i)= \mathrm{GKdim}(L).
\end{displaymath}
Moreover, for the (column) vector $\mathbf{B}_{L}=(\mathbf{b}(X_i))_{i\in\mathtt{I}}$, 
from the last property in Subsection~\ref{s2.6} it follows that
\begin{displaymath}
\mathbf{B}_{L}^t[\mathrm{F}]=3\mathbf{B}_{L}^t. 
\end{displaymath}
In particular, $3$ is the Perron-Frobenius eigenvalue of 
$[\mathrm{F}]$ and $\mathbf{B}_{L}$ is the unique (up to a positive scalar)
left Perron-Frobenius eigenvector for that eigenvalue.

Similarly, we can define the vector $\mathbf{B}_{\mathcal{M}}$ for any transitive
$\mathscr{C}$-module category $\mathcal{M}$ and have
$\mathbf{B}_{\mathcal{M}}^t[\mathrm{F}]=3\mathbf{B}_{\mathcal{M}}^t$.

\section{The left regular module category}\label{s4}

\subsection{Basic combinatorics}\label{s4.1}

The left regular $\mathscr{C}$-module category ${}_\mathscr{C}\mathscr{C}$
is transitive with semi-simple underlying category. 
Consequently, ${}_\mathscr{C}\mathscr{C}$ is a  simple
$\mathscr{C}$-module category and coincides with 
$\mathrm{add}(\mathscr{C}\cdot L)$, for any simple finite
dimensional $\mathfrak{sl}_3$-module $L$.

The indecomposable objects in $\mathscr{C}$ are
$\{L((i,j))\,:\,i,j\in\mathbb{Z}_{\geq 0}\}$.
We have:

\resizebox{\textwidth}{!}{
$
L((1,0))\otimes_{\mathbb{C}}L((i,j))\cong
\begin{cases}
L((1,0)),& (i,j)=(0,0);\\
L((i+1,0))\oplus L((i-1,1)),& i>0, j=0;\\
L((1,j))\oplus L((i,j-1)),& i=0, j>0;\\
L((i+1,j))\oplus L((i-1,j+1))\oplus L((i,j-1)),& i,j>0.
\end{cases}
$
}

This implies that the graph $\Gamma_\mathrm{F}$
is as in Figure~\ref{fig3}. Since the underlying category
${}_\mathscr{C}\mathscr{C}$ is semi-simple, the graph 
$\Gamma_\mathrm{G}$ is just the opposite 
of $\Gamma_\mathrm{F}$.

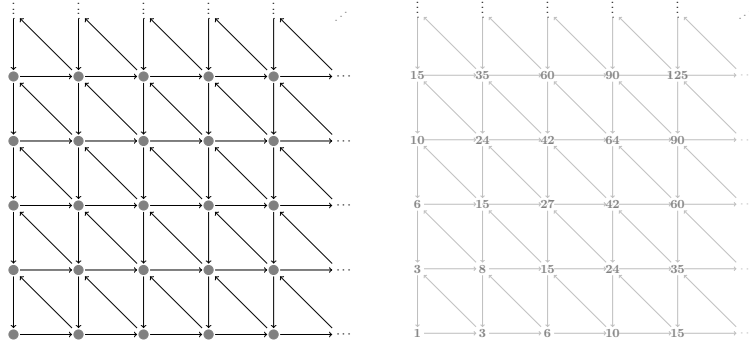
\begin{figure}
\resizebox{10cm}{!}{
\begin{tikzpicture}
\draw[black, thick,  ->] (0.2,0) -- (1.8,0);
\draw[black, thick,  ->] (0,1.8) -- (0,0.2);
\draw[black, thick,  ->] (1.8,0.2) -- (0.2,1.8);
\draw[black, thick,  ->] (2.2,0) -- (3.8,0);
\draw[black, thick,  ->] (2,1.8) -- (2,0.2);
\draw[black, thick,  ->] (3.8,0.2) -- (2.2,1.8);
\draw[black, thick,  ->] (4.2,0) -- (5.8,0);
\draw[black, thick,  ->] (4,1.8) -- (4,0.2);
\draw[black, thick,  ->] (5.8,0.2) -- (4.2,1.8);
\draw[black, thick,  ->] (6.2,0) -- (7.8,0);
\draw[black, thick,  ->] (6,1.8) -- (6,0.2);
\draw[black, thick,  ->] (7.8,0.2) -- (6.2,1.8);
\draw[black, thick,  ->] (8,1.8) -- (8,0.2);
\draw[black, thick,  ->] (0.2,2) -- (1.8,2);
\draw[black, thick,  ->] (0,3.8) -- (0,2.2);
\draw[black, thick,  ->] (1.8,2.2) -- (0.2,3.8);
\draw[black, thick,  ->] (2.2,2) -- (3.8,2);
\draw[black, thick,  ->] (2,3.8) -- (2,2.2);
\draw[black, thick,  ->] (3.8,2.2) -- (2.2,3.8);
\draw[black, thick,  ->] (4.2,2) -- (5.8,2);
\draw[black, thick,  ->] (4,3.8) -- (4,2.2);
\draw[black, thick,  ->] (5.8,2.2) -- (4.2,3.8);
\draw[black, thick,  ->] (6.2,2) -- (7.8,2);
\draw[black, thick,  ->] (6,3.8) -- (6,2.2);
\draw[black, thick,  ->] (7.8,2.2) -- (6.2,3.8);
\draw[black, thick,  ->] (8,3.8) -- (8,2.2);
\draw[black, thick,  ->] (0.2,4) -- (1.8,4);
\draw[black, thick,  ->] (0,5.8) -- (0,4.2);
\draw[black, thick,  ->] (1.8,4.2) -- (0.2,5.8);
\draw[black, thick,  ->] (2.2,4) -- (3.8,4);
\draw[black, thick,  ->] (2,5.8) -- (2,4.2);
\draw[black, thick,  ->] (3.8,4.2) -- (2.2,5.8);
\draw[black, thick,  ->] (4.2,4) -- (5.8,4);
\draw[black, thick,  ->] (4,5.8) -- (4,4.2);
\draw[black, thick,  ->] (5.8,4.2) -- (4.2,5.8);
\draw[black, thick,  ->] (6.2,4) -- (7.8,4);
\draw[black, thick,  ->] (6,5.8) -- (6,4.2);
\draw[black, thick,  ->] (7.8,4.2) -- (6.2,5.8);
\draw[black, thick,  ->] (8,5.8) -- (8,4.2);
\draw[black, thick,  ->] (0.2,6) -- (1.8,6);
\draw[black, thick,  ->] (0,7.8) -- (0,6.2);
\draw[black, thick,  ->] (1.8,6.2) -- (0.2,7.8);
\draw[black, thick,  ->] (2.2,6) -- (3.8,6);
\draw[black, thick,  ->] (2,7.8) -- (2,6.2);
\draw[black, thick,  ->] (3.8,6.2) -- (2.2,7.8);
\draw[black, thick,  ->] (4.2,6) -- (5.8,6);
\draw[black, thick,  ->] (4,7.8) -- (4,6.2);
\draw[black, thick,  ->] (5.8,6.2) -- (4.2,7.8);
\draw[black, thick,  ->] (6.2,6) -- (7.8,6);
\draw[black, thick,  ->] (6,7.8) -- (6,6.2);
\draw[black, thick,  ->] (7.8,6.2) -- (6.2,7.8);
\draw[black, thick,  ->] (8,7.8) -- (8,6.2);
\draw[black, thick,  ->] (8.2,0) -- (9.8,0) node[anchor= west] {\large$\cdots$};
\draw[black, thick,  ->] (8.2,2) -- (9.8,2) node[anchor= west] {\large$\cdots$};
\draw[black, thick,  ->] (8.2,4) -- (9.8,4) node[anchor= west] {\large$\cdots$};
\draw[black, thick,  ->] (8.2,6) -- (9.8,6) node[anchor= west] {\large$\cdots$};
\draw[black, thick,  ->] (8.2,8) -- (9.8,8) node[anchor= west] {\large$\cdots$};
\draw[black, thick,  ->] (9.8,0.2) -- (8.2,1.8);
\draw[black, thick,  ->] (9.8,2.2) -- (8.2,3.8);
\draw[black, thick,  ->] (9.8,4.2) -- (8.2,5.8);
\draw[black, thick,  ->] (9.8,6.2) -- (8.2,7.8);
\draw[black, thick,  ->]  (0,9.8) -- (0,8.2);
\draw[black, thick,  ->]  (2,9.8) -- (2,8.2);
\draw[black, thick,  ->]  (4,9.8) -- (4,8.2);
\draw[black, thick,  ->]  (6,9.8) -- (6,8.2);
\draw[black, thick,  ->]  (8,9.8) -- (8,8.2);
\draw[black, thin,  -] (0,9.8) -- (0,9.8) node[anchor= south] {\large$\vdots$};
\draw[black, thin,  -] (2,9.8) -- (2,9.8) node[anchor= south] {\large$\vdots$};
\draw[black, thin,  -] (4,9.8) -- (4,9.8) node[anchor= south] {\large$\vdots$};
\draw[black, thin,  -] (6,9.8) -- (6,9.8) node[anchor= south] {\large$\vdots$};
\draw[black, thin,  -] (8,9.8) -- (8,9.8) node[anchor= south] {\large$\vdots$};
\draw[gray, thin,  -] (10.4,10.4) -- (10.4,10.4) node[anchor= north east] {\large$\iddots$};
\draw[black, thick,  ->] (1.8,8.2) -- (0.2,9.8);
\draw[black, thick,  ->] (3.8,8.2) -- (2.2,9.8);
\draw[black, thick,  ->] (5.8,8.2) -- (4.2,9.8);
\draw[black, thick,  ->] (7.8,8.2) -- (6.2,9.8);
\draw[black, thick,  ->] (9.8,8.2) -- (8.2,9.8);
\draw[black, thick,  ->] (0.2,8) -- (1.8,8);
\draw[black, thick,  ->] (2.2,8) -- (3.8,8);
\draw[black, thick,  ->] (4.2,8) -- (5.8,8);
\draw[black, thick,  ->] (6.2,8) -- (7.8,8);
\draw[gray,fill=gray] (0,0) circle (.9ex);
\draw[gray,fill=gray] (2,0) circle (.9ex);
\draw[gray,fill=gray] (4,0) circle (.9ex);
\draw[gray,fill=gray] (6,0) circle (.9ex);
\draw[gray,fill=gray] (8,0) circle (.9ex);
\draw[gray,fill=gray] (0,2) circle (.9ex);
\draw[gray,fill=gray] (2,2) circle (.9ex);
\draw[gray,fill=gray] (4,2) circle (.9ex);
\draw[gray,fill=gray] (6,2) circle (.9ex);
\draw[gray,fill=gray] (8,2) circle (.9ex);
\draw[gray,fill=gray] (0,4) circle (.9ex);
\draw[gray,fill=gray] (2,4) circle (.9ex);
\draw[gray,fill=gray] (4,4) circle (.9ex);
\draw[gray,fill=gray] (6,4) circle (.9ex);
\draw[gray,fill=gray] (8,4) circle (.9ex);
\draw[gray,fill=gray] (0,6) circle (.9ex);
\draw[gray,fill=gray] (2,6) circle (.9ex);
\draw[gray,fill=gray] (4,6) circle (.9ex);
\draw[gray,fill=gray] (6,6) circle (.9ex);
\draw[gray,fill=gray] (8,6) circle (.9ex);
\draw[gray,fill=gray] (0,8) circle (.9ex);
\draw[gray,fill=gray] (2,8) circle (.9ex);
\draw[gray,fill=gray] (4,8) circle (.9ex);
\draw[gray,fill=gray] (6,8) circle (.9ex);
\draw[gray,fill=gray] (8,8) circle (.9ex);
\end{tikzpicture}
\qquad\qquad
\begin{tikzpicture}
\draw[lightgray, thin,  ->] (0.2,0) -- (1.8,0);
\draw[lightgray, thin,  ->] (0,1.8) -- (0,0.2);
\draw[lightgray, thin,  ->] (1.8,0.2) -- (0.2,1.8);
\draw[lightgray, thin,  ->] (2.2,0) -- (3.8,0);
\draw[lightgray, thin,  ->] (2,1.8) -- (2,0.2);
\draw[lightgray, thin,  ->] (3.8,0.2) -- (2.2,1.8);
\draw[lightgray, thin,  ->] (4.2,0) -- (5.8,0);
\draw[lightgray, thin,  ->] (4,1.8) -- (4,0.2);
\draw[lightgray, thin,  ->] (5.8,0.2) -- (4.2,1.8);
\draw[lightgray, thin,  ->] (6.2,0) -- (7.8,0);
\draw[lightgray, thin,  ->] (6,1.8) -- (6,0.2);
\draw[lightgray, thin,  ->] (7.8,0.2) -- (6.2,1.8);
\draw[lightgray, thin,  ->] (8,1.8) -- (8,0.2);
\draw[lightgray, thin,  ->] (0.2,2) -- (1.8,2);
\draw[lightgray, thin,  ->] (0,3.8) -- (0,2.2);
\draw[lightgray, thin,  ->] (1.8,2.2) -- (0.2,3.8);
\draw[lightgray, thin,  ->] (2.2,2) -- (3.8,2);
\draw[lightgray, thin,  ->] (2,3.8) -- (2,2.2);
\draw[lightgray, thin,  ->] (3.8,2.2) -- (2.2,3.8);
\draw[lightgray, thin,  ->] (4.2,2) -- (5.8,2);
\draw[lightgray, thin,  ->] (4,3.8) -- (4,2.2);
\draw[lightgray, thin,  ->] (5.8,2.2) -- (4.2,3.8);
\draw[lightgray, thin,  ->] (6.2,2) -- (7.8,2);
\draw[lightgray, thin,  ->] (6,3.8) -- (6,2.2);
\draw[lightgray, thin,  ->] (7.8,2.2) -- (6.2,3.8);
\draw[lightgray, thin,  ->] (8,3.8) -- (8,2.2);
\draw[lightgray, thin,  ->] (0.2,4) -- (1.8,4);
\draw[lightgray, thin,  ->] (0,5.8) -- (0,4.2);
\draw[lightgray, thin,  ->] (1.8,4.2) -- (0.2,5.8);
\draw[lightgray, thin,  ->] (2.2,4) -- (3.8,4);
\draw[lightgray, thin,  ->] (2,5.8) -- (2,4.2);
\draw[lightgray, thin,  ->] (3.8,4.2) -- (2.2,5.8);
\draw[lightgray, thin,  ->] (4.2,4) -- (5.8,4);
\draw[lightgray, thin,  ->] (4,5.8) -- (4,4.2);
\draw[lightgray, thin,  ->] (5.8,4.2) -- (4.2,5.8);
\draw[lightgray, thin,  ->] (6.2,4) -- (7.8,4);
\draw[lightgray, thin,  ->] (6,5.8) -- (6,4.2);
\draw[lightgray, thin,  ->] (7.8,4.2) -- (6.2,5.8);
\draw[lightgray, thin,  ->] (8,5.8) -- (8,4.2);
\draw[lightgray, thin,  ->] (0.2,6) -- (1.8,6);
\draw[lightgray, thin,  ->] (0,7.8) -- (0,6.2);
\draw[lightgray, thin,  ->] (1.8,6.2) -- (0.2,7.8);
\draw[lightgray, thin,  ->] (2.2,6) -- (3.8,6);
\draw[lightgray, thin,  ->] (2,7.8) -- (2,6.2);
\draw[lightgray, thin,  ->] (3.8,6.2) -- (2.2,7.8);
\draw[lightgray, thin,  ->] (4.2,6) -- (5.8,6);
\draw[lightgray, thin,  ->] (4,7.8) -- (4,6.2);
\draw[lightgray, thin,  ->] (5.8,6.2) -- (4.2,7.8);
\draw[lightgray, thin,  ->] (6.2,6) -- (7.8,6);
\draw[lightgray, thin,  ->] (6,7.8) -- (6,6.2);
\draw[lightgray, thin,  ->] (7.8,6.2) -- (6.2,7.8);
\draw[lightgray, thin,  ->] (8,7.8) -- (8,6.2);
\draw[lightgray, thin,  ->] (8.2,0) -- (9.8,0) node[anchor= west] {\large$\cdots$};
\draw[lightgray, thin,  ->] (8.2,2) -- (9.8,2) node[anchor= west] {\large$\cdots$};
\draw[lightgray, thin,  ->] (8.2,4) -- (9.8,4) node[anchor= west] {\large$\cdots$};
\draw[lightgray, thin,  ->] (8.2,6) -- (9.8,6) node[anchor= west] {\large$\cdots$};
\draw[lightgray, thin,  ->] (8.2,8) -- (9.8,8) node[anchor= west] {\large$\cdots$};
\draw[lightgray, thin,  ->] (9.8,0.2) -- (8.2,1.8);
\draw[lightgray, thin,  ->] (9.8,2.2) -- (8.2,3.8);
\draw[lightgray, thin,  ->] (9.8,4.2) -- (8.2,5.8);
\draw[lightgray, thin,  ->] (9.8,6.2) -- (8.2,7.8);
\draw[lightgray, thin,  ->]  (0,9.8) -- (0,8.2);
\draw[lightgray, thin,  ->]  (2,9.8) -- (2,8.2);
\draw[lightgray, thin,  ->]  (4,9.8) -- (4,8.2);
\draw[lightgray, thin,  ->]  (6,9.8) -- (6,8.2);
\draw[lightgray, thin,  ->]  (8,9.8) -- (8,8.2);
\draw[black, thin,  -] (0,9.8) -- (0,9.8) node[anchor= south] {\large$\vdots$};
\draw[black, thin,  -] (2,9.8) -- (2,9.8) node[anchor= south] {\large$\vdots$};
\draw[black, thin,  -] (4,9.8) -- (4,9.8) node[anchor= south] {\large$\vdots$};
\draw[black, thin,  -] (6,9.8) -- (6,9.8) node[anchor= south] {\large$\vdots$};
\draw[black, thin,  -] (8,9.8) -- (8,9.8) node[anchor= south] {\large$\vdots$};
\draw[gray, thin,  -] (10.4,10.4) -- (10.4,10.4) node[anchor= north east] {\large$\iddots$};
\draw[lightgray, thin,  ->] (1.8,8.2) -- (0.2,9.8);
\draw[lightgray, thin,  ->] (3.8,8.2) -- (2.2,9.8);
\draw[lightgray, thin,  ->] (5.8,8.2) -- (4.2,9.8);
\draw[lightgray, thin,  ->] (7.8,8.2) -- (6.2,9.8);
\draw[lightgray, thin,  ->] (9.8,8.2) -- (8.2,9.8);
\draw[lightgray, thin,  ->] (0.2,8) -- (1.8,8);
\draw[lightgray, thin,  ->] (2.2,8) -- (3.8,8);
\draw[lightgray, thin,  ->] (4.2,8) -- (5.8,8);
\draw[lightgray, thin,  ->] (6.2,8) -- (7.8,8);
\draw[gray,fill=gray] (0,0) circle (.01ex) node {\large$\mathbf{1}$};
\draw[gray,fill=gray] (2,0) circle (.01ex) node {\large$\mathbf{3}$};
\draw[gray,fill=gray] (4,0) circle (.01ex) node {\large$\mathbf{6}$};
\draw[gray,fill=gray] (6,0) circle (.01ex) node {\large$\mathbf{10}$};
\draw[gray,fill=gray] (8,0) circle (.01ex) node {\large$\mathbf{15}$};
\draw[gray,fill=gray] (0,2) circle (.01ex) node {\large$\mathbf{3}$};
\draw[gray,fill=gray] (2,2) circle (.01ex) node {\large$\mathbf{8}$};
\draw[gray,fill=gray] (4,2) circle (.01ex) node {\large$\mathbf{15}$};
\draw[gray,fill=gray] (6,2) circle (.01ex) node {\large$\mathbf{24}$};
\draw[gray,fill=gray] (8,2) circle (.01ex) node {\large$\mathbf{35}$};
\draw[gray,fill=gray] (0,4) circle (.01ex) node {\large$\mathbf{6}$};
\draw[gray,fill=gray] (2,4) circle (.01ex) node {\large$\mathbf{15}$};
\draw[gray,fill=gray] (4,4) circle (.01ex) node {\large$\mathbf{27}$};
\draw[gray,fill=gray] (6,4) circle (.01ex) node {\large$\mathbf{42}$};
\draw[gray,fill=gray] (8,4) circle (.01ex) node {\large$\mathbf{60}$};
\draw[gray,fill=gray] (0,6) circle (.01ex) node {\large$\mathbf{10}$};
\draw[gray,fill=gray] (2,6) circle (.01ex) node {\large$\mathbf{24}$};
\draw[gray,fill=gray] (4,6) circle (.01ex) node {\large$\mathbf{42}$};
\draw[gray,fill=gray] (6,6) circle (.01ex) node {\large$\mathbf{64}$};
\draw[gray,fill=gray] (8,6) circle (.01ex) node {\large$\mathbf{90}$};
\draw[gray,fill=gray] (0,8) circle (.01ex) node {\large$\mathbf{15}$};
\draw[gray,fill=gray] (2,8) circle (.01ex) node {\large$\mathbf{35}$};
\draw[gray,fill=gray] (4,8) circle (.01ex) node {\large$\mathbf{60}$};
\draw[gray,fill=gray] (6,8) circle (.01ex) node {\large$\mathbf{90}$};
\draw[gray,fill=gray] (8,8) circle (.01ex) node {\large$\mathbf{125}$};
\end{tikzpicture}
}
\caption{The graph of the left regular $\mathscr{C}$-module category
and the corresponding eigenvector}\label{fig3}
\end{figure}

\subsection{Eigenvector}\label{s4.2}

Since $\Gamma_{\mathrm{F}}$ is strongly connected
and all vertices have finite degree, from the 
infinite analogue of the Perron-Frobenius Theorem, 
the matrix $[\mathrm{F}]$ has a unique, up to scalar, 
eigenvector with strictly positive entries. The corresponding
eigenvalue is the Perron-Frobenius eigenvalue of $[\mathrm{F}]$.

Clearly, the vector $(\dim(L((i,j))))_{(i,j)\in\mathbb{Z}_{\geq 0}^2}$
is an eigenvector for the matrix $[\mathrm{F}]$ with eigenvalue $3=\dim(L((1,0)))$.
From the definition of $\mathtt{U}_{(i,j)}(X,Y)$ it follows directly that
$\dim(L((i,j)))=\mathtt{U}_{(i,j)}(3,3)$. It is easy to check that
the assignment
\begin{displaymath}
(i,j)\mapsto \frac{(i+1)(j+1)(i+j+2)}{2} 
\end{displaymath}
satisfies the same recursion and hence equals $\dim(L((i,j)))$.
An alternative way to come to this is to use the Weyl character formula.

This means that the integral sequence 
\begin{displaymath}
\dim(L((0,0))), \dim(L((1,0))), \dim(L((0,1))), 
\dim(L((2,0))), \dim(L((1,1))), \dots 
\end{displaymath}
is the sequence A107985 in \cite{OEIS}.

The above eigenvector is depicted in Figure~\ref{fig3}. 
Clearly, it is also an eigenvector for $[\mathrm{G}]$,
with the same eigenvalue.

\subsection{Simple $\mathscr{C}$-module categories with the same combinatorics}\label{s4.3}

The following theorem shows that the $\mathscr{C}$-module 
category ${}_\mathscr{C}\mathscr{C}$ is determined
uniquely, up to isomorphism,
by its graph and dual graph in the class of 
admissible simple $\mathscr{C}$-module categories.

\begin{theorem}\label{thm-s4.3-1}
Let $\mathcal{M}$ be an admissible simple $\mathscr{C}$-module category
with the same graph and dual graph as ${}_\mathscr{C}\mathscr{C}$.
Then $\mathcal{M}$ is equivalent to ${}_\mathscr{C}\mathscr{C}$.
\end{theorem}

\begin{proof}
Let $\{M_{(i,j)}\,:\,(i,j)\in\mathbb{Z}_{\geq 0}^2\}$ be a complete
and irredundant list of representatives of isomorphism classes of 
indecomposable objects in $\mathcal{M}$, indexed such that the map
$L((i,j))\mapsto M_{(i,j)}$ induces an isomorphism between the 
graphs of ${}_\mathscr{C}\mathscr{C}$ and $\mathcal{M}$. Clearly,
the graph in Figure~\ref{fig3} has no non-trivial automorphisms,
which implies that we automatically have the induced isomorphism 
between the  dual graphs of ${}_\mathscr{C}\mathscr{C}$ and $\mathcal{M}$.

Consider $\overline{\mathcal{M}}$ and let $N_{(i,j)}$ be the simple
top of $M_{(i,j)}$ in $\overline{\mathcal{M}}$. Then we have the matrices
$\llbracket \mathrm{F}\rrbracket$ and $\llbracket \mathrm{F}^*\rrbracket$
which are the transposes of $[\mathrm{F}^*]$ and $[\mathrm{F}]$,
respectively. Now recall that the underline category of 
${}_\mathscr{C}\mathscr{C}$ is semi-simple. Hence $[\mathrm{F}^*]$ 
is the transpose of $[\mathrm{F}]$. Consequently, 
$\llbracket \mathrm{F}\rrbracket=[\mathrm{F}]$.
Similarly, 
\begin{equation}\label{eq-s4.3-2}
\llbracket \theta\rrbracket=[\theta],
\end{equation}
for any $\theta\in\mathscr{C}$.

The last observation allows us to adopt the argument from 
\cite[Proposition~7]{MZ}. Denote by $\mathcal{N}$ the category of 
all semi-simple objects in $\overline{\mathcal{M}}$. 
For $(i,j)\in\mathbb{Z}_{\geq 0}^2$, we have
$L((i,j))\otimes_{\mathbb{C}}L((0,0))\cong L((i,j))$.
Therefore \eqref{eq-s4.3-2} implies
$L((i,j))(N_{(0,0)})\cong N_{(i,j)}$. 

For $\theta\in\mathscr{C}$, we thus can write 
$\theta(N_{(i,j)})$ as $(\theta\otimes_{\mathbb{C}} L((i,j)))\big(N_{(0,0)}\big)$
and, from the previous paragraph, it follows that 
$\theta(N_{(i,j)})\in\mathcal{N}$. Consequently,
$\mathcal{N}$ is stable under the action of $\mathscr{C}$.

In order to prove the claim of the theorem, we need to show that
the radical of ${\mathcal{M}}$ is $\mathscr{C}$-invariant.
Since $\mathscr{C}$ is monoidally generated by $\mathrm{F}$, it is
enough to prove that the radical of ${\mathcal{M}}$ is 
$\mathrm{F}$-invariant. The radical of ${\mathcal{M}}$ is
generated by the morphisms of the form $\varphi:X\to Y$, where 
$X$ and $Y$ are two non-isomorphism indecomposable objects in
$\mathcal{M}$, and the radical morphisms of the form $\psi:X\to X$, 
where  $X$ is an indecomposable objects in $\mathcal{M}$.

In the latter case, $\psi$ is nilpotent as $\mathcal{M}$ is locally
finitary. Therefore $\mathrm{F}(\psi)$ is nilpotent. Since
$\mathrm{F}(X)$ is multiplicity free (the graph in Figure~\ref{fig3}
does not have any multiple edges), it follows that $\mathrm{F}(\psi)$
is a radical morphism.

In the former case, the image of $\varphi$ belongs to the radical
of $Y$ and we need to show that the image of $\mathrm{F}(\varphi)$
belongs to the radical of $\mathrm{F}(Y)$. For this we apply 
$\mathrm{F}$ to the short exact sequence
\begin{displaymath}
0\to \mathrm{Rad}(Y)\to Y\to N_Y\to 0, 
\end{displaymath}
where  $N_Y$ is the simple top of $Y$, and obtain a short exact sequence
\begin{displaymath}
0\to \mathrm{F}(\mathrm{Rad}(Y))\to \mathrm{F}(Y)\to \mathrm{F}(N_Y)\to 0. 
\end{displaymath}
From \eqref{eq-s4.3-2}, we have that the top of 
$\mathrm{F}(Y)$ is isomorphic to $\mathrm{F}(N_Y)$. Hence
the radical of $\mathrm{F}(Y)$ is isomorphic to 
$\mathrm{F}(\mathrm{Rad}(Y))$. The claim follows.

The above shows that $\mathcal{M}$ is equivalent to $\mathcal{N}$.
The Yoneda Lemma implies that, mapping $L((0,0))$ to $N_{(0,0)}$
induces a morphism of $\mathscr{C}$-module categories 
from ${}_\mathscr{C}\mathscr{C}$ to
$\mathcal{N}$. As $L((i,j))(N_{(0,0)})\cong N_{(i,j)}$ and both
categories are semi-simple, this morphism is an equivalence.
\end{proof}

\begin{remark}\label{rem-s4.3-7}
{\em
The argument in the proof of Theorem~\ref{thm-s4.3-1} generalizes to the
following situation: Let $\mathscr{X}$ be a semi-simple rigid monoidal
category with only one $\mathcal{H}$-cell in the sense of \cite{MM1},
then every admissible simple $\mathscr{X}$-module category with the 
same action matrices as for the left regular module category is equivalent 
to the left regular module category.
}
\end{remark}

\section{$\mathscr{C}$-module categories from category $\mathcal{O}$}\label{s5}

\subsection{Category $\mathcal{O}$}\label{s5.1}

Consider the BGG category $\mathcal{O}$ for $\mathfrak{g}$, see \cite{BGG,Hu}.
For $\lambda\in\mathfrak{h}^*$, the corresponding Verma module $\Delta(\lambda)$
and its simple top $L(\lambda)$ are objects in $\mathcal{O}$. Each 
$L(\lambda)$ has an indecomposable projective cover in $\mathcal{O}$, which we
denote by $P(\lambda)$.
The set $\{L(\lambda)\,:\,\lambda\in \mathfrak{h}^*\}$ is a complete and
irredundant list of representatives of the isomorphism classes of simple
objects in $\mathcal{O}$.

Category $\mathcal{O}$ is a subcategory of $\mathcal{Z}$ and inherits from the 
latter the decomposition
\begin{displaymath}
\mathcal{O}\cong\bigoplus_{\chi}  \mathcal{O}_{{}_\chi}.
\end{displaymath}
Category $\mathcal{O}$ is stable under the action of $\mathscr{C}$
and of all projective functors.

\subsection{The integral part of $\mathcal{O}$}\label{s5.2}

Consider the full subcategory $\mathcal{O}^{\mathrm{int}}$ of $\mathcal{O}$
consisting of all modules with integral supports, that is,
consisting of all $M\in\mathcal{O}$ such that $\mathrm{supp}(M)\subset \Lambda$.
We will call $\lambda=(\lambda_1,\lambda_2)$
\begin{itemize}
\item a {\em top weight} provided that $\lambda_1,\lambda_2\geq 0$;
\item a {\em middle weight} provided that ($\lambda_1<0$
and $\lambda_2\geq 0$) or ($\lambda_1\geq 0$
and $\lambda_2< 0$);
\item a {\em bottom weight} provided that $\lambda_1,\lambda_2< 0$.
\end{itemize}
Note that 
\begin{itemize}
\item the top weights are exactly the highest weights of simple
finite dimensional modules, that is simple objects in $\mathscr{C}$;
\item the bottom weights are exactly the highest weights of 
(integral) simple Verma modules.
\end{itemize}
The dot-orbit of an integral $\lambda$ may have $6$, $3$ or $1$ element.
In the first case, $\lambda$ is {\em regular}, otherwise it is {\em singular}.
In Figure~\ref{fig4}, one can find a depiction of integral weights with 
{\color{magenta}top}, {\color{teal}middle} and 
{\color{blue}bottom} weights highlighted with different colors;
the areas with regular weights shaded, and singular weights lying on dotted lines
(the latter are the reflection hyperplanes for the dot action of $W$).

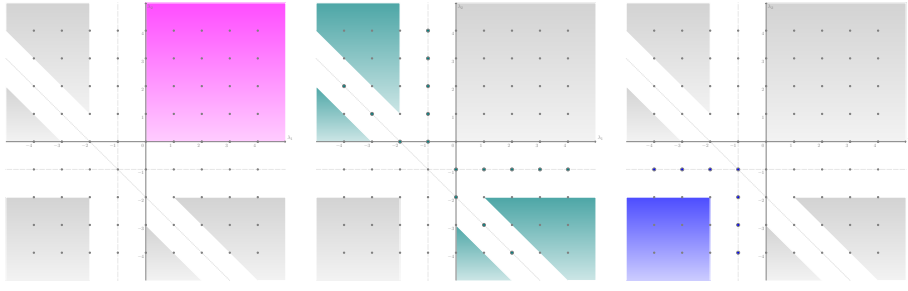
\begin{figure}
\resizebox{12cm}{!}{
\begin{tikzpicture}
\shadedraw[color=white,top color=lightgray!70,bottom color = lightgray!20]  (0,0) --
(0,6) --  (6,6) -- (6,0);
\shadedraw[color=white,top color=lightgray!70,bottom color = lightgray!20]  (0,10) --
(0,14) --  (4,10);
\shadedraw[color=white,top color=lightgray!70,bottom color = lightgray!20]  (0,18) --
(0,20) --  (6,20) -- (6,12);
\shadedraw[color=white,top color=magenta!70,bottom color = magenta!20]  (10,10) --
(10,20) --  (20,20) -- (20,10);
\shadedraw[color=white,top color=lightgray!70,bottom color = lightgray!20]  (12,6) --
(20,6) --  (20,0) -- (18,0);
\shadedraw[color=white,top color=lightgray!70,bottom color = lightgray!20]  (10,0) --
(10,4) --  (14,0);
\draw[gray, thin,  dotted] (0,8) -- (20,8);
\draw[gray, thin,  dotted] (8,0) -- (8,20);
\draw[gray, thin,  dotted] (0,16) -- (16,0);
\draw[gray, thin,  ->] (0,10) -- (20,10) node[anchor=south west] {\large$\lambda_1$};
\draw[gray, thin,  ->] (10,00) -- (10,20) node[anchor=north west] {\large$\lambda_2$};
\draw[gray,fill=gray] (10,10) circle (.3ex) node[anchor=north east] {\color{gray}$0$};
\draw[gray,fill=gray] (12,10) circle (.3ex) node[anchor=north east] {\color{gray}$1$};
\draw[gray,fill=gray] (14,10) circle (.3ex) node[anchor=north east] {\color{gray}$2$};
\draw[gray,fill=gray] (16,10) circle (.3ex) node[anchor=north east] {\color{gray}$3$};
\draw[gray,fill=gray] (18,10) circle (.3ex) node[anchor=north east] {\color{gray}$4$};
\draw[gray,fill=gray] (8,10) circle (.3ex) node[anchor=north east] {\color{gray}$-1$};
\draw[gray,fill=gray] (6,10) circle (.3ex) node[anchor=north east] {\color{gray}$-2$};
\draw[gray,fill=gray] (4,10) circle (.3ex) node[anchor=north east] {\color{gray}$-3$};
\draw[gray,fill=gray] (2,10) circle (.3ex) node[anchor=north east] {\color{gray}$-4$};
\draw[gray,fill=gray] (10,12) circle (.3ex) node[anchor=north east] {\color{gray}$1$};
\draw[gray,fill=gray] (10,14) circle (.3ex) node[anchor=north east] {\color{gray}$2$};
\draw[gray,fill=gray] (10,16) circle (.3ex) node[anchor=north east] {\color{gray}$3$};
\draw[gray,fill=gray] (10,18) circle (.3ex) node[anchor=north east] {\color{gray}$4$};
\draw[gray,fill=gray] (10,8) circle (.3ex) node[anchor=north east] {\color{gray}$-1$};
\draw[gray,fill=gray] (10,6) circle (.3ex) node[anchor=north east] {\color{gray}$-2$};
\draw[gray,fill=gray] (10,4) circle (.3ex) node[anchor=north east] {\color{gray}$-3$};
\draw[gray,fill=gray] (10,2) circle (.3ex) node[anchor=north east] {\color{gray}$-4$};
\draw[gray,fill=gray] (2,2) circle (.3ex);
\draw[gray,fill=gray] (2,4) circle (.3ex);
\draw[gray,fill=gray] (2,6) circle (.3ex);
\draw[gray,fill=gray] (2,8) circle (.3ex);
\draw[gray,fill=gray] (2,10) circle (.3ex);
\draw[gray,fill=gray] (2,12) circle (.3ex);
\draw[gray,fill=gray] (2,14) circle (.3ex);
\draw[gray,fill=gray] (2,16) circle (.3ex);
\draw[gray,fill=gray] (2,18) circle (.3ex);
\draw[gray,fill=gray] (4,2) circle (.3ex);
\draw[gray,fill=gray] (4,4) circle (.3ex);
\draw[gray,fill=gray] (4,6) circle (.3ex);
\draw[gray,fill=gray] (4,8) circle (.3ex);
\draw[gray,fill=gray] (4,10) circle (.3ex);
\draw[gray,fill=gray] (4,12) circle (.3ex);
\draw[gray,fill=gray] (4,14) circle (.3ex);
\draw[gray,fill=gray] (4,16) circle (.3ex);
\draw[gray,fill=gray] (4,18) circle (.3ex);
\draw[gray,fill=gray] (6,2) circle (.3ex);
\draw[gray,fill=gray] (6,4) circle (.3ex);
\draw[gray,fill=gray] (6,6) circle (.3ex);
\draw[gray,fill=gray] (6,8) circle (.3ex);
\draw[gray,fill=gray] (6,10) circle (.3ex);
\draw[gray,fill=gray] (6,12) circle (.3ex);
\draw[gray,fill=gray] (6,14) circle (.3ex);
\draw[gray,fill=gray] (6,16) circle (.3ex);
\draw[gray,fill=gray] (6,18) circle (.3ex);
\draw[gray,fill=gray] (8,2) circle (.3ex);
\draw[gray,fill=gray] (8,4) circle (.3ex);
\draw[gray,fill=gray] (8,6) circle (.3ex);
\draw[gray,fill=gray] (8,8) circle (.3ex);
\draw[gray,fill=gray] (8,10) circle (.3ex);
\draw[gray,fill=gray] (8,12) circle (.3ex);
\draw[gray,fill=gray] (8,14) circle (.3ex);
\draw[gray,fill=gray] (8,16) circle (.3ex);
\draw[gray,fill=gray] (8,18) circle (.3ex);
\draw[gray,fill=gray] (12,2) circle (.3ex);
\draw[gray,fill=gray] (12,4) circle (.3ex);
\draw[gray,fill=gray] (12,6) circle (.3ex);
\draw[gray,fill=gray] (12,8) circle (.3ex);
\draw[gray,fill=gray] (12,10) circle (.3ex);
\draw[gray,fill=gray] (12,12) circle (.3ex);
\draw[gray,fill=gray] (12,14) circle (.3ex);
\draw[gray,fill=gray] (12,16) circle (.3ex);
\draw[gray,fill=gray] (12,18) circle (.3ex);
\draw[gray,fill=gray] (14,2) circle (.3ex);
\draw[gray,fill=gray] (14,4) circle (.3ex);
\draw[gray,fill=gray] (14,6) circle (.3ex);
\draw[gray,fill=gray] (14,8) circle (.3ex);
\draw[gray,fill=gray] (14,10) circle (.3ex);
\draw[gray,fill=gray] (14,12) circle (.3ex);
\draw[gray,fill=gray] (14,14) circle (.3ex);
\draw[gray,fill=gray] (14,16) circle (.3ex);
\draw[gray,fill=gray] (14,18) circle (.3ex);
\draw[gray,fill=gray] (16,2) circle (.3ex);
\draw[gray,fill=gray] (16,4) circle (.3ex);
\draw[gray,fill=gray] (16,6) circle (.3ex);
\draw[gray,fill=gray] (16,8) circle (.3ex);
\draw[gray,fill=gray] (16,10) circle (.3ex);
\draw[gray,fill=gray] (16,12) circle (.3ex);
\draw[gray,fill=gray] (16,14) circle (.3ex);
\draw[gray,fill=gray] (16,16) circle (.3ex);
\draw[gray,fill=gray] (16,18) circle (.3ex);
\draw[gray,fill=gray] (18,2) circle (.3ex);
\draw[gray,fill=gray] (18,4) circle (.3ex);
\draw[gray,fill=gray] (18,6) circle (.3ex);
\draw[gray,fill=gray] (18,8) circle (.3ex);
\draw[gray,fill=gray] (18,10) circle (.3ex);
\draw[gray,fill=gray] (18,12) circle (.3ex);
\draw[gray,fill=gray] (18,14) circle (.3ex);
\draw[gray,fill=gray] (18,16) circle (.3ex);
\draw[gray,fill=gray] (18,18) circle (.3ex);
\end{tikzpicture}
\qquad\qquad
\begin{tikzpicture}
\shadedraw[color=white,top color=lightgray!70,bottom color = lightgray!20]  (0,0) --
(0,6) --  (6,6) -- (6,0);
\shadedraw[color=white,top color=teal!70,bottom color = teal!20]  (0,10) --
(0,14) --  (4,10);
\shadedraw[color=white,top color=teal!70,bottom color = teal!20]  (0,18) --
(0,20) --  (6,20) -- (6,12);
\shadedraw[color=white,top color=lightgray!70,bottom color = lightgray!20]  (10,10) --
(10,20) --  (20,20) -- (20,10);
\shadedraw[color=white,top color=teal!70,bottom color = teal!20]  (12,6) --
(20,6) --  (20,0) -- (18,0);
\shadedraw[color=white,top color=teal!70,bottom color = teal!20]  (10,0) --
(10,4) --  (14,0);
\draw[gray, thin,  dotted] (0,8) -- (20,8);
\draw[gray, thin,  dotted] (8,0) -- (8,20);
\draw[gray, thin,  dotted] (0,16) -- (16,0);
\draw[gray, thin,  ->] (0,10) -- (20,10) node[anchor=south west] {\large$\lambda_1$};
\draw[gray, thin,  ->] (10,00) -- (10,20) node[anchor=north west] {\large$\lambda_2$};
\draw[gray,fill=gray] (10,10) circle (.3ex) node[anchor=north east] {\color{gray}$0$};
\draw[gray,fill=gray] (12,10) circle (.3ex) node[anchor=north east] {\color{gray}$1$};
\draw[gray,fill=gray] (14,10) circle (.3ex) node[anchor=north east] {\color{gray}$2$};
\draw[gray,fill=gray] (16,10) circle (.3ex) node[anchor=north east] {\color{gray}$3$};
\draw[gray,fill=gray] (18,10) circle (.3ex) node[anchor=north east] {\color{gray}$4$};
\draw[gray,fill=gray] (8,10) circle (.3ex) node[anchor=north east] {\color{gray}$-1$};
\draw[gray,fill=gray] (6,10) circle (.3ex) node[anchor=north east] {\color{gray}$-2$};
\draw[gray,fill=gray] (4,10) circle (.3ex) node[anchor=north east] {\color{gray}$-3$};
\draw[gray,fill=gray] (2,10) circle (.3ex) node[anchor=north east] {\color{gray}$-4$};
\draw[gray,fill=gray] (10,12) circle (.3ex) node[anchor=north east] {\color{gray}$1$};
\draw[gray,fill=gray] (10,14) circle (.3ex) node[anchor=north east] {\color{gray}$2$};
\draw[gray,fill=gray] (10,16) circle (.3ex) node[anchor=north east] {\color{gray}$3$};
\draw[gray,fill=gray] (10,18) circle (.3ex) node[anchor=north east] {\color{gray}$4$};
\draw[gray,fill=teal] (10,8) circle (.7ex) node[anchor=north east] {\color{gray}$-1$};
\draw[gray,fill=teal] (10,6) circle (.7ex) node[anchor=north east] {\color{gray}$-2$};
\draw[gray,fill=gray] (10,4) circle (.3ex) node[anchor=north east] {\color{gray}$-3$};
\draw[gray,fill=gray] (10,2) circle (.3ex) node[anchor=north east] {\color{gray}$-4$};
\draw[gray,fill=gray] (2,2) circle (.3ex);
\draw[gray,fill=gray] (2,4) circle (.3ex);
\draw[gray,fill=gray] (2,6) circle (.3ex);
\draw[gray,fill=gray] (2,8) circle (.3ex);
\draw[gray,fill=gray] (2,10) circle (.3ex);
\draw[gray,fill=gray] (2,12) circle (.3ex);
\draw[gray,fill=teal] (2,14) circle (.7ex);
\draw[gray,fill=gray] (2,16) circle (.3ex);
\draw[gray,fill=gray] (2,18) circle (.3ex);
\draw[gray,fill=gray] (4,2) circle (.3ex);
\draw[gray,fill=gray] (4,4) circle (.3ex);
\draw[gray,fill=gray] (4,6) circle (.3ex);
\draw[gray,fill=gray] (4,8) circle (.3ex);
\draw[gray,fill=gray] (4,10) circle (.3ex);
\draw[gray,fill=teal] (4,12) circle (.7ex);
\draw[gray,fill=gray] (4,14) circle (.3ex);
\draw[gray,fill=gray] (4,16) circle (.3ex);
\draw[gray,fill=gray] (4,18) circle (.3ex);
\draw[gray,fill=gray] (6,2) circle (.3ex);
\draw[gray,fill=gray] (6,4) circle (.3ex);
\draw[gray,fill=gray] (6,6) circle (.3ex);
\draw[gray,fill=gray] (6,8) circle (.3ex);
\draw[gray,fill=teal] (6,10) circle (.7ex);
\draw[gray,fill=gray] (6,12) circle (.3ex);
\draw[gray,fill=gray] (6,14) circle (.3ex);
\draw[gray,fill=gray] (6,16) circle (.3ex);
\draw[gray,fill=gray] (6,18) circle (.3ex);
\draw[gray,fill=gray] (8,2) circle (.3ex);
\draw[gray,fill=gray] (8,4) circle (.3ex);
\draw[gray,fill=gray] (8,6) circle (.3ex);
\draw[gray,fill=gray] (8,8) circle (.3ex);
\draw[gray,fill=teal] (8,10) circle (.7ex);
\draw[gray,fill=teal] (8,12) circle (.7ex);
\draw[gray,fill=teal] (8,14) circle (.7ex);
\draw[gray,fill=teal] (8,16) circle (.7ex);
\draw[gray,fill=teal] (8,18) circle (.7ex);
\draw[gray,fill=gray] (12,2) circle (.3ex);
\draw[gray,fill=teal] (12,4) circle (.7ex);
\draw[gray,fill=gray] (12,6) circle (.3ex);
\draw[gray,fill=teal] (12,8) circle (.7ex);
\draw[gray,fill=gray] (12,10) circle (.3ex);
\draw[gray,fill=gray] (12,12) circle (.3ex);
\draw[gray,fill=gray] (12,14) circle (.3ex);
\draw[gray,fill=gray] (12,16) circle (.3ex);
\draw[gray,fill=gray] (12,18) circle (.3ex);
\draw[gray,fill=teal] (14,2) circle (.7ex);
\draw[gray,fill=gray] (14,4) circle (.3ex);
\draw[gray,fill=gray] (14,6) circle (.3ex);
\draw[gray,fill=teal] (14,8) circle (.7ex);
\draw[gray,fill=gray] (14,10) circle (.3ex);
\draw[gray,fill=gray] (14,12) circle (.3ex);
\draw[gray,fill=gray] (14,14) circle (.3ex);
\draw[gray,fill=gray] (14,16) circle (.3ex);
\draw[gray,fill=gray] (14,18) circle (.3ex);
\draw[gray,fill=gray] (16,2) circle (.3ex);
\draw[gray,fill=gray] (16,4) circle (.3ex);
\draw[gray,fill=gray] (16,6) circle (.3ex);
\draw[gray,fill=teal] (16,8) circle (.7ex);
\draw[gray,fill=gray] (16,10) circle (.3ex);
\draw[gray,fill=gray] (16,12) circle (.3ex);
\draw[gray,fill=gray] (16,14) circle (.3ex);
\draw[gray,fill=gray] (16,16) circle (.3ex);
\draw[gray,fill=gray] (16,18) circle (.3ex);
\draw[gray,fill=gray] (18,2) circle (.3ex);
\draw[gray,fill=gray] (18,4) circle (.3ex);
\draw[gray,fill=gray] (18,6) circle (.3ex);
\draw[gray,fill=teal] (18,8) circle (.7ex);
\draw[gray,fill=gray] (18,10) circle (.3ex);
\draw[gray,fill=gray] (18,12) circle (.3ex);
\draw[gray,fill=gray] (18,14) circle (.3ex);
\draw[gray,fill=gray] (18,16) circle (.3ex);
\draw[gray,fill=gray] (18,18) circle (.3ex);
\end{tikzpicture}
\qquad\qquad
\begin{tikzpicture}
\shadedraw[color=white,top color=blue!70,bottom color = blue!20]  (0,0) --
(0,6) --  (6,6) -- (6,0);
\shadedraw[color=white,top color=lightgray!70,bottom color = lightgray!20]  (0,10) --
(0,14) --  (4,10);
\shadedraw[color=white,top color=lightgray!70,bottom color = lightgray!20]  (0,18) --
(0,20) --  (6,20) -- (6,12);
\shadedraw[color=white,top color=lightgray!70,bottom color = lightgray!20]  (10,10) --
(10,20) --  (20,20) -- (20,10);
\shadedraw[color=white,top color=lightgray!70,bottom color = lightgray!20]  (12,6) --
(20,6) --  (20,0) -- (18,0);
\shadedraw[color=white,top color=lightgray!70,bottom color = lightgray!20]  (10,0) --
(10,4) --  (14,0);
\draw[gray, thin,  dotted] (0,8) -- (20,8);
\draw[gray, thin,  dotted] (8,0) -- (8,20);
\draw[gray, thin,  dotted] (0,16) -- (16,0);
\draw[gray, thin,  ->] (0,10) -- (20,10) node[anchor=south west] {\large$\lambda_1$};
\draw[gray, thin,  ->] (10,00) -- (10,20) node[anchor=north west] {\large$\lambda_2$};
\draw[gray,fill=gray] (10,10) circle (.3ex) node[anchor=north east] {\color{gray}$0$};
\draw[gray,fill=gray] (12,10) circle (.3ex) node[anchor=north east] {\color{gray}$1$};
\draw[gray,fill=gray] (14,10) circle (.3ex) node[anchor=north east] {\color{gray}$2$};
\draw[gray,fill=gray] (16,10) circle (.3ex) node[anchor=north east] {\color{gray}$3$};
\draw[gray,fill=gray] (18,10) circle (.3ex) node[anchor=north east] {\color{gray}$4$};
\draw[gray,fill=gray] (8,10) circle (.3ex) node[anchor=north east] {\color{gray}$-1$};
\draw[gray,fill=gray] (6,10) circle (.3ex) node[anchor=north east] {\color{gray}$-2$};
\draw[gray,fill=gray] (4,10) circle (.3ex) node[anchor=north east] {\color{gray}$-3$};
\draw[gray,fill=gray] (2,10) circle (.3ex) node[anchor=north east] {\color{gray}$-4$};
\draw[gray,fill=gray] (10,12) circle (.3ex) node[anchor=north east] {\color{gray}$1$};
\draw[gray,fill=gray] (10,14) circle (.3ex) node[anchor=north east] {\color{gray}$2$};
\draw[gray,fill=gray] (10,16) circle (.3ex) node[anchor=north east] {\color{gray}$3$};
\draw[gray,fill=gray] (10,18) circle (.3ex) node[anchor=north east] {\color{gray}$4$};
\draw[gray,fill=gray] (10,8) circle (.3ex) node[anchor=north east] {\color{gray}$-1$};
\draw[gray,fill=gray] (10,6) circle (.3ex) node[anchor=north east] {\color{gray}$-2$};
\draw[gray,fill=gray] (10,4) circle (.3ex) node[anchor=north east] {\color{gray}$-3$};
\draw[gray,fill=gray] (10,2) circle (.3ex) node[anchor=north east] {\color{gray}$-4$};
\draw[gray,fill=gray] (2,2) circle (.3ex);
\draw[gray,fill=gray] (2,4) circle (.3ex);
\draw[gray,fill=gray] (2,6) circle (.3ex);
\draw[gray,fill=blue] (2,8) circle (.7ex);
\draw[gray,fill=gray] (2,10) circle (.3ex);
\draw[gray,fill=gray] (2,12) circle (.3ex);
\draw[gray,fill=gray] (2,14) circle (.3ex);
\draw[gray,fill=gray] (2,16) circle (.3ex);
\draw[gray,fill=gray] (2,18) circle (.3ex);
\draw[gray,fill=gray] (4,2) circle (.3ex);
\draw[gray,fill=gray] (4,4) circle (.3ex);
\draw[gray,fill=gray] (4,6) circle (.3ex);
\draw[gray,fill=blue] (4,8) circle (.7ex);
\draw[gray,fill=gray] (4,10) circle (.3ex);
\draw[gray,fill=gray] (4,12) circle (.3ex);
\draw[gray,fill=gray] (4,14) circle (.3ex);
\draw[gray,fill=gray] (4,16) circle (.3ex);
\draw[gray,fill=gray] (4,18) circle (.3ex);
\draw[gray,fill=gray] (6,2) circle (.3ex);
\draw[gray,fill=gray] (6,4) circle (.3ex);
\draw[gray,fill=gray] (6,6) circle (.3ex);
\draw[gray,fill=blue] (6,8) circle (.7ex);
\draw[gray,fill=gray] (6,10) circle (.3ex);
\draw[gray,fill=gray] (6,12) circle (.3ex);
\draw[gray,fill=gray] (6,14) circle (.3ex);
\draw[gray,fill=gray] (6,16) circle (.3ex);
\draw[gray,fill=gray] (6,18) circle (.3ex);
\draw[gray,fill=blue] (8,2) circle (.7ex);
\draw[gray,fill=blue] (8,4) circle (.7ex);
\draw[gray,fill=blue] (8,6) circle (.7ex);
\draw[gray,fill=blue] (8,8) circle (.7ex);
\draw[gray,fill=gray] (8,10) circle (.3ex);
\draw[gray,fill=gray] (8,12) circle (.3ex);
\draw[gray,fill=gray] (8,14) circle (.3ex);
\draw[gray,fill=gray] (8,16) circle (.3ex);
\draw[gray,fill=gray] (8,18) circle (.3ex);
\draw[gray,fill=gray] (12,2) circle (.3ex);
\draw[gray,fill=gray] (12,4) circle (.3ex);
\draw[gray,fill=gray] (12,6) circle (.3ex);
\draw[gray,fill=gray] (12,8) circle (.3ex);
\draw[gray,fill=gray] (12,10) circle (.3ex);
\draw[gray,fill=gray] (12,12) circle (.3ex);
\draw[gray,fill=gray] (12,14) circle (.3ex);
\draw[gray,fill=gray] (12,16) circle (.3ex);
\draw[gray,fill=gray] (12,18) circle (.3ex);
\draw[gray,fill=gray] (14,2) circle (.3ex);
\draw[gray,fill=gray] (14,4) circle (.3ex);
\draw[gray,fill=gray] (14,6) circle (.3ex);
\draw[gray,fill=gray] (14,8) circle (.3ex);
\draw[gray,fill=gray] (14,10) circle (.3ex);
\draw[gray,fill=gray] (14,12) circle (.3ex);
\draw[gray,fill=gray] (14,14) circle (.3ex);
\draw[gray,fill=gray] (14,16) circle (.3ex);
\draw[gray,fill=gray] (14,18) circle (.3ex);
\draw[gray,fill=gray] (16,2) circle (.3ex);
\draw[gray,fill=gray] (16,4) circle (.3ex);
\draw[gray,fill=gray] (16,6) circle (.3ex);
\draw[gray,fill=gray] (16,8) circle (.3ex);
\draw[gray,fill=gray] (16,10) circle (.3ex);
\draw[gray,fill=gray] (16,12) circle (.3ex);
\draw[gray,fill=gray] (16,14) circle (.3ex);
\draw[gray,fill=gray] (16,16) circle (.3ex);
\draw[gray,fill=gray] (16,18) circle (.3ex);
\draw[gray,fill=gray] (18,2) circle (.3ex);
\draw[gray,fill=gray] (18,4) circle (.3ex);
\draw[gray,fill=gray] (18,6) circle (.3ex);
\draw[gray,fill=gray] (18,8) circle (.3ex);
\draw[gray,fill=gray] (18,10) circle (.3ex);
\draw[gray,fill=gray] (18,12) circle (.3ex);
\draw[gray,fill=gray] (18,14) circle (.3ex);
\draw[gray,fill=gray] (18,16) circle (.3ex);
\draw[gray,fill=gray] (18,18) circle (.3ex);
\end{tikzpicture}
}
\caption{{\color{magenta}Top}, {\color{teal}middle} and 
{\color{blue}bottom} integral weights}\label{fig4}
\end{figure}

For a regular, integral and dominant $\lambda$ and $w\in W$,
denote $\theta_{\lambda,w\cdot \lambda}$ simply by $\theta_w$.

For each integral $\lambda$, define a module $N(\lambda)$ as follows:
\begin{itemize}
\item $N(\lambda)=L(\lambda)$, if $\lambda$ is a top weight
or a singular middle weight;
\item $N(\lambda)=\theta_{t}L(\lambda)$,
where $t\in\{r,s\}$ is the unique element such that
$\theta_{t}L(\lambda)\neq 0$, if $\lambda$ is a regular middle weight; 
\item $N(\lambda)=P(\lambda)$, if $\lambda$ is a bottom weight.
\end{itemize}

If $\lambda$ is a top weight, then $\mathrm{add}(\mathscr{C}\cdot L(\lambda))=\mathscr{C}$
and this case has already been studied in Section~\ref{s4}. Hence in what
follows we concentrate on middle and bottom $\lambda$.

\subsection{Upper middle weights}\label{s5.3}

A middle weight $\lambda$ will be called an upper middle weight
provided that $\lambda_2\geq 0$ and a lower middle weight
provided that $\lambda_2<0$. In this subsection we fix
an upper middle weight $\lambda$.

Denote by $\mathcal{N}_1$ the additive closure of all 
$N(\mu)$, where $\mu$ is an upper middle weight.

\begin{proposition}\label{prop-s5.3-1}
If $\lambda$ is a singular upper middle weight, then 
$\mathrm{add}(\mathscr{C}\cdot L(\lambda))$ coincides
with $\mathcal{N}_1$, moreover, the latter is a simple
transitive $\mathscr{C}$-module category whose graph
and positive eigenvector are depicted in Figure~\ref{fig5}.
\end{proposition}

\begin{figure}
\resizebox{10cm}{!}{
\begin{tikzpicture}
\draw[black, thick,  ->] (0.2,0) -- (1.8,0);
\draw[black, thick,  ->] (0.2,2) -- (1.8,2);
\draw[black, thick,  ->] (0.2,6.1) to [out=40,in=140]  (1.8,6.1);
\draw[black, thick,  ->] (0.2,5.9) to [out=-40,in=-140]  (1.8,5.9);
\draw[black, thick,  ->] (0.2,4) -- (1.8,4);
\draw[black, thick,  ->] (8,9.8) -- (8,8.2);
\draw[black, thick,  ->] (6,9.8) -- (6,8.2);
\draw[black, thick,  ->] (4,9.8) -- (4,8.2);
\draw[black, thick,  ->] (2,9.8) -- (2,8.2);
\draw[black, thick,  ->] (10,9.8) -- (10,8.2);
\draw[black, thick,  ->] (2.2,0) -- (3.8,0);
\draw[black, thick,  ->] (2.2,2) -- (3.8,2);
\draw[black, thick,  ->] (2.2,8) -- (3.8,8);
\draw[black, thick,  ->] (2.2,4.1) to [out=40,in=140]  (3.8,4.1);
\draw[black, thick,  ->] (2.2,3.9) to [out=-40,in=-140]  (3.8,3.9);
\draw[black, thick,  ->] (4.2,0) -- (5.8,0);
\draw[black, thick,  ->] (4.2,6) -- (5.8,6);
\draw[black, thick,  ->] (4.2,8) -- (5.8,8);
\draw[black, thick,  ->] (4.2,2.1) to [out=40,in=140]  (5.8,2.1);
\draw[black, thick,  ->] (4.2,1.9) to [out=-40,in=-140]  (5.8,1.9);
\draw[black, thick,  ->] (6.2,4) -- (7.8,4);
\draw[black, thick,  ->] (6.2,6) -- (7.8,6);
\draw[black, thick,  ->] (6.2,8) -- (7.8,8);
\draw[black, thick,  ->] (6.2,0.1) to [out=40,in=140]  (7.8,0.1);
\draw[black, thick,  ->] (6.2,-0.1) to [out=-40,in=-140]  (7.8,-0.1);
\draw[black, thick,  ->] (8.2,0) -- (9.8,0);
\draw[black, thick,  ->] (8.2,2.1) to [out=40,in=140]  (9.8,2.1);
\draw[black, thick,  ->] (8.2,1.9) to [out=-40,in=-140]  (9.8,1.9);
\draw[black, thick,  ->] (8.2,4.1) to [out=40,in=140]  (9.8,4.1);
\draw[black, thick,  ->] (8.2,3.9) to [out=-40,in=-140]  (9.8,3.9);
\draw[black, thick,  ->] (8.2,6.1) to [out=40,in=140]  (9.8,6.1);
\draw[black, thick,  ->] (8.2,5.9) to [out=-40,in=-140]  (9.8,5.9);
\draw[black, thick,  ->] (8.2,8.1) to [out=40,in=140]  (9.8,8.1);
\draw[black, thick,  ->] (8.2,7.9) to [out=-40,in=-140]  (9.8,7.9);
\draw[black, thick,  ->] (2,1.8) -- (2,0.2);
\draw[black, thick,  ->] (4,1.8) -- (4,0.2);
\draw[black, thick,  ->] (6,1.8) -- (6,0.2);
\draw[black, thick,  ->] (8,1.8) -- (8,0.2);
\draw[black, thick,  ->] (10,1.8) -- (10,0.2);
\draw[black, thick,  ->] (2,3.8) -- (2,2.2);
\draw[black, thick,  ->] (4,3.8) -- (4,2.2);
\draw[black, thick,  ->] (6,3.8) -- (6,2.2);
\draw[black, thick,  ->] (8,3.8) -- (8,2.2);
\draw[black, thick,  ->] (10,3.8) -- (10,2.2);
\draw[black, thick,  ->] (2,5.8) -- (2,4.2);
\draw[black, thick,  ->] (4,5.8) -- (4,4.2);
\draw[black, thick,  ->] (6,5.8) -- (6,4.2);
\draw[black, thick,  ->] (8,5.8) -- (8,4.2);
\draw[black, thick,  ->] (10,5.8) -- (10,4.2);
\draw[black, thick,  ->] (2,7.8) -- (2,6.2);
\draw[black, thick,  ->] (4,7.8) -- (4,6.2);
\draw[black, thick,  ->] (6,7.8) -- (6,6.2);
\draw[black, thick,  ->] (8,7.8) -- (8,6.2);
\draw[black, thick,  ->] (10,7.8) -- (10,6.2);
\draw[black, thick,  ->] (1.8,0.2) -- (0.2,1.8);
\draw[black, thick,  ->] (3.8,0.2) -- (2.2,1.8);
\draw[black, thick,  ->] (5.8,0.2) -- (4.2,1.8);
\draw[black, thick,  ->] (7.8,0.2) -- (6.2,1.8);
\draw[black, thick,  ->] (9.8,0.2) -- (8.2,1.8);
\draw[black, thick,  ->] (1.8,2.2) -- (0.2,3.8);
\draw[black, thick,  ->] (3.8,2.2) -- (2.2,3.8);
\draw[black, thick,  ->] (5.8,2.2) -- (4.2,3.8);
\draw[black, thick,  ->] (7.8,2.2) -- (6.2,3.8);
\draw[black, thick,  ->] (9.8,2.2) -- (8.2,3.8);
\draw[black, thick,  ->] (1.8,4.2) -- (0.2,5.8);
\draw[black, thick,  ->] (3.8,4.2) -- (2.2,5.8);
\draw[black, thick,  ->] (5.8,4.2) -- (4.2,5.8);
\draw[black, thick,  ->] (7.8,4.2) -- (6.2,5.8);
\draw[black, thick,  ->] (9.8,4.2) -- (8.2,5.8);
\draw[black, thick,  ->] (1.8,6.2) -- (0.2,7.8);
\draw[black, thick,  ->] (3.8,6.2) -- (2.2,7.8);
\draw[black, thick,  ->] (5.8,6.2) -- (4.2,7.8);
\draw[black, thick,  ->] (7.8,6.2) -- (6.2,7.8);
\draw[black, thick,  ->] (9.8,6.2) -- (8.2,7.8);
\draw[black, thick,  ->] (1.8,8.2) -- (0.2,9.8);
\draw[black, thick,  ->] (3.8,8.2) -- (2.2,9.8);
\draw[black, thick,  ->] (5.8,8.2) -- (4.2,9.8);
\draw[black, thick,  ->] (7.8,8.2) -- (6.2,9.8);
\draw[black, thick,  ->] (9.8,8.2) -- (8.2,9.8);
\draw[black, thick,  ->] (4.2,0.2) to [out=60,in=230]  (9.8,3.6);
\draw[black, thick,  ->] (6.2,0.3) to [out=50,in=230]  (9.8,1.6);
\draw[black, thick,  ->] (2.2,0.3) to [out=70,in=240]  (9.8,5.6);
\draw[black, thick,  ->] (0.2,0.3) to [out=70,in=240]  (9.8,7.6);
\draw[black, thin,  -] (8,9.8) -- (8,9.8) node[anchor= south] {\large$\vdots$};
\draw[black, thin,  -] (6,9.8) -- (6,9.8) node[anchor= south] {\large$\vdots$};
\draw[black, thin,  -] (4,9.8) -- (4,9.8) node[anchor= south] {\large$\vdots$};
\draw[black, thin,  -] (2,9.8) -- (2,9.8) node[anchor= south] {\large$\vdots$};
\draw[black, thin,  -] (10,9.8) -- (10,9.8) node[anchor= south] {\large$\vdots$};
\draw[black, thin,  -] (0,0) -- (0,0) node[anchor= east] {\large$\cdots$};
\draw[black, thin,  -] (0,2) -- (0,2) node[anchor= east] {\large$\cdots$};
\draw[black, thin,  -] (0,4) -- (0,4) node[anchor= east] {\large$\cdots$};
\draw[black, thin,  -] (0,6) -- (0,6) node[anchor= east] {\large$\cdots$};
\draw[black, thin,  -] (0,8) -- (0,8) node[anchor= east] {\large$\cdots$};
\draw[black, thin,  -] (0,10) -- (0,10) node[anchor= south east] {\large$\ddots$};
\draw[gray,fill=gray] (10,0) circle (.9ex);
\draw[gray,fill=gray] (2,0) circle (.9ex);
\draw[gray,fill=gray] (4,0) circle (.9ex);
\draw[gray,fill=gray] (6,0) circle (.9ex);
\draw[gray,fill=gray] (8,0) circle (.9ex);
\draw[gray,fill=gray] (10,2) circle (.9ex);
\draw[gray,fill=gray] (2,2) circle (.9ex);
\draw[gray,fill=gray] (4,2) circle (.9ex);
\draw[gray,fill=gray] (6,2) circle (.9ex);
\draw[gray,fill=gray] (8,2) circle (.9ex);
\draw[gray,fill=gray] (10,4) circle (.9ex);
\draw[gray,fill=gray] (2,4) circle (.9ex);
\draw[gray,fill=gray] (4,4) circle (.9ex);
\draw[gray,fill=gray] (6,4) circle (.9ex);
\draw[gray,fill=gray] (8,4) circle (.9ex);
\draw[gray,fill=gray] (10,6) circle (.9ex);
\draw[gray,fill=gray] (2,6) circle (.9ex);
\draw[gray,fill=gray] (4,6) circle (.9ex);
\draw[gray,fill=gray] (6,6) circle (.9ex);
\draw[gray,fill=gray] (8,6) circle (.9ex);
\draw[gray,fill=gray] (10,8) circle (.9ex);
\draw[gray,fill=gray] (2,8) circle (.9ex);
\draw[gray,fill=gray] (4,8) circle (.9ex);
\draw[gray,fill=gray] (6,8) circle (.9ex);
\draw[gray,fill=gray] (8,8) circle (.9ex);
\end{tikzpicture}
\qquad\qquad
\begin{tikzpicture}
\draw[lightgray, thin,  ->] (0.2,0) -- (1.8,0);
\draw[lightgray, thin,  ->] (0.2,2) -- (1.8,2);
\draw[lightgray, thin,  ->] (0.2,6.1) to [out=40,in=140]  (1.8,6.1);
\draw[lightgray, thin,  ->] (0.2,5.9) to [out=-40,in=-140]  (1.8,5.9);
\draw[lightgray, thin,  ->] (0.2,4) -- (1.8,4);
\draw[lightgray, thin,  ->] (8,9.8) -- (8,8.2);
\draw[lightgray, thin,  ->] (6,9.8) -- (6,8.2);
\draw[lightgray, thin,  ->] (4,9.8) -- (4,8.2);
\draw[lightgray, thin,  ->] (2,9.8) -- (2,8.2);
\draw[lightgray, thin,  ->] (10,9.8) -- (10,8.2);
\draw[lightgray, thin,  ->] (2.2,0) -- (3.8,0);
\draw[lightgray, thin,  ->] (2.2,2) -- (3.8,2);
\draw[lightgray, thin,  ->] (2.2,8) -- (3.8,8);
\draw[lightgray, thin,  ->] (2.2,4.1) to [out=40,in=140]  (3.8,4.1);
\draw[lightgray, thin,  ->] (2.2,3.9) to [out=-40,in=-140]  (3.8,3.9);
\draw[lightgray, thin,  ->] (4.2,0) -- (5.8,0);
\draw[lightgray, thin,  ->] (4.2,6) -- (5.8,6);
\draw[lightgray, thin,  ->] (4.2,8) -- (5.8,8);
\draw[lightgray, thin,  ->] (4.2,2.1) to [out=40,in=140]  (5.8,2.1);
\draw[lightgray, thin,  ->] (4.2,1.9) to [out=-40,in=-140]  (5.8,1.9);
\draw[lightgray, thin,  ->] (6.2,4) -- (7.8,4);
\draw[lightgray, thin,  ->] (6.2,6) -- (7.8,6);
\draw[lightgray, thin,  ->] (6.2,8) -- (7.8,8);
\draw[lightgray, thin,  ->] (6.2,0.1) to [out=40,in=140]  (7.8,0.1);
\draw[lightgray, thin,  ->] (6.2,-0.1) to [out=-40,in=-140]  (7.8,-0.1);
\draw[lightgray, thin,  ->] (8.2,0) -- (9.8,0);
\draw[lightgray, thin,  ->] (8.2,2.1) to [out=40,in=140]  (9.8,2.1);
\draw[lightgray, thin,  ->] (8.2,1.9) to [out=-40,in=-140]  (9.8,1.9);
\draw[lightgray, thin,  ->] (8.2,4.1) to [out=40,in=140]  (9.8,4.1);
\draw[lightgray, thin,  ->] (8.2,3.9) to [out=-40,in=-140]  (9.8,3.9);
\draw[lightgray, thin,  ->] (8.2,6.1) to [out=40,in=140]  (9.8,6.1);
\draw[lightgray, thin,  ->] (8.2,5.9) to [out=-40,in=-140]  (9.8,5.9);
\draw[lightgray, thin,  ->] (8.2,8.1) to [out=40,in=140]  (9.8,8.1);
\draw[lightgray, thin,  ->] (8.2,7.9) to [out=-40,in=-140]  (9.8,7.9);
\draw[lightgray, thin,  ->] (2,1.8) -- (2,0.2);
\draw[lightgray, thin,  ->] (4,1.8) -- (4,0.2);
\draw[lightgray, thin,  ->] (6,1.8) -- (6,0.2);
\draw[lightgray, thin,  ->] (8,1.8) -- (8,0.2);
\draw[lightgray, thin,  ->] (10,1.8) -- (10,0.2);
\draw[lightgray, thin,  ->] (2,3.8) -- (2,2.2);
\draw[lightgray, thin,  ->] (4,3.8) -- (4,2.2);
\draw[lightgray, thin,  ->] (6,3.8) -- (6,2.2);
\draw[lightgray, thin,  ->] (8,3.8) -- (8,2.2);
\draw[lightgray, thin,  ->] (10,3.8) -- (10,2.2);
\draw[lightgray, thin,  ->] (2,5.8) -- (2,4.2);
\draw[lightgray, thin,  ->] (4,5.8) -- (4,4.2);
\draw[lightgray, thin,  ->] (6,5.8) -- (6,4.2);
\draw[lightgray, thin,  ->] (8,5.8) -- (8,4.2);
\draw[lightgray, thin,  ->] (10,5.8) -- (10,4.2);
\draw[lightgray, thin,  ->] (2,7.8) -- (2,6.2);
\draw[lightgray, thin,  ->] (4,7.8) -- (4,6.2);
\draw[lightgray, thin,  ->] (6,7.8) -- (6,6.2);
\draw[lightgray, thin,  ->] (8,7.8) -- (8,6.2);
\draw[lightgray, thin,  ->] (10,7.8) -- (10,6.2);
\draw[lightgray, thin,  ->] (1.8,0.2) -- (0.2,1.8);
\draw[lightgray, thin,  ->] (3.8,0.2) -- (2.2,1.8);
\draw[lightgray, thin,  ->] (5.8,0.2) -- (4.2,1.8);
\draw[lightgray, thin,  ->] (7.8,0.2) -- (6.2,1.8);
\draw[lightgray, thin,  ->] (9.8,0.2) -- (8.2,1.8);
\draw[lightgray, thin,  ->] (1.8,2.2) -- (0.2,3.8);
\draw[lightgray, thin,  ->] (3.8,2.2) -- (2.2,3.8);
\draw[lightgray, thin,  ->] (5.8,2.2) -- (4.2,3.8);
\draw[lightgray, thin,  ->] (7.8,2.2) -- (6.2,3.8);
\draw[lightgray, thin,  ->] (9.8,2.2) -- (8.2,3.8);
\draw[lightgray, thin,  ->] (1.8,4.2) -- (0.2,5.8);
\draw[lightgray, thin,  ->] (3.8,4.2) -- (2.2,5.8);
\draw[lightgray, thin,  ->] (5.8,4.2) -- (4.2,5.8);
\draw[lightgray, thin,  ->] (7.8,4.2) -- (6.2,5.8);
\draw[lightgray, thin,  ->] (9.8,4.2) -- (8.2,5.8);
\draw[lightgray, thin,  ->] (1.8,6.2) -- (0.2,7.8);
\draw[lightgray, thin,  ->] (3.8,6.2) -- (2.2,7.8);
\draw[lightgray, thin,  ->] (5.8,6.2) -- (4.2,7.8);
\draw[lightgray, thin,  ->] (7.8,6.2) -- (6.2,7.8);
\draw[lightgray, thin,  ->] (9.8,6.2) -- (8.2,7.8);
\draw[lightgray, thin,  ->] (1.8,8.2) -- (0.2,9.8);
\draw[lightgray, thin,  ->] (3.8,8.2) -- (2.2,9.8);
\draw[lightgray, thin,  ->] (5.8,8.2) -- (4.2,9.8);
\draw[lightgray, thin,  ->] (7.8,8.2) -- (6.2,9.8);
\draw[lightgray, thin,  ->] (9.8,8.2) -- (8.2,9.8);
\draw[lightgray, thin,  ->] (4.2,0.2) to [out=60,in=230]  (9.8,3.6);
\draw[lightgray, thin,  ->] (6.2,0.3) to [out=50,in=230]  (9.8,1.6);
\draw[lightgray, thin,  ->] (2.2,0.3) to [out=70,in=240]  (9.8,5.6);
\draw[lightgray, thin,  ->] (0.2,0.3) to [out=70,in=240]  (9.8,7.6);
\draw[lightgray, thin,  -] (8,9.8) -- (8,9.8) node[anchor= south] {\large$\vdots$};
\draw[lightgray, thin,  -] (6,9.8) -- (6,9.8) node[anchor= south] {\large$\vdots$};
\draw[lightgray, thin,  -] (4,9.8) -- (4,9.8) node[anchor= south] {\large$\vdots$};
\draw[lightgray, thin,  -] (2,9.8) -- (2,9.8) node[anchor= south] {\large$\vdots$};
\draw[lightgray, thin,  -] (10,9.8) -- (10,9.8) node[anchor= south] {\large$\vdots$};
\draw[lightgray, thin,  -] (0,0) -- (0,0) node[anchor= east] {\large$\cdots$};
\draw[lightgray, thin,  -] (0,2) -- (0,2) node[anchor= east] {\large$\cdots$};
\draw[lightgray, thin,  -] (0,4) -- (0,4) node[anchor= east] {\large$\cdots$};
\draw[lightgray, thin,  -] (0,6) -- (0,6) node[anchor= east] {\large$\cdots$};
\draw[lightgray, thin,  -] (0,8) -- (0,8) node[anchor= east] {\large$\cdots$};
\draw[lightgray, thin,  -] (0,10) -- (0,10) node[anchor= south east] {\large$\ddots$};
\draw[gray,fill=gray] (10,0) circle (.01ex) node {\large$\mathbf{1}$};
\draw[gray,fill=gray] (2,0) circle (.01ex) node {\large$\mathbf{5}$};
\draw[gray,fill=gray] (4,0) circle (.01ex) node {\large$\mathbf{4}$};
\draw[gray,fill=gray] (6,0) circle (.01ex) node {\large$\mathbf{3}$};
\draw[gray,fill=gray] (8,0) circle (.01ex) node {\large$\mathbf{1}$};
\draw[gray,fill=gray] (10,2) circle (.01ex) node {\large$\mathbf{2}$};
\draw[gray,fill=gray] (2,2) circle (.01ex) node {\large$\mathbf{6}$};
\draw[gray,fill=gray] (4,2) circle (.01ex) node {\large$\mathbf{5}$};
\draw[gray,fill=gray] (6,2) circle (.01ex) node {\large$\mathbf{2}$};
\draw[gray,fill=gray] (8,2) circle (.01ex) node {\large$\mathbf{3}$};
\draw[gray,fill=gray] (10,4) circle (.01ex) node {\large$\mathbf{3}$};
\draw[gray,fill=gray] (2,4) circle (.01ex) node {\large$\mathbf{7}$};
\draw[gray,fill=gray] (4,4) circle (.01ex) node {\large$\mathbf{3}$};
\draw[gray,fill=gray] (6,4) circle (.01ex) node {\large$\mathbf{4}$};
\draw[gray,fill=gray] (8,4) circle (.01ex) node {\large$\mathbf{5}$};
\draw[gray,fill=gray] (10,6) circle (.01ex) node {\large$\mathbf{4}$};
\draw[gray,fill=gray] (2,6) circle (.01ex) node {\large$\mathbf{4}$};
\draw[gray,fill=gray] (4,6) circle (.01ex) node {\large$\mathbf{5}$};
\draw[gray,fill=gray] (6,6) circle (.01ex) node {\large$\mathbf{6}$};
\draw[gray,fill=gray] (8,6) circle (.01ex) node {\large$\mathbf{7}$};
\draw[gray,fill=gray] (10,8) circle (.01ex) node {\large$\mathbf{5}$};
\draw[gray,fill=gray] (2,8) circle (.01ex) node {\large$\mathbf{6}$};
\draw[gray,fill=gray] (4,8) circle (.01ex) node {\large$\mathbf{7}$};
\draw[gray,fill=gray] (6,8) circle (.01ex) node {\large$\mathbf{8}$};
\draw[gray,fill=gray] (8,8) circle (.01ex) node {\large$\mathbf{9}$};
\end{tikzpicture}
\qquad
}
\caption{The graph of the upper middle $\mathscr{C}$-module category
$\mathcal{N}_1$ and the corresponding positive eigenvector}\label{fig5}
\end{figure}
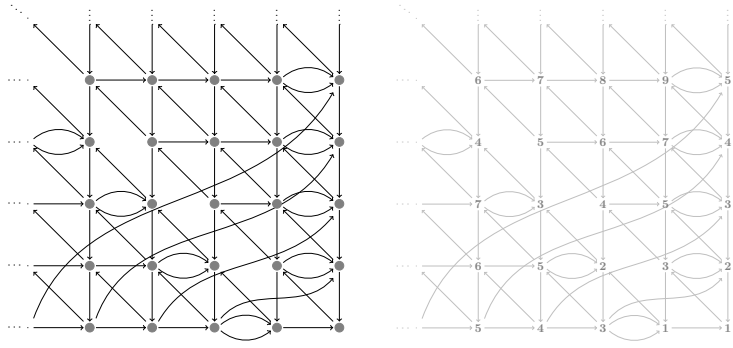

\begin{proof}
We start by observing that all singular middle weights belong to exactly
one wall (either an $r$-wall or an $s$-wall, but not both). Therefore,
our $L(\lambda)$ can be seen as the translation of some $L(\mu)$,
where $\mu$ is an upper middle regular weight, 
to this wall. Translating back, we get exactly $N(\mu)$.

The module $N(\mu)$ has, by construction, simple top and socle
isomorphic to $L(\mu)$. It also has one more simple subquotient
of the form $L(\mu')$, where $\mu'$ is 
obtained from $\mu$ by reflection with respect to the line
$x+y=-2$ (here $x$ and $y$ are standard coordinates in the plane),
which is again an upper middle regular
weight. Note that, sometimes, $N(\mu)$ also has a finite dimensional subquotient
which is not very relevant for our story as it is always killed
by translation to any wall. Translating 
$N(\mu)$ to some singular upper middle weight produces 
either a simple module (if $L(\mu)$ is killed and $L(\mu')$ survives)
or a direct sum of two simple modules (if $L(\mu')$ is killed and 
$L(\mu)$ survives). This implies that 
$\mathrm{add}(\mathscr{C}\cdot L(\lambda))=\mathcal{N}_1$.

Now let us argue that $\mathcal{N}_1$ is simple transitive. 
Any $N(\mu)$, where $\mu$ is upper middle regular, can be 
translated to any wall and then back to get all other
$N(\mu')$ (possibly with some non-zero multiplicities).
This implies that  $\mathcal{N}_1$ is transitive. 

If $\varphi:N(\mu)\to N(\mu')$ is a non-zero morphism, we can translate
it to a wall such that $L(\mu)$ survives and becomes
some $L(\nu)$. The part of $\mathcal{N}_1$ on the wall
that corresponds to the central character of $L(\nu)$
is just a copy of vector spaces with $L(\nu)$ being
the simple object. The morphism $\varphi$ survives translation
(as $L(\mu)$ survives it) and hence its translation becomes
a non-zero morphism in the category of vector spaces. 
As the latter category is simple, this non-zero morphism
generates the whole category, in particular, it generates
the identity on $L(\nu)$. Translating back, we get the
identity on all $N(\mu)$. This shows that $\mathcal{N}_1$ is simple.

It remains to determine the combinatorics of $\mathcal{N}_1$.
For this, we need to do a case-by-case analysis.
Recall that $\mathrm{supp}(L(1,0))=\{(1,0),(0,-1),(-1,1)\}$.

{\bf Case 1.} {\color{violet}Assume that the weights $\mu$, $\mu+(1,0)$,
$\mu+(0,-1)$ and $\mu+(-1,1)$ are all regular.} In this case, 
tensoring with $L((1,0))$ splits into a direct sum of three
equivalences: from the block containing $L(\mu)$ to the three
blocks containing $L(\mu+(1,0))$, $L(\mu+(0,-1))$ and 
$L(\mu+(-1,1))$, respectively. This implies that
the module $L(1,0)\otimes_{\mathbb{C}}N(\mu)$ splits as 
\begin{displaymath}
N(\mu+(1,0))\oplus N(\mu+(0,-1))\oplus N(\mu+(-1,1)).
\end{displaymath}
This gives the following local picture in the graph
(the node $\mu$ is in the middle): 

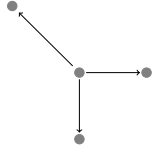
\begin{equation}
\resizebox{2cm}{!}{
\begin{tikzpicture}
\draw[gray,fill=gray] (2,2) circle (.9ex);
\draw[gray,fill=gray] (2,0) circle (.9ex);
\draw[gray,fill=gray] (4,2) circle (.9ex);
\draw[gray,fill=gray] (0,4) circle (.9ex);
\draw[black, thick,  ->] (2.2,2) -- (3.8,2);
\draw[black, thick,  ->] (2,1.8) -- (2,0.2);
\draw[black, thick,  ->] (1.8,2.2) -- (0.2,3.8);
\end{tikzpicture}
}\label{eq-s5.3-1}
\end{equation}

{\bf Case 2.} {\color{violet}Assume that $\mu$
is regular, $\mu_1=-2$ and $\mu_2\geq 2$.} In this case 
both $\mu+(0,-1)$ and $\mu+(-1,1)$ are regular
but $\mu+(1,0)$ is singular. Therefore, we get the
direct sum of of two equivalences, from the block 
containing $L(\mu)$ to the two blocks containing
$L(\mu+(0,-1))$ and  $L(\mu+(-1,1))$, respectively,
together with translation to the wall, from the
block  containing $L(\mu)$ to the block containing
$L(\mu+(1,0))$. Since $N(\mu)$ has two copies of
$L(\mu)$, translating to the wall we obtain two
copies of the simple module on the wall. 
This implies that $L((1,0))\otimes_\mathbb{C}N(\mu)$
is isomorphic to
\begin{displaymath}
N(\mu+(-1,1))\oplus N(\mu+(0,-1))\oplus 
N(\mu+(1,0))\oplus N(\mu+(1,0)),
\end{displaymath}
which gives the following local picture in the graph
(the node $\mu$ is in the middle): 

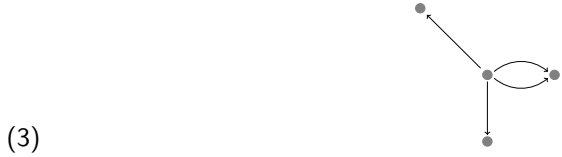
\begin{equation}\label{eq-s5.3-2}
\resizebox{2cm}{!}{
\begin{tikzpicture}
\draw[gray,fill=gray] (2,2) circle (.9ex);
\draw[gray,fill=gray] (2,0) circle (.9ex);
\draw[gray,fill=gray] (4,2) circle (.9ex);
\draw[gray,fill=gray] (0,4) circle (.9ex);
\draw[black, thick,  ->] (2.2,2.1) to [out=40,in=140]  (3.8,2.1);
\draw[black, thick,  ->] (2.2,1.9) to [out=-40,in=-140]  (3.8,1.9);
\draw[black, thick,  ->] (2,1.8) -- (2,0.2);
\draw[black, thick,  ->] (1.8,2.2) -- (0.2,3.8);
\end{tikzpicture}
}
\end{equation}

{\bf Case 3.} {\color{violet}Assume that $\mu=(-2,1)$.} 
In this case $\mu$ and $\mu+(-1,1)$ are regular, while 
$\mu+(0,-1)$ and $\mu+(1,0)$ are singular but with
different singularities. As in Case~2, we get 
an equivalence for $\mu+(-1,1)$ which produces 
a summand $N(\mu+(-1,1))$ and two copies of
$L(\mu+(1,0))$. At the same time, we will only get
one copy of $L(\mu+(0,-1))$ as $N(\mu)$ has
only one copy of $L((-3,0))$. This means that 
the combinatorics in this case is exactly the same 
as in Case~2.

{\bf Case 4.} {\color{violet}Assume that $\mu=(-1,i)$,
for $i>0$.}  In this case $\mu$ and $\mu+(0,-1)$ 
are singular with the same singularity, while 
$\mu+(-1,1)$ is regular and $\mu+(1,0)$ is not a
middle weight. Therefore tensoring with $L((1,0))$
splits into an equivalence between the blocks
containing $L(\mu)$ and $L(\mu+(0,-1))$ and a
translation out of the wall from the blocks
containing $L(\mu)$ to the block containing 
$L(\mu+(-1,1))$. This implies that 
$L((1,0))\otimes_\mathbb{C}N(\mu)$
is isomorphic to
\begin{displaymath}
N(\mu+(-1,1))\oplus N(\mu+(0,-1)),
\end{displaymath}
which gives the following local picture in the graph
(the node $\mu$ is in the middle): 

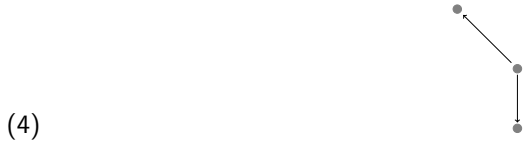
\begin{equation}\label{eq-s5.3-3}
\resizebox{1cm}{!}{
\begin{tikzpicture}
\draw[gray,fill=gray] (2,2) circle (.9ex);
\draw[gray,fill=gray] (2,0) circle (.9ex);
\draw[gray,fill=gray] (0,4) circle (.9ex);
\draw[black, thick,  ->] (2,1.8) -- (2,0.2);
\draw[black, thick,  ->] (1.8,2.2) -- (0.2,3.8);
\end{tikzpicture}
}
\end{equation}

{\bf Case 5.} {\color{violet}Assume that $\mu=(-1,0)$.}  
In this case $\mu$ is singular, while 
$\mu+(-1,1)$ is regular and neither $\mu+(1,0)$
nor $\mu+(0,-1)$ is a middle weight. Therefore tensoring 
with $L((1,0))$ gives just translation out of the
wall. This implies that  $L((1,0))\otimes_\mathbb{C}N(\mu)$
is isomorphic to $N(\mu+(-1,1))$
which gives the following local picture in the graph
(the node $\mu$ is on the south-east): 

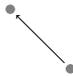
\begin{center}
\resizebox{1cm}{!}{
\begin{tikzpicture}
\draw[gray,fill=gray] (2,2) circle (.9ex);
\draw[gray,fill=gray] (0,4) circle (.9ex);
\draw[black, thick,  ->] (1.8,2.2) -- (0.2,3.8);
\end{tikzpicture}
}
\end{center}

{\bf Case 6.} {\color{violet}Assume that $\mu=(-2-i,1+i)$,
for some $i>0$.} 
In this case $\mu$, $\mu+(-1,1)$ and $\mu+(1,0)$
are regular, while $\mu+(0,-1)$ is singular.
Thus, we get  two equivalences, from $\mu$
to $\mu+(-1,1)$ and $\mu+(1,0)$, and one translation
to the wall, from $\mu$ to $\mu+(0,-1)$.
As $N(\mu)$ has only one copy of $L(\mu+(-1,-1))$,
the translation to the wall outputs one copy of 
the corresponding simple module. This means that
the combinatorics of this case is exactly the same
as in Case~1.

{\bf Case 7.} {\color{violet}Assume that $\mu=(-2,0)$.}  
In this case $\mu$ and $\mu+(-1,1)$
are singular with the same singularity, while 
$\mu+(1,0)$ is singular with different singularity
and $\mu+(0,-1)$ is not a middle weight. This gives
us an equivalence from $\mu$ to $\mu+(-1,1)$
together with the composition of translation out of 
the wall (from $\mu$) to some regular weight
followed by translation to the wall (to $\mu+(1,0)$).
Translation of a simple out of the wall contains only one copy 
of the simple which survives the translation to the
other wall. Therefore $L((1,0))\otimes_\mathbb{C}N(\mu)$
is isomorphic to $N(\mu+(-1,1))\oplus N(\mu+(1,0))$
which gives the following local picture in the graph
(the node $\mu$ is in the middle): 

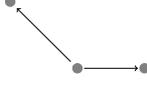
\begin{center}
\resizebox{2cm}{!}{
\begin{tikzpicture}
\draw[gray,fill=gray] (4,2) circle (.9ex);
\draw[gray,fill=gray] (2,2) circle (.9ex);
\draw[gray,fill=gray] (0,4) circle (.9ex);
\draw[black, thick,  ->] (1.8,2.2) -- (0.2,3.8);
\draw[black, thick,  ->] (2.2,2) -- (3.8,2);
\end{tikzpicture}
}
\end{center}

{\bf Case 8.} {\color{violet}Assume that $\mu=(-2-i,i)$,
for some $i>0$.} 
In this case $\mu$ and $\mu+(-1,1)$ are singular
with the same singularity, while $\mu+(1,0)$
$\mu+(0,-1)$ are regular but belong to two 
different connected shaded components.
Also, in this case $\mu+(1,0)$ and 
$\mu+(0,-1)$ belong to the same dot-orbit of $W$.
Thus, we only get an equivalence from 
$\mu$ to $\mu+(-1,1)$ and one translation out
of the wall from $\mu$ to the block containing 
both $\mu+(1,0)$ and $\mu+(0,-1)$. 
This means that
the combinatorics of this case is exactly the same
as in Case~4.

{\bf Case 9.} {\color{violet}Assume that $\mu=(-3-i,i)$,
for some $i>0$.} In this case, the situation
is exactly the same as in Case~2.

We are now left with two most complicated cases.

{\bf Case 10.} {\color{violet}Assume that $\mu=(-3,0)$.} 
In this case $\mu$ and $\mu+(-1,1)=(-4,1)$
are regular, $\mu+(1,0)=(-2,0)$ is singular and 
$\mu+(0,-1)=(-3,-1)$ is not a middle weight. However,
we have the middle weight $(-1,1)$ which belongs to
the same dot-orbit as $(-3,-1)$ and the difference
between $(-1,1)$ and the weight $(-2,1)$ which 
belongs to the same dot-orbit as $\mu$ is
$(1,0)$. This means that tensoring with $L((1,0))$
splits into three summands: an equivalence from
$\mu$ and $(-4,1)$ and two translations to walls:
from $\mu$ to $(-2,0)$ and from $(-2,1)$ to $(-1,1)$.
As $[N(\mu):L(\mu)]=2$, the first translation produces
two simples. As $[N(\mu):L((-2,1))]=1$, the second 
translation produces one simple. 
Therefore $L((1,0))\otimes_\mathbb{C}N(\mu)$
is isomorphic to 
\begin{displaymath}
N((-4,1))\oplus N((-2,0))\oplus N((-2,0))
\oplus N((-1,1)),
\end{displaymath}
which gives the following local picture in the graph
(the node $\mu$ is the source): 

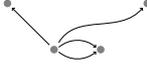
\begin{center}
\resizebox{2cm}{!}{
\begin{tikzpicture}
\draw[gray,fill=gray] (0,2) circle (.9ex);
\draw[gray,fill=gray] (2,0) circle (.9ex);
\draw[gray,fill=gray] (4,0) circle (.9ex);
\draw[gray,fill=gray] (6,2) circle (.9ex);
\draw[black, thick,  ->] (1.8,0.2) -- (0.2,1.8);
\draw[black, thick,  ->] (2.2,0.1) to [out=40,in=140]  (3.8,0.1);
\draw[black, thick,  ->] (2.2,-0.1) to [out=-40,in=-140]  (3.8,-0.1);
\draw[black, thick,  ->] (2.2,0.3) to [out=60,in=230]  (5.8,1.8);
\end{tikzpicture}
}
\end{center}

{\bf Case 11.} {\color{violet}Assume that $\mu=(-i,0)$,
for some $i>3$.} In this case $\mu$, $\mu+(-1,1)$
and $\mu+(1,0)$ are regular while $\mu+(0,-1)$ is
not a middle weight. However, just like in Case~10,
the dot-orbit of the weight $\mu+(0,-1)$ contains a middle weight,
namely, $(-1,i-2)$,
which has difference $(1,0)$ to the middle weight $(-2,i-2)$
which is in the dot-orbit of $\mu$. Therefore, 
tensoring with $L((1,0))$
splits again into three summands: two equivalences, from
$\mu$ to  $\mu+(-1,1)$ and $\mu+(1,0)$, 
and a translation to the wall, namely from $(-2,i-2)$ to
$(-1,i-2)$. 
Therefore $L((1,0))\otimes_\mathbb{C}N(\mu)$
is isomorphic to 
\begin{displaymath}
N(\mu+(-1,1))\oplus N(\mu+(1,0))
\oplus N((-1,i-2)),
\end{displaymath}
which gives the following local picture in the graph
(the node $\mu$ is the source): 

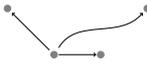
\begin{center}
\resizebox{2cm}{!}{
\begin{tikzpicture}
\draw[gray,fill=gray] (0,2) circle (.9ex);
\draw[gray,fill=gray] (2,0) circle (.9ex);
\draw[gray,fill=gray] (4,0) circle (.9ex);
\draw[gray,fill=gray] (6,2) circle (.9ex);
\draw[black, thick,  ->] (1.8,0.2) -- (0.2,1.8);
\draw[black, thick,  ->] (2.2,0) -- (3.8,0);
\draw[black, thick,  ->] (2.2,0.3) to [out=60,in=230]  (5.8,1.8);
\end{tikzpicture}
}
\end{center}

Combining all these cases together, gives 
the left graph in Figure~\ref{fig5}.

Now let us construct a Perron-Frobenius
eigenvector for the  corresponding matrix,
we will see that the corresponding eigenvalue is $3$.
As the defining property of the middle weights
(among all integral weights) one can take the 
property that the corresponding simple highest
weight modules are, on the one hand, 
infinite dimensional but, on the other hand,
not isomorphic to Verma modules. This, in particular,
implies that all simple highest weight modules
for middle weights have uniformly bounded 
weight spaces. 

For a module $M$ with uniformly bounded weight spaces,
denote by $\mathtt{d}(M)$ the maximum value of 
$\dim M_\mu$, taken over all $\mu$. For any simple
highest weight module $L$ and for any negative root
$\gamma$, non-zero elements from $\mathfrak{g}_\gamma$
act either locally nilpotently or injectively on
$L$. Consequently, as soon as some $L_\mu$
has dimension $\mathtt{d}(L)$, for any negative root
$\gamma$ for which the non-zero elements from 
$\mathfrak{g}_\gamma$ act injectively on $L$, we have
$\dim L_{\mu+\gamma}= \mathtt{d}(L)$.
Note that, for all upper middle weights $\nu$, 
the non-zero elements from both
$\mathfrak{g}_{-\alpha}$ and $\mathfrak{g}_{-\alpha-\beta}$
act injectively on $L(\nu)$. 

This implies that, for all upper middle $\nu$,
we have $\mathtt{d}(L(\nu)\otimes L((1,0)))=3 \mathtt{d}(L(\nu))$
and that $\mathtt{d}(N(\nu))$ is the sum of
$\mathtt{d}(L(\mu))$, taken over all upper middle 
subquotients  $L(\mu)$ of $N(\nu)$ (counted with 
the corresponding multiplicities). This implies  the equality
$\mathtt{d}(N(\nu)\otimes L((1,0)))=3 \mathtt{d}(N(\nu))$
and therefore, assigning to an upper middle weight $\nu$
the value $\mathtt{d}(N(\nu))$ produces an eigenvector 
for the matrix corresponding to our graph, with the eigenvalue $3$.
It remains to compute these values $\mathtt{d}(N(\nu))$.

For any Verma module $\Delta(\nu)$
and for all $a,b\in\mathbb{Z}_{\geq 0}$, we have
\begin{displaymath}
\dim \Delta(\nu)_{\nu-a\alpha-b\beta}=\min(a+1,b+1),
\end{displaymath}
which equals the value of Kostant's partition function
at $a\alpha+b\beta$, that is the number of ways to write
$a\alpha+b\beta$ as a linear combination of 
$\alpha$, $\beta$ and $\alpha+\beta$ with non-negative
integer coefficients. Consequently, for the quotient
$M=\Delta(\nu)/\Delta(\nu-a\alpha-b\beta)$,
we have $\mathtt{d}(M)=\max(a,b)$.

From \cite[Theorem~7.6.23]{Di}, for any middle upper
weight $\mu=(\mu_1,\mu_2)$, we have the embedding
\begin{displaymath}
\Delta(\mu-(\mu_2+1)(-1,2))\subset \Delta(\mu)
\end{displaymath}
as well as the embedding 
\begin{displaymath}
\Delta(\mu-(\mu_2-\mu_1+2)(1,1))\subset \Delta(\mu).
\end{displaymath}
This implies that $\mathtt{d}(L(\mu))$ is given by the
following table:
\begin{displaymath}
\begin{array}{c|c|c|c|c|c||c}
\ddots&\vdots&\vdots &\vdots&\vdots& \vdots& \vdots\\\hline
\dots&1&2&3&4&{\color{purple}5}&{\color{cyan}4}\\ \hline
\dots&{\color{purple}4}&1&2&3&{\color{purple}4}&{\color{cyan}3}\\ \hline
\dots&3&{\color{purple}3}&1&2&{\color{purple}3}&{\color{cyan}2}\\ \hline
\dots&2&2&{\color{purple}2}&1&{\color{purple}2}&{\color{cyan}1}\\ \hline
\dots&1&1&1&{\color{purple}1}&{\color{purple}1}&{\color{cyan}0}\\  \hline\hline
\dots& {\color{violet}-5}&{\color{violet}-4}
&{\color{violet}-3}&{\color{violet}-2}&
{\color{violet}-1}&{\color{violet}\mu_1}\setminus {\color{cyan}\mu_2}\\
\end{array}
\end{displaymath}
Now, for {\color{purple}singular weights (shown in purple)}, we have
$L(\mu)=N(\mu)$, thus the table above already outputs
$\mathtt{d}(N(\mu))$. For regular weights (shown in black), we know
that $[N(\mu):L(\mu)]=2$ and $[N(\mu):L(\mu')]=1$, where
$\mu'$ is a reflection of $\mu$ with respect to the diagonal singular line.
Taking the corresponding linear combination of the values gives
exactly the right hand side picture in Figure~\ref{fig5}.
For the record, here is an explicit formula for $\mathtt{d}(N(\mu))$:
\begin{displaymath}
\mathtt{d}(N(\mu))=
\begin{cases}
\mu_2+1,& \mu_1=-1\text{ or }\mu_1+\mu_2=-2;\\
2\mu_2+\mu_1+3,& -2<\mu_1+\mu_2\text{ and }\mu_1<-1;\\
\mu_2-\mu_1,& \mu_1+\mu_2<-2.
\end{cases}
\end{displaymath}
It is easy to check that this is, indeed, an eigenvector with 
eigenvalue $3$, that is, that three times the value at each vertex
$v$ coincides with the sum, taken over all arrows $\eta:v\to w$
that start at $v$, of the values at the the corresponding endpoints $w$.
This completes the proof of our proposition.
\end{proof}

\begin{proposition}\label{prop-s5.3-2}
Let $\lambda$ be a regular upper middle weight. Then we have:

\begin{enumerate}[$($a$)$]
\item\label{prop-s5.3-2.1} The $\mathscr{C}$-module category
$\mathrm{add}(\mathscr{C}\cdot L(\lambda))$ contains
$\mathcal{N}_1$ as a $\mathscr{C}$-module subcategory.
\item\label{prop-s5.3-2.2}
The quotient of $\mathrm{add}(\mathscr{C}\cdot L(\lambda))$
by the ideal $\mathcal{I}$ generated by $\mathcal{N}_1$ is a simple 
transitive $\mathscr{C}$-module category that is 
equivalent to ${}_\mathscr{C}\mathscr{C}$. 
\item\label{prop-s5.3-2.3}
The indecomposable
objects in  $\mathrm{add}(\mathscr{C}\cdot L(\lambda))/\mathcal{I}$  
are given by the images of 
$L(\mu)$, where $\mu$ is a 
regular upper middle weight from the same shaded connected
component as $\lambda$, see Figure~\ref{fig4}.
\end{enumerate}
\end{proposition}

\begin{proof}
We can translate $L(\lambda)$ to a wall to obtain some singular
$L(\lambda')\in \mathrm{add}(\mathscr{C}\cdot L(\lambda))$.
Now Claim~\eqref{prop-s5.3-2.1} follows from Proposition~\ref{prop-s5.3-1}.

If we do not translate to any wall, the only projective functors
that remain are equivalences between regular blocks. These 
send simples to simples. Also note that the two connected shaded
upper middle weight components are separated by a wall.
To go from one of these components to another one, we must
cross this wall in which case we end up in the 
$\mathscr{C}$-module subcategory $\mathcal{N}_1$.
This implies Claim~\eqref{prop-s5.3-2.3}.

Finally, for Claim~\eqref{prop-s5.3-2.2}, we 
can now apply Theorem~\ref{thm-s4.3-1}. For this 
we just note that the combinatorics of tensoring
with $L((1,0))$ between the regular simples within
one connected shaded component is the same as
the combinatorics of tensoring
with $L((1,0))$ in ${}_\mathscr{C}\mathscr{C}$.
This is clear because we can transform our connected
component to the component of top weights by applying
some Weyl group element. This completes the proof.
\end{proof}

\subsection{Lower middle weights}\label{s5.4}

In this subsection we fix a lower middle weight $\lambda$.
Denote by $\mathcal{N}_2$ the additive closure of all 
$N(\mu)$, where $\mu$ is a lower middle weight.

\begin{proposition}\label{prop-s5.4-1}
If $\lambda$ is a singular lower middle weight, then 
$\mathrm{add}(\mathscr{C}\cdot L(\lambda))$ coincides
with $\mathcal{N}_2$, moreover, the latter is a simple
transitive $\mathscr{C}$-module category whose graph
and positive eigenvector are depicted in Figure~\ref{fig6}.
\end{proposition}

\begin{proof}
Mutatis mutandis the proof of Proposition~\ref{prop-s5.3-1}. 
\end{proof}

We note that the $\mathscr{C}$-module categories $\mathcal{N}_1$
and $\mathcal{N}_2$ are not equivalent. In fact, they have 
different combinatorics: comparing Figure~\ref{fig5}
with Figure~\ref{fig6} we see that the graphs are not isomorphic.
Indeed, in Figure~\ref{fig5}, the points of indegree $3$ form
an infinite chain of the form
\begin{displaymath}
\xymatrix{\dots\ar[r]&\bullet\ar[r]&\bullet\ar[r]&\bullet} 
\end{displaymath}
while, in Figure~\ref{fig6}, the points of indegree $3$ form
an infinite chain of the form
\begin{displaymath}
\xymatrix{&\bullet\ar[r]&\bullet\ar[r]&\bullet\ar[r]&\dots}. 
\end{displaymath}
This is explained by the lack of symmetry between the two
cases. While all upper middle weights can be transferred to
lower middle weights by a Weyl group element, the additional
influence of the support of $L((1,0))$ is asymmetric in the two cases.
This stays in a sharp contrast to the action of projective functors
on a fixed block of category $\mathcal{O}$. In that case, 
the actions on the upper middle part and on the lower middle part
are equivalent, see \cite{MM1,MS1}.

\begin{figure}
\resizebox{10cm}{!}{
\begin{tikzpicture}
\draw[gray,fill=gray] (0,10) circle (.9ex);
\draw[gray,fill=gray] (0,8) circle (.9ex);
\draw[gray,fill=gray] (0,6) circle (.9ex);
\draw[gray,fill=gray] (0,4) circle (.9ex);
\draw[gray,fill=gray] (0,2) circle (.9ex);
\draw[gray,fill=gray] (2,10) circle (.9ex);
\draw[gray,fill=gray] (2,8) circle (.9ex);
\draw[gray,fill=gray] (2,6) circle (.9ex);
\draw[gray,fill=gray] (2,4) circle (.9ex);
\draw[gray,fill=gray] (2,2) circle (.9ex);
\draw[gray,fill=gray] (4,10) circle (.9ex);
\draw[gray,fill=gray] (4,8) circle (.9ex);
\draw[gray,fill=gray] (4,6) circle (.9ex);
\draw[gray,fill=gray] (4,4) circle (.9ex);
\draw[gray,fill=gray] (4,2) circle (.9ex);
\draw[gray,fill=gray] (6,10) circle (.9ex);
\draw[gray,fill=gray] (6,8) circle (.9ex);
\draw[gray,fill=gray] (6,6) circle (.9ex);
\draw[gray,fill=gray] (6,4) circle (.9ex);
\draw[gray,fill=gray] (6,2) circle (.9ex);
\draw[gray,fill=gray] (8,10) circle (.9ex);
\draw[gray,fill=gray] (8,8) circle (.9ex);
\draw[gray,fill=gray] (8,6) circle (.9ex);
\draw[gray,fill=gray] (8,4) circle (.9ex);
\draw[gray,fill=gray] (8,2) circle (.9ex);
\draw[black, thin,  -] (0,0) -- (0,0) node[anchor= north] {\large$\vdots$};
\draw[black, thin,  -] (2,0) -- (2,0) node[anchor= north] {\large$\vdots$};
\draw[black, thin,  -] (4,0) -- (4,0) node[anchor= north] {\large$\vdots$};
\draw[black, thin,  -] (6,0) -- (6,0) node[anchor= north] {\large$\vdots$};
\draw[black, thin,  -] (8,0) -- (8,0) node[anchor= north] {\large$\vdots$};
\draw[black, thin,  -] (9.8,0.3) -- (9.8,0.3) node[anchor=north west] {\large$\ddots$};
\draw[black, thick,  ->] (8.2,2) -- (9.8,2) node[anchor= west] {\large$\dots$};
\draw[black, thick,  ->] (8.2,4) -- (9.8,4) node[anchor= west] {\large$\dots$};
\draw[black, thick,  ->] (8.2,6) -- (9.8,6) node[anchor= west] {\large$\dots$};
\draw[black, thick,  ->] (8.2,8) -- (9.8,8) node[anchor= west] {\large$\dots$};
\draw[black, thick,  ->] (8.2,10) -- (9.8,10) node[anchor= west] {\large$\dots$};
\draw[black, thick,  ->] (0,1.8) -- (0,0.2);
\draw[black, thick,  ->] (2,1.8) -- (2,0.2);
\draw[black, thick,  ->] (4,1.8) -- (4,0.2);
\draw[black, thick,  ->] (6,1.8) -- (6,0.2);
\draw[black, thick,  ->] (8,1.8) -- (8,0.2);
\draw[black, thick,  ->] (9.8,0.2) -- (8.2,1.8);
\draw[black, thick,  ->] (0.2,10) -- (1.8,10);
\draw[black, thick,  ->] (2.2,10) -- (3.8,10);
\draw[black, thick,  ->] (4.2,10) -- (5.8,10);
\draw[black, thick,  ->] (6.2,10) -- (7.8,10);
\draw[black, thick,  ->] (2.2,8) -- (3.8,8);
\draw[black, thick,  ->] (4.2,8) -- (5.8,8);
\draw[black, thick,  ->] (6.2,8) -- (7.8,8);
\draw[black, thick,  ->] (4.2,6) -- (5.8,6);
\draw[black, thick,  ->] (6.2,6) -- (7.8,6);
\draw[black, thick,  ->] (0.2,4) -- (1.8,4);
\draw[black, thick,  ->] (6.2,4) -- (7.8,4);
\draw[black, thick,  ->] (0.2,2) -- (1.8,2);
\draw[black, thick,  ->] (2.2,2) -- (3.8,2);
\draw[black, thick,  ->] (0.2,6.1) to [out=40,in=140]  (1.8,6.1);
\draw[black, thick,  ->] (0.2,5.9) to [out=-40,in=-140]  (1.8,5.9);
\draw[black, thick,  ->] (2.2,4.1) to [out=40,in=140]  (3.8,4.1);
\draw[black, thick,  ->] (2.2,3.9) to [out=-40,in=-140]  (3.8,3.9);
\draw[black, thick,  ->] (4.2,2.1) to [out=40,in=140]  (5.8,2.1);
\draw[black, thick,  ->] (4.2,1.9) to [out=-40,in=-140]  (5.8,1.9);
\draw[black, thick,  ->] (0,3.8) -- (0,2.2);
\draw[black, thick,  ->] (2,3.8) -- (2,2.2);
\draw[black, thick,  ->] (4,3.8) -- (4,2.2);
\draw[black, thick,  ->] (6,3.8) -- (6,2.2);
\draw[black, thick,  ->] (8,3.8) -- (8,2.2);
\draw[black, thick,  ->] (0,5.8) -- (0,4.2);
\draw[black, thick,  ->] (2,5.8) -- (2,4.2);
\draw[black, thick,  ->] (4,5.8) -- (4,4.2);
\draw[black, thick,  ->] (6,5.8) -- (6,4.2);
\draw[black, thick,  ->] (8,5.8) -- (8,4.2);
\draw[black, thick,  ->] (0,7.8) -- (0,6.2);
\draw[black, thick,  ->] (2,7.8) -- (2,6.2);
\draw[black, thick,  ->] (4,7.8) -- (4,6.2);
\draw[black, thick,  ->] (6,7.8) -- (6,6.2);
\draw[black, thick,  ->] (8,7.8) -- (8,6.2);
\draw[black, thick,  ->] (0,9.8) -- (0,8.2);
\draw[black, thick,  ->] (2,9.8) -- (2,8.2);
\draw[black, thick,  ->] (4,9.8) -- (4,8.2);
\draw[black, thick,  ->] (6,9.8) -- (6,8.2);
\draw[black, thick,  ->] (8,9.8) -- (8,8.2);
\draw[black, thick,  ->] (1.8,0.2) -- (0.2,1.8);
\draw[black, thick,  ->] (3.8,0.2) -- (2.2,1.8);
\draw[black, thick,  ->] (5.8,0.2) -- (4.2,1.8);
\draw[black, thick,  ->] (7.8,0.2) -- (6.2,1.8);
\draw[black, thick,  ->] (1.8,2.2) -- (0.2,3.8);
\draw[black, thick,  ->] (3.8,2.2) -- (2.2,3.8);
\draw[black, thick,  ->] (5.8,2.2) -- (4.2,3.8);
\draw[black, thick,  ->] (7.8,2.2) -- (6.2,3.8);
\draw[black, thick,  ->] (9.8,2.2) -- (8.2,3.8);
\draw[black, thick,  ->] (1.8,4.2) -- (0.2,5.8);
\draw[black, thick,  ->] (3.8,4.2) -- (2.2,5.8);
\draw[black, thick,  ->] (5.8,4.2) -- (4.2,5.8);
\draw[black, thick,  ->] (7.8,4.2) -- (6.2,5.8);
\draw[black, thick,  ->] (9.8,4.2) -- (8.2,5.8);
\draw[black, thick,  ->] (1.8,6.2) -- (0.2,7.8);
\draw[black, thick,  ->] (3.8,6.2) -- (2.2,7.8);
\draw[black, thick,  ->] (5.8,6.2) -- (4.2,7.8);
\draw[black, thick,  ->] (7.8,6.2) -- (6.2,7.8);
\draw[black, thick,  ->] (9.8,6.2) -- (8.2,7.8);
\draw[black, thick,  ->] (1.7,8.3) to [out=150,in=290]  (0.3,9.7);
\draw[black, thick,  ->] (1.8,8.4) to [out=120,in=330]  (0.4,9.8);
\draw[black, thick,  ->] (3.7,8.3) to [out=150,in=290]  (2.3,9.7);
\draw[black, thick,  ->] (3.8,8.4) to [out=120,in=330]  (2.4,9.8);
\draw[black, thick,  ->] (5.7,8.3) to [out=150,in=290]  (4.3,9.7);
\draw[black, thick,  ->] (5.8,8.4) to [out=120,in=330]  (4.4,9.8);
\draw[black, thick,  ->] (7.7,8.3) to [out=150,in=290]  (6.3,9.7);
\draw[black, thick,  ->] (7.8,8.4) to [out=120,in=330]  (6.4,9.8);
\draw[black, thick,  ->] (-0.2,6.2) to [out=110,in=220]  (-0.2,9.8);
\draw[black, thick,  ->] (0.2,4.2) to [out=60,in=230]  (1.8,9.8);
\draw[black, thick,  ->] (0.2,2.2) to [out=1,in=245]  (3.8,9.8);
\draw[black, thick,  ->] (0.2,0.2) to [out=1,in=245]  (5.8,9.8);
\end{tikzpicture}
\qquad\qquad
\begin{tikzpicture}
\draw[gray,fill=gray] (0,10) circle (.01ex) node {\large$\mathbf{1}$};
\draw[gray,fill=gray] (0,8) circle (.01ex) node {\large$\mathbf{1}$};
\draw[gray,fill=gray] (0,6) circle (.01ex) node {\large$\mathbf{3}$};
\draw[gray,fill=gray] (0,4) circle (.01ex) node {\large$\mathbf{4}$};
\draw[gray,fill=gray] (0,2) circle (.01ex) node {\large$\mathbf{5}$};
\draw[gray,fill=gray] (2,10) circle (.01ex) node {\large$\mathbf{2}$};
\draw[gray,fill=gray] (2,8) circle (.01ex) node {\large$\mathbf{3}$};
\draw[gray,fill=gray] (2,6) circle (.01ex) node {\large$\mathbf{2}$};
\draw[gray,fill=gray] (2,4) circle (.01ex) node {\large$\mathbf{5}$};
\draw[gray,fill=gray] (2,2) circle (.01ex) node {\large$\mathbf{6}$};
\draw[gray,fill=gray] (4,10) circle (.01ex) node {\large$\mathbf{3}$};
\draw[gray,fill=gray] (4,8) circle (.01ex) node {\large$\mathbf{5}$};
\draw[gray,fill=gray] (4,6) circle (.01ex) node {\large$\mathbf{4}$};
\draw[gray,fill=gray] (4,4) circle (.01ex) node {\large$\mathbf{3}$};
\draw[gray,fill=gray] (4,2) circle (.01ex) node {\large$\mathbf{7}$};
\draw[gray,fill=gray] (6,10) circle (.01ex) node {\large$\mathbf{4}$};
\draw[gray,fill=gray] (6,8) circle (.01ex) node {\large$\mathbf{7}$};
\draw[gray,fill=gray] (6,6) circle (.01ex) node {\large$\mathbf{6}$};
\draw[gray,fill=gray] (6,4) circle (.01ex) node {\large$\mathbf{5}$};
\draw[gray,fill=gray] (6,2) circle (.01ex) node {\large$\mathbf{4}$};
\draw[gray,fill=gray] (8,10) circle (.01ex) node {\large$\mathbf{5}$};
\draw[gray,fill=gray] (8,8) circle (.01ex) node {\large$\mathbf{9}$};
\draw[gray,fill=gray] (8,6) circle (.01ex) node {\large$\mathbf{8}$};
\draw[gray,fill=gray] (8,4) circle (.01ex) node {\large$\mathbf{7}$};
\draw[gray,fill=gray] (8,2) circle (.01ex) node {\large$\mathbf{6}$};
\draw[lightgray, thin,  -] (0,0) -- (0,0) node[anchor= north] {\large$\vdots$};
\draw[lightgray, thin,  -] (2,0) -- (2,0) node[anchor= north] {\large$\vdots$};
\draw[lightgray, thin,  -] (4,0) -- (4,0) node[anchor= north] {\large$\vdots$};
\draw[lightgray, thin,  -] (6,0) -- (6,0) node[anchor= north] {\large$\vdots$};
\draw[lightgray, thin,  -] (8,0) -- (8,0) node[anchor= north] {\large$\vdots$};
\draw[lightgray, thin,  -] (9.8,0.3) -- (9.8,0.3) node[anchor=north west] {\large$\ddots$};
\draw[lightgray, thin,  ->] (8.2,2) -- (9.8,2) node[anchor= west] {\large$\dots$};
\draw[lightgray, thin,  ->] (8.2,4) -- (9.8,4) node[anchor= west] {\large$\dots$};
\draw[lightgray, thin,  ->] (8.2,6) -- (9.8,6) node[anchor= west] {\large$\dots$};
\draw[lightgray, thin,  ->] (8.2,8) -- (9.8,8) node[anchor= west] {\large$\dots$};
\draw[lightgray, thin,  ->] (8.2,10) -- (9.8,10) node[anchor= west] {\large$\dots$};
\draw[lightgray, thin,  ->] (0,1.8) -- (0,0.2);
\draw[lightgray, thin,  ->] (2,1.8) -- (2,0.2);
\draw[lightgray, thin,  ->] (4,1.8) -- (4,0.2);
\draw[lightgray, thin,  ->] (6,1.8) -- (6,0.2);
\draw[lightgray, thin,  ->] (8,1.8) -- (8,0.2);
\draw[lightgray, thin,  ->] (9.8,0.2) -- (8.2,1.8);
\draw[lightgray, thin,  ->] (0.2,10) -- (1.8,10);
\draw[lightgray, thin,  ->] (2.2,10) -- (3.8,10);
\draw[lightgray, thin,  ->] (4.2,10) -- (5.8,10);
\draw[lightgray, thin,  ->] (6.2,10) -- (7.8,10);
\draw[lightgray, thin,  ->] (2.2,8) -- (3.8,8);
\draw[lightgray, thin,  ->] (4.2,8) -- (5.8,8);
\draw[lightgray, thin,  ->] (6.2,8) -- (7.8,8);
\draw[lightgray, thin,  ->] (4.2,6) -- (5.8,6);
\draw[lightgray, thin,  ->] (6.2,6) -- (7.8,6);
\draw[lightgray, thin,  ->] (0.2,4) -- (1.8,4);
\draw[lightgray, thin,  ->] (6.2,4) -- (7.8,4);
\draw[lightgray, thin,  ->] (0.2,2) -- (1.8,2);
\draw[lightgray, thin,  ->] (2.2,2) -- (3.8,2);
\draw[lightgray, thin,  ->] (0.2,6.1) to [out=40,in=140]  (1.8,6.1);
\draw[lightgray, thin,  ->] (0.2,5.9) to [out=-40,in=-140]  (1.8,5.9);
\draw[lightgray, thin,  ->] (2.2,4.1) to [out=40,in=140]  (3.8,4.1);
\draw[lightgray, thin,  ->] (2.2,3.9) to [out=-40,in=-140]  (3.8,3.9);
\draw[lightgray, thin,  ->] (4.2,2.1) to [out=40,in=140]  (5.8,2.1);
\draw[lightgray, thin,  ->] (4.2,1.9) to [out=-40,in=-140]  (5.8,1.9);
\draw[lightgray, thin,  ->] (0,3.8) -- (0,2.2);
\draw[lightgray, thin,  ->] (2,3.8) -- (2,2.2);
\draw[lightgray, thin,  ->] (4,3.8) -- (4,2.2);
\draw[lightgray, thin,  ->] (6,3.8) -- (6,2.2);
\draw[lightgray, thin,  ->] (8,3.8) -- (8,2.2);
\draw[lightgray, thin,  ->] (0,5.8) -- (0,4.2);
\draw[lightgray, thin,  ->] (2,5.8) -- (2,4.2);
\draw[lightgray, thin,  ->] (4,5.8) -- (4,4.2);
\draw[lightgray, thin,  ->] (6,5.8) -- (6,4.2);
\draw[lightgray, thin,  ->] (8,5.8) -- (8,4.2);
\draw[lightgray, thin,  ->] (0,7.8) -- (0,6.2);
\draw[lightgray, thin,  ->] (2,7.8) -- (2,6.2);
\draw[lightgray, thin,  ->] (4,7.8) -- (4,6.2);
\draw[lightgray, thin,  ->] (6,7.8) -- (6,6.2);
\draw[lightgray, thin,  ->] (8,7.8) -- (8,6.2);
\draw[lightgray, thin,  ->] (0,9.8) -- (0,8.2);
\draw[lightgray, thin,  ->] (2,9.8) -- (2,8.2);
\draw[lightgray, thin,  ->] (4,9.8) -- (4,8.2);
\draw[lightgray, thin,  ->] (6,9.8) -- (6,8.2);
\draw[lightgray, thin,  ->] (8,9.8) -- (8,8.2);
\draw[lightgray, thin,  ->] (1.8,0.2) -- (0.2,1.8);
\draw[lightgray, thin,  ->] (3.8,0.2) -- (2.2,1.8);
\draw[lightgray, thin,  ->] (5.8,0.2) -- (4.2,1.8);
\draw[lightgray, thin,  ->] (7.8,0.2) -- (6.2,1.8);
\draw[lightgray, thin,  ->] (1.8,2.2) -- (0.2,3.8);
\draw[lightgray, thin,  ->] (3.8,2.2) -- (2.2,3.8);
\draw[lightgray, thin,  ->] (5.8,2.2) -- (4.2,3.8);
\draw[lightgray, thin,  ->] (7.8,2.2) -- (6.2,3.8);
\draw[lightgray, thin,  ->] (9.8,2.2) -- (8.2,3.8);
\draw[lightgray, thin,  ->] (1.8,4.2) -- (0.2,5.8);
\draw[lightgray, thin,  ->] (3.8,4.2) -- (2.2,5.8);
\draw[lightgray, thin,  ->] (5.8,4.2) -- (4.2,5.8);
\draw[lightgray, thin,  ->] (7.8,4.2) -- (6.2,5.8);
\draw[lightgray, thin,  ->] (9.8,4.2) -- (8.2,5.8);
\draw[lightgray, thin,  ->] (1.8,6.2) -- (0.2,7.8);
\draw[lightgray, thin,  ->] (3.8,6.2) -- (2.2,7.8);
\draw[lightgray, thin,  ->] (5.8,6.2) -- (4.2,7.8);
\draw[lightgray, thin,  ->] (7.8,6.2) -- (6.2,7.8);
\draw[lightgray, thin,  ->] (9.8,6.2) -- (8.2,7.8);
\draw[lightgray, thin,  ->] (1.7,8.3) to [out=150,in=290]  (0.3,9.7);
\draw[lightgray, thin,  ->] (1.8,8.4) to [out=120,in=330]  (0.4,9.8);
\draw[lightgray, thin,  ->] (3.7,8.3) to [out=150,in=290]  (2.3,9.7);
\draw[lightgray, thin,  ->] (3.8,8.4) to [out=120,in=330]  (2.4,9.8);
\draw[lightgray, thin,  ->] (5.7,8.3) to [out=150,in=290]  (4.3,9.7);
\draw[lightgray, thin,  ->] (5.8,8.4) to [out=120,in=330]  (4.4,9.8);
\draw[lightgray, thin,  ->] (7.7,8.3) to [out=150,in=290]  (6.3,9.7);
\draw[lightgray, thin,  ->] (7.8,8.4) to [out=120,in=330]  (6.4,9.8);
\draw[lightgray, thin,  ->] (-0.2,6.2) to [out=110,in=220]  (-0.2,9.8);
\draw[lightgray, thin,  ->] (0.2,4.2) to [out=60,in=230]  (1.8,9.8);
\draw[lightgray, thin,  ->] (0.2,2.2) to [out=1,in=245]  (3.8,9.8);
\draw[lightgray, thin,  ->] (0.2,0.2) to [out=1,in=245]  (5.8,9.8);
\end{tikzpicture}
\qquad
}
\caption{The graph of the lower middle $\mathscr{C}$-module category
$\mathcal{N}_2$ and the corresponding positive eigenvector}\label{fig6}
\end{figure}
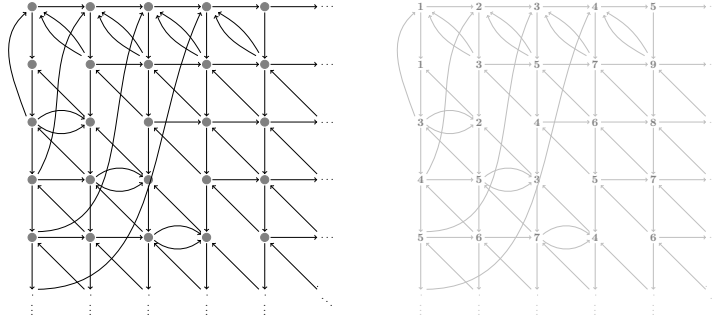

\begin{proposition}\label{prop-s5.4-2}
Let $\lambda$ be a regular lower middle weight. Then we have:

\begin{enumerate}[$($a$)$]
\item\label{prop-s5.4-2.1} The $\mathscr{C}$-module category
$\mathrm{add}(\mathscr{C}\cdot L(\lambda))$ contains
$\mathcal{N}_2$ as a $\mathscr{C}$-module subcategory.
\item\label{prop-s5.4-2.2}
The quotient of $\mathrm{add}(\mathscr{C}\cdot L(\lambda))$
by the ideal $\mathcal{I}$ generated by $\mathcal{N}_2$ is a simple 
transitive $\mathscr{C}$-module category that is 
equivalent to ${}_\mathscr{C}\mathscr{C}$. 
\item\label{prop-s5.4-2.3}
The indecomposable
objects in  $\mathrm{add}(\mathscr{C}\cdot L(\lambda))/\mathcal{I}$  
are given by the images of 
$L(\mu)$, where $\mu$ is a 
regular lower middle weight from the same shaded connected
component as $\lambda$, see Figure~\ref{fig4}.
\end{enumerate}
\end{proposition}

\begin{proof}
Mutatis mutandis the proof of Proposition~\ref{prop-s5.3-2}. 
\end{proof}

\subsection{Bottom weights}\label{s5.5}

In this subsection we fix a bottom weight $\lambda$.
Denote by $\mathcal{N}_3$ the additive closure of all 
$N(\mu)$, where $\mu$ is a bottom weight.

\begin{proposition}\label{prop-s5.5-1}
The $\mathscr{C}$-module category
$\mathrm{add}(\mathscr{C}\cdot L((-1,-1)))$ coincides
with $\mathcal{N}_3$, moreover, the latter is a simple
transitive $\mathscr{C}$-module category whose graph
and positive eigenvector are depicted in Figure~\ref{fig7}.
\end{proposition}

\begin{figure}
\resizebox{10cm}{!}{
\begin{tikzpicture}
\draw[gray,fill=gray] (2,2) circle (.9ex);
\draw[gray,fill=gray] (2,4) circle (.9ex);
\draw[gray,fill=gray] (2,6) circle (.9ex);
\draw[gray,fill=gray] (2,8) circle (.9ex);
\draw[gray,fill=gray] (2,10) circle (.9ex);
\draw[gray,fill=gray] (4,2) circle (.9ex);
\draw[gray,fill=gray] (4,4) circle (.9ex);
\draw[gray,fill=gray] (4,6) circle (.9ex);
\draw[gray,fill=gray] (4,8) circle (.9ex);
\draw[gray,fill=gray] (4,10) circle (.9ex);
\draw[gray,fill=gray] (6,2) circle (.9ex);
\draw[gray,fill=gray] (6,4) circle (.9ex);
\draw[gray,fill=gray] (6,6) circle (.9ex);
\draw[gray,fill=gray] (6,8) circle (.9ex);
\draw[gray,fill=gray] (6,10) circle (.9ex);
\draw[gray,fill=gray] (8,2) circle (.9ex);
\draw[gray,fill=gray] (8,4) circle (.9ex);
\draw[gray,fill=gray] (8,6) circle (.9ex);
\draw[gray,fill=gray] (8,8) circle (.9ex);
\draw[gray,fill=gray] (8,10) circle (.9ex);
\draw[gray,fill=gray] (10,2) circle (.9ex);
\draw[gray,fill=gray] (10,4) circle (.9ex);
\draw[gray,fill=gray] (10,6) circle (.9ex);
\draw[gray,fill=gray] (10,8) circle (.9ex);
\draw[gray,fill=gray] (10,10) circle (.9ex);
\draw[black, thick,  -] (0,2) -- (0,2) node[anchor= east] {\large$\dots$};
\draw[black, thick,  -] (0,4) -- (0,4) node[anchor= east] {\large$\dots$};
\draw[black, thick,  -] (0,6) -- (0,6) node[anchor= east] {\large$\dots$};
\draw[black, thick,  -] (0,8) -- (0,8) node[anchor= east] {\large$\dots$};
\draw[black, thick,  -] (0,10) -- (0,10) node[anchor= east] {\large$\dots$};
\draw[black, thick,  ->] (0.2,2) -- (1.8,2);
\draw[black, thick,  ->] (0.2,4) -- (1.8,4);
\draw[black, thick,  ->] (0.2,6) -- (1.8,6);
\draw[black, thick,  ->] (0.2,8) -- (1.8,8);
\draw[black, thick,  ->] (0.2,10) -- (1.8,10);
\draw[black, thin,  -] (2,0) -- (2,0) node[anchor= north] {\large$\vdots$};
\draw[black, thin,  -] (4,0) -- (4,0) node[anchor= north] {\large$\vdots$};
\draw[black, thin,  -] (6,0) -- (6,0) node[anchor= north] {\large$\vdots$};
\draw[black, thin,  -] (8,0) -- (8,0) node[anchor= north] {\large$\vdots$};
\draw[black, thin,  -] (10,0) -- (10,0) node[anchor= north] {\large$\vdots$};
\draw[black, thick,  ->] (2,1.8) -- (2,0.2);
\draw[black, thick,  ->] (4,1.8) -- (4,0.2);
\draw[black, thick,  ->] (6,1.8) -- (6,0.2);
\draw[black, thick,  ->] (8,1.8) -- (8,0.2);
\draw[black, thick,  ->] (10,1.8) -- (10,0.2);
\draw[black, thick,  ->] (2.2,2) -- (3.8,2);
\draw[black, thick,  ->] (2.2,4) -- (3.8,4);
\draw[black, thick,  ->] (2.2,6) -- (3.8,6);
\draw[black, thick,  ->] (2.2,8) -- (3.8,8);
\draw[black, thick,  ->] (2.2,10) -- (3.8,10);
\draw[black, thick,  ->] (4.2,2) -- (5.8,2);
\draw[black, thick,  ->] (4.2,4) -- (5.8,4);
\draw[black, thick,  ->] (4.2,6) -- (5.8,6);
\draw[black, thick,  ->] (4.2,8) -- (5.8,8);
\draw[black, thick,  ->] (4.2,10) -- (5.8,10);
\draw[black, thick,  ->] (6.2,2) -- (7.8,2);
\draw[black, thick,  ->] (6.2,4) -- (7.8,4);
\draw[black, thick,  ->] (6.2,6) -- (7.8,6);
\draw[black, thick,  ->] (6.2,8) -- (7.8,8);
\draw[black, thick,  ->] (6.2,10) -- (7.8,10);
\draw[black, thick,  ->] (2,3.8) -- (2,2.2);
\draw[black, thick,  ->] (4,3.8) -- (4,2.2);
\draw[black, thick,  ->] (6,3.8) -- (6,2.2);
\draw[black, thick,  ->] (8,3.8) -- (8,2.2);
\draw[black, thick,  ->] (10,3.8) -- (10,2.2);
\draw[black, thick,  ->] (2,5.8) -- (2,4.2);
\draw[black, thick,  ->] (4,5.8) -- (4,4.2);
\draw[black, thick,  ->] (6,5.8) -- (6,4.2);
\draw[black, thick,  ->] (8,5.8) -- (8,4.2);
\draw[black, thick,  ->] (10,5.8) -- (10,4.2);
\draw[black, thick,  ->] (2,7.8) -- (2,6.2);
\draw[black, thick,  ->] (4,7.8) -- (4,6.2);
\draw[black, thick,  ->] (6,7.8) -- (6,6.2);
\draw[black, thick,  ->] (8,7.8) -- (8,6.2);
\draw[black, thick,  ->] (10,7.8) -- (10,6.2);
\draw[black, thick,  ->] (2,9.8) -- (2,8.2);
\draw[black, thick,  ->] (4,9.8) -- (4,8.2);
\draw[black, thick,  ->] (6,9.8) -- (6,8.2);
\draw[black, thick,  ->] (8,9.8) -- (8,8.2);
\draw[black, thick,  ->] (10,9.8) -- (10,8.2);
\draw[black, thick,  ->] (8.2,2.1) to [out=40,in=140]  (9.8,2.1);
\draw[black, thick,  ->] (8.2,1.9) to [out=-40,in=-140]  (9.8,1.9);
\draw[black, thick,  ->] (8.2,4.1) to [out=40,in=140]  (9.8,4.1);
\draw[black, thick,  ->] (8.2,3.9) to [out=-40,in=-140]  (9.8,3.9);
\draw[black, thick,  ->] (8.2,6.1) to [out=40,in=140]  (9.8,6.1);
\draw[black, thick,  ->] (8.2,5.9) to [out=-40,in=-140]  (9.8,5.9);
\draw[black, thick,  ->] (8.2,8.1) to [out=40,in=140]  (9.8,8.1);
\draw[black, thick,  ->] (8.2,7.9) to [out=-40,in=-140]  (9.8,7.9);
\draw[black, thick,  ->] (8.2,10.1) to [out=40,in=140]  (9.8,10.1);
\draw[black, thick,  ->] (8.2,10.3) to [out=80,in=110]  (9.8,10.3);
\draw[black, thick,  ->] (8.2,10) -- (9.6,10);
\draw[black, thick,  ->] (1.8,0.2) -- (0.2,1.8);
\draw[black, thick,  ->] (3.8,0.2) -- (2.2,1.8);
\draw[black, thick,  ->] (5.8,0.2) -- (4.2,1.8);
\draw[black, thick,  ->] (7.8,0.2) -- (6.2,1.8);
\draw[black, thick,  ->] (9.8,0.2) -- (8.2,1.8);
\draw[black, thick,  ->] (1.8,2.2) -- (0.2,3.8);
\draw[black, thick,  ->] (3.8,2.2) -- (2.2,3.8);
\draw[black, thick,  ->] (5.8,2.2) -- (4.2,3.8);
\draw[black, thick,  ->] (7.8,2.2) -- (6.2,3.8);
\draw[black, thick,  ->] (9.8,2.2) -- (8.2,3.8);
\draw[black, thick,  ->] (1.8,4.2) -- (0.2,5.8);
\draw[black, thick,  ->] (3.8,4.2) -- (2.2,5.8);
\draw[black, thick,  ->] (5.8,4.2) -- (4.2,5.8);
\draw[black, thick,  ->] (7.8,4.2) -- (6.2,5.8);
\draw[black, thick,  ->] (9.8,4.2) -- (8.2,5.8);
\draw[black, thick,  ->] (1.8,6.2) -- (0.2,7.8);
\draw[black, thick,  ->] (3.8,6.2) -- (2.2,7.8);
\draw[black, thick,  ->] (5.8,6.2) -- (4.2,7.8);
\draw[black, thick,  ->] (7.8,6.2) -- (6.2,7.8);
\draw[black, thick,  ->] (9.8,6.2) -- (8.2,7.8);
\draw[black, thick,  ->] (1.7,8.2) to [out=170,in=-80]  (0.2,9.7);
\draw[black, thick,  ->] (1.8,8.3) to [out=100,in=-10]  (0.3,9.8);
\draw[black, thick,  ->] (3.7,8.2) to [out=170,in=-80]  (2.2,9.7);
\draw[black, thick,  ->] (3.8,8.3) to [out=100,in=-10]  (2.3,9.8);
\draw[black, thick,  ->] (5.7,8.2) to [out=170,in=-80]  (4.2,9.7);
\draw[black, thick,  ->] (5.8,8.3) to [out=100,in=-10]  (4.3,9.8);
\draw[black, thick,  ->] (7.7,8.2) to [out=170,in=-80]  (6.2,9.7);
\draw[black, thick,  ->] (7.8,8.3) to [out=100,in=-10]  (6.3,9.8);
\draw[black, thick,  ->] (9.6,8.4) to [out=135,in=-45]  (8.4,9.6);
\draw[black, thick,  ->] (9.8,8.3) to [out=90,in=0]  (8.3,9.8);
\end{tikzpicture}
\qquad\qquad
\begin{tikzpicture}
\draw[gray,fill=gray] (2,2) circle (.01ex) node {\large$\mathbf{6}$};
\draw[gray,fill=gray] (2,4) circle (.01ex) node {\large$\mathbf{6}$};
\draw[gray,fill=gray] (2,6) circle (.01ex) node {\large$\mathbf{6}$};
\draw[gray,fill=gray] (2,8) circle (.01ex) node {\large$\mathbf{6}$};
\draw[gray,fill=gray] (2,10) circle (.01ex) node {\large$\mathbf{3}$};
\draw[gray,fill=gray] (4,2) circle (.01ex) node {\large$\mathbf{6}$};
\draw[gray,fill=gray] (4,4) circle (.01ex) node {\large$\mathbf{6}$};
\draw[gray,fill=gray] (4,6) circle (.01ex) node {\large$\mathbf{6}$};
\draw[gray,fill=gray] (4,8) circle (.01ex) node {\large$\mathbf{6}$};
\draw[gray,fill=gray] (4,10) circle (.01ex) node {\large$\mathbf{3}$};
\draw[gray,fill=gray] (6,2) circle (.01ex) node {\large$\mathbf{6}$};
\draw[gray,fill=gray] (6,4) circle (.01ex) node {\large$\mathbf{6}$};
\draw[gray,fill=gray] (6,6) circle (.01ex) node {\large$\mathbf{6}$};
\draw[gray,fill=gray] (6,8) circle (.01ex) node {\large$\mathbf{6}$};
\draw[gray,fill=gray] (6,10) circle (.01ex) node {\large$\mathbf{3}$};
\draw[gray,fill=gray] (8,2) circle (.01ex) node {\large$\mathbf{6}$};
\draw[gray,fill=gray] (8,4) circle (.01ex) node {\large$\mathbf{6}$};
\draw[gray,fill=gray] (8,6) circle (.01ex) node {\large$\mathbf{6}$};
\draw[gray,fill=gray] (8,8) circle (.01ex) node {\large$\mathbf{6}$};
\draw[gray,fill=gray] (8,10) circle (.01ex) node {\large$\mathbf{3}$};
\draw[gray,fill=gray] (10,2) circle (.01ex) node {\large$\mathbf{3}$};
\draw[gray,fill=gray] (10,4) circle (.01ex) node {\large$\mathbf{3}$};
\draw[gray,fill=gray] (10,6) circle (.01ex) node {\large$\mathbf{3}$};
\draw[gray,fill=gray] (10,8) circle (.01ex) node {\large$\mathbf{3}$};
\draw[gray,fill=gray] (10,10) circle (.01ex) node {\large$\mathbf{1}$};
\draw[lightgray, thin,  -] (0,2) -- (0,2) node[anchor= east] {\large$\dots$};
\draw[lightgray, thin,  -] (0,4) -- (0,4) node[anchor= east] {\large$\dots$};
\draw[lightgray, thin,  -] (0,6) -- (0,6) node[anchor= east] {\large$\dots$};
\draw[lightgray, thin,  -] (0,8) -- (0,8) node[anchor= east] {\large$\dots$};
\draw[lightgray, thin,  -] (0,10) -- (0,10) node[anchor= east] {\large$\dots$};
\draw[lightgray, thin,  ->] (0.2,2) -- (1.8,2);
\draw[lightgray, thin,  ->] (0.2,4) -- (1.8,4);
\draw[lightgray, thin,  ->] (0.2,6) -- (1.8,6);
\draw[lightgray, thin,  ->] (0.2,8) -- (1.8,8);
\draw[lightgray, thin,  ->] (0.2,10) -- (1.8,10);
\draw[black, thin,  -] (2,0) -- (2,0) node[anchor= north] {\large$\vdots$};
\draw[black, thin,  -] (4,0) -- (4,0) node[anchor= north] {\large$\vdots$};
\draw[black, thin,  -] (6,0) -- (6,0) node[anchor= north] {\large$\vdots$};
\draw[black, thin,  -] (8,0) -- (8,0) node[anchor= north] {\large$\vdots$};
\draw[black, thin,  -] (10,0) -- (10,0) node[anchor= north] {\large$\vdots$};
\draw[lightgray, thin,  ->] (2,1.8) -- (2,0.2);
\draw[lightgray, thin,  ->] (4,1.8) -- (4,0.2);
\draw[lightgray, thin,  ->] (6,1.8) -- (6,0.2);
\draw[lightgray, thin,  ->] (8,1.8) -- (8,0.2);
\draw[lightgray, thin,  ->] (10,1.8) -- (10,0.2);
\draw[lightgray, thin,  ->] (2.2,2) -- (3.8,2);
\draw[lightgray, thin,  ->] (2.2,4) -- (3.8,4);
\draw[lightgray, thin,  ->] (2.2,6) -- (3.8,6);
\draw[lightgray, thin,  ->] (2.2,8) -- (3.8,8);
\draw[lightgray, thin,  ->] (2.2,10) -- (3.8,10);
\draw[lightgray, thin,  ->] (4.2,2) -- (5.8,2);
\draw[lightgray, thin,  ->] (4.2,4) -- (5.8,4);
\draw[lightgray, thin,  ->] (4.2,6) -- (5.8,6);
\draw[lightgray, thin,  ->] (4.2,8) -- (5.8,8);
\draw[lightgray, thin,  ->] (4.2,10) -- (5.8,10);
\draw[lightgray, thin,  ->] (6.2,2) -- (7.8,2);
\draw[lightgray, thin,  ->] (6.2,4) -- (7.8,4);
\draw[lightgray, thin,  ->] (6.2,6) -- (7.8,6);
\draw[lightgray, thin,  ->] (6.2,8) -- (7.8,8);
\draw[lightgray, thin,  ->] (6.2,10) -- (7.8,10);
\draw[lightgray, thin,  ->] (2,3.8) -- (2,2.2);
\draw[lightgray, thin,  ->] (4,3.8) -- (4,2.2);
\draw[lightgray, thin,  ->] (6,3.8) -- (6,2.2);
\draw[lightgray, thin,  ->] (8,3.8) -- (8,2.2);
\draw[lightgray, thin,  ->] (10,3.8) -- (10,2.2);
\draw[lightgray, thin,  ->] (2,5.8) -- (2,4.2);
\draw[lightgray, thin,  ->] (4,5.8) -- (4,4.2);
\draw[lightgray, thin,  ->] (6,5.8) -- (6,4.2);
\draw[lightgray, thin,  ->] (8,5.8) -- (8,4.2);
\draw[lightgray, thin,  ->] (10,5.8) -- (10,4.2);
\draw[lightgray, thin,  ->] (2,7.8) -- (2,6.2);
\draw[lightgray, thin,  ->] (4,7.8) -- (4,6.2);
\draw[lightgray, thin,  ->] (6,7.8) -- (6,6.2);
\draw[lightgray, thin,  ->] (8,7.8) -- (8,6.2);
\draw[lightgray, thin,  ->] (10,7.8) -- (10,6.2);
\draw[lightgray, thin,  ->] (2,9.8) -- (2,8.2);
\draw[lightgray, thin,  ->] (4,9.8) -- (4,8.2);
\draw[lightgray, thin,  ->] (6,9.8) -- (6,8.2);
\draw[lightgray, thin,  ->] (8,9.8) -- (8,8.2);
\draw[lightgray, thin,  ->] (10,9.8) -- (10,8.2);
\draw[lightgray, thin,  ->] (8.2,2.1) to [out=40,in=140]  (9.8,2.1);
\draw[lightgray, thin,  ->] (8.2,1.9) to [out=-40,in=-140]  (9.8,1.9);
\draw[lightgray, thin,  ->] (8.2,4.1) to [out=40,in=140]  (9.8,4.1);
\draw[lightgray, thin,  ->] (8.2,3.9) to [out=-40,in=-140]  (9.8,3.9);
\draw[lightgray, thin,  ->] (8.2,6.1) to [out=40,in=140]  (9.8,6.1);
\draw[lightgray, thin,  ->] (8.2,5.9) to [out=-40,in=-140]  (9.8,5.9);
\draw[lightgray, thin,  ->] (8.2,8.1) to [out=40,in=140]  (9.8,8.1);
\draw[lightgray, thin,  ->] (8.2,7.9) to [out=-40,in=-140]  (9.8,7.9);
\draw[lightgray, thin,  ->] (8.2,10.1) to [out=40,in=140]  (9.8,10.1);
\draw[lightgray, thin,  ->] (8.2,10.3) to [out=80,in=110]  (9.8,10.3);
\draw[lightgray, thin,  ->] (8.2,10) -- (9.6,10);
\draw[lightgray, thin,  ->] (1.8,0.2) -- (0.2,1.8);
\draw[lightgray, thin,  ->] (3.8,0.2) -- (2.2,1.8);
\draw[lightgray, thin,  ->] (5.8,0.2) -- (4.2,1.8);
\draw[lightgray, thin,  ->] (7.8,0.2) -- (6.2,1.8);
\draw[lightgray, thin,  ->] (9.8,0.2) -- (8.2,1.8);
\draw[lightgray, thin,  ->] (1.8,2.2) -- (0.2,3.8);
\draw[lightgray, thin,  ->] (3.8,2.2) -- (2.2,3.8);
\draw[lightgray, thin,  ->] (5.8,2.2) -- (4.2,3.8);
\draw[lightgray, thin,  ->] (7.8,2.2) -- (6.2,3.8);
\draw[lightgray, thin,  ->] (9.8,2.2) -- (8.2,3.8);
\draw[lightgray, thin,  ->] (1.8,4.2) -- (0.2,5.8);
\draw[lightgray, thin,  ->] (3.8,4.2) -- (2.2,5.8);
\draw[lightgray, thin,  ->] (5.8,4.2) -- (4.2,5.8);
\draw[lightgray, thin,  ->] (7.8,4.2) -- (6.2,5.8);
\draw[lightgray, thin,  ->] (9.8,4.2) -- (8.2,5.8);
\draw[lightgray, thin,  ->] (1.8,6.2) -- (0.2,7.8);
\draw[lightgray, thin,  ->] (3.8,6.2) -- (2.2,7.8);
\draw[lightgray, thin,  ->] (5.8,6.2) -- (4.2,7.8);
\draw[lightgray, thin,  ->] (7.8,6.2) -- (6.2,7.8);
\draw[lightgray, thin,  ->] (9.8,6.2) -- (8.2,7.8);
\draw[lightgray, thin,  ->] (1.7,8.2) to [out=170,in=-80]  (0.2,9.7);
\draw[lightgray, thin,  ->] (1.8,8.3) to [out=100,in=-10]  (0.3,9.8);
\draw[lightgray, thin,  ->] (3.7,8.2) to [out=170,in=-80]  (2.2,9.7);
\draw[lightgray, thin,  ->] (3.8,8.3) to [out=100,in=-10]  (2.3,9.8);
\draw[lightgray, thin,  ->] (5.7,8.2) to [out=170,in=-80]  (4.2,9.7);
\draw[lightgray, thin,  ->] (5.8,8.3) to [out=100,in=-10]  (4.3,9.8);
\draw[lightgray, thin,  ->] (7.7,8.2) to [out=170,in=-80]  (6.2,9.7);
\draw[lightgray, thin,  ->] (7.8,8.3) to [out=100,in=-10]  (6.3,9.8);
\draw[lightgray, thin,  ->] (9.6,8.4) to [out=135,in=-45]  (8.4,9.6);
\draw[lightgray, thin,  ->] (9.8,8.3) to [out=90,in=0]  (8.3,9.8);
\end{tikzpicture}
\qquad
}
\caption{The graph of the bottom $\mathscr{C}$-module category
$\mathcal{N}_3$ and the corresponding positive eigenvector}\label{fig7}
\end{figure}
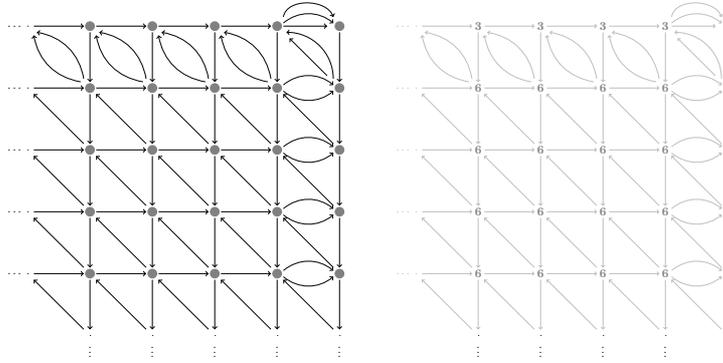

\begin{proof}
To start with, let us note that the bottom weights are exactly 
those integral weights for which the corresponding projective
modules are injective. As projective functors have biadjoints,
they preserve projective-injective modules. This implies that
$\mathcal{N}_3$ is a $\mathscr{C}$-module category.

The weight $(-1,-1)$ is a fixed point for the dot-action of $W$.
This implies that $L((-1,-1))=P((-1,-1))$, moreover, this module
is also injective. As projective functors preserve 
projective-injective modules, it follows that 
$\mathrm{add}(\mathscr{C}\cdot L((-1,-1)))$
is a subcategory of $\mathcal{N}_3$. On the other hand,
any integral projective-injective module can be translated to the 
most singular wall, that is to the additive closure of 
$L((-1,-1))$. This gives us the reversed inclusion and
thus implies that $\mathrm{add}(\mathscr{C}\cdot L((-1,-1)))$
coincides with $\mathcal{N}_3$. 

The simple tops (and socles) of projective-injective modules
are simple Verma modules. Any non-zero homomorphism between 
two projective-injective modules hence induces the identity map
between the corresponding simple Verma subquotients.
Translating this to the additive closure of 
$L((-1,-1))$, we get the identity map on
$L((-1,-1))$. By the previous paragraph, that generates
the whole of $\mathcal{N}_3$ under the action of $\mathscr{C}$.
This implies that $\mathcal{N}_3$ is a simple transitive 
$\mathscr{C}$-module category.

It remains to determine the combinatorics. 
For this, let us recall that every projective object
in $\mathcal{O}$ has a Verma flag, that is a filtration
whose subquotients are Verma modules. For a module
$M$ with a Verma flag, let us denote by
$\mathtt{v}(M)$ the length of a Verma flag
of $M$, which can be alternatively described as 
follows: $\displaystyle\mathtt{v}(M)=\sum_{\text{ bottom }\mu}[M:L(\mu)]$.
Clearly, $\mathtt{v}(M\otimes L((1,0)))=3\mathtt{v}(M)$,
which implies that sending a bottom weight $\mu$
to $\mathtt{v}(N(\mu))$ gives an eigenvector with
eigenvalue $3$ for the graph of $\mathcal{N}_3$.

Note that $\mathtt{v}(N(\mu))$ equals $\frac{|W|}{|G|}$,
where $G$ is the dot-stabilizer of $\mu$ in $W$.
Consequently, $\mathtt{v}(N(\mu))=|W|=6$
if $\mu$ is regular, further $\mathtt{v}(N(-1,-1))=1$ and, finally,
$\mathtt{v}(N(\mu))=3$ for all other singular weights.
This is depicted on the right hand side of Figure~\ref{fig7}.
To determine the graph of $\mathcal{N}_3$, we should again 
go into a case-by-case analysis.

{\bf Case 1.} {\color{violet}Assume that the weights $\mu$, $\mu+(1,0)$,
$\mu+(0,-1)$ and $\mu+(-1,1)$ are all regular.}
This is exactly analogous to Case~1 in the proof of 
Proposition~\ref{prop-s5.3-1} and produces the local
picture \eqref{eq-s5.3-1}.

{\bf Case 2.} {\color{violet}Assume that the weights $\mu$, 
$\mu+(0,-1)$ and $\mu+(-1,1)$ are regular and $\mu+(1,0)$ is singular.}
This is exactly analogous to Case~2 in the proof of 
Proposition~\ref{prop-s5.3-1} and produces the local
picture \eqref{eq-s5.3-2}.

{\bf Case 3.} {\color{violet}Assume that the weights $\mu$, 
$\mu+(0,-1)$ and $\mu+(1,0)$ are regular and $\mu+(-1,1)$ is singular.}
The symmetry argument, when compared with the previous
case, outputs the following local picture:

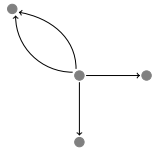
\begin{center}
\resizebox{2cm}{!}{
\begin{tikzpicture}
\draw[gray,fill=gray] (2,2) circle (.9ex);
\draw[gray,fill=gray] (2,0) circle (.9ex);
\draw[gray,fill=gray] (4,2) circle (.9ex);
\draw[gray,fill=gray] (0,4) circle (.9ex);
\draw[black, thick,  ->] (1.8,2.1) to [out=180,in=270]  (0.1,3.8);
\draw[black, thick,  ->] (1.9,2.2) to [out=90,in=350]  (0.2,3.9);
\draw[black, thick,  ->] (2.2,2) -- (3.8,2);
\draw[black, thick,  ->] (2,1.8) -- (2,0.2);
\end{tikzpicture}
}
\end{center}

{\bf Case 4.} {\color{violet}Assume that the weights $\mu$ and
$\mu+(0,-1)$ are regular while  $\mu+(1,0)$ and $\mu+(-1,1)$ are singular.}
In this case $\mu=(-2,-2)$. Here, by merging the Cases 2 and 3 we obtain
the following local picture:

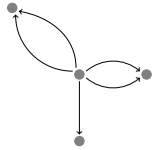
\begin{center}
\resizebox{2cm}{!}{
\begin{tikzpicture}
\draw[gray,fill=gray] (2,2) circle (.9ex);
\draw[gray,fill=gray] (2,0) circle (.9ex);
\draw[gray,fill=gray] (4,2) circle (.9ex);
\draw[gray,fill=gray] (0,4) circle (.9ex);
\draw[black, thick,  ->] (1.8,2.1) to [out=180,in=270]  (0.1,3.8);
\draw[black, thick,  ->] (1.9,2.2) to [out=90,in=350]  (0.2,3.9);
\draw[black, thick,  ->] (2,1.8) -- (2,0.2);
\draw[black, thick,  ->] (2.2,2.1) to [out=40,in=140]  (3.8,2.1);
\draw[black, thick,  ->] (2.2,1.9) to [out=-40,in=-140]  (3.8,1.9);
\end{tikzpicture}
}
\end{center}

{\bf Case 5.} {\color{violet}Assume that the weights $\mu$ and
$\mu+(0,-1)$ are singular with the same singularity,  $\mu+(-1,1)$ is regular
and $\mu+(1,0)$ is not a bottom weight.} 
This is exactly analogous to Case~4 in the proof of 
Proposition~\ref{prop-s5.3-1} and produces the local
picture \eqref{eq-s5.3-3}.

{\bf Case 6.} {\color{violet}Assume that the weights $\mu$ and 
$\mu+(1,0)$ are singular with the same singularity, 
$\mu+(0,-1)$ is regular and $\mu+(-1,1)$ is not a bottom weight.}
The symmetry argument, when compared with the previous
case, outputs the following local picture:

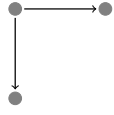
\begin{center}
\resizebox{15mm}{!}{
\begin{tikzpicture}
\draw[gray,fill=gray] (2,2) circle (.9ex);
\draw[gray,fill=gray] (2,0) circle (.9ex);
\draw[gray,fill=gray] (4,2) circle (.9ex);
\draw[black, thick,  ->] (2.2,2) -- (3.8,2);
\draw[black, thick,  ->] (2,1.8) -- (2,0.2);
\end{tikzpicture}
}
\end{center}

{\bf Case 7.} {\color{violet}Assume that the weights $\mu$, 
$\mu+(-1,1)$ and $\mu+(0,-1)$ are singular while 
$\mu+(1,0)$ is not a bottom weight.} 
In this case $\mu=(-1,-2)$. One summand of tensoring
with $L((1,0))$ will be an equivalence between the 
singular blocks of $(-1,-2)$ and $(-1,-3)$. 
As $\mu+(1,0)$ is in the same dot orbit as $\mu+(-1,1)$,
we will get a double contribution for 
translation from the $(-1,-2)$-block to the 
$(-2,-1)$-block. Hence we get the following local picture:

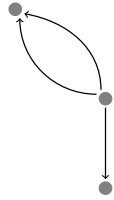
\begin{center}
\resizebox{15mm}{!}{
\begin{tikzpicture}
\draw[gray,fill=gray] (2,2) circle (.9ex);
\draw[gray,fill=gray] (2,0) circle (.9ex);
\draw[gray,fill=gray] (0,4) circle (.9ex);
\draw[black, thick,  ->] (1.8,2.1) to [out=180,in=270]  (0.1,3.8);
\draw[black, thick,  ->] (1.9,2.2) to [out=90,in=350]  (0.2,3.9);
\draw[black, thick,  ->] (2,1.8) -- (2,0.2);
\end{tikzpicture}
}
\end{center}

{\bf Case 8.} {\color{violet}Assume that the weights $\mu$
and $\mu+(1,0)$ are singular, $\mu+(0,-1)$ is regular while 
$\mu+(-1,1)$ is not a bottom weight.} 
In this case $\mu=(-2,-1)$. Note that $\mu+(-1,1)$ is in 
the same dot orbit as the regular weight $\mu+(0,-1)$.
So, they together give one summand of translation 
out of the wall from $(-2,-1)$ and $(-2,-2)$.
Additionally, we also have a translation 
from $(-2,-1)$ to the most singular weight 
$(-1,-1)$. As here the stabilizer grows by a factor 
of $3$, we will have three arrows in the graph,
giving us the following local picture:

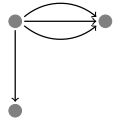
\begin{center}
\resizebox{15mm}{!}{
\begin{tikzpicture}
\draw[gray,fill=gray] (2,2) circle (.9ex);
\draw[gray,fill=gray] (2,0) circle (.9ex);
\draw[gray,fill=gray] (4,2) circle (.9ex);
\draw[black, thick,  ->] (2.2,2.1) to [out=40,in=140]  (3.8,2.1);
\draw[black, thick,  ->] (2.2,1.9) to [out=-40,in=-140]  (3.8,1.9);
\draw[black, thick,  ->] (2,1.8) -- (2,0.2);
\draw[black, thick,  ->] (2.2,2) -- (3.8,2);
\end{tikzpicture}
}
\end{center}

{\bf Case 9.} {\color{violet}Assume that $\mu=(-1,-1)$.}
This is the most singular point. For it, tensoring with
$L((1,0))$ produces just a translation out of the most
singular wall to the singular weight $(-1,-2)$. This
results in one arrow from $\mu$ to $(-1,-2)$.

The final graph of $\mathcal{N}_3$
is depicted on the left hand side of 
Figure~\ref{fig7}.
\end{proof}

\begin{proposition}\label{prop-s5.5-2}
Assume that $\lambda$ is a bottom weight 
of the form $(a,-1)$, for some element 
$a\in\{-2,-3,\dots\}$. Then we have:

\begin{enumerate}[$($a$)$]
\item\label{prop-s5.5-2.1} The $\mathscr{C}$-module category
$\mathrm{add}(\mathscr{C}\cdot L(\lambda))$ contains
$\mathcal{N}_3$ as a $\mathscr{C}$-module subcategory.
\item\label{prop-s5.5-2.2}
The quotient of $\mathrm{add}(\mathscr{C}\cdot L(\lambda))$
by the ideal $\mathcal{I}$ generated by $\mathcal{N}_3$ is a simple 
transitive $\mathscr{C}$-module category that is 
equivalent to $\mathcal{N}_1$. 
\end{enumerate}
\end{proposition}

\begin{proof}
We can translate the module $L(\lambda)$ to the most singular weight
$(-1,-1)$ to obtain $L((-1,-1))$. This implies that 
$\mathrm{add}(\mathscr{C}\cdot L((-1,-1)))=\mathcal{N}_3$
is a $\mathscr{C}$-module subcategory of
$\mathrm{add}(\mathscr{C}\cdot L(\lambda))$, proving
Claim~\eqref{prop-s5.5-3.1}.

To prove Claim~\eqref{prop-s5.5-3.2}, we start by noting that 
$L(\lambda)$ is a Verma module, since $\lambda$ is a bottom weight.
Therefore all modules in $\mathrm{add}(\mathscr{C}\cdot L(\lambda))$
are tilting modules in $\mathcal{O}$, that is, self-dual modules
with Verma flag. Consider the adjoint of the  twisting functor 
$\top_{w_0}$, see \cite{AS,KM}. This functor maps anti-dominant
Verma modules to dominant Verma modules, commutes with the 
action of $\mathscr{C}$ and induces an equivalence between the 
categories of tilting and projective modules in $\mathcal{O}$. 
This means that $\mathrm{add}(\mathscr{C}\cdot L(\lambda))$
is equivalent, as a $\mathscr{C}$-module subcategory, to
$\mathrm{add}(\mathscr{C}\cdot \Delta(\tilde{\lambda}))$,
where $\tilde{\lambda}$ is the dominant weight in the
dot $W$-orbit of $\lambda$. Note that $\tilde{\lambda}$
is an upper middle weight.

Since $\tilde{\lambda}$ is dominant, the module
$\Delta(\tilde{\lambda})$ is projective in $\mathcal{O}$.
As projective functors preserve projective modules, in follows
that $\mathrm{add}(\mathscr{C}\cdot \Delta(\tilde{\lambda}))$
is a subcategory of the category of all projective modules in
$\mathcal{O}$. We claim that 
$\mathrm{add}(\mathscr{C}\cdot \Delta(\tilde{\lambda}))$
is the additive closure of all $P(\mu)$, where $\mu$ is 
either an upper middle weight or a bottom weight. 

Indeed, let $\nu$ be either a top weight or a lower
middle weight and $\theta$ a projective functor. Then,
by adjunction,
\begin{displaymath}
\mathrm{Hom}(\theta P(\tilde{\lambda}),L(\nu))\cong
\mathrm{Hom}(P(\tilde{\lambda}),\theta^* L(\nu)).
\end{displaymath}
In the previous subsections, we have seen that 
$\theta^* L(\nu)$ can only have top or lower middle
subquotients. Therefore
$\mathrm{Hom}(P(\tilde{\lambda}),\theta^* L(\nu))=0$
and we conclude that all indecomposable summands of
$\theta P(\tilde{\lambda})$ are given by either 
bottom or upper middle weights.

The additive closure of all indecomposable 
projective modules indexed by bottom weights is exactly
$\mathcal{N}_3$. Factoring out from
$\mathrm{add}(\mathscr{C}\cdot \Delta(\tilde{\lambda}))$
the ideal generated
by the bottom projectives, each $P(\mu)$, where 
$\mu$ is an upper middle weight, is replaced by 
$P(\mu)/K$, where $K$ is the trace of all bottom
projectives in $P(\mu)$. This quotient surjects onto
$N(\mu)$ and the kernel of this surjection has only simple
subquotients indexed by either lower middle weights or  top weights.
Since we do not have any corresponding projectives in our
picture, the kernel makes no contribution to our story. 
Now Claim~\eqref{prop-s5.5-3.2}
follows from Proposition~\ref{prop-s5.3-1}
and the proof is complete.
\end{proof}

\begin{proposition}\label{prop-s5.5-3}
Assume that $\lambda$ is a bottom weight 
of the form $(-1,b)$, for some element 
$b\in\{-2,-3,\dots\}$. Then we have:

\begin{enumerate}[$($a$)$]
\item\label{prop-s5.5-3.1} The $\mathscr{C}$-module category
$\mathrm{add}(\mathscr{C}\cdot L(\lambda))$ contains
$\mathcal{N}_3$ as a $\mathscr{C}$-module subcategory.
\item\label{prop-s5.5-3.2}
The quotient of $\mathrm{add}(\mathscr{C}\cdot L(\lambda))$
by the ideal $\mathcal{I}$ generated by $\mathcal{N}_3$ is a simple 
transitive $\mathscr{C}$-module category that is 
equivalent to $\mathcal{N}_2$. 
\end{enumerate}
\end{proposition}

\begin{proof}
Mutatis mutandis the proof of Proposition~\ref{prop-s5.5-2}. 
\end{proof}

\begin{proposition}\label{prop-s5.5-4}
Assume that $\lambda$ is a regular bottom weight. Then we have:

\begin{enumerate}[$($a$)$]
\item\label{prop-s5.5-4.1} The $\mathscr{C}$-module category
$\mathrm{add}(\mathscr{C}\cdot L(\lambda))$ contains
$\mathcal{N}_3$ as a $\mathscr{C}$-module subcategory.
\item\label{prop-s5.5-4.2}
The $\mathscr{C}$-module category
$\mathrm{add}(\mathscr{C}\cdot L(\lambda))$ contains
$\mathrm{add}(\mathscr{C}\cdot L((-1,-2)))$ as a $\mathscr{C}$-module subcategory. 
\item\label{prop-s5.5-4.3}
The $\mathscr{C}$-module category
$\mathrm{add}(\mathscr{C}\cdot L(\lambda))$ contains
$\mathrm{add}(\mathscr{C}\cdot L((-2,-1)))$ as a $\mathscr{C}$-module subcategory. 
\item\label{prop-s5.5-4.4}
The category $\mathcal{N}_3$
coincides with the intersection of 
$\mathrm{add}(\mathscr{C}\cdot L((-1,-2)))$ with
$\mathrm{add}(\mathscr{C}\cdot L((-2,-1)))$ inside
$\mathrm{add}(\mathscr{C}\cdot L(\lambda))$. 
\item\label{prop-s5.5-4.5}
The quotient of $\mathrm{add}(\mathscr{C}\cdot L(\lambda))$
by the ideal $\mathcal{I}$, generated by $\mathcal{N}_3$,
$\mathrm{add}(\mathscr{C}\cdot L((-1,-2)))$
and $\mathrm{add}(\mathscr{C}\cdot L((-2,-1)))$,
is a simple 
transitive $\mathscr{C}$-module category that is 
equivalent to ${}_{\mathscr{C}}\mathscr{C}$. 
\end{enumerate}
\end{proposition}

\begin{proof}
As $\lambda$ is regular and anti-dominant, the module $L(\lambda)$
is a regular tilting module in $\mathcal{O}$ and hence 
$\mathrm{add}(\mathscr{C}\cdot L(\lambda))$ coincides with the
category of (integral) tilting modules in $\mathcal{O}$. Applying
$\top_{w_0}$, just like in the proof of Proposition~\ref{prop-s5.5-2},
we get an equivalence of $\mathscr{C}$-module categories
between $\mathrm{add}(\mathscr{C}\cdot L(\lambda))$ and the
category of integral projective modules in $\mathcal{O}$. 

The subcategory of all integral projective-injective modules is
exactly $\mathcal{N}_3$, proving Claim~\eqref{prop-s5.5-4.1}.
Extending this $\mathcal{N}_3$ by all projective modules 
indexed by lower middle weights and taking the additive closure, 
we get exactly  $\mathrm{add}(\mathscr{C}\cdot L((-1,-2)))$, as 
was shown in the proof of Proposition~\ref{prop-s5.5-3}. 
This proves Claim~\eqref{prop-s5.5-4.2}.
Extending $\mathcal{N}_3$ by all projective modules  indexed by 
upper middle weights, we similarly get 
$\mathrm{add}(\mathscr{C}\cdot L((-2,-1)))$.
This proves Claim~\eqref{prop-s5.5-4.3} and Claim~\eqref{prop-s5.5-4.4}.

Factoring from the category of projective modules out
the ideal generated
by all bottom and middle projectives, each $P(\mu)$, where 
$\mu$ is a top weight, is replaced by 
$P(\mu)/K$, where $K$ is the trace of all middle and bottom
projectives in $P(\mu)$. This quotient is 
exactly $L(\mu)$. Now Claim~\eqref{prop-s5.5-4.5}
follows from Theorem~\ref{thm-s4.3-1}
and the proof is complete.
\end{proof}

\subsection{Partially integral weights}\label{s5.6}

For $a\in\mathbb{C}\setminus\mathbb{Z}$, consider
the set $(0,a)+\Lambda$. We split this set into a
disjoint union $(0,a)+\Lambda= X_1\cup X_2$, where
$X_1$ consists of all $\lambda\in (0,a)+\Lambda$ such that
$\lambda_1\geq 0$ and $X_2$ is the complement of 
$X_1$ inside $(0,a)+\Lambda$. Denote by $\mathcal{M}_1^a$
the additive closure of all $L(\lambda)$, where
$\lambda\in X_1$.

\begin{proposition}\label{prop-s5.6-1}
For any $\lambda\in X_1$, the $\mathscr{C}$-module category
$\mathrm{add}(\mathscr{C}\cdot L(\lambda))$ coincides
with $\mathcal{M}_1^a$, moreover, the latter is a simple
transitive $\mathscr{C}$-module category whose graph
and positive eigenvector are depicted in Figure~\ref{fig8}.
\end{proposition}

\begin{figure}
\resizebox{10cm}{!}{
\begin{tikzpicture}
\draw[black, thick,  ->] (0,1.8) -- (0,0.2) node[anchor= north] {\large$\vdots$};
\draw[black, thick,  ->] (1.8,0.2) -- (0.2,1.8);
\draw[black, thick,  ->] (2,1.8) -- (2,0.2) node[anchor= north] {\large$\vdots$};
\draw[black, thick,  ->] (3.8,0.2) -- (2.2,1.8);
\draw[black, thick,  ->] (4,1.8) -- (4,0.2) node[anchor= north] {\large$\vdots$};
\draw[black, thick,  ->] (5.8,0.2) -- (4.2,1.8);
\draw[black, thick,  ->] (6,1.8) -- (6,0.2) node[anchor= north] {\large$\vdots$};
\draw[black, thick,  ->] (7.8,0.2) -- (6.2,1.8);
\draw[black, thick,  ->] (8,1.8) -- (8,0.2) node[anchor= north] {\large$\vdots$};
\draw[black, thick,  ->] (0.2,2) -- (1.8,2);
\draw[black, thick,  ->] (0,3.8) -- (0,2.2);
\draw[black, thick,  ->] (1.8,2.2) -- (0.2,3.8);
\draw[black, thick,  ->] (2.2,2) -- (3.8,2);
\draw[black, thick,  ->] (2,3.8) -- (2,2.2);
\draw[black, thick,  ->] (3.8,2.2) -- (2.2,3.8);
\draw[black, thick,  ->] (4.2,2) -- (5.8,2);
\draw[black, thick,  ->] (4,3.8) -- (4,2.2);
\draw[black, thick,  ->] (5.8,2.2) -- (4.2,3.8);
\draw[black, thick,  ->] (6.2,2) -- (7.8,2);
\draw[black, thick,  ->] (6,3.8) -- (6,2.2);
\draw[black, thick,  ->] (7.8,2.2) -- (6.2,3.8);
\draw[black, thick,  ->] (8,3.8) -- (8,2.2);
\draw[black, thick,  ->] (0.2,4) -- (1.8,4);
\draw[black, thick,  ->] (0,5.8) -- (0,4.2);
\draw[black, thick,  ->] (1.8,4.2) -- (0.2,5.8);
\draw[black, thick,  ->] (2.2,4) -- (3.8,4);
\draw[black, thick,  ->] (2,5.8) -- (2,4.2);
\draw[black, thick,  ->] (3.8,4.2) -- (2.2,5.8);
\draw[black, thick,  ->] (4.2,4) -- (5.8,4);
\draw[black, thick,  ->] (4,5.8) -- (4,4.2);
\draw[black, thick,  ->] (5.8,4.2) -- (4.2,5.8);
\draw[black, thick,  ->] (6.2,4) -- (7.8,4);
\draw[black, thick,  ->] (6,5.8) -- (6,4.2);
\draw[black, thick,  ->] (7.8,4.2) -- (6.2,5.8);
\draw[black, thick,  ->] (8,5.8) -- (8,4.2);
\draw[black, thick,  ->] (0.2,6) -- (1.8,6);
\draw[black, thick,  ->] (0,7.8) -- (0,6.2);
\draw[black, thick,  ->] (1.8,6.2) -- (0.2,7.8);
\draw[black, thick,  ->] (2.2,6) -- (3.8,6);
\draw[black, thick,  ->] (2,7.8) -- (2,6.2);
\draw[black, thick,  ->] (3.8,6.2) -- (2.2,7.8);
\draw[black, thick,  ->] (4.2,6) -- (5.8,6);
\draw[black, thick,  ->] (4,7.8) -- (4,6.2);
\draw[black, thick,  ->] (5.8,6.2) -- (4.2,7.8);
\draw[black, thick,  ->] (6.2,6) -- (7.8,6);
\draw[black, thick,  ->] (6,7.8) -- (6,6.2);
\draw[black, thick,  ->] (7.8,6.2) -- (6.2,7.8);
\draw[black, thick,  ->] (8,7.8) -- (8,6.2);
\draw[black, thick,  -] (9.8,0.35) -- (9.8,0.35) node[anchor= north west] {\large$\ddots$};
\draw[black, thick,  ->] (8.2,2) -- (9.8,2) node[anchor= west] {\large$\cdots$};
\draw[black, thick,  ->] (8.2,4) -- (9.8,4) node[anchor= west] {\large$\cdots$};
\draw[black, thick,  ->] (8.2,6) -- (9.8,6) node[anchor= west] {\large$\cdots$};
\draw[black, thick,  ->] (8.2,8) -- (9.8,8) node[anchor= west] {\large$\cdots$};
\draw[black, thick,  ->] (9.8,0.2) -- (8.2,1.8);
\draw[black, thick,  ->] (9.8,2.2) -- (8.2,3.8);
\draw[black, thick,  ->] (9.8,4.2) -- (8.2,5.8);
\draw[black, thick,  ->] (9.8,6.2) -- (8.2,7.8);
\draw[black, thick,  ->]  (0,9.8) -- (0,8.2);
\draw[black, thick,  ->]  (2,9.8) -- (2,8.2);
\draw[black, thick,  ->]  (4,9.8) -- (4,8.2);
\draw[black, thick,  ->]  (6,9.8) -- (6,8.2);
\draw[black, thick,  ->]  (8,9.8) -- (8,8.2);
\draw[black, thin,  -] (0,9.8) -- (0,9.8) node[anchor= south] {\large$\vdots$};
\draw[black, thin,  -] (2,9.8) -- (2,9.8) node[anchor= south] {\large$\vdots$};
\draw[black, thin,  -] (4,9.8) -- (4,9.8) node[anchor= south] {\large$\vdots$};
\draw[black, thin,  -] (6,9.8) -- (6,9.8) node[anchor= south] {\large$\vdots$};
\draw[black, thin,  -] (8,9.8) -- (8,9.8) node[anchor= south] {\large$\vdots$};
\draw[gray, thin,  -] (10.4,10.4) -- (10.4,10.4) node[anchor= north east] {\large$\iddots$};
\draw[black, thick,  ->] (1.8,8.2) -- (0.2,9.8);
\draw[black, thick,  ->] (3.8,8.2) -- (2.2,9.8);
\draw[black, thick,  ->] (5.8,8.2) -- (4.2,9.8);
\draw[black, thick,  ->] (7.8,8.2) -- (6.2,9.8);
\draw[black, thick,  ->] (9.8,8.2) -- (8.2,9.8);
\draw[black, thick,  ->] (0.2,8) -- (1.8,8);
\draw[black, thick,  ->] (2.2,8) -- (3.8,8);
\draw[black, thick,  ->] (4.2,8) -- (5.8,8);
\draw[black, thick,  ->] (6.2,8) -- (7.8,8);
\draw[gray,fill=gray] (0,2) circle (.9ex);
\draw[gray,fill=gray] (2,2) circle (.9ex);
\draw[gray,fill=gray] (4,2) circle (.9ex);
\draw[gray,fill=gray] (6,2) circle (.9ex);
\draw[gray,fill=gray] (8,2) circle (.9ex);
\draw[gray,fill=gray] (0,4) circle (.9ex);
\draw[gray,fill=gray] (2,4) circle (.9ex);
\draw[gray,fill=gray] (4,4) circle (.9ex);
\draw[gray,fill=gray] (6,4) circle (.9ex);
\draw[gray,fill=gray] (8,4) circle (.9ex);
\draw[gray,fill=gray] (0,6) circle (.9ex);
\draw[gray,fill=gray] (2,6) circle (.9ex);
\draw[gray,fill=gray] (4,6) circle (.9ex);
\draw[gray,fill=gray] (6,6) circle (.9ex);
\draw[gray,fill=gray] (8,6) circle (.9ex);
\draw[gray,fill=gray] (0,8) circle (.9ex);
\draw[gray,fill=gray] (2,8) circle (.9ex);
\draw[gray,fill=gray] (4,8) circle (.9ex);
\draw[gray,fill=gray] (6,8) circle (.9ex);
\draw[gray,fill=gray] (8,8) circle (.9ex);
\end{tikzpicture}
\qquad\qquad
\begin{tikzpicture}
\draw[lightgray, thin,  ->] (0,1.8) -- (0,0.2) node[anchor= north] {\large$\vdots$};
\draw[lightgray, thin,  ->] (1.8,0.2) -- (0.2,1.8);
\draw[lightgray, thin,  ->] (2,1.8) -- (2,0.2) node[anchor= north] {\large$\vdots$};
\draw[lightgray, thin,  ->] (3.8,0.2) -- (2.2,1.8);
\draw[lightgray, thin,  ->] (4,1.8) -- (4,0.2) node[anchor= north] {\large$\vdots$};
\draw[lightgray, thin,  ->] (5.8,0.2) -- (4.2,1.8);
\draw[lightgray, thin,  ->] (6,1.8) -- (6,0.2) node[anchor= north] {\large$\vdots$};
\draw[lightgray, thin,  ->] (7.8,0.2) -- (6.2,1.8);
\draw[lightgray, thin,  ->] (8,1.8) -- (8,0.2) node[anchor= north] {\large$\vdots$};
\draw[lightgray, thin,  ->] (0.2,2) -- (1.8,2);
\draw[lightgray, thin,  ->] (0,3.8) -- (0,2.2);
\draw[lightgray, thin,  ->] (1.8,2.2) -- (0.2,3.8);
\draw[lightgray, thin,  ->] (2.2,2) -- (3.8,2);
\draw[lightgray, thin,  ->] (2,3.8) -- (2,2.2);
\draw[lightgray, thin,  ->] (3.8,2.2) -- (2.2,3.8);
\draw[lightgray, thin,  ->] (4.2,2) -- (5.8,2);
\draw[lightgray, thin,  ->] (4,3.8) -- (4,2.2);
\draw[lightgray, thin,  ->] (5.8,2.2) -- (4.2,3.8);
\draw[lightgray, thin,  ->] (6.2,2) -- (7.8,2);
\draw[lightgray, thin,  ->] (6,3.8) -- (6,2.2);
\draw[lightgray, thin,  ->] (7.8,2.2) -- (6.2,3.8);
\draw[lightgray, thin,  ->] (8,3.8) -- (8,2.2);
\draw[lightgray, thin,  ->] (0.2,4) -- (1.8,4);
\draw[lightgray, thin,  ->] (0,5.8) -- (0,4.2);
\draw[lightgray, thin,  ->] (1.8,4.2) -- (0.2,5.8);
\draw[lightgray, thin,  ->] (2.2,4) -- (3.8,4);
\draw[lightgray, thin,  ->] (2,5.8) -- (2,4.2);
\draw[lightgray, thin,  ->] (3.8,4.2) -- (2.2,5.8);
\draw[lightgray, thin,  ->] (4.2,4) -- (5.8,4);
\draw[lightgray, thin,  ->] (4,5.8) -- (4,4.2);
\draw[lightgray, thin,  ->] (5.8,4.2) -- (4.2,5.8);
\draw[lightgray, thin,  ->] (6.2,4) -- (7.8,4);
\draw[lightgray, thin,  ->] (6,5.8) -- (6,4.2);
\draw[lightgray, thin,  ->] (7.8,4.2) -- (6.2,5.8);
\draw[lightgray, thin,  ->] (8,5.8) -- (8,4.2);
\draw[lightgray, thin,  ->] (0.2,6) -- (1.8,6);
\draw[lightgray, thin,  ->] (0,7.8) -- (0,6.2);
\draw[lightgray, thin,  ->] (1.8,6.2) -- (0.2,7.8);
\draw[lightgray, thin,  ->] (2.2,6) -- (3.8,6);
\draw[lightgray, thin,  ->] (2,7.8) -- (2,6.2);
\draw[lightgray, thin,  ->] (3.8,6.2) -- (2.2,7.8);
\draw[lightgray, thin,  ->] (4.2,6) -- (5.8,6);
\draw[lightgray, thin,  ->] (4,7.8) -- (4,6.2);
\draw[lightgray, thin,  ->] (5.8,6.2) -- (4.2,7.8);
\draw[lightgray, thin,  ->] (6.2,6) -- (7.8,6);
\draw[lightgray, thin,  ->] (6,7.8) -- (6,6.2);
\draw[lightgray, thin,  ->] (7.8,6.2) -- (6.2,7.8);
\draw[lightgray, thin,  ->] (8,7.8) -- (8,6.2);
\draw[lightgray, thin,  ->] (9.8,0.35) -- (9.8,0.35) node[anchor= north west] {\large$\ddots$};
\draw[lightgray, thin,  ->] (8.2,2) -- (9.8,2) node[anchor= west] {\large$\cdots$};
\draw[lightgray, thin,  ->] (8.2,4) -- (9.8,4) node[anchor= west] {\large$\cdots$};
\draw[lightgray, thin,  ->] (8.2,6) -- (9.8,6) node[anchor= west] {\large$\cdots$};
\draw[lightgray, thin,  ->] (8.2,8) -- (9.8,8) node[anchor= west] {\large$\cdots$};
\draw[lightgray, thin,  ->] (9.8,0.2) -- (8.2,1.8);
\draw[lightgray, thin,  ->] (9.8,2.2) -- (8.2,3.8);
\draw[lightgray, thin,  ->] (9.8,4.2) -- (8.2,5.8);
\draw[lightgray, thin,  ->] (9.8,6.2) -- (8.2,7.8);
\draw[lightgray, thin,  ->]  (0,9.8) -- (0,8.2);
\draw[lightgray, thin,  ->]  (2,9.8) -- (2,8.2);
\draw[lightgray, thin,  ->]  (4,9.8) -- (4,8.2);
\draw[lightgray, thin,  ->]  (6,9.8) -- (6,8.2);
\draw[lightgray, thin,  ->]  (8,9.8) -- (8,8.2);
\draw[black, thin,  -] (0,9.8) -- (0,9.8) node[anchor= south] {\large$\vdots$};
\draw[black, thin,  -] (2,9.8) -- (2,9.8) node[anchor= south] {\large$\vdots$};
\draw[black, thin,  -] (4,9.8) -- (4,9.8) node[anchor= south] {\large$\vdots$};
\draw[black, thin,  -] (6,9.8) -- (6,9.8) node[anchor= south] {\large$\vdots$};
\draw[black, thin,  -] (8,9.8) -- (8,9.8) node[anchor= south] {\large$\vdots$};
\draw[gray, thin,  -] (10.4,10.4) -- (10.4,10.4) node[anchor= north east] {\large$\iddots$};
\draw[lightgray, thin,  ->] (1.8,8.2) -- (0.2,9.8);
\draw[lightgray, thin,  ->] (3.8,8.2) -- (2.2,9.8);
\draw[lightgray, thin,  ->] (5.8,8.2) -- (4.2,9.8);
\draw[lightgray, thin,  ->] (7.8,8.2) -- (6.2,9.8);
\draw[lightgray, thin,  ->] (9.8,8.2) -- (8.2,9.8);
\draw[lightgray, thin,  ->] (0.2,8) -- (1.8,8);
\draw[lightgray, thin,  ->] (2.2,8) -- (3.8,8);
\draw[lightgray, thin,  ->] (4.2,8) -- (5.8,8);
\draw[lightgray, thin,  ->] (6.2,8) -- (7.8,8);
\draw[gray,fill=gray] (0,2) circle (.01ex) node {\large$\mathbf{1}$};
\draw[gray,fill=gray] (2,2) circle (.01ex) node {\large$\mathbf{2}$};
\draw[gray,fill=gray] (4,2) circle (.01ex) node {\large$\mathbf{3}$};
\draw[gray,fill=gray] (6,2) circle (.01ex) node {\large$\mathbf{4}$};
\draw[gray,fill=gray] (8,2) circle (.01ex) node {\large$\mathbf{5}$};
\draw[gray,fill=gray] (0,4) circle (.01ex) node {\large$\mathbf{1}$};
\draw[gray,fill=gray] (2,4) circle (.01ex) node {\large$\mathbf{2}$};
\draw[gray,fill=gray] (4,4) circle (.01ex) node {\large$\mathbf{3}$};
\draw[gray,fill=gray] (6,4) circle (.01ex) node {\large$\mathbf{4}$};
\draw[gray,fill=gray] (8,4) circle (.01ex) node {\large$\mathbf{5}$};
\draw[gray,fill=gray] (0,6) circle (.01ex) node {\large$\mathbf{1}$};
\draw[gray,fill=gray] (2,6) circle (.01ex) node {\large$\mathbf{2}$};
\draw[gray,fill=gray] (4,6) circle (.01ex) node {\large$\mathbf{3}$};
\draw[gray,fill=gray] (6,6) circle (.01ex) node {\large$\mathbf{4}$};
\draw[gray,fill=gray] (8,6) circle (.01ex) node {\large$\mathbf{5}$};
\draw[gray,fill=gray] (0,8) circle (.01ex) node {\large$\mathbf{1}$};
\draw[gray,fill=gray] (2,8) circle (.01ex) node {\large$\mathbf{2}$};
\draw[gray,fill=gray] (4,8) circle (.01ex) node {\large$\mathbf{3}$};
\draw[gray,fill=gray] (6,8) circle (.01ex) node {\large$\mathbf{4}$};
\draw[gray,fill=gray] (8,8) circle (.01ex) node {\large$\mathbf{5}$};
\end{tikzpicture}
}
\caption{The first case of partially integral weights}\label{fig8}
\end{figure}

\begin{proof}
For $\lambda\in X_1$ the intersection of the dot-orbit of $\lambda$
with $(0,a)+\Lambda$ consists of $\lambda$ and $s\cdot \lambda$.
Therefore, for each $\lambda\in X_1$, we have a short exact sequence
\begin{displaymath}
0\to L(s\cdot \lambda)\to\Delta(\lambda)\to L(\lambda)\to 0. 
\end{displaymath}
The module $L((1,0))\otimes_{\mathbb{C}}\Delta(\lambda)$ has a 
standard filtration (Verma flag)
whose subquotients are $\Delta(\lambda+(1,0))$, $\Delta(\lambda+(0,-1))$
and $\Delta(\lambda+(-1,1))$, each occurring with multiplicity one. 
For $\lambda\in X_1$, the weights 
$\lambda+(1,0)$, $\lambda+(0,-1)$ and $\lambda+(-1,1)$ belong
to different dot-orbits and hence
the modules $\Delta(\lambda+(1,0))$, 
$\Delta(\lambda+(0,-1))$ and $\Delta(\lambda+(-1,1))$ have different central characters.
Therefore, $L((1,0))\otimes_{\mathbb{C}}\Delta(\lambda)$ splits as their direct sum.

Next we observe that, if $\lambda_1\neq 0$, then all three weights 
$\lambda+(1,0)$, $\lambda+(0,-1)$ and $\lambda+(-1,1)$ still
belong to $X_1$. Therefore, in this case the module 
$L((1,0))\otimes_{\mathbb{C}} L(\lambda)$ splits as a direct sum of 
$L(\lambda+(1,0))$, $L(\lambda+(0,-1))$ and $L(\lambda+(-1,1))$.

If $\lambda_1=0$, then $\lambda+(1,0)\in X_1$ and $\lambda+(0,-1)\in X_1$ 
but $\lambda+(-1,1)\not\in X_1$.  Therefore, in this case the module 
$L((1,0))\otimes_{\mathbb{C}} L(\lambda)$ splits as a direct sum of 
$L(\lambda+(1,0))$ and $L(\lambda+(0,-1))$. This implies 
both that $\mathrm{add}(\mathscr{C}\cdot L(\lambda))$ coincides
with $\mathcal{M}_1^a$ and that the latter is simple transitive.
We also have that the combinatorics of this simple transitive 
$\mathscr{C}$-module category is described by the left 
part of Figure~\ref{fig8}.

It is straightforward to verify that the right  
part of Figure~\ref{fig8} provides an eigenvector, as claimed.
\end{proof}

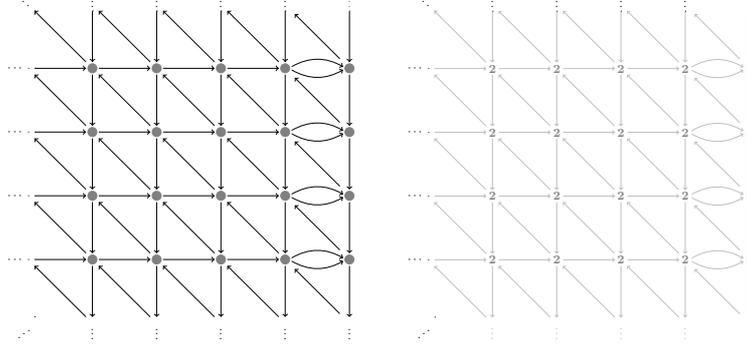
\begin{figure}
\resizebox{10cm}{!}{
\begin{tikzpicture}
\draw[black, thick,  ->] (10,1.8) -- (10,0.2) node[anchor= north] {\large$\vdots$};
\draw[black, thick,  ->] (1.8,0.2) -- (0.2,1.8);
\draw[black, thick,  ->] (2,1.8) -- (2,0.2) node[anchor= north] {\large$\vdots$};
\draw[black, thick,  ->] (3.8,0.2) -- (2.2,1.8);
\draw[black, thick,  ->] (4,1.8) -- (4,0.2) node[anchor= north] {\large$\vdots$};
\draw[black, thick,  ->] (5.8,0.2) -- (4.2,1.8);
\draw[black, thick,  ->] (6,1.8) -- (6,0.2) node[anchor= north] {\large$\vdots$};
\draw[black, thick,  ->] (7.8,0.2) -- (6.2,1.8);
\draw[black, thick,  ->] (8,1.8) -- (8,0.2) node[anchor= north] {\large$\vdots$};
\draw[black, thick,  ->] (0.2,2) -- (1.8,2);
\draw[black, thick,  ->] (10,3.8) -- (10,2.2);
\draw[black, thick,  ->] (1.8,2.2) -- (0.2,3.8);
\draw[black, thick,  ->] (2.2,2) -- (3.8,2);
\draw[black, thick,  ->] (2,3.8) -- (2,2.2);
\draw[black, thick,  ->] (3.8,2.2) -- (2.2,3.8);
\draw[black, thick,  ->] (4.2,2) -- (5.8,2);
\draw[black, thick,  ->] (4,3.8) -- (4,2.2);
\draw[black, thick,  ->] (5.8,2.2) -- (4.2,3.8);
\draw[black, thick,  ->] (6.2,2) -- (7.8,2);
\draw[black, thick,  ->] (6,3.8) -- (6,2.2);
\draw[black, thick,  ->] (7.8,2.2) -- (6.2,3.8);
\draw[black, thick,  ->] (8,3.8) -- (8,2.2);
\draw[black, thick,  ->] (0.2,4) -- (1.8,4);
\draw[black, thick,  ->] (10,5.8) -- (10,4.2);
\draw[black, thick,  ->] (1.8,4.2) -- (0.2,5.8);
\draw[black, thick,  ->] (2.2,4) -- (3.8,4);
\draw[black, thick,  ->] (2,5.8) -- (2,4.2);
\draw[black, thick,  ->] (3.8,4.2) -- (2.2,5.8);
\draw[black, thick,  ->] (4.2,4) -- (5.8,4);
\draw[black, thick,  ->] (4,5.8) -- (4,4.2);
\draw[black, thick,  ->] (5.8,4.2) -- (4.2,5.8);
\draw[black, thick,  ->] (6.2,4) -- (7.8,4);
\draw[black, thick,  ->] (6,5.8) -- (6,4.2);
\draw[black, thick,  ->] (7.8,4.2) -- (6.2,5.8);
\draw[black, thick,  ->] (8,5.8) -- (8,4.2);
\draw[black, thick,  ->] (0.2,6) -- (1.8,6);
\draw[black, thick,  ->] (10,7.8) -- (10,6.2);
\draw[black, thick,  ->] (1.8,6.2) -- (0.2,7.8);
\draw[black, thick,  ->] (2.2,6) -- (3.8,6);
\draw[black, thick,  ->] (2,7.8) -- (2,6.2);
\draw[black, thick,  ->] (3.8,6.2) -- (2.2,7.8);
\draw[black, thick,  ->] (4.2,6) -- (5.8,6);
\draw[black, thick,  ->] (4,7.8) -- (4,6.2);
\draw[black, thick,  ->] (5.8,6.2) -- (4.2,7.8);
\draw[black, thick,  ->] (6.2,6) -- (7.8,6);
\draw[black, thick,  ->] (6,7.8) -- (6,6.2);
\draw[black, thick,  ->] (7.8,6.2) -- (6.2,7.8);
\draw[black, thick,  ->] (8,7.8) -- (8,6.2);
\draw[black, thick,  ->] (8.2,2.05) to [out=30,in=150]   (9.8,2.05);
\draw[black, thick,  ->] (8.2,1.95) to [out=-30,in=-150]   (9.8,1.95);
\draw[black, thick,  ->] (8.2,4.05) to [out=30,in=150]   (9.8,4.05);
\draw[black, thick,  ->] (8.2,3.95) to [out=-30,in=-150]   (9.8,3.95);
\draw[black, thick,  ->] (8.2,6.05) to [out=30,in=150]   (9.8,6.05);
\draw[black, thick,  ->] (8.2,5.95) to [out=-30,in=-150]   (9.8,5.95);
\draw[black, thick,  ->] (8.2,8.05) to [out=30,in=150]   (9.8,8.05);
\draw[black, thick,  ->] (8.2,7.95) to [out=-30,in=-150]   (9.8,7.95);
\draw[black, thick,  ->] (9.7,0.3)  --  (8.3,1.7);
\draw[black, thick,  ->] (9.7,2.3)  --  (8.3,3.7);
\draw[black, thick,  ->] (9.7,4.3)  --  (8.3,5.7);
\draw[black, thick,  ->] (9.7,6.3)  --  (8.3,7.7);
\draw[black, thick,  ->] (9.7,8.3)  --  (8.3,9.7);
\draw[black, thick,  ->]  (10,9.8) -- (10,8.2);
\draw[black, thick,  ->]  (2,9.8) -- (2,8.2);
\draw[black, thick,  ->]  (4,9.8) -- (4,8.2);
\draw[black, thick,  ->]  (6,9.8) -- (6,8.2);
\draw[black, thick,  ->]  (8,9.8) -- (8,8.2);
\draw[black, thin,  -] (0,2) -- (0,2) node[anchor= east] {\large$\cdots$};
\draw[black, thin,  -] (0,4) -- (0,4) node[anchor= east] {\large$\cdots$};
\draw[black, thin,  -] (0,6) -- (0,6) node[anchor= east] {\large$\cdots$};
\draw[black, thin,  -] (0,8) -- (0,8) node[anchor= east] {\large$\cdots$};
\draw[black, thin,  -] (0.2,0.2) -- (0.2,0.2) node[anchor= north east] {\large$\iddots$};
\draw[black, thin,  -] (0.2,9.8) -- (0.2,9.8) node[anchor= south east] {\large$\ddots$};
\draw[black, thin,  -] (2,9.8) -- (2,9.8) node[anchor= south] {\large$\vdots$};
\draw[black, thin,  -] (4,9.8) -- (4,9.8) node[anchor= south] {\large$\vdots$};
\draw[black, thin,  -] (6,9.8) -- (6,9.8) node[anchor= south] {\large$\vdots$};
\draw[black, thin,  -] (8,9.8) -- (8,9.8) node[anchor= south] {\large$\vdots$};
\draw[gray, thin,  -] (10,9.8) -- (10,9.8) node[anchor= south] {\large$\vdots$};
\draw[black, thick,  ->] (1.8,8.2) -- (0.2,9.8);
\draw[black, thick,  ->] (3.8,8.2) -- (2.2,9.8);
\draw[black, thick,  ->] (5.8,8.2) -- (4.2,9.8);
\draw[black, thick,  ->] (7.8,8.2) -- (6.2,9.8);
\draw[black, thick,  ->] (0.2,8) -- (1.8,8);
\draw[black, thick,  ->] (2.2,8) -- (3.8,8);
\draw[black, thick,  ->] (4.2,8) -- (5.8,8);
\draw[black, thick,  ->] (6.2,8) -- (7.8,8);
\draw[gray,fill=gray] (10,2) circle (.9ex);
\draw[gray,fill=gray] (2,2) circle (.9ex);
\draw[gray,fill=gray] (4,2) circle (.9ex);
\draw[gray,fill=gray] (6,2) circle (.9ex);
\draw[gray,fill=gray] (8,2) circle (.9ex);
\draw[gray,fill=gray] (10,4) circle (.9ex);
\draw[gray,fill=gray] (2,4) circle (.9ex);
\draw[gray,fill=gray] (4,4) circle (.9ex);
\draw[gray,fill=gray] (6,4) circle (.9ex);
\draw[gray,fill=gray] (8,4) circle (.9ex);
\draw[gray,fill=gray] (10,6) circle (.9ex);
\draw[gray,fill=gray] (2,6) circle (.9ex);
\draw[gray,fill=gray] (4,6) circle (.9ex);
\draw[gray,fill=gray] (6,6) circle (.9ex);
\draw[gray,fill=gray] (8,6) circle (.9ex);
\draw[gray,fill=gray] (10,8) circle (.9ex);
\draw[gray,fill=gray] (2,8) circle (.9ex);
\draw[gray,fill=gray] (4,8) circle (.9ex);
\draw[gray,fill=gray] (6,8) circle (.9ex);
\draw[gray,fill=gray] (8,8) circle (.9ex);
\end{tikzpicture}
\qquad\qquad
\begin{tikzpicture}
\draw[lightgray, thin,  ->] (10,1.8) -- (10,0.2) node[anchor= north] {\large$\vdots$};
\draw[lightgray, thin,  ->] (1.8,0.2) -- (0.2,1.8);
\draw[lightgray, thin,  ->] (2,1.8) -- (2,0.2) node[anchor= north] {\large$\vdots$};
\draw[lightgray, thin,  ->] (3.8,0.2) -- (2.2,1.8);
\draw[lightgray, thin,  ->] (4,1.8) -- (4,0.2) node[anchor= north] {\large$\vdots$};
\draw[lightgray, thin,  ->] (5.8,0.2) -- (4.2,1.8);
\draw[lightgray, thin,  ->] (6,1.8) -- (6,0.2) node[anchor= north] {\large$\vdots$};
\draw[lightgray, thin,  ->] (7.8,0.2) -- (6.2,1.8);
\draw[lightgray, thin,  ->] (8,1.8) -- (8,0.2) node[anchor= north] {\large$\vdots$};
\draw[lightgray, thin,  ->] (0.2,2) -- (1.8,2);
\draw[lightgray, thin,  ->] (10,3.8) -- (10,2.2);
\draw[lightgray, thin,  ->] (1.8,2.2) -- (0.2,3.8);
\draw[lightgray, thin,  ->] (2.2,2) -- (3.8,2);
\draw[lightgray, thin,  ->] (2,3.8) -- (2,2.2);
\draw[lightgray, thin,  ->] (3.8,2.2) -- (2.2,3.8);
\draw[lightgray, thin,  ->] (4.2,2) -- (5.8,2);
\draw[lightgray, thin,  ->] (4,3.8) -- (4,2.2);
\draw[lightgray, thin,  ->] (5.8,2.2) -- (4.2,3.8);
\draw[lightgray, thin,  ->] (6.2,2) -- (7.8,2);
\draw[lightgray, thin,  ->] (6,3.8) -- (6,2.2);
\draw[lightgray, thin,  ->] (7.8,2.2) -- (6.2,3.8);
\draw[lightgray, thin,  ->] (8,3.8) -- (8,2.2);
\draw[lightgray, thin,  ->] (0.2,4) -- (1.8,4);
\draw[lightgray, thin,  ->] (10,5.8) -- (10,4.2);
\draw[lightgray, thin,  ->] (1.8,4.2) -- (0.2,5.8);
\draw[lightgray, thin,  ->] (2.2,4) -- (3.8,4);
\draw[lightgray, thin,  ->] (2,5.8) -- (2,4.2);
\draw[lightgray, thin,  ->] (3.8,4.2) -- (2.2,5.8);
\draw[lightgray, thin,  ->] (4.2,4) -- (5.8,4);
\draw[lightgray, thin,  ->] (4,5.8) -- (4,4.2);
\draw[lightgray, thin,  ->] (5.8,4.2) -- (4.2,5.8);
\draw[lightgray, thin,  ->] (6.2,4) -- (7.8,4);
\draw[lightgray, thin,  ->] (6,5.8) -- (6,4.2);
\draw[lightgray, thin,  ->] (7.8,4.2) -- (6.2,5.8);
\draw[lightgray, thin,  ->] (8,5.8) -- (8,4.2);
\draw[lightgray, thin,  ->] (0.2,6) -- (1.8,6);
\draw[lightgray, thin,  ->] (10,7.8) -- (10,6.2);
\draw[lightgray, thin,  ->] (1.8,6.2) -- (0.2,7.8);
\draw[lightgray, thin,  ->] (2.2,6) -- (3.8,6);
\draw[lightgray, thin,  ->] (2,7.8) -- (2,6.2);
\draw[lightgray, thin,  ->] (3.8,6.2) -- (2.2,7.8);
\draw[lightgray, thin,  ->] (4.2,6) -- (5.8,6);
\draw[lightgray, thin,  ->] (4,7.8) -- (4,6.2);
\draw[lightgray, thin,  ->] (5.8,6.2) -- (4.2,7.8);
\draw[lightgray, thin,  ->] (6.2,6) -- (7.8,6);
\draw[lightgray, thin,  ->] (6,7.8) -- (6,6.2);
\draw[lightgray, thin,  ->] (7.8,6.2) -- (6.2,7.8);
\draw[lightgray, thin,  ->] (8,7.8) -- (8,6.2);
\draw[lightgray, thin,  ->] (8.2,2.05) to [out=30,in=150]   (9.8,2.05);
\draw[lightgray, thin,  ->] (8.2,1.95) to [out=-30,in=-150]   (9.8,1.95);
\draw[lightgray, thin,  ->] (8.2,4.05) to [out=30,in=150]   (9.8,4.05);
\draw[lightgray, thin,  ->] (8.2,3.95) to [out=-30,in=-150]   (9.8,3.95);
\draw[lightgray, thin,  ->] (8.2,6.05) to [out=30,in=150]   (9.8,6.05);
\draw[lightgray, thin,  ->] (8.2,5.95) to [out=-30,in=-150]   (9.8,5.95);
\draw[lightgray, thin,  ->] (8.2,8.05) to [out=30,in=150]   (9.8,8.05);
\draw[lightgray, thin,  ->] (8.2,7.95) to [out=-30,in=-150]   (9.8,7.95);
\draw[lightgray, thin,  ->] (9.7,0.3)  --  (8.3,1.7);
\draw[lightgray, thin,  ->] (9.7,2.3)  --  (8.3,3.7);
\draw[lightgray, thin,  ->] (9.7,4.3)  --  (8.3,5.7);
\draw[lightgray, thin,  ->] (9.7,6.3)  --  (8.3,7.7);
\draw[lightgray, thin,  ->] (9.7,8.3)  --  (8.3,9.7);
\draw[lightgray, thin,  ->]  (10,9.8) -- (10,8.2);
\draw[lightgray, thin,  ->]  (2,9.8) -- (2,8.2);
\draw[lightgray, thin,  ->]  (4,9.8) -- (4,8.2);
\draw[lightgray, thin,  ->]  (6,9.8) -- (6,8.2);
\draw[lightgray, thin,  ->]  (8,9.8) -- (8,8.2);
\draw[black, thin,  -] (0,2) -- (0,2) node[anchor= east] {\large$\cdots$};
\draw[black, thin,  -] (0,4) -- (0,4) node[anchor= east] {\large$\cdots$};
\draw[black, thin,  -] (0,6) -- (0,6) node[anchor= east] {\large$\cdots$};
\draw[black, thin,  -] (0,8) -- (0,8) node[anchor= east] {\large$\cdots$};
\draw[black, thin,  -] (0.2,0.2) -- (0.2,0.2) node[anchor= north east] {\large$\iddots$};
\draw[black, thin,  -] (0.2,9.8) -- (0.2,9.8) node[anchor= south east] {\large$\ddots$};
\draw[black, thin,  -] (2,9.8) -- (2,9.8) node[anchor= south] {\large$\vdots$};
\draw[black, thin,  -] (4,9.8) -- (4,9.8) node[anchor= south] {\large$\vdots$};
\draw[black, thin,  -] (6,9.8) -- (6,9.8) node[anchor= south] {\large$\vdots$};
\draw[black, thin,  -] (8,9.8) -- (8,9.8) node[anchor= south] {\large$\vdots$};
\draw[gray, thin,  -] (10,9.8) -- (10,9.8) node[anchor= south] {\large$\vdots$};
\draw[lightgray, thin,  ->] (1.8,8.2) -- (0.2,9.8);
\draw[lightgray, thin,  ->] (3.8,8.2) -- (2.2,9.8);
\draw[lightgray, thin,  ->] (5.8,8.2) -- (4.2,9.8);
\draw[lightgray, thin,  ->] (7.8,8.2) -- (6.2,9.8);
\draw[lightgray, thin,  ->] (0.2,8) -- (1.8,8);
\draw[lightgray, thin,  ->] (2.2,8) -- (3.8,8);
\draw[lightgray, thin,  ->] (4.2,8) -- (5.8,8);
\draw[lightgray, thin,  ->] (6.2,8) -- (7.8,8);
\draw[gray,fill=gray] (10,2) circle (.01ex) node {\large$\mathbf{1}$};
\draw[gray,fill=gray] (2,2) circle (.01ex) node {\large$\mathbf{2}$};
\draw[gray,fill=gray] (4,2) circle (.01ex) node {\large$\mathbf{2}$};
\draw[gray,fill=gray] (6,2) circle (.01ex) node {\large$\mathbf{2}$};
\draw[gray,fill=gray] (8,2) circle (.01ex) node {\large$\mathbf{2}$};
\draw[gray,fill=gray] (10,4) circle (.01ex) node {\large$\mathbf{1}$};
\draw[gray,fill=gray] (2,4) circle (.01ex) node {\large$\mathbf{2}$};
\draw[gray,fill=gray] (4,4) circle (.01ex) node {\large$\mathbf{2}$};
\draw[gray,fill=gray] (6,4) circle (.01ex) node {\large$\mathbf{2}$};
\draw[gray,fill=gray] (8,4) circle (.01ex) node {\large$\mathbf{2}$};
\draw[gray,fill=gray] (10,6) circle (.01ex) node {\large$\mathbf{1}$};
\draw[gray,fill=gray] (2,6) circle (.01ex) node {\large$\mathbf{2}$};
\draw[gray,fill=gray] (4,6) circle (.01ex) node {\large$\mathbf{2}$};
\draw[gray,fill=gray] (6,6) circle (.01ex) node {\large$\mathbf{2}$};
\draw[gray,fill=gray] (8,6) circle (.01ex) node {\large$\mathbf{2}$};
\draw[gray,fill=gray] (10,8) circle (.01ex) node {\large$\mathbf{1}$};
\draw[gray,fill=gray] (2,8) circle (.01ex) node {\large$\mathbf{2}$};
\draw[gray,fill=gray] (4,8) circle (.01ex) node {\large$\mathbf{2}$};
\draw[gray,fill=gray] (6,8) circle (.01ex) node {\large$\mathbf{2}$};
\draw[gray,fill=gray] (8,8) circle (.01ex) node {\large$\mathbf{2}$};
\end{tikzpicture}
}
\caption{The second case of partially integral weights}\label{fig9}
\end{figure}

Denote by $\mathcal{M}_2^a$
the additive closure of all $P(\lambda)$, where
$\lambda\in X_2$. 

\begin{proposition}\label{prop-s5.6-2}
\begin{enumerate}[$($a$)$]
\item\label{prop-s5.6-2.1} 
For $\lambda\in X_2$ such that $\lambda_1=-1$, 
$\mathrm{add}(\mathscr{C}\cdot L(\lambda))$ coincides with
$\mathcal{M}_2^a$ and is a simple transitive $\mathscr{C}$-module 
category whose graph
and positive eigenvector are depicted in Figure~\ref{fig9}.
\item\label{prop-s5.6-2.2}
For $\lambda\in X_2$ such that $\lambda_1\neq -1$,
$\mathrm{add}(\mathscr{C}\cdot L(\lambda))$
contains $\mathcal{M}_2^a$ as a $\mathscr{C}$-module 
subcategory, moreover, the quotient of $\mathrm{add}(\mathscr{C}\cdot L(\lambda))$
by the ideal $\mathcal{I}$ generated by $\mathcal{M}_2^a$ is a simple 
transitive $\mathscr{C}$-module category that is 
equivalent to $\mathcal{M}_1^a$. 
\item\label{prop-s5.6-2.3}
If $\lambda_1\neq -1$, then the indecomposable
objects in  $\mathrm{add}(\mathscr{C}\cdot L(\lambda))/\mathcal{I}$  
are given by the images of 
$L(\mu)$, where $\mu\in X_2$ is such that $\mu_1\neq -1$.
\end{enumerate}
\end{proposition}

\begin{proof}
A weight $\lambda\in X_2$ is singular if and only if $\lambda_1=-1$.
The integral Weyl group of all weights in $(0,a)+\Lambda$
equals $\{e,s\}$. Consequently, $L(\lambda)=P(\lambda)=I(\lambda)$
for all singular weights in $X_2$. 

Note that $X_2$ coincides with the set of all 
antidominant weights in $(0,a)+\Lambda$ (with respect to the 
integral Weyl group), equivalently, the set of all weights
for which the corresponding projective module is injective.
As applying projective functors
preserves both projectivity and injectivity, it follows that,
for each singular $\lambda\in X_2$, the category
$\mathrm{add}(\mathscr{C}\cdot L(\lambda))$
is a subcategory of $\mathcal{M}_2^a$.

Let us determine the combinatorics of $\mathcal{M}_2^a$.
If $\lambda\in X_2$ is regular, then we have a short exact sequence
\begin{displaymath}
0\to \Delta(s\cdot\lambda)\to P(\lambda)\to\Delta(\lambda)\to 0. 
\end{displaymath}
We have to consider three cases

{\bf Case 1.} Assume $\lambda_1\leq -3$.
Looking at the Verma filtration of $L((1,0))\otimes_\mathbb{C}P(\lambda)$, 
we see that the highest weights in $X_2$ of the Verma subquotients appearing 
in that filtration are $\lambda+(1,0)$, $\lambda+(0,-1)$ and $\lambda+(-1,1)$,
each appearing with multiplicity one. As all these weights belong to $X_2$ 
and are regular, it follows that 
\begin{displaymath}
L((1,0))\otimes_\mathbb{C}P(\lambda)\cong
P(\lambda+(1,0))\oplus P(\lambda+(0,-1))\oplus P(\lambda+(-1,1)).
\end{displaymath}

{\bf Case 2.} Assume $\lambda_1=-2$.
Looking at the Verma filtration of $L((1,0))\otimes_\mathbb{C}P(\lambda)$, 
we see that the highest weights in $X_2$ of the Verma subquotients appearing 
in that filtration are $\lambda+(1,0)$, $\lambda+(0,-1)$ and $\lambda+(-1,1)$.
However, the weights $\lambda+(0,-1)$ and $\lambda+(-1,1)$
appear with multiplicity one while the weight $\lambda+(1,0)$
appears with multiplicity two (as $\lambda+(1,0)=s\cdot \lambda+(-1,1)$).
As all these weights belong to $X_2$, it follows that 
\begin{displaymath}
L((1,0))\otimes_\mathbb{C}P(\lambda)\cong
P(\lambda+(1,0))\oplus 
P(\lambda+(1,0))\oplus P(\lambda+(0,-1))\oplus P(\lambda+(-1,1)).
\end{displaymath}

{\bf Case 3.} Assume $\lambda_1=-1$.
Looking at the Verma filtration of $L((1,0))\otimes_\mathbb{C}P(\lambda)$, 
we see that the highest weights in $X_2$ of the Verma subquotients appearing 
in that filtration are $\lambda+(0,-1)$ and $\lambda+(-1,1)$, both
appearing with multiplicity one. It follows that 
\begin{displaymath}
L((1,0))\otimes_\mathbb{C}P(\lambda)\cong
P(\lambda+(0,-1))\oplus P(\lambda+(-1,1)).
\end{displaymath}

Combining the three cases, we get that the combinatorics of 
$\mathcal{M}_2^a$ is given by the graph on the left hand side of 
Figure~\ref{fig9}. As this graph is strongly connected, it follows
that $\mathcal{M}_2^a$ is transitive. That the right hand side
of Figure~\ref{fig9} provides an eigenvector follows from a direct
computation. Note that this eigenvector can be interpreted as 
the Verma length of the projective object corresponding to a vertex.

The category $\mathcal{M}_2^a$
is not semi-simple, however, there are no 
homomorphisms between different indecomposable objects in 
$\mathcal{M}_2^a$. In fact, for $\lambda\in X_2$, the endomorphism
algebra of $P(\lambda)$ is trivial if and only if $\lambda$
is singular. If $\lambda$ is regular, this endomorphism algebra is
isomorphic to the algebra of dual numbers. Let $\mathscr{J}$ be a 
$\mathscr{C}$-stable ideal of $\mathcal{M}_2^a$ different from 
$\mathcal{M}_2^a$. Then $\mathscr{J}$ must contain a non-zero
radical endomorphism of some regular $P(\lambda)$. Translating
this onto the $s$-wall gives a non-zero endomorphism there.
Since the block on the wall is semi-simple, it follows that 
$\mathscr{J}$ contains the identity morphism of some 
non-zero singular object. Since $\mathcal{M}_2^a$ is transitive,
it follows that $\mathscr{J}=\mathcal{M}_2^a$, a contradiction.
This proves that $\mathcal{M}_2^a$ is simple transitive.
This proves Claim~\eqref{prop-s5.6-2.1}.

Denote that $\mathcal{M}$ the additive closure of all $P(\lambda)$,
where $\lambda\in (0,a)+\Lambda$. For a regular $\lambda\in X_2$,
the functor $\top_s$ induces an equivalence of $\mathscr{C}$-module
categories between $\mathcal{M}$ and $\mathrm{add}(\mathscr{C}\cdot L(\lambda))$.
The category $\mathcal{M}_2^a$ is a $\mathscr{C}$-module subcategory of
$\mathcal{M}$. Factoring out of $\mathcal{M}$ the ideal generated by $\mathcal{M}_2^a$ 
outputs the additive closure of all $P(\lambda)$, where $\lambda\in X_1$.
As all these modules have trivial endomorphism algebra, it follows that
this quotient is a semi-simple category. By mapping each $P(\lambda)$, where $\lambda\in X_1$,
to its simple top, defines an equivalence of  $\mathscr{C}$-module
categories between this quotient and $\mathcal{M}_1^a$.
This implies Claims~\eqref{prop-s5.6-2.2} and \eqref{prop-s5.6-2.3}
and completes the proof.
\end{proof}

For $a\in\mathbb{C}\setminus\mathbb{Z}$, consider
the set $(a,0)+\Lambda$. We split this set into a
disjoint union $(a,0)+\Lambda= X_3\cup X_4$, where
$X_3$ consists of all $\lambda\in (a,0)+\Lambda$ such that
$\lambda_2\geq 0$ and $X_4$ is the complement of 
$X_3$ inside $(a,0)+\Lambda$. Denote by $\mathcal{M}_3^a$
the additive closure of all $L(\lambda)$, where
$\lambda\in X_3$.

\begin{proposition}\label{prop-s5.6-3}
For any $\lambda\in X_3$, the $\mathscr{C}$-module category
$\mathrm{add}(\mathscr{C}\cdot L(\lambda))$ coincides
with $\mathcal{M}_3^a$, moreover, the latter is a simple
transitive $\mathscr{C}$-module category whose graph
and positive eigenvector are depicted in Figure~\ref{fig10}.
\end{proposition}

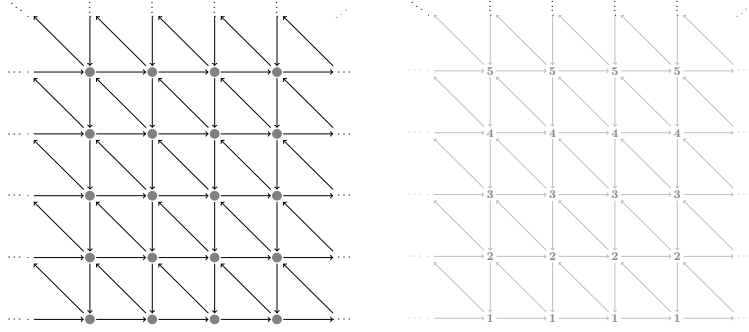
\begin{figure}
\resizebox{10cm}{!}{
\begin{tikzpicture}
\draw[black, thick,  ->] (2,1.8) -- (2,0.2);
\draw[black, thick,  ->] (4,1.8) -- (4,0.2);
\draw[black, thick,  ->] (6,1.8) -- (6,0.2);
\draw[black, thick,  ->] (8,1.8) -- (8,0.2);
\draw[black, thick,  ->] (1.8,0.2) -- (0.2,1.8);
\draw[black, thick,  ->] (3.8,0.2) -- (2.2,1.8);
\draw[black, thick,  ->] (5.8,0.2) -- (4.2,1.8);
\draw[black, thick,  ->] (7.8,0.2) -- (6.2,1.8);
\draw[black, thick,  ->] (0.2,2) -- (1.8,2);
\draw[black, thick,  ->] (2.2,0) -- (3.8,0);
\draw[black, thick,  ->] (1.8,2.2) -- (0.2,3.8);
\draw[black, thick,  ->] (2.2,2) -- (3.8,2);
\draw[black, thick,  ->] (2,3.8) -- (2,2.2);
\draw[black, thick,  ->] (3.8,2.2) -- (2.2,3.8);
\draw[black, thick,  ->] (4.2,2) -- (5.8,2);
\draw[black, thick,  ->] (4,3.8) -- (4,2.2);
\draw[black, thick,  ->] (5.8,2.2) -- (4.2,3.8);
\draw[black, thick,  ->] (6.2,2) -- (7.8,2);
\draw[black, thick,  ->] (6,3.8) -- (6,2.2);
\draw[black, thick,  ->] (7.8,2.2) -- (6.2,3.8);
\draw[black, thick,  ->] (8,3.8) -- (8,2.2);
\draw[black, thick,  ->] (0.2,4) -- (1.8,4);
\draw[black, thick,  ->] (4.2,0) -- (5.8,0);
\draw[black, thick,  ->] (1.8,4.2) -- (0.2,5.8);
\draw[black, thick,  ->] (2.2,4) -- (3.8,4);
\draw[black, thick,  ->] (2,5.8) -- (2,4.2);
\draw[black, thick,  ->] (3.8,4.2) -- (2.2,5.8);
\draw[black, thick,  ->] (4.2,4) -- (5.8,4);
\draw[black, thick,  ->] (4,5.8) -- (4,4.2);
\draw[black, thick,  ->] (5.8,4.2) -- (4.2,5.8);
\draw[black, thick,  ->] (6.2,4) -- (7.8,4);
\draw[black, thick,  ->] (6,5.8) -- (6,4.2);
\draw[black, thick,  ->] (7.8,4.2) -- (6.2,5.8);
\draw[black, thick,  ->] (8,5.8) -- (8,4.2);
\draw[black, thick,  ->] (0.2,6) -- (1.8,6);
\draw[black, thick,  ->] (0.2,0) -- (1.8,0);
\draw[black, thick,  ->] (6.2,0) -- (7.8,0);
\draw[black, thick,  ->] (1.8,6.2) -- (0.2,7.8);
\draw[black, thick,  ->] (2.2,6) -- (3.8,6);
\draw[black, thick,  ->] (2,7.8) -- (2,6.2);
\draw[black, thick,  ->] (3.8,6.2) -- (2.2,7.8);
\draw[black, thick,  ->] (4.2,6) -- (5.8,6);
\draw[black, thick,  ->] (4,7.8) -- (4,6.2);
\draw[black, thick,  ->] (5.8,6.2) -- (4.2,7.8);
\draw[black, thick,  ->] (6.2,6) -- (7.8,6);
\draw[black, thick,  ->] (6,7.8) -- (6,6.2);
\draw[black, thick,  ->] (7.8,6.2) -- (6.2,7.8);
\draw[black, thick,  ->] (8,7.8) -- (8,6.2);
\draw[black, thick,  -] (9.8,0) -- (9.8,0) node[anchor= west] {\large$\cdots$};
\draw[black, thick,  ->] (8.2,2) -- (9.8,2) node[anchor= west] {\large$\cdots$};
\draw[black, thick,  ->] (8.2,4) -- (9.8,4) node[anchor= west] {\large$\cdots$};
\draw[black, thick,  ->] (8.2,6) -- (9.8,6) node[anchor= west] {\large$\cdots$};
\draw[black, thick,  ->] (8.2,8) -- (9.8,8) node[anchor= west] {\large$\cdots$};
\draw[black, thick,  -] (0,0) -- (0,0) node[anchor= east] {\large$\cdots$};
\draw[black, thick,  -] (0,2) -- (0,2) node[anchor= east] {\large$\cdots$};
\draw[black, thick,  -] (0,4) -- (0,4) node[anchor= east] {\large$\cdots$};
\draw[black, thick,  -] (0,6) -- (0,6) node[anchor= east] {\large$\cdots$};
\draw[black, thick,  -] (0,8) -- (0,8) node[anchor= east] {\large$\cdots$};
\draw[black, thick,  ->] (9.8,0.2) -- (8.2,1.8);
\draw[black, thick,  ->] (9.8,2.2) -- (8.2,3.8);
\draw[black, thick,  ->] (9.8,4.2) -- (8.2,5.8);
\draw[black, thick,  ->] (9.8,6.2) -- (8.2,7.8);
\draw[black, thick,  ->]  (8.2,0) -- (9.8,0);
\draw[black, thick,  ->]  (2,9.8) -- (2,8.2);
\draw[black, thick,  ->]  (4,9.8) -- (4,8.2);
\draw[black, thick,  ->]  (6,9.8) -- (6,8.2);
\draw[black, thick,  ->]  (8,9.8) -- (8,8.2);
\draw[black, thin,  -] (0,9.8) -- (0,9.8) node[anchor= south east] {\large$\ddots$};
\draw[black, thin,  -] (2,9.8) -- (2,9.8) node[anchor= south] {\large$\vdots$};
\draw[black, thin,  -] (4,9.8) -- (4,9.8) node[anchor= south] {\large$\vdots$};
\draw[black, thin,  -] (6,9.8) -- (6,9.8) node[anchor= south] {\large$\vdots$};
\draw[black, thin,  -] (8,9.8) -- (8,9.8) node[anchor= south] {\large$\vdots$};
\draw[gray, thin,  -] (10.4,10.4) -- (10.4,10.4) node[anchor= north east] {\large$\iddots$};
\draw[black, thick,  ->] (1.8,8.2) -- (0.2,9.8);
\draw[black, thick,  ->] (3.8,8.2) -- (2.2,9.8);
\draw[black, thick,  ->] (5.8,8.2) -- (4.2,9.8);
\draw[black, thick,  ->] (7.8,8.2) -- (6.2,9.8);
\draw[black, thick,  ->] (9.8,8.2) -- (8.2,9.8);
\draw[black, thick,  ->] (0.2,8) -- (1.8,8);
\draw[black, thick,  ->] (2.2,8) -- (3.8,8);
\draw[black, thick,  ->] (4.2,8) -- (5.8,8);
\draw[black, thick,  ->] (6.2,8) -- (7.8,8);
\draw[gray,fill=gray] (2,0) circle (.9ex);
\draw[gray,fill=gray] (2,2) circle (.9ex);
\draw[gray,fill=gray] (4,2) circle (.9ex);
\draw[gray,fill=gray] (6,2) circle (.9ex);
\draw[gray,fill=gray] (8,2) circle (.9ex);
\draw[gray,fill=gray] (4,0) circle (.9ex);
\draw[gray,fill=gray] (2,4) circle (.9ex);
\draw[gray,fill=gray] (4,4) circle (.9ex);
\draw[gray,fill=gray] (6,4) circle (.9ex);
\draw[gray,fill=gray] (8,4) circle (.9ex);
\draw[gray,fill=gray] (6,0) circle (.9ex);
\draw[gray,fill=gray] (2,6) circle (.9ex);
\draw[gray,fill=gray] (4,6) circle (.9ex);
\draw[gray,fill=gray] (6,6) circle (.9ex);
\draw[gray,fill=gray] (8,6) circle (.9ex);
\draw[gray,fill=gray] (8,0) circle (.9ex);
\draw[gray,fill=gray] (2,8) circle (.9ex);
\draw[gray,fill=gray] (4,8) circle (.9ex);
\draw[gray,fill=gray] (6,8) circle (.9ex);
\draw[gray,fill=gray] (8,8) circle (.9ex);
\end{tikzpicture}
\qquad\qquad
\begin{tikzpicture}
\draw[lightgray, thin,  ->] (2,1.8) -- (2,0.2);
\draw[lightgray, thin,  ->] (4,1.8) -- (4,0.2);
\draw[lightgray, thin,  ->] (6,1.8) -- (6,0.2);
\draw[lightgray, thin,  ->] (8,1.8) -- (8,0.2);
\draw[lightgray, thin,  ->] (1.8,0.2) -- (0.2,1.8);
\draw[lightgray, thin,  ->] (3.8,0.2) -- (2.2,1.8);
\draw[lightgray, thin,  ->] (5.8,0.2) -- (4.2,1.8);
\draw[lightgray, thin,  ->] (7.8,0.2) -- (6.2,1.8);
\draw[lightgray, thin,  ->] (0.2,2) -- (1.8,2);
\draw[lightgray, thin,  ->] (2.2,0) -- (3.8,0);
\draw[lightgray, thin,  ->] (1.8,2.2) -- (0.2,3.8);
\draw[lightgray, thin,  ->] (2.2,2) -- (3.8,2);
\draw[lightgray, thin,  ->] (2,3.8) -- (2,2.2);
\draw[lightgray, thin,  ->] (3.8,2.2) -- (2.2,3.8);
\draw[lightgray, thin,  ->] (4.2,2) -- (5.8,2);
\draw[lightgray, thin,  ->] (4,3.8) -- (4,2.2);
\draw[lightgray, thin,  ->] (5.8,2.2) -- (4.2,3.8);
\draw[lightgray, thin,  ->] (6.2,2) -- (7.8,2);
\draw[lightgray, thin,  ->] (6,3.8) -- (6,2.2);
\draw[lightgray, thin,  ->] (7.8,2.2) -- (6.2,3.8);
\draw[lightgray, thin,  ->] (8,3.8) -- (8,2.2);
\draw[lightgray, thin,  ->] (0.2,4) -- (1.8,4);
\draw[lightgray, thin,  ->] (4.2,0) -- (5.8,0);
\draw[lightgray, thin,  ->] (1.8,4.2) -- (0.2,5.8);
\draw[lightgray, thin,  ->] (2.2,4) -- (3.8,4);
\draw[lightgray, thin,  ->] (2,5.8) -- (2,4.2);
\draw[lightgray, thin,  ->] (3.8,4.2) -- (2.2,5.8);
\draw[lightgray, thin,  ->] (4.2,4) -- (5.8,4);
\draw[lightgray, thin,  ->] (4,5.8) -- (4,4.2);
\draw[lightgray, thin,  ->] (5.8,4.2) -- (4.2,5.8);
\draw[lightgray, thin,  ->] (6.2,4) -- (7.8,4);
\draw[lightgray, thin,  ->] (6,5.8) -- (6,4.2);
\draw[lightgray, thin,  ->] (7.8,4.2) -- (6.2,5.8);
\draw[lightgray, thin,  ->] (8,5.8) -- (8,4.2);
\draw[lightgray, thin,  ->] (0.2,6) -- (1.8,6);
\draw[lightgray, thin,  ->] (0.2,0) -- (1.8,0);
\draw[lightgray, thin,  ->] (6.2,0) -- (7.8,0);
\draw[lightgray, thin,  ->] (1.8,6.2) -- (0.2,7.8);
\draw[lightgray, thin,  ->] (2.2,6) -- (3.8,6);
\draw[lightgray, thin,  ->] (2,7.8) -- (2,6.2);
\draw[lightgray, thin,  ->] (3.8,6.2) -- (2.2,7.8);
\draw[lightgray, thin,  ->] (4.2,6) -- (5.8,6);
\draw[lightgray, thin,  ->] (4,7.8) -- (4,6.2);
\draw[lightgray, thin,  ->] (5.8,6.2) -- (4.2,7.8);
\draw[lightgray, thin,  ->] (6.2,6) -- (7.8,6);
\draw[lightgray, thin,  ->] (6,7.8) -- (6,6.2);
\draw[lightgray, thin,  ->] (7.8,6.2) -- (6.2,7.8);
\draw[lightgray, thin,  ->] (8,7.8) -- (8,6.2);
\draw[lightgray, thin,  -] (9.8,0) -- (9.8,0) node[anchor= west] {\large$\cdots$};
\draw[lightgray, thin,  ->] (8.2,2) -- (9.8,2) node[anchor= west] {\large$\cdots$};
\draw[lightgray, thin,  ->] (8.2,4) -- (9.8,4) node[anchor= west] {\large$\cdots$};
\draw[lightgray, thin,  ->] (8.2,6) -- (9.8,6) node[anchor= west] {\large$\cdots$};
\draw[lightgray, thin,  ->] (8.2,8) -- (9.8,8) node[anchor= west] {\large$\cdots$};
\draw[lightgray, thin,  -] (0,0) -- (0,0) node[anchor= east] {\large$\cdots$};
\draw[lightgray, thin,  -] (0,2) -- (0,2) node[anchor= east] {\large$\cdots$};
\draw[lightgray, thin,  -] (0,4) -- (0,4) node[anchor= east] {\large$\cdots$};
\draw[lightgray, thin,  -] (0,6) -- (0,6) node[anchor= east] {\large$\cdots$};
\draw[lightgray, thin,  -] (0,8) -- (0,8) node[anchor= east] {\large$\cdots$};
\draw[lightgray, thin,  ->] (9.8,0.2) -- (8.2,1.8);
\draw[lightgray, thin,  ->] (9.8,2.2) -- (8.2,3.8);
\draw[lightgray, thin,  ->] (9.8,4.2) -- (8.2,5.8);
\draw[lightgray, thin,  ->] (9.8,6.2) -- (8.2,7.8);
\draw[lightgray, thin,  ->]  (8.2,0) -- (9.8,0);
\draw[lightgray, thin,  ->]  (2,9.8) -- (2,8.2);
\draw[lightgray, thin,  ->]  (4,9.8) -- (4,8.2);
\draw[lightgray, thin,  ->]  (6,9.8) -- (6,8.2);
\draw[lightgray, thin,  ->]  (8,9.8) -- (8,8.2);
\draw[black, thin,  -] (0,9.8) -- (0,9.8) node[anchor= south east] {\large$\ddots$};
\draw[black, thin,  -] (2,9.8) -- (2,9.8) node[anchor= south] {\large$\vdots$};
\draw[black, thin,  -] (4,9.8) -- (4,9.8) node[anchor= south] {\large$\vdots$};
\draw[black, thin,  -] (6,9.8) -- (6,9.8) node[anchor= south] {\large$\vdots$};
\draw[black, thin,  -] (8,9.8) -- (8,9.8) node[anchor= south] {\large$\vdots$};
\draw[gray, thin,  -] (10.4,10.4) -- (10.4,10.4) node[anchor= north east] {\large$\iddots$};
\draw[lightgray, thin,  ->] (1.8,8.2) -- (0.2,9.8);
\draw[lightgray, thin,  ->] (3.8,8.2) -- (2.2,9.8);
\draw[lightgray, thin,  ->] (5.8,8.2) -- (4.2,9.8);
\draw[lightgray, thin,  ->] (7.8,8.2) -- (6.2,9.8);
\draw[lightgray, thin,  ->] (9.8,8.2) -- (8.2,9.8);
\draw[lightgray, thin,  ->] (0.2,8) -- (1.8,8);
\draw[lightgray, thin,  ->] (2.2,8) -- (3.8,8);
\draw[lightgray, thin,  ->] (4.2,8) -- (5.8,8);
\draw[lightgray, thin,  ->] (6.2,8) -- (7.8,8);
\draw[gray,fill=gray] (2,0) circle (.01ex) node {\large$\mathbf{1}$};
\draw[gray,fill=gray] (2,2) circle (.01ex) node {\large$\mathbf{2}$};
\draw[gray,fill=gray] (4,2) circle (.01ex) node {\large$\mathbf{2}$};
\draw[gray,fill=gray] (6,2) circle (.01ex) node {\large$\mathbf{2}$};
\draw[gray,fill=gray] (8,2) circle (.01ex) node {\large$\mathbf{2}$};
\draw[gray,fill=gray] (4,0) circle (.01ex) node {\large$\mathbf{1}$};
\draw[gray,fill=gray] (2,4) circle (.01ex) node {\large$\mathbf{3}$};
\draw[gray,fill=gray] (4,4) circle (.01ex) node {\large$\mathbf{3}$};
\draw[gray,fill=gray] (6,4) circle (.01ex) node {\large$\mathbf{3}$};
\draw[gray,fill=gray] (8,4) circle (.01ex) node {\large$\mathbf{3}$};
\draw[gray,fill=gray] (6,0) circle (.01ex) node {\large$\mathbf{1}$};
\draw[gray,fill=gray] (2,6) circle (.01ex) node {\large$\mathbf{4}$};
\draw[gray,fill=gray] (4,6) circle (.01ex) node {\large$\mathbf{4}$};
\draw[gray,fill=gray] (6,6) circle (.01ex) node {\large$\mathbf{4}$};
\draw[gray,fill=gray] (8,6) circle (.01ex) node {\large$\mathbf{4}$};
\draw[gray,fill=gray] (8,0) circle (.01ex) node {\large$\mathbf{1}$};
\draw[gray,fill=gray] (2,8) circle (.01ex) node {\large$\mathbf{5}$};
\draw[gray,fill=gray] (4,8) circle (.01ex) node {\large$\mathbf{5}$};
\draw[gray,fill=gray] (6,8) circle (.01ex) node {\large$\mathbf{5}$};
\draw[gray,fill=gray] (8,8) circle (.01ex) node {\large$\mathbf{5}$};
\end{tikzpicture}
}
\caption{The third case of partially integral weights}\label{fig10}
\end{figure}

\begin{proof}
Mutatis mutandis the proof of Proposition~\ref{prop-s5.6-1}.
\end{proof}

\begin{figure}
\resizebox{10cm}{!}{
\begin{tikzpicture}
\draw[black, thick,  ->] (8,1.8) -- (8,0.2);
\draw[black, thick,  ->] (6,1.8) -- (6,0.2);
\draw[black, thick,  ->] (4,1.8) -- (4,0.2);
\draw[black, thick,  ->] (2,1.8) -- (2,0.2);
\draw[black, thick,  ->] (8,3.8) -- (8,2.2);
\draw[black, thick,  ->] (6,3.8) -- (6,2.2);
\draw[black, thick,  ->] (4,3.8) -- (4,2.2);
\draw[black, thick,  ->] (2,3.8) -- (2,2.2);
\draw[black, thick,  ->] (8,5.8) -- (8,4.2);
\draw[black, thick,  ->] (6,5.8) -- (6,4.2);
\draw[black, thick,  ->] (4,5.8) -- (4,4.2);
\draw[black, thick,  ->] (2,5.8) -- (2,4.2);
\draw[black, thick,  ->] (8,7.8) -- (8,6.2);
\draw[black, thick,  ->] (6,7.8) -- (6,6.2);
\draw[black, thick,  ->] (4,7.8) -- (4,6.2);
\draw[black, thick,  ->] (2,7.8) -- (2,6.2);
\draw[black, thick,  ->] (8,9.8) -- (8,8.2);
\draw[black, thick,  ->] (6,9.8) -- (6,8.2);
\draw[black, thick,  ->] (4,9.8) -- (4,8.2);
\draw[black, thick,  ->] (2,9.8) -- (2,8.2);
\draw[black, thick,  ->] (0.2,2) -- (1.8,2);
\draw[black, thick,  ->] (2.2,2) -- (3.8,2);
\draw[black, thick,  ->] (4.2,2) -- (5.8,2);
\draw[black, thick,  ->] (6.2,2) -- (7.8,2);
\draw[black, thick,  ->] (8.2,2) -- (9.8,2);
\draw[black, thick,  ->] (0.2,4) -- (1.8,4);
\draw[black, thick,  ->] (2.2,4) -- (3.8,4);
\draw[black, thick,  ->] (4.2,4) -- (5.8,4);
\draw[black, thick,  ->] (6.2,4) -- (7.8,4);
\draw[black, thick,  ->] (8.2,4) -- (9.8,4);
\draw[black, thick,  ->] (0.2,6) -- (1.8,6);
\draw[black, thick,  ->] (2.2,6) -- (3.8,6);
\draw[black, thick,  ->] (4.2,6) -- (5.8,6);
\draw[black, thick,  ->] (6.2,6) -- (7.8,6);
\draw[black, thick,  ->] (8.2,6) -- (9.8,6);
\draw[black, thick,  ->] (0.2,8) -- (1.8,8);
\draw[black, thick,  ->] (2.2,8) -- (3.8,8);
\draw[black, thick,  ->] (4.2,8) -- (5.8,8);
\draw[black, thick,  ->] (6.2,8) -- (7.8,8);
\draw[black, thick,  ->] (8.2,8) -- (9.8,8);
\draw[black, thick,  ->] (0.2,10) -- (1.8,10);
\draw[black, thick,  ->] (2.2,10) -- (3.8,10);
\draw[black, thick,  ->] (4.2,10) -- (5.8,10);
\draw[black, thick,  ->] (6.2,10) -- (7.8,10);
\draw[black, thick,  ->] (8.2,10) -- (9.8,10);
\draw[black, thick,  ->] (1.8,0.2) -- (0.2,1.8);
\draw[black, thick,  ->] (3.8,0.2) -- (2.2,1.8);
\draw[black, thick,  ->] (5.8,0.2) -- (4.2,1.8);
\draw[black, thick,  ->] (7.8,0.2) -- (6.2,1.8);
\draw[black, thick,  ->] (9.8,0.2) -- (8.2,1.8);
\draw[black, thick,  ->] (1.8,2.2) -- (0.2,3.8);
\draw[black, thick,  ->] (3.8,2.2) -- (2.2,3.8);
\draw[black, thick,  ->] (5.8,2.2) -- (4.2,3.8);
\draw[black, thick,  ->] (7.8,2.2) -- (6.2,3.8);
\draw[black, thick,  ->] (9.8,2.2) -- (8.2,3.8);
\draw[black, thick,  ->] (1.8,4.2) -- (0.2,5.8);
\draw[black, thick,  ->] (3.8,4.2) -- (2.2,5.8);
\draw[black, thick,  ->] (5.8,4.2) -- (4.2,5.8);
\draw[black, thick,  ->] (7.8,4.2) -- (6.2,5.8);
\draw[black, thick,  ->] (9.8,4.2) -- (8.2,5.8);
\draw[black, thick,  ->] (1.8,6.2) -- (0.2,7.8);
\draw[black, thick,  ->] (3.8,6.2) -- (2.2,7.8);
\draw[black, thick,  ->] (5.8,6.2) -- (4.2,7.8);
\draw[black, thick,  ->] (7.8,6.2) -- (6.2,7.8);
\draw[black, thick,  ->] (9.8,6.2) -- (8.2,7.8);
\draw[black, thick,  ->] (1.7,8.2) to [out=150,in=300]   (0.2,9.7);
\draw[black, thick,  ->] (1.8,8.3) to [out=120,in=330]   (0.3,9.8);
\draw[black, thick,  ->] (3.7,8.2) to [out=150,in=300]   (2.2,9.7);
\draw[black, thick,  ->] (3.8,8.3) to [out=120,in=330]   (2.3,9.8);
\draw[black, thick,  ->] (5.7,8.2) to [out=150,in=300]   (4.2,9.7);
\draw[black, thick,  ->] (5.8,8.3) to [out=120,in=330]   (4.3,9.8);
\draw[black, thick,  ->] (7.7,8.2) to [out=150,in=300]   (6.2,9.7);
\draw[black, thick,  ->] (7.8,8.3) to [out=120,in=330]   (6.3,9.8);
\draw[black, thick,  ->] (9.7,8.2) to [out=150,in=300]   (8.2,9.7);
\draw[black, thick,  ->] (9.8,8.3) to [out=120,in=330]   (8.3,9.8);
\draw[black, thin,  -] (10,2) -- (10,2) node[anchor= west] {\large$\dots$};
\draw[black, thin,  -] (10,4) -- (10,4) node[anchor= west] {\large$\dots$};
\draw[black, thin,  -] (10,6) -- (10,6) node[anchor= west] {\large$\dots$};
\draw[black, thin,  -] (10,8) -- (10,8) node[anchor= west] {\large$\dots$};
\draw[black, thin,  -] (10,10) -- (10,10) node[anchor= west] {\large$\dots$};
\draw[black, thin,  -] (0,2) -- (0,2) node[anchor= east] {\large$\dots$};
\draw[black, thin,  -] (0,4) -- (0,4) node[anchor= east] {\large$\dots$};
\draw[black, thin,  -] (0,6) -- (0,6) node[anchor= east] {\large$\dots$};
\draw[black, thin,  -] (0,8) -- (0,8) node[anchor= east] {\large$\dots$};
\draw[black, thin,  -] (0,10) -- (0,10) node[anchor= east] {\large$\dots$};
\draw[black, thin,  -] (2,0) -- (2,0) node[anchor= north] {\large$\vdots$};
\draw[black, thin,  -] (4,0) -- (4,0) node[anchor= north] {\large$\vdots$};
\draw[black, thin,  -] (6,0) -- (6,0) node[anchor= north] {\large$\vdots$};
\draw[black, thin,  -] (8,0) -- (8,0) node[anchor= north] {\large$\vdots$};
\draw[black, thin,  -] (10,0) -- (10,0) node[anchor= north west] {\large$\ddots$};
\draw[black, thin,  -] (0,0) -- (0,0) node[anchor= north east] {\large$\iddots$};
\draw[gray,fill=gray] (2,10) circle (.9ex);
\draw[gray,fill=gray] (2,2) circle (.9ex);
\draw[gray,fill=gray] (4,2) circle (.9ex);
\draw[gray,fill=gray] (6,2) circle (.9ex);
\draw[gray,fill=gray] (8,2) circle (.9ex);
\draw[gray,fill=gray] (4,10) circle (.9ex);
\draw[gray,fill=gray] (2,4) circle (.9ex);
\draw[gray,fill=gray] (4,4) circle (.9ex);
\draw[gray,fill=gray] (6,4) circle (.9ex);
\draw[gray,fill=gray] (8,4) circle (.9ex);
\draw[gray,fill=gray] (6,10) circle (.9ex);
\draw[gray,fill=gray] (2,6) circle (.9ex);
\draw[gray,fill=gray] (4,6) circle (.9ex);
\draw[gray,fill=gray] (6,6) circle (.9ex);
\draw[gray,fill=gray] (8,6) circle (.9ex);
\draw[gray,fill=gray] (8,10) circle (.9ex);
\draw[gray,fill=gray] (2,8) circle (.9ex);
\draw[gray,fill=gray] (4,8) circle (.9ex);
\draw[gray,fill=gray] (6,8) circle (.9ex);
\draw[gray,fill=gray] (8,8) circle (.9ex);
\end{tikzpicture}
\qquad\qquad
\begin{tikzpicture}
\draw[lightgray, thin,  ->] (8,1.8) -- (8,0.2);
\draw[lightgray, thin,  ->] (6,1.8) -- (6,0.2);
\draw[lightgray, thin,  ->] (4,1.8) -- (4,0.2);
\draw[lightgray, thin,  ->] (2,1.8) -- (2,0.2);
\draw[lightgray, thin,  ->] (8,3.8) -- (8,2.2);
\draw[lightgray, thin,  ->] (6,3.8) -- (6,2.2);
\draw[lightgray, thin,  ->] (4,3.8) -- (4,2.2);
\draw[lightgray, thin,  ->] (2,3.8) -- (2,2.2);
\draw[lightgray, thin,  ->] (8,5.8) -- (8,4.2);
\draw[lightgray, thin,  ->] (6,5.8) -- (6,4.2);
\draw[lightgray, thin,  ->] (4,5.8) -- (4,4.2);
\draw[lightgray, thin,  ->] (2,5.8) -- (2,4.2);
\draw[lightgray, thin,  ->] (8,7.8) -- (8,6.2);
\draw[lightgray, thin,  ->] (6,7.8) -- (6,6.2);
\draw[lightgray, thin,  ->] (4,7.8) -- (4,6.2);
\draw[lightgray, thin,  ->] (2,7.8) -- (2,6.2);
\draw[lightgray, thin,  ->] (8,9.8) -- (8,8.2);
\draw[lightgray, thin,  ->] (6,9.8) -- (6,8.2);
\draw[lightgray, thin,  ->] (4,9.8) -- (4,8.2);
\draw[lightgray, thin,  ->] (2,9.8) -- (2,8.2);
\draw[lightgray, thin,  ->] (0.2,2) -- (1.8,2);
\draw[lightgray, thin,  ->] (2.2,2) -- (3.8,2);
\draw[lightgray, thin,  ->] (4.2,2) -- (5.8,2);
\draw[lightgray, thin,  ->] (6.2,2) -- (7.8,2);
\draw[lightgray, thin,  ->] (8.2,2) -- (9.8,2);
\draw[lightgray, thin,  ->] (0.2,4) -- (1.8,4);
\draw[lightgray, thin,  ->] (2.2,4) -- (3.8,4);
\draw[lightgray, thin,  ->] (4.2,4) -- (5.8,4);
\draw[lightgray, thin,  ->] (6.2,4) -- (7.8,4);
\draw[lightgray, thin,  ->] (8.2,4) -- (9.8,4);
\draw[lightgray, thin,  ->] (0.2,6) -- (1.8,6);
\draw[lightgray, thin,  ->] (2.2,6) -- (3.8,6);
\draw[lightgray, thin,  ->] (4.2,6) -- (5.8,6);
\draw[lightgray, thin,  ->] (6.2,6) -- (7.8,6);
\draw[lightgray, thin,  ->] (8.2,6) -- (9.8,6);
\draw[lightgray, thin,  ->] (0.2,8) -- (1.8,8);
\draw[lightgray, thin,  ->] (2.2,8) -- (3.8,8);
\draw[lightgray, thin,  ->] (4.2,8) -- (5.8,8);
\draw[lightgray, thin,  ->] (6.2,8) -- (7.8,8);
\draw[lightgray, thin,  ->] (8.2,8) -- (9.8,8);
\draw[lightgray, thin,  ->] (0.2,10) -- (1.8,10);
\draw[lightgray, thin,  ->] (2.2,10) -- (3.8,10);
\draw[lightgray, thin,  ->] (4.2,10) -- (5.8,10);
\draw[lightgray, thin,  ->] (6.2,10) -- (7.8,10);
\draw[lightgray, thin,  ->] (8.2,10) -- (9.8,10);
\draw[lightgray, thin,  ->] (1.8,0.2) -- (0.2,1.8);
\draw[lightgray, thin,  ->] (3.8,0.2) -- (2.2,1.8);
\draw[lightgray, thin,  ->] (5.8,0.2) -- (4.2,1.8);
\draw[lightgray, thin,  ->] (7.8,0.2) -- (6.2,1.8);
\draw[lightgray, thin,  ->] (9.8,0.2) -- (8.2,1.8);
\draw[lightgray, thin,  ->] (1.8,2.2) -- (0.2,3.8);
\draw[lightgray, thin,  ->] (3.8,2.2) -- (2.2,3.8);
\draw[lightgray, thin,  ->] (5.8,2.2) -- (4.2,3.8);
\draw[lightgray, thin,  ->] (7.8,2.2) -- (6.2,3.8);
\draw[lightgray, thin,  ->] (9.8,2.2) -- (8.2,3.8);
\draw[lightgray, thin,  ->] (1.8,4.2) -- (0.2,5.8);
\draw[lightgray, thin,  ->] (3.8,4.2) -- (2.2,5.8);
\draw[lightgray, thin,  ->] (5.8,4.2) -- (4.2,5.8);
\draw[lightgray, thin,  ->] (7.8,4.2) -- (6.2,5.8);
\draw[lightgray, thin,  ->] (9.8,4.2) -- (8.2,5.8);
\draw[lightgray, thin,  ->] (1.8,6.2) -- (0.2,7.8);
\draw[lightgray, thin,  ->] (3.8,6.2) -- (2.2,7.8);
\draw[lightgray, thin,  ->] (5.8,6.2) -- (4.2,7.8);
\draw[lightgray, thin,  ->] (7.8,6.2) -- (6.2,7.8);
\draw[lightgray, thin,  ->] (9.8,6.2) -- (8.2,7.8);
\draw[lightgray, thin,  ->] (1.7,8.2) to [out=150,in=300]   (0.2,9.7);
\draw[lightgray, thin,  ->] (1.8,8.3) to [out=120,in=330]   (0.3,9.8);
\draw[lightgray, thin,  ->] (3.7,8.2) to [out=150,in=300]   (2.2,9.7);
\draw[lightgray, thin,  ->] (3.8,8.3) to [out=120,in=330]   (2.3,9.8);
\draw[lightgray, thin,  ->] (5.7,8.2) to [out=150,in=300]   (4.2,9.7);
\draw[lightgray, thin,  ->] (5.8,8.3) to [out=120,in=330]   (4.3,9.8);
\draw[lightgray, thin,  ->] (7.7,8.2) to [out=150,in=300]   (6.2,9.7);
\draw[lightgray, thin,  ->] (7.8,8.3) to [out=120,in=330]   (6.3,9.8);
\draw[lightgray, thin,  ->] (9.7,8.2) to [out=150,in=300]   (8.2,9.7);
\draw[lightgray, thin,  ->] (9.8,8.3) to [out=120,in=330]   (8.3,9.8);
\draw[black, thin,  -] (10,2) -- (10,2) node[anchor= west] {\large$\dots$};
\draw[black, thin,  -] (10,4) -- (10,4) node[anchor= west] {\large$\dots$};
\draw[black, thin,  -] (10,6) -- (10,6) node[anchor= west] {\large$\dots$};
\draw[black, thin,  -] (10,8) -- (10,8) node[anchor= west] {\large$\dots$};
\draw[black, thin,  -] (10,10) -- (10,10) node[anchor= west] {\large$\dots$};
\draw[black, thin,  -] (0,2) -- (0,2) node[anchor= east] {\large$\dots$};
\draw[black, thin,  -] (0,4) -- (0,4) node[anchor= east] {\large$\dots$};
\draw[black, thin,  -] (0,6) -- (0,6) node[anchor= east] {\large$\dots$};
\draw[black, thin,  -] (0,8) -- (0,8) node[anchor= east] {\large$\dots$};
\draw[black, thin,  -] (0,10) -- (0,10) node[anchor= east] {\large$\dots$};
\draw[black, thin,  -] (2,0) -- (2,0) node[anchor= north] {\large$\vdots$};
\draw[black, thin,  -] (4,0) -- (4,0) node[anchor= north] {\large$\vdots$};
\draw[black, thin,  -] (6,0) -- (6,0) node[anchor= north] {\large$\vdots$};
\draw[black, thin,  -] (8,0) -- (8,0) node[anchor= north] {\large$\vdots$};
\draw[black, thin,  -] (10,0) -- (10,0) node[anchor= north west] {\large$\ddots$};
\draw[black, thin,  -] (0,0) -- (0,0) node[anchor= north east] {\large$\iddots$};
\draw[gray,fill=gray] (2,10) circle (.01ex) node {\large$\mathbf{1}$};
\draw[gray,fill=gray] (2,2) circle (.01ex) node {\large$\mathbf{2}$};
\draw[gray,fill=gray] (4,2) circle (.01ex) node {\large$\mathbf{2}$};
\draw[gray,fill=gray] (6,2) circle (.01ex) node {\large$\mathbf{2}$};
\draw[gray,fill=gray] (8,2) circle (.01ex) node {\large$\mathbf{2}$};
\draw[gray,fill=gray] (4,10) circle (.01ex) node {\large$\mathbf{1}$};
\draw[gray,fill=gray] (2,4) circle (.01ex) node {\large$\mathbf{2}$};
\draw[gray,fill=gray] (4,4) circle (.01ex) node {\large$\mathbf{2}$};
\draw[gray,fill=gray] (6,4) circle (.01ex) node {\large$\mathbf{2}$};
\draw[gray,fill=gray] (8,4) circle (.01ex) node {\large$\mathbf{2}$};
\draw[gray,fill=gray] (6,10) circle (.01ex) node {\large$\mathbf{1}$};
\draw[gray,fill=gray] (2,6) circle (.01ex) node {\large$\mathbf{2}$};
\draw[gray,fill=gray] (4,6) circle (.01ex) node {\large$\mathbf{2}$};
\draw[gray,fill=gray] (6,6) circle (.01ex) node {\large$\mathbf{2}$};
\draw[gray,fill=gray] (8,6) circle (.01ex) node {\large$\mathbf{2}$};
\draw[gray,fill=gray] (8,10) circle (.01ex) node {\large$\mathbf{1}$};
\draw[gray,fill=gray] (2,8) circle (.01ex) node {\large$\mathbf{2}$};
\draw[gray,fill=gray] (4,8) circle (.01ex) node {\large$\mathbf{2}$};
\draw[gray,fill=gray] (6,8) circle (.01ex) node {\large$\mathbf{2}$};
\draw[gray,fill=gray] (8,8) circle (.01ex) node {\large$\mathbf{2}$};

\end{tikzpicture}
}
\caption{The forth case of partially integral weights}\label{fig11}
\end{figure}

Denote by $\mathcal{M}_4^a$
the additive closure of all $P(\lambda)$, where
$\lambda\in X_4$. 

\begin{proposition}\label{prop-s5.6-4}
\begin{enumerate}[$($a$)$]
\item\label{prop-s5.6-4.1} 
For $\lambda\in X_4$ such that $\lambda_2=-1$, 
$\mathrm{add}(\mathscr{C}\cdot L(\lambda))$ coincides with
$\mathcal{M}_4^a$ and is a simple transitive $\mathscr{C}$-module 
category whose graph
and positive eigenvector are depicted in Figure~\ref{fig11}.
\item\label{prop-s5.6-4.2}
For $\lambda\in X_4$ such that $\lambda_2\neq -1$,
$\mathrm{add}(\mathscr{C}\cdot L(\lambda))$
contains $\mathcal{M}_4^a$ as a $\mathscr{C}$-module 
subcategory, moreover, the quotient of $\mathrm{add}(\mathscr{C}\cdot L(\lambda))$
by the ideal $\mathcal{I}$ generated by $\mathcal{M}_4^a$ is a simple 
transitive $\mathscr{C}$-module category that is 
equivalent to $\mathcal{M}_3^a$. 
\item\label{prop-s5.6-4.3}
If $\lambda_2\neq -1$, then the indecomposable
objects in  $\mathrm{add}(\mathscr{C}\cdot L(\lambda))/\mathcal{I}$  
are given by the images of 
$L(\mu)$, where $\mu\in X_4$ is such that $\mu_2\neq -1$.
\end{enumerate}
\end{proposition}

\begin{proof}
Mutatis mutandis the proof of Proposition~\ref{prop-s5.6-2}. 
\end{proof}

For $a\in\mathbb{C}\setminus\mathbb{Z}$, consider
the set $(a,-a)+\Lambda$. We split this set into a
disjoint union $(a,-a)+\Lambda= X_5\cup X_6$, where
$X_5$ consists of all $\lambda\in (a,-a)+\Lambda$ such that
$\lambda_1+\lambda_2\geq 0$ and $X_6$ is the complement of 
$X_5$ inside $(a,-a)+\Lambda$. Denote by $\mathcal{M}_5^a$
the additive closure of all $L(\lambda)$, where
$\lambda\in X_5$.

\begin{proposition}\label{prop-s5.6-5}
For any $\lambda\in X_5$, the $\mathscr{C}$-module category
$\mathrm{add}(\mathscr{C}\cdot L(\lambda))$ coincides
with $\mathcal{M}_5^a$, moreover, the latter is a simple
transitive $\mathscr{C}$-module category whose graph
and positive eigenvector are depicted in Figure~\ref{fig12}.
\end{proposition}

\begin{figure}
\resizebox{10cm}{!}{
\begin{tikzpicture}
\draw[black, thick,  -] (8,0) -- (8,0) node[anchor= north west] {\large$\ddots$};
\draw[black, thick,  -] (10,0) -- (10,0) node[anchor= north west] {\large$\ddots$};
\draw[black, thick,  -] (10,2) -- (10,2) node[anchor= west] {\large$\cdots$};
\draw[black, thick,  -] (10,4) -- (10,4) node[anchor= west] {\large$\cdots$};
\draw[black, thick,  -] (10,6) -- (10,6) node[anchor= west] {\large$\cdots$};
\draw[black, thick,  -] (10,8) -- (10,8) node[anchor= west] {\large$\cdots$};
\draw[black, thick,  -] (10,10) -- (10,10) node[anchor= south west] {\large$\iddots$};
\draw[black, thick,  -] (8,10) -- (8,10) node[anchor= south] {\large$\vdots$};
\draw[black, thick,  -] (6,10) -- (6,10) node[anchor= south] {\large$\vdots$};
\draw[black, thick,  -] (4,10) -- (4,10) node[anchor= south] {\large$\vdots$};
\draw[black, thick,  -] (2,10) -- (2,10) node[anchor= south] {\large$\vdots$};
\draw[black, thick,  -] (0,10) -- (0,10) node[anchor= south east] {\large$\ddots$};
\draw[black, thick,  -] (0,8) -- (0,8) node[anchor= south east] {\large$\ddots$};
\draw[black, thick,  ->] (8.2,0) -- (9.8,0);
\draw[black, thick,  ->] (6.2,2) -- (7.8,2);
\draw[black, thick,  ->] (8.2,2) -- (9.8,2);
\draw[black, thick,  ->] (6.2,4) -- (7.8,4);
\draw[black, thick,  ->] (8.2,4) -- (9.8,4);
\draw[black, thick,  ->] (4.2,4) -- (5.8,4);
\draw[black, thick,  ->] (6.2,6) -- (7.8,6);
\draw[black, thick,  ->] (8.2,6) -- (9.8,6);
\draw[black, thick,  ->] (4.2,6) -- (5.8,6);
\draw[black, thick,  ->] (2.2,6) -- (3.8,6);
\draw[black, thick,  ->] (6.2,8) -- (7.8,8);
\draw[black, thick,  ->] (8.2,8) -- (9.8,8);
\draw[black, thick,  ->] (4.2,8) -- (5.8,8);
\draw[black, thick,  ->] (2.2,8) -- (3.8,8);
\draw[black, thick,  ->] (0.2,8) -- (1.8,8);
\draw[black, thick,  ->] (8,1.8) -- (8,0.2);
\draw[black, thick,  ->] (8,3.8) -- (8,2.2);
\draw[black, thick,  ->] (6,3.8) -- (6,2.2);
\draw[black, thick,  ->] (8,5.8) -- (8,4.2);
\draw[black, thick,  ->] (6,5.8) -- (6,4.2);
\draw[black, thick,  ->] (4,5.8) -- (4,4.2);
\draw[black, thick,  ->] (8,7.8) -- (8,6.2);
\draw[black, thick,  ->] (6,7.8) -- (6,6.2);
\draw[black, thick,  ->] (4,7.8) -- (4,6.2);
\draw[black, thick,  ->] (2,7.8) -- (2,6.2);
\draw[black, thick,  ->] (8,9.8) -- (8,8.2);
\draw[black, thick,  ->] (6,9.8) -- (6,8.2);
\draw[black, thick,  ->] (4,9.8) -- (4,8.2);
\draw[black, thick,  ->] (2,9.8) -- (2,8.2);
\draw[black, thick,  ->] (0,9.8) -- (0,8.2);
\draw[black, thick,  ->] (1.8,6.2) -- (0.2,7.8);
\draw[black, thick,  ->] (3.8,6.2) -- (2.2,7.8);
\draw[black, thick,  ->] (5.8,6.2) -- (4.2,7.8);
\draw[black, thick,  ->] (7.8,6.2) -- (6.2,7.8);
\draw[black, thick,  ->] (9.8,6.2) -- (8.2,7.8);
\draw[black, thick,  ->] (1.8,8.2) -- (0.2,9.8);
\draw[black, thick,  ->] (3.8,8.2) -- (2.2,9.8);
\draw[black, thick,  ->] (5.8,8.2) -- (4.2,9.8);
\draw[black, thick,  ->] (7.8,8.2) -- (6.2,9.8);
\draw[black, thick,  ->] (9.8,8.2) -- (8.2,9.8);
\draw[black, thick,  ->] (3.8,4.2) -- (2.2,5.8);
\draw[black, thick,  ->] (5.8,4.2) -- (4.2,5.8);
\draw[black, thick,  ->] (7.8,4.2) -- (6.2,5.8);
\draw[black, thick,  ->] (9.8,4.2) -- (8.2,5.8);
\draw[black, thick,  ->] (5.8,2.2) -- (4.2,3.8);
\draw[black, thick,  ->] (7.8,2.2) -- (6.2,3.8);
\draw[black, thick,  ->] (9.8,2.2) -- (8.2,3.8);
\draw[black, thick,  ->] (7.8,0.2) -- (6.2,1.8);
\draw[black, thick,  ->] (9.8,0.2) -- (8.2,1.8);
\draw[gray,fill=gray] (6,2) circle (.9ex);
\draw[gray,fill=gray] (4,4) circle (.9ex);
\draw[gray,fill=gray] (2,6) circle (.9ex);
\draw[gray,fill=gray] (8,2) circle (.9ex);
\draw[gray,fill=gray] (6,4) circle (.9ex);
\draw[gray,fill=gray] (4,6) circle (.9ex);
\draw[gray,fill=gray] (2,8) circle (.9ex);
\draw[gray,fill=gray] (8,4) circle (.9ex);
\draw[gray,fill=gray] (6,6) circle (.9ex);
\draw[gray,fill=gray] (4,8) circle (.9ex);
\draw[gray,fill=gray] (8,6) circle (.9ex);
\draw[gray,fill=gray] (6,8) circle (.9ex);
\draw[gray,fill=gray] (8,8) circle (.9ex);
\end{tikzpicture}
\qquad\qquad
\begin{tikzpicture}
\draw[lightgray, thin,  -] (8,0) -- (8,0) node[anchor= north west] {\large$\ddots$};
\draw[lightgray, thin,  -] (10,0) -- (10,0) node[anchor= north west] {\large$\ddots$};
\draw[lightgray, thin,  -] (10,2) -- (10,2) node[anchor= west] {\large$\cdots$};
\draw[lightgray, thin,  -] (10,4) -- (10,4) node[anchor= west] {\large$\cdots$};
\draw[lightgray, thin,  -] (10,6) -- (10,6) node[anchor= west] {\large$\cdots$};
\draw[lightgray, thin,  -] (10,8) -- (10,8) node[anchor= west] {\large$\cdots$};
\draw[lightgray, thin,  -] (10,10) -- (10,10) node[anchor= south west] {\large$\iddots$};
\draw[lightgray, thin,  -] (8,10) -- (8,10) node[anchor= south] {\large$\vdots$};
\draw[lightgray, thin,  -] (6,10) -- (6,10) node[anchor= south] {\large$\vdots$};
\draw[lightgray, thin,  -] (4,10) -- (4,10) node[anchor= south] {\large$\vdots$};
\draw[lightgray, thin,  -] (2,10) -- (2,10) node[anchor= south] {\large$\vdots$};
\draw[lightgray, thin,  -] (0,10) -- (0,10) node[anchor= south east] {\large$\ddots$};
\draw[lightgray, thin,  -] (0,8) -- (0,8) node[anchor= south east] {\large$\ddots$};
\draw[lightgray, thin,  ->] (8.2,0) -- (9.8,0);
\draw[lightgray, thin,  ->] (6.2,2) -- (7.8,2);
\draw[lightgray, thin,  ->] (8.2,2) -- (9.8,2);
\draw[lightgray, thin,  ->] (6.2,4) -- (7.8,4);
\draw[lightgray, thin,  ->] (8.2,4) -- (9.8,4);
\draw[lightgray, thin,  ->] (4.2,4) -- (5.8,4);
\draw[lightgray, thin,  ->] (6.2,6) -- (7.8,6);
\draw[lightgray, thin,  ->] (8.2,6) -- (9.8,6);
\draw[lightgray, thin,  ->] (4.2,6) -- (5.8,6);
\draw[lightgray, thin,  ->] (2.2,6) -- (3.8,6);
\draw[lightgray, thin,  ->] (6.2,8) -- (7.8,8);
\draw[lightgray, thin,  ->] (8.2,8) -- (9.8,8);
\draw[lightgray, thin,  ->] (4.2,8) -- (5.8,8);
\draw[lightgray, thin,  ->] (2.2,8) -- (3.8,8);
\draw[lightgray, thin,  ->] (0.2,8) -- (1.8,8);
\draw[lightgray, thin,  ->] (8,1.8) -- (8,0.2);
\draw[lightgray, thin,  ->] (8,3.8) -- (8,2.2);
\draw[lightgray, thin,  ->] (6,3.8) -- (6,2.2);
\draw[lightgray, thin,  ->] (8,5.8) -- (8,4.2);
\draw[lightgray, thin,  ->] (6,5.8) -- (6,4.2);
\draw[lightgray, thin,  ->] (4,5.8) -- (4,4.2);
\draw[lightgray, thin,  ->] (8,7.8) -- (8,6.2);
\draw[lightgray, thin,  ->] (6,7.8) -- (6,6.2);
\draw[lightgray, thin,  ->] (4,7.8) -- (4,6.2);
\draw[lightgray, thin,  ->] (2,7.8) -- (2,6.2);
\draw[lightgray, thin,  ->] (8,9.8) -- (8,8.2);
\draw[lightgray, thin,  ->] (6,9.8) -- (6,8.2);
\draw[lightgray, thin,  ->] (4,9.8) -- (4,8.2);
\draw[lightgray, thin,  ->] (2,9.8) -- (2,8.2);
\draw[lightgray, thin,  ->] (0,9.8) -- (0,8.2);
\draw[lightgray, thin,  ->] (1.8,6.2) -- (0.2,7.8);
\draw[lightgray, thin,  ->] (3.8,6.2) -- (2.2,7.8);
\draw[lightgray, thin,  ->] (5.8,6.2) -- (4.2,7.8);
\draw[lightgray, thin,  ->] (7.8,6.2) -- (6.2,7.8);
\draw[lightgray, thin,  ->] (9.8,6.2) -- (8.2,7.8);
\draw[lightgray, thin,  ->] (1.8,8.2) -- (0.2,9.8);
\draw[lightgray, thin,  ->] (3.8,8.2) -- (2.2,9.8);
\draw[lightgray, thin,  ->] (5.8,8.2) -- (4.2,9.8);
\draw[lightgray, thin,  ->] (7.8,8.2) -- (6.2,9.8);
\draw[lightgray, thin,  ->] (9.8,8.2) -- (8.2,9.8);
\draw[lightgray, thin,  ->] (3.8,4.2) -- (2.2,5.8);
\draw[lightgray, thin,  ->] (5.8,4.2) -- (4.2,5.8);
\draw[lightgray, thin,  ->] (7.8,4.2) -- (6.2,5.8);
\draw[lightgray, thin,  ->] (9.8,4.2) -- (8.2,5.8);
\draw[lightgray, thin,  ->] (5.8,2.2) -- (4.2,3.8);
\draw[lightgray, thin,  ->] (7.8,2.2) -- (6.2,3.8);
\draw[lightgray, thin,  ->] (9.8,2.2) -- (8.2,3.8);
\draw[lightgray, thin,  ->] (7.8,0.2) -- (6.2,1.8);
\draw[lightgray, thin,  ->] (9.8,0.2) -- (8.2,1.8);
\draw[gray,fill=gray] (6,2) circle (.01ex) node {\large$\mathbf{1}$};
\draw[gray,fill=gray] (4,4) circle (.01ex) node {\large$\mathbf{1}$};
\draw[gray,fill=gray] (2,6) circle (.01ex) node {\large$\mathbf{1}$};
\draw[gray,fill=gray] (8,2) circle (.01ex) node {\large$\mathbf{2}$};
\draw[gray,fill=gray] (6,4) circle (.01ex) node {\large$\mathbf{2}$};
\draw[gray,fill=gray] (4,6) circle (.01ex) node {\large$\mathbf{2}$};
\draw[gray,fill=gray] (2,8) circle (.01ex) node {\large$\mathbf{2}$};
\draw[gray,fill=gray] (8,4) circle (.01ex) node {\large$\mathbf{3}$};
\draw[gray,fill=gray] (6,6) circle (.01ex) node {\large$\mathbf{3}$};
\draw[gray,fill=gray] (4,8) circle (.01ex) node {\large$\mathbf{3}$};
\draw[gray,fill=gray] (8,6) circle (.01ex) node {\large$\mathbf{4}$};
\draw[gray,fill=gray] (6,8) circle (.01ex) node {\large$\mathbf{4}$};
\draw[gray,fill=gray] (8,8) circle (.01ex) node {\large$\mathbf{5}$};
\end{tikzpicture}
}
\caption{The fifth case of partially integral weights}\label{fig12}
\end{figure}
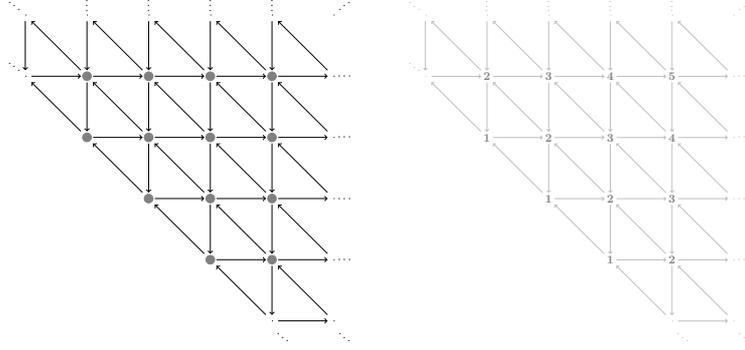

\begin{proof}
Mutatis mutandis the proof of Proposition~\ref{prop-s5.6-1}.
\end{proof}

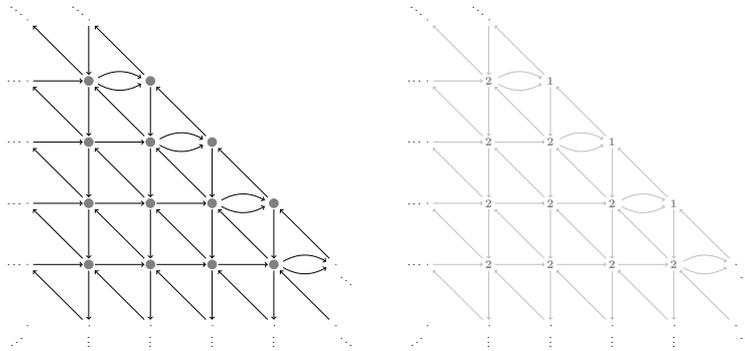
\begin{figure}
\resizebox{10cm}{!}{
\begin{tikzpicture}
\draw[black, thick,  ->] (8,1.8) -- (8,0.2);
\draw[black, thick,  ->] (6,1.8) -- (6,0.2);
\draw[black, thick,  ->] (4,1.8) -- (4,0.2);
\draw[black, thick,  ->] (2,1.8) -- (2,0.2);
\draw[black, thick,  ->] (8,3.8) -- (8,2.2);
\draw[black, thick,  ->] (6,3.8) -- (6,2.2);
\draw[black, thick,  ->] (4,3.8) -- (4,2.2);
\draw[black, thick,  ->] (2,3.8) -- (2,2.2);
\draw[black, thick,  ->] (6,5.8) -- (6,4.2);
\draw[black, thick,  ->] (4,5.8) -- (4,4.2);
\draw[black, thick,  ->] (2,5.8) -- (2,4.2);
\draw[black, thick,  ->] (4,7.8) -- (4,6.2);
\draw[black, thick,  ->] (2,7.8) -- (2,6.2);
\draw[black, thick,  ->] (2,9.8) -- (2,8.2);
\draw[black, thick,  ->] (2.3,8.1) to [out=30,in=150]   (3.7,8.1);
\draw[black, thick,  ->] (2.3,7.9) to [out=-30,in=-150]   (3.7,7.9);
\draw[black, thick,  ->] (4.3,6.1) to [out=30,in=150]   (5.7,6.1);
\draw[black, thick,  ->] (4.3,5.9) to [out=-30,in=-150]   (5.7,5.9);
\draw[black, thick,  ->] (6.3,4.1) to [out=30,in=150]   (7.7,4.1);
\draw[black, thick,  ->] (6.3,3.9) to [out=-30,in=-150]   (7.7,3.9);
\draw[black, thick,  ->] (8.3,2.1) to [out=30,in=150]   (9.7,2.1);
\draw[black, thick,  ->] (8.3,1.9) to [out=-30,in=-150]   (9.7,1.9);
\draw[black, thick,  ->] (0.2,2) -- (1.8,2);
\draw[black, thick,  ->] (2.2,2) -- (3.8,2);
\draw[black, thick,  ->] (4.2,2) -- (5.8,2);
\draw[black, thick,  ->] (6.2,2) -- (7.8,2);
\draw[black, thick,  ->] (0.2,4) -- (1.8,4);
\draw[black, thick,  ->] (2.2,4) -- (3.8,4);
\draw[black, thick,  ->] (4.2,4) -- (5.8,4);
\draw[black, thick,  ->] (0.2,6) -- (1.8,6);
\draw[black, thick,  ->] (2.2,6) -- (3.8,6);
\draw[black, thick,  ->] (0.2,8) -- (1.8,8);
\draw[black, thick,  ->] (1.8,0.2) -- (0.2,1.8);
\draw[black, thick,  ->] (3.8,0.2) -- (2.2,1.8);
\draw[black, thick,  ->] (5.8,0.2) -- (4.2,1.8);
\draw[black, thick,  ->] (7.8,0.2) -- (6.2,1.8);
\draw[black, thick,  ->] (9.8,0.2) -- (8.2,1.8);
\draw[black, thick,  ->] (1.8,2.2) -- (0.2,3.8);
\draw[black, thick,  ->] (3.8,2.2) -- (2.2,3.8);
\draw[black, thick,  ->] (5.8,2.2) -- (4.2,3.8);
\draw[black, thick,  ->] (7.8,2.2) -- (6.2,3.8);
\draw[black, thick,  ->] (9.8,2.2) -- (8.2,3.8);
\draw[black, thick,  ->] (1.8,4.2) -- (0.2,5.8);
\draw[black, thick,  ->] (3.8,4.2) -- (2.2,5.8);
\draw[black, thick,  ->] (5.8,4.2) -- (4.2,5.8);
\draw[black, thick,  ->] (7.8,4.2) -- (6.2,5.8);
\draw[black, thick,  ->] (1.8,8.2) -- (0.2,9.8);
\draw[black, thick,  ->] (1.8,6.2) -- (0.2,7.8);
\draw[black, thick,  ->] (3.8,6.2) -- (2.2,7.8);
\draw[black, thick,  ->] (5.8,6.2) -- (4.2,7.8);
\draw[black, thick,  ->] (3.8,8.2) -- (2.2,9.8);
\draw[black, thin,  -] (10,2) -- (10,2) node[anchor= north west] {\large$\ddots$};
\draw[black, thin,  -] (0,2) -- (0,2) node[anchor= east] {\large$\dots$};
\draw[black, thin,  -] (0,4) -- (0,4) node[anchor= east] {\large$\dots$};
\draw[black, thin,  -] (0,6) -- (0,6) node[anchor= east] {\large$\dots$};
\draw[black, thin,  -] (0,8) -- (0,8) node[anchor= east] {\large$\dots$};
\draw[black, thin,  -] (0,10) -- (0,10) node[anchor= south east] {\large$\ddots$};
\draw[black, thin,  -] (2,10) -- (2,10) node[anchor= south east] {\large$\ddots$};
\draw[black, thin,  -] (2,0) -- (2,0) node[anchor= north] {\large$\vdots$};
\draw[black, thin,  -] (4,0) -- (4,0) node[anchor= north] {\large$\vdots$};
\draw[black, thin,  -] (6,0) -- (6,0) node[anchor= north] {\large$\vdots$};
\draw[black, thin,  -] (8,0) -- (8,0) node[anchor= north] {\large$\vdots$};
\draw[black, thin,  -] (10,0) -- (10,0) node[anchor= north west] {\large$\ddots$};
\draw[black, thin,  -] (0,0) -- (0,0) node[anchor= north east] {\large$\iddots$};
\draw[gray,fill=gray] (2,2) circle (.9ex);
\draw[gray,fill=gray] (4,2) circle (.9ex);
\draw[gray,fill=gray] (6,2) circle (.9ex);
\draw[gray,fill=gray] (8,2) circle (.9ex);
\draw[gray,fill=gray] (2,4) circle (.9ex);
\draw[gray,fill=gray] (4,4) circle (.9ex);
\draw[gray,fill=gray] (6,4) circle (.9ex);
\draw[gray,fill=gray] (8,4) circle (.9ex);
\draw[gray,fill=gray] (2,6) circle (.9ex);
\draw[gray,fill=gray] (4,6) circle (.9ex);
\draw[gray,fill=gray] (6,6) circle (.9ex);
\draw[gray,fill=gray] (2,8) circle (.9ex);
\draw[gray,fill=gray] (4,8) circle (.9ex);
\end{tikzpicture}
\qquad\qquad
\begin{tikzpicture}
\draw[lightgray, thin,  ->] (8,1.8) -- (8,0.2);
\draw[lightgray, thin,  ->] (6,1.8) -- (6,0.2);
\draw[lightgray, thin,  ->] (4,1.8) -- (4,0.2);
\draw[lightgray, thin,  ->] (2,1.8) -- (2,0.2);
\draw[lightgray, thin,  ->] (8,3.8) -- (8,2.2);
\draw[lightgray, thin,  ->] (6,3.8) -- (6,2.2);
\draw[lightgray, thin,  ->] (4,3.8) -- (4,2.2);
\draw[lightgray, thin,  ->] (2,3.8) -- (2,2.2);
\draw[lightgray, thin,  ->] (6,5.8) -- (6,4.2);
\draw[lightgray, thin,  ->] (4,5.8) -- (4,4.2);
\draw[lightgray, thin,  ->] (2,5.8) -- (2,4.2);
\draw[lightgray, thin,  ->] (4,7.8) -- (4,6.2);
\draw[lightgray, thin,  ->] (2,7.8) -- (2,6.2);
\draw[lightgray, thin,  ->] (2,9.8) -- (2,8.2);
\draw[lightgray, thin,  ->] (2.3,8.1) to [out=30,in=150]   (3.7,8.1);
\draw[lightgray, thin,  ->] (2.3,7.9) to [out=-30,in=-150]   (3.7,7.9);
\draw[lightgray, thin,  ->] (4.3,6.1) to [out=30,in=150]   (5.7,6.1);
\draw[lightgray, thin,  ->] (4.3,5.9) to [out=-30,in=-150]   (5.7,5.9);
\draw[lightgray, thin,  ->] (6.3,4.1) to [out=30,in=150]   (7.7,4.1);
\draw[lightgray, thin,  ->] (6.3,3.9) to [out=-30,in=-150]   (7.7,3.9);
\draw[lightgray, thin,  ->] (8.3,2.1) to [out=30,in=150]   (9.7,2.1);
\draw[lightgray, thin,  ->] (8.3,1.9) to [out=-30,in=-150]   (9.7,1.9);
\draw[lightgray, thin,  ->] (0.2,2) -- (1.8,2);
\draw[lightgray, thin,  ->] (2.2,2) -- (3.8,2);
\draw[lightgray, thin,  ->] (4.2,2) -- (5.8,2);
\draw[lightgray, thin,  ->] (6.2,2) -- (7.8,2);
\draw[lightgray, thin,  ->] (0.2,4) -- (1.8,4);
\draw[lightgray, thin,  ->] (2.2,4) -- (3.8,4);
\draw[lightgray, thin,  ->] (4.2,4) -- (5.8,4);
\draw[lightgray, thin,  ->] (0.2,6) -- (1.8,6);
\draw[lightgray, thin,  ->] (2.2,6) -- (3.8,6);
\draw[lightgray, thin,  ->] (0.2,8) -- (1.8,8);
\draw[lightgray, thin,  ->] (1.8,0.2) -- (0.2,1.8);
\draw[lightgray, thin,  ->] (3.8,0.2) -- (2.2,1.8);
\draw[lightgray, thin,  ->] (5.8,0.2) -- (4.2,1.8);
\draw[lightgray, thin,  ->] (7.8,0.2) -- (6.2,1.8);
\draw[lightgray, thin,  ->] (9.8,0.2) -- (8.2,1.8);
\draw[lightgray, thin,  ->] (1.8,2.2) -- (0.2,3.8);
\draw[lightgray, thin,  ->] (3.8,2.2) -- (2.2,3.8);
\draw[lightgray, thin,  ->] (5.8,2.2) -- (4.2,3.8);
\draw[lightgray, thin,  ->] (7.8,2.2) -- (6.2,3.8);
\draw[lightgray, thin,  ->] (9.8,2.2) -- (8.2,3.8);
\draw[lightgray, thin,  ->] (1.8,4.2) -- (0.2,5.8);
\draw[lightgray, thin,  ->] (3.8,4.2) -- (2.2,5.8);
\draw[lightgray, thin,  ->] (5.8,4.2) -- (4.2,5.8);
\draw[lightgray, thin,  ->] (7.8,4.2) -- (6.2,5.8);
\draw[lightgray, thin,  ->] (1.8,8.2) -- (0.2,9.8);
\draw[lightgray, thin,  ->] (1.8,6.2) -- (0.2,7.8);
\draw[lightgray, thin,  ->] (3.8,6.2) -- (2.2,7.8);
\draw[lightgray, thin,  ->] (5.8,6.2) -- (4.2,7.8);
\draw[lightgray, thin,  ->] (3.8,8.2) -- (2.2,9.8);
\draw[black, thin,  -] (10,2) -- (10,2) node[anchor= north west] {\large$\ddots$};
\draw[black, thin,  -] (0,2) -- (0,2) node[anchor= east] {\large$\dots$};
\draw[black, thin,  -] (0,4) -- (0,4) node[anchor= east] {\large$\dots$};
\draw[black, thin,  -] (0,6) -- (0,6) node[anchor= east] {\large$\dots$};
\draw[black, thin,  -] (0,8) -- (0,8) node[anchor= east] {\large$\dots$};
\draw[black, thin,  -] (0,10) -- (0,10) node[anchor= south east] {\large$\ddots$};
\draw[black, thin,  -] (2,10) -- (2,10) node[anchor= south east] {\large$\ddots$};
\draw[black, thin,  -] (2,0) -- (2,0) node[anchor= north] {\large$\vdots$};
\draw[black, thin,  -] (4,0) -- (4,0) node[anchor= north] {\large$\vdots$};
\draw[black, thin,  -] (6,0) -- (6,0) node[anchor= north] {\large$\vdots$};
\draw[black, thin,  -] (8,0) -- (8,0) node[anchor= north] {\large$\vdots$};
\draw[black, thin,  -] (10,0) -- (10,0) node[anchor= north west] {\large$\ddots$};
\draw[black, thin,  -] (0,0) -- (0,0) node[anchor= north east] {\large$\iddots$};
\draw[gray,fill=gray] (2,2) circle (.01ex) node {\large$\mathbf{2}$};
\draw[gray,fill=gray] (4,2) circle (.01ex) node {\large$\mathbf{2}$};
\draw[gray,fill=gray] (6,2) circle (.01ex) node {\large$\mathbf{2}$};
\draw[gray,fill=gray] (8,2) circle (.01ex) node {\large$\mathbf{2}$};
\draw[gray,fill=gray] (2,4) circle (.01ex) node {\large$\mathbf{2}$};
\draw[gray,fill=gray] (4,4) circle (.01ex) node {\large$\mathbf{2}$};
\draw[gray,fill=gray] (6,4) circle (.01ex) node {\large$\mathbf{2}$};
\draw[gray,fill=gray] (8,4) circle (.01ex) node {\large$\mathbf{1}$};
\draw[gray,fill=gray] (2,6) circle (.01ex) node {\large$\mathbf{2}$};
\draw[gray,fill=gray] (4,6) circle (.01ex) node {\large$\mathbf{2}$};
\draw[gray,fill=gray] (6,6) circle (.01ex) node {\large$\mathbf{1}$};
\draw[gray,fill=gray] (2,8) circle (.01ex) node {\large$\mathbf{2}$};
\draw[gray,fill=gray] (4,8) circle (.01ex) node {\large$\mathbf{1}$};
\end{tikzpicture}
}
\caption{The sixth case of partially integral weights}\label{fig13}
\end{figure}

Denote by $\mathcal{M}_6^a$
the additive closure of all $P(\lambda)$, where
$\lambda\in X_6$. 

\begin{proposition}\label{prop-s5.6-6}
\begin{enumerate}[$($a$)$]
\item\label{prop-s5.6-6.1} 
For $\lambda\in X_4$ such that $\lambda_1+\lambda_2=-1$, 
$\mathrm{add}(\mathscr{C}\cdot L(\lambda))$ coincides with
$\mathcal{M}_6^a$ and is a simple transitive $\mathscr{C}$-module 
category whose graph
and positive eigenvector are depicted in Figure~\ref{fig13}.
\item\label{prop-s5.6-6.2}
For $\lambda\in X_6$ such that $\lambda_1+\lambda_2\neq -1$,
$\mathrm{add}(\mathscr{C}\cdot L(\lambda))$
contains $\mathcal{M}_6^a$ as a $\mathscr{C}$-module 
subcategory, moreover, the quotient of $\mathrm{add}(\mathscr{C}\cdot L(\lambda))$
by the ideal $\mathcal{I}$ generated by $\mathcal{M}_6^a$ is a simple 
transitive $\mathscr{C}$-module category that is 
equivalent to $\mathcal{M}_5^a$. 
\item\label{prop-s5.6-6.3}
If $\lambda_1+\lambda_2\neq -1$, then the indecomposable
objects in  $\mathrm{add}(\mathscr{C}\cdot L(\lambda))/\mathcal{I}$  
are given by the images of 
$L(\mu)$, where $\mu\in X_6$ is such that $\mu_1+\mu_2\neq -1$.
\end{enumerate}
\end{proposition}

\begin{proof}
Mutatis mutandis the proof of Proposition~\ref{prop-s5.6-2}. 
\end{proof}

The weights $\lambda$ that appeared in this subsection will be
called {\em partially integral weights}.

It is easy to see that the graphs shown in Figures~\ref{fig8}, \ref{fig10}, \ref{fig12}
are isomorphic and, similarly,  the graphs shown in Figures~\ref{fig9}, \ref{fig11}, \ref{fig13}
are isomorphic.

\subsection{Generic weights}\label{s5.9}

Let now $\lambda$ be a weight that is neither integral nor
partially integral. We will call such $\lambda$ {\em generic}. 
Denote by $\mathcal{K}$ the additive
closure of all $L(\mu)$, where $\mu\in\lambda+\Lambda$.

\begin{proposition}\label{prop-s5.9-1}
For any generic $\lambda$, the $\mathscr{C}$-module category
$\mathrm{add}(\mathscr{C}\cdot L(\lambda))$ coincides
with $\mathcal{K}$, moreover, the latter is a simple
transitive $\mathscr{C}$-module category whose graph
and positive eigenvector are depicted in Figure~\ref{fig14}. 
\end{proposition}

\begin{proof}
Note that, for a generic $\lambda$, we have $L(\lambda)=\Delta(\lambda)$.
Moreover, all generic blocks of category $\mathcal{O}$
are semi-simple. Consequently, for a generic $\lambda$, the 
usual standard filtration of the module
$L((1,0))\otimes_{\mathbb{C}}\Delta(\lambda)$ with subquotients
$\Delta(\lambda+(1,0))$, $\Delta(\lambda+(0,-1))$ and 
$\Delta(\lambda+(-1,1))$ gives rise to a decomposition of
$L((1,0))\otimes_{\mathbb{C}}L(\lambda)$ into a direct sum of 
$L(\lambda+(1,0))$, $L(\lambda+(0,-1))$ and 
$L(\lambda+(-1,1))$. This implies that $\mathcal{K}$
is a simple $\mathscr{C}$-module category whose graph
is  depicted on the left picture in Figure~\ref{fig14}.
That the positive eigenvector is as 
depicted on the right picture in Figure~\ref{fig14}
is easy to check.

Since the graph is simply connected, 
the $\mathscr{C}$-module $\mathcal{K}$
is transitive, hence also simple transitive
due to the semi-simplicity of the underlying category $\mathcal{K}$.
We also, clearly, have that $\mathrm{add}(\mathscr{C}\cdot L(\lambda))=\mathcal{K}$
and the proof is complete.
\end{proof}

This completes our description of the combinatorics of the 
$\mathscr{C}$-module categories
of the form $\mathrm{add}(\mathscr{C}\cdot L(\lambda))$, for
$\lambda\in\mathfrak{h}^*$.

\begin{figure}
\resizebox{10cm}{!}{
\begin{tikzpicture}
\draw[black, thick,  ->] (1.8,0.2) -- (0.2,1.8);
\draw[black, thick,  ->] (2,1.8) -- (2,0.2) node[anchor= north] {\large$\vdots$};
\draw[black, thick,  ->] (3.8,0.2) -- (2.2,1.8);
\draw[black, thick,  ->] (4,1.8) -- (4,0.2) node[anchor= north] {\large$\vdots$};
\draw[black, thick,  ->] (5.8,0.2) -- (4.2,1.8);
\draw[black, thick,  ->] (6,1.8) -- (6,0.2) node[anchor= north] {\large$\vdots$};
\draw[black, thick,  ->] (7.8,0.2) -- (6.2,1.8);
\draw[black, thick,  ->] (8,1.8) -- (8,0.2) node[anchor= north] {\large$\vdots$};
\draw[black, thick,  ->] (0.2,2) -- (1.8,2);
\draw[black, thick,  ->] (1.8,2.2) -- (0.2,3.8);
\draw[black, thick,  ->] (2.2,2) -- (3.8,2);
\draw[black, thick,  ->] (2,3.8) -- (2,2.2);
\draw[black, thick,  ->] (3.8,2.2) -- (2.2,3.8);
\draw[black, thick,  ->] (4.2,2) -- (5.8,2);
\draw[black, thick,  ->] (4,3.8) -- (4,2.2);
\draw[black, thick,  ->] (5.8,2.2) -- (4.2,3.8);
\draw[black, thick,  ->] (6.2,2) -- (7.8,2);
\draw[black, thick,  ->] (6,3.8) -- (6,2.2);
\draw[black, thick,  ->] (7.8,2.2) -- (6.2,3.8);
\draw[black, thick,  ->] (8,3.8) -- (8,2.2);
\draw[black, thick,  ->] (0.2,4) -- (1.8,4);
\draw[black, thick,  ->] (1.8,4.2) -- (0.2,5.8);
\draw[black, thick,  ->] (2.2,4) -- (3.8,4);
\draw[black, thick,  ->] (2,5.8) -- (2,4.2);
\draw[black, thick,  ->] (3.8,4.2) -- (2.2,5.8);
\draw[black, thick,  ->] (4.2,4) -- (5.8,4);
\draw[black, thick,  ->] (4,5.8) -- (4,4.2);
\draw[black, thick,  ->] (5.8,4.2) -- (4.2,5.8);
\draw[black, thick,  ->] (6.2,4) -- (7.8,4);
\draw[black, thick,  ->] (6,5.8) -- (6,4.2);
\draw[black, thick,  ->] (7.8,4.2) -- (6.2,5.8);
\draw[black, thick,  ->] (8,5.8) -- (8,4.2);
\draw[black, thick,  ->] (0.2,6) -- (1.8,6);
\draw[black, thick,  ->] (1.8,6.2) -- (0.2,7.8);
\draw[black, thick,  ->] (2.2,6) -- (3.8,6);
\draw[black, thick,  ->] (2,7.8) -- (2,6.2);
\draw[black, thick,  ->] (3.8,6.2) -- (2.2,7.8);
\draw[black, thick,  ->] (4.2,6) -- (5.8,6);
\draw[black, thick,  ->] (4,7.8) -- (4,6.2);
\draw[black, thick,  ->] (5.8,6.2) -- (4.2,7.8);
\draw[black, thick,  ->] (6.2,6) -- (7.8,6);
\draw[black, thick,  ->] (6,7.8) -- (6,6.2);
\draw[black, thick,  ->] (7.8,6.2) -- (6.2,7.8);
\draw[black, thick,  ->] (8,7.8) -- (8,6.2);
\draw[black, thick,  -] (9.8,0.35) -- (9.8,0.35) node[anchor= north west] {\large$\ddots$};
\draw[black, thick,  ->] (8.2,2) -- (9.8,2) node[anchor= west] {\large$\cdots$};
\draw[black, thick,  ->] (8.2,4) -- (9.8,4) node[anchor= west] {\large$\cdots$};
\draw[black, thick,  ->] (8.2,6) -- (9.8,6) node[anchor= west] {\large$\cdots$};
\draw[black, thick,  ->] (8.2,8) -- (9.8,8) node[anchor= west] {\large$\cdots$};
\draw[black, thick,  ->] (9.8,0.2) -- (8.2,1.8);
\draw[black, thick,  ->] (9.8,2.2) -- (8.2,3.8);
\draw[black, thick,  ->] (9.8,4.2) -- (8.2,5.8);
\draw[black, thick,  ->] (9.8,6.2) -- (8.2,7.8);
\draw[black, thick,  ->]  (2,9.8) -- (2,8.2);
\draw[black, thick,  ->]  (4,9.8) -- (4,8.2);
\draw[black, thick,  ->]  (6,9.8) -- (6,8.2);
\draw[black, thick,  ->]  (8,9.8) -- (8,8.2);
\draw[black, thin,  -] (0,9.8) -- (0,9.8) node[anchor= south east] {\large$\ddots$};
\draw[black, thin,  -] (0.2,0.2) -- (0.2,0.2) node[anchor= north east] {\large$\iddots$};
\draw[black, thin,  -] (0,8) -- (0,8) node[anchor= east] {\large$\dots$};
\draw[black, thin,  -] (0,6) -- (0,6) node[anchor= east] {\large$\dots$};
\draw[black, thin,  -] (0,4) -- (0,4) node[anchor= east] {\large$\dots$};
\draw[black, thin,  -] (0,2) -- (0,2) node[anchor= east] {\large$\dots$};
\draw[black, thin,  -] (2,9.8) -- (2,9.8) node[anchor= south] {\large$\vdots$};
\draw[black, thin,  -] (4,9.8) -- (4,9.8) node[anchor= south] {\large$\vdots$};
\draw[black, thin,  -] (6,9.8) -- (6,9.8) node[anchor= south] {\large$\vdots$};
\draw[black, thin,  -] (8,9.8) -- (8,9.8) node[anchor= south] {\large$\vdots$};
\draw[gray, thin,  -] (10.4,10.4) -- (10.4,10.4) node[anchor= north east] {\large$\iddots$};
\draw[black, thick,  ->] (1.8,8.2) -- (0.2,9.8);
\draw[black, thick,  ->] (3.8,8.2) -- (2.2,9.8);
\draw[black, thick,  ->] (5.8,8.2) -- (4.2,9.8);
\draw[black, thick,  ->] (7.8,8.2) -- (6.2,9.8);
\draw[black, thick,  ->] (9.8,8.2) -- (8.2,9.8);
\draw[black, thick,  ->] (0.2,8) -- (1.8,8);
\draw[black, thick,  ->] (2.2,8) -- (3.8,8);
\draw[black, thick,  ->] (4.2,8) -- (5.8,8);
\draw[black, thick,  ->] (6.2,8) -- (7.8,8);
\draw[gray,fill=gray] (2,2) circle (.9ex);
\draw[gray,fill=gray] (4,2) circle (.9ex);
\draw[gray,fill=gray] (6,2) circle (.9ex);
\draw[gray,fill=gray] (8,2) circle (.9ex);
\draw[gray,fill=gray] (2,4) circle (.9ex);
\draw[gray,fill=gray] (4,4) circle (.9ex);
\draw[gray,fill=gray] (6,4) circle (.9ex);
\draw[gray,fill=gray] (8,4) circle (.9ex);
\draw[gray,fill=gray] (2,6) circle (.9ex);
\draw[gray,fill=gray] (4,6) circle (.9ex);
\draw[gray,fill=gray] (6,6) circle (.9ex);
\draw[gray,fill=gray] (8,6) circle (.9ex);
\draw[gray,fill=gray] (2,8) circle (.9ex);
\draw[gray,fill=gray] (4,8) circle (.9ex);
\draw[gray,fill=gray] (6,8) circle (.9ex);
\draw[gray,fill=gray] (8,8) circle (.9ex);
\end{tikzpicture}
\qquad\qquad
\begin{tikzpicture}
\draw[lightgray, thin,  ->] (1.8,0.2) -- (0.2,1.8);
\draw[lightgray, thin,  ->] (2,1.8) -- (2,0.2) node[anchor= north] {\large$\vdots$};
\draw[lightgray, thin,  ->] (3.8,0.2) -- (2.2,1.8);
\draw[lightgray, thin,  ->] (4,1.8) -- (4,0.2) node[anchor= north] {\large$\vdots$};
\draw[lightgray, thin,  ->] (5.8,0.2) -- (4.2,1.8);
\draw[lightgray, thin,  ->] (6,1.8) -- (6,0.2) node[anchor= north] {\large$\vdots$};
\draw[lightgray, thin,  ->] (7.8,0.2) -- (6.2,1.8);
\draw[lightgray, thin,  ->] (8,1.8) -- (8,0.2) node[anchor= north] {\large$\vdots$};
\draw[lightgray, thin,  -] (0,9.8) -- (0,9.8) node[anchor= south east] {\large$\ddots$};
\draw[lightgray, thin,  -] (0.2,0.2) -- (0.2,0.2) node[anchor= north east] {\large$\iddots$};
\draw[lightgray, thin,  -] (0,8) -- (0,8) node[anchor= east] {\large$\dots$};
\draw[lightgray, thin,  -] (0,6) -- (0,6) node[anchor= east] {\large$\dots$};
\draw[lightgray, thin,  -] (0,4) -- (0,4) node[anchor= east] {\large$\dots$};
\draw[lightgray, thin,  -] (0,2) -- (0,2) node[anchor= east] {\large$\dots$};
\draw[lightgray, thin,  ->] (0.2,2) -- (1.8,2);
\draw[lightgray, thin,  ->] (1.8,2.2) -- (0.2,3.8);
\draw[lightgray, thin,  ->] (2.2,2) -- (3.8,2);
\draw[lightgray, thin,  ->] (2,3.8) -- (2,2.2);
\draw[lightgray, thin,  ->] (3.8,2.2) -- (2.2,3.8);
\draw[lightgray, thin,  ->] (4.2,2) -- (5.8,2);
\draw[lightgray, thin,  ->] (4,3.8) -- (4,2.2);
\draw[lightgray, thin,  ->] (5.8,2.2) -- (4.2,3.8);
\draw[lightgray, thin,  ->] (6.2,2) -- (7.8,2);
\draw[lightgray, thin,  ->] (6,3.8) -- (6,2.2);
\draw[lightgray, thin,  ->] (7.8,2.2) -- (6.2,3.8);
\draw[lightgray, thin,  ->] (8,3.8) -- (8,2.2);
\draw[lightgray, thin,  ->] (0.2,4) -- (1.8,4);
\draw[lightgray, thin,  ->] (1.8,4.2) -- (0.2,5.8);
\draw[lightgray, thin,  ->] (2.2,4) -- (3.8,4);
\draw[lightgray, thin,  ->] (2,5.8) -- (2,4.2);
\draw[lightgray, thin,  ->] (3.8,4.2) -- (2.2,5.8);
\draw[lightgray, thin,  ->] (4.2,4) -- (5.8,4);
\draw[lightgray, thin,  ->] (4,5.8) -- (4,4.2);
\draw[lightgray, thin,  ->] (5.8,4.2) -- (4.2,5.8);
\draw[lightgray, thin,  ->] (6.2,4) -- (7.8,4);
\draw[lightgray, thin,  ->] (6,5.8) -- (6,4.2);
\draw[lightgray, thin,  ->] (7.8,4.2) -- (6.2,5.8);
\draw[lightgray, thin,  ->] (8,5.8) -- (8,4.2);
\draw[lightgray, thin,  ->] (0.2,6) -- (1.8,6);
\draw[lightgray, thin,  ->] (1.8,6.2) -- (0.2,7.8);
\draw[lightgray, thin,  ->] (2.2,6) -- (3.8,6);
\draw[lightgray, thin,  ->] (2,7.8) -- (2,6.2);
\draw[lightgray, thin,  ->] (3.8,6.2) -- (2.2,7.8);
\draw[lightgray, thin,  ->] (4.2,6) -- (5.8,6);
\draw[lightgray, thin,  ->] (4,7.8) -- (4,6.2);
\draw[lightgray, thin,  ->] (5.8,6.2) -- (4.2,7.8);
\draw[lightgray, thin,  ->] (6.2,6) -- (7.8,6);
\draw[lightgray, thin,  ->] (6,7.8) -- (6,6.2);
\draw[lightgray, thin,  ->] (7.8,6.2) -- (6.2,7.8);
\draw[lightgray, thin,  ->] (8,7.8) -- (8,6.2);
\draw[lightgray, thin,  ->] (9.8,0.35) -- (9.8,0.35) node[anchor= north west] {\large$\ddots$};
\draw[lightgray, thin,  ->] (8.2,2) -- (9.8,2) node[anchor= west] {\large$\cdots$};
\draw[lightgray, thin,  ->] (8.2,4) -- (9.8,4) node[anchor= west] {\large$\cdots$};
\draw[lightgray, thin,  ->] (8.2,6) -- (9.8,6) node[anchor= west] {\large$\cdots$};
\draw[lightgray, thin,  ->] (8.2,8) -- (9.8,8) node[anchor= west] {\large$\cdots$};
\draw[lightgray, thin,  ->] (9.8,0.2) -- (8.2,1.8);
\draw[lightgray, thin,  ->] (9.8,2.2) -- (8.2,3.8);
\draw[lightgray, thin,  ->] (9.8,4.2) -- (8.2,5.8);
\draw[lightgray, thin,  ->] (9.8,6.2) -- (8.2,7.8);
\draw[lightgray, thin,  ->]  (2,9.8) -- (2,8.2);
\draw[lightgray, thin,  ->]  (4,9.8) -- (4,8.2);
\draw[lightgray, thin,  ->]  (6,9.8) -- (6,8.2);
\draw[lightgray, thin,  ->]  (8,9.8) -- (8,8.2);
\draw[lightgray, thin,  -] (0,9.8) -- (0,9.8) node[anchor= south] {\large$\vdots$};
\draw[lightgray, thin,  -] (2,9.8) -- (2,9.8) node[anchor= south] {\large$\vdots$};
\draw[lightgray, thin,  -] (4,9.8) -- (4,9.8) node[anchor= south] {\large$\vdots$};
\draw[lightgray, thin,  -] (6,9.8) -- (6,9.8) node[anchor= south] {\large$\vdots$};
\draw[lightgray, thin,  -] (8,9.8) -- (8,9.8) node[anchor= south] {\large$\vdots$};
\draw[lightgray, thin,  -] (10.4,10.4) -- (10.4,10.4) node[anchor= north east] {\large$\iddots$};
\draw[lightgray, thin,  ->] (1.8,8.2) -- (0.2,9.8);
\draw[lightgray, thin,  ->] (3.8,8.2) -- (2.2,9.8);
\draw[lightgray, thin,  ->] (5.8,8.2) -- (4.2,9.8);
\draw[lightgray, thin,  ->] (7.8,8.2) -- (6.2,9.8);
\draw[lightgray, thin,  ->] (9.8,8.2) -- (8.2,9.8);
\draw[lightgray, thin,  ->] (0.2,8) -- (1.8,8);
\draw[lightgray, thin,  ->] (2.2,8) -- (3.8,8);
\draw[lightgray, thin,  ->] (4.2,8) -- (5.8,8);
\draw[lightgray, thin,  ->] (6.2,8) -- (7.8,8);
\draw[gray,fill=gray] (2,2) circle (.01ex) node {\large$\mathbf{1}$};
\draw[gray,fill=gray] (4,2) circle (.01ex) node {\large$\mathbf{1}$};
\draw[gray,fill=gray] (6,2) circle (.01ex) node {\large$\mathbf{1}$};
\draw[gray,fill=gray] (8,2) circle (.01ex) node {\large$\mathbf{1}$};
\draw[gray,fill=gray] (2,4) circle (.01ex) node {\large$\mathbf{1}$};
\draw[gray,fill=gray] (4,4) circle (.01ex) node {\large$\mathbf{1}$};
\draw[gray,fill=gray] (6,4) circle (.01ex) node {\large$\mathbf{1}$};
\draw[gray,fill=gray] (8,4) circle (.01ex) node {\large$\mathbf{1}$};
\draw[gray,fill=gray] (2,6) circle (.01ex) node {\large$\mathbf{1}$};
\draw[gray,fill=gray] (4,6) circle (.01ex) node {\large$\mathbf{1}$};
\draw[gray,fill=gray] (6,6) circle (.01ex) node {\large$\mathbf{1}$};
\draw[gray,fill=gray] (8,6) circle (.01ex) node {\large$\mathbf{1}$};
\draw[gray,fill=gray] (2,8) circle (.01ex) node {\large$\mathbf{1}$};
\draw[gray,fill=gray] (4,8) circle (.01ex) node {\large$\mathbf{1}$};
\draw[gray,fill=gray] (6,8) circle (.01ex) node {\large$\mathbf{1}$};
\draw[gray,fill=gray] (8,8) circle (.01ex) node {\large$\mathbf{1}$};
\end{tikzpicture}
}
\caption{Generic weights}\label{fig14}
\end{figure}
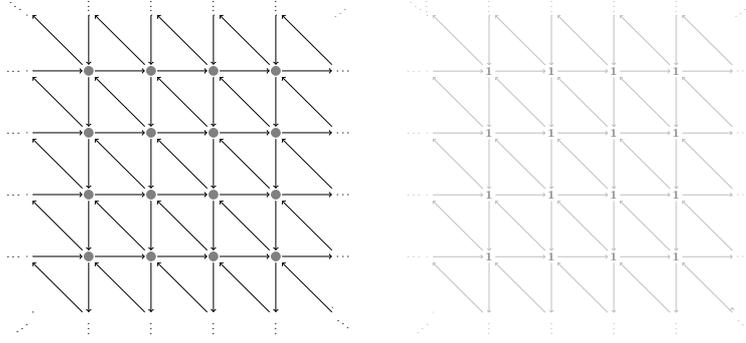

\section{$\mathscr{C}$-module categories from Whittaker modules}\label{s6}

\subsection{Whittaker modules}\label{s6.1}

Fix a Lie algebra homomorphism $\eta:\mathfrak{n}_+\to \mathbb{C}$
such that both $\eta(e_{1,2})\neq 0$  and $\eta(e_{2,3})\neq 0$.
Associated to $\eta$, we have the corresponding universal Whittaker
module $\mathbf{K}_\eta$ defined as the quotient of $U(\mathfrak{g})$
modulo the left ideal generated by $e_{1,2}-\eta(e_{1,2})$,
$e_{2,3}-\eta(e_{2,3})$ and $e_{1,3}$ (note that $\eta(e_{1,3})=0$), 
see \cite{Ko}.

Now, given a central character $\chi:Z(\mathfrak{g})\to\mathbb{C}$,
the quotient of $\mathbf{K}_\eta$ modulo $\mathrm{Ker}(\chi)\mathbf{K}_\eta$
is a simple $\mathfrak{g}$-module denoted by $\mathbf{L}(\eta,\chi)$.
The module $\mathbf{L}(\eta,\chi)$ is the unique, up to isomorphism,
simple $\mathfrak{g}$-module with central character $\chi$, on which
the action of $v-\eta(v)$ is locally nilpotent, for any $v\in \mathfrak{n}_+$.

\subsection{Some weight combinatorics}\label{s6.2}

Here, for a weight $\lambda\in\mathfrak{h}^*$, we determine all possible cases 
for how $\lambda+\Lambda$ may intersect $W\cdot \lambda$.

\begin{proposition}\label{prop-s6.2-1}
For $\lambda\in\mathfrak{h}^*$, we have:
\begin{enumerate}[$($a$)$]
\item\label{prop-s6.2-1.1} If $\lambda\in\Lambda$,
then $(W\cdot \lambda)\cap (\lambda+\Lambda)=W\cdot \lambda$.
\item\label{prop-s6.2-1.2} If $\lambda\in (0,a)+\Lambda$,
for some $a\in\mathbb{C}\setminus\mathbb{Z}$, 
then $(W\cdot \lambda)\cap (\lambda+\Lambda)=\{\lambda,s\cdot \lambda\}$.
\item\label{prop-s6.2-1.3} If $\lambda\in (a,0)+\Lambda$,
for some $a\in\mathbb{C}\setminus\mathbb{Z}$, 
then $(W\cdot \lambda)\cap (\lambda+\Lambda)=\{\lambda,r\cdot \lambda\}$.
\item\label{prop-s6.2-1.4} If $\lambda\in (a,-a)+\Lambda$,
for some $a\in\mathbb{C}\setminus\mathbb{Z}$,
then $(W\cdot \lambda)\cap (\lambda+\Lambda)=\{\lambda,w_0\cdot \lambda\}$.
\item\label{prop-s6.2-1.5} If $\lambda\in (\frac{1}{3},\frac{1}{3})+\Lambda$
or $\lambda\in (\frac{2}{3},\frac{2}{3})+\Lambda$,
then $(W\cdot \lambda)\cap (\lambda+\Lambda)=
\{\lambda,sr\cdot \lambda,rs\cdot \lambda\}$.
\item\label{prop-s6.2-1.6} In all other cases, 
$(W\cdot \lambda)\cap (\lambda+\Lambda)=\{\lambda\}$. 
\end{enumerate}
\end{proposition}

We note that the most interesting case for us in this proposition
is Case~\eqref{prop-s6.2-1.5}. In the setup of  category $\mathcal{O}$,
the weights described by Case~\eqref{prop-s6.2-1.5} are generic.
As we will see, in the setup of Whittaker modules, they are not
(only the weights described by Case~\eqref{prop-s6.2-1.6} are generic).

\begin{proof}
The group $S_3$ has exactly six subgroups: $\{e\}$, $\{e,s\}$,
$\{e,r\}$, $\{e,w_0\}$, $\{e,rs,sr\}$ and $S_3$. This gives us 
six different possibilities for the stabilizer of the set
$\lambda+\Lambda$ with respect to the dot-action.

In the standard basis of $\mathfrak{h}^*$, the actions of the linear
operators $s$, $r$ and $sr$ are given by the following matrices:
\begin{displaymath}
[s]=\left(\begin{array}{cc}-1&0\\1&1\end{array}\right),\qquad 
[r]=\left(\begin{array}{cc}1&1\\0&-1\end{array}\right),\qquad 
[sr]=\left(\begin{array}{cc}-1&-1\\1&0\end{array}\right).  
\end{displaymath}
Using these, the condition $s\cdot \lambda\in \lambda+\Lambda$
can be rewritten as the conjunction of $-\lambda_1\in \lambda_1+\mathbb{Z}$
and $\lambda_1+\lambda_2\in \lambda_2+\mathbb{Z}$. This 
is equivalent to $\lambda_1\in \mathbb{Z}$. Similarly, 
$r\cdot \lambda\in \lambda+\Lambda$ is equivalent to 
$\lambda_2\in \mathbb{Z}$.

Therefore the conjunction of $s\cdot \lambda\in \lambda+\Lambda$ and
$r\cdot \lambda\in \lambda+\Lambda$ implies Claim~\eqref{prop-s6.2-1.1}.
Taking this into account, each condition separately implies
Claim~\eqref{prop-s6.2-1.2} and Claim~\eqref{prop-s6.2-1.3}.
Claim~\eqref{prop-s6.2-1.4} is proved similarly.

To prove Claim~\eqref{prop-s6.2-1.5}, we note that 
the condition $sr\cdot \lambda\in \lambda+\Lambda$
can be rewritten as the conjunction of $-\lambda_1-\lambda_2\in \lambda_1+\mathbb{Z}$
and $\lambda_1\in \lambda_2+\mathbb{Z}$. This is equivalent to
$\lambda_2\in \frac{1}{3}\mathbb{Z}$ and $\lambda_1\in \lambda_2+\mathbb{Z}$.
Together with already established Claim~\eqref{prop-s6.2-1.1}, this implies
Claim~\eqref{prop-s6.2-1.5}.

Claim~\eqref{prop-s6.2-1.6} follows as the logical complement 
to the union of all the preceding claims.
\end{proof}

\subsection{$\mathscr{C}$-module categories generated by simple
Whittaker modules}\label{s6.3}

For a central character, $\chi:Z(\mathfrak{g})\to \mathbb{C}$, 
fix some {\em anti-dominant} $\lambda_{\chi}\in\mathfrak{h}^*$ such that
$\chi=\chi_{{}_{\lambda_{\chi}}}$. Note that 
$\lambda_{\chi}$ is not unique, in general.

For $\eta:\mathfrak{n}_+\to \mathbb{C}$
such that both $\eta(e_{1,2})\neq 0$  and $\eta(e_{2,3})\neq 0$,
and for $\chi:Z(\mathfrak{g})\to \mathbb{C}$, consider the module
$\mathbf{L}(\eta,\chi)$ and the $\mathscr{C}$-module
category $\mathrm{add}(\mathscr{C}\cdot \mathbf{L}(\eta,\chi))$.

\begin{proposition}\label{prop-s6.3-1}
For $\eta$ and $\chi$ as above, we have:
\begin{enumerate}[$($a$)$]
\item\label{prop-s6.3-1.1} If 
$\lambda_{\chi}\in (\frac{1}{3},\frac{1}{3})+\Lambda$
or $\lambda_{\chi}\in (\frac{2}{3},\frac{2}{3})+\Lambda$,
then $\mathrm{add}(\mathscr{C}\cdot \mathbf{L}(\eta,\chi))$
is semi-simple and is a simple transitive 
$\mathscr{C}$-module category whose graph
and positive eigenvector are depicted in Figure~\ref{fig15}.
\item\label{prop-s6.3-1.2} 
In all other cases, $\mathrm{add}(\mathscr{C}\cdot \mathbf{L}(\eta,\chi))$
is equivalent to $\mathrm{add}(\mathscr{C}\cdot L(\lambda_{\chi}))$,
as a $\mathscr{C}$-module category.
\end{enumerate}
\end{proposition}

\begin{proof}
Claim~\eqref{prop-s6.3-1.2} follows from \cite[Theorem~5.3]{MiSo}.
Claim~\eqref{prop-s6.3-1.1} differs from Claim~\eqref{prop-s6.3-1.2}
due to the fact that, in this case, the second condition in
\cite[Proposition~5.5]{MiSo} is violated. 

Indeed, let $\lambda_{\chi}\in (\frac{1}{3},\frac{1}{3})+\Lambda$
or $\lambda_{\chi}\in (\frac{2}{3},\frac{2}{3})+\Lambda$.
Then $(W\cdot\lambda)\cap\lambda_{\chi}+\Lambda$ consists of 
$\lambda_{\chi}$, $sr\cdot \lambda_{\chi}$ and $rs\cdot \lambda_{\chi}$. At the same time,
the integral Weyl group of $\lambda_{\chi}$ is trivial. Our classification
of projective functors tells us that we have two non-trival projective functors,
namely $\theta_{\lambda_{\chi},sr\cdot \lambda_{\chi}}$ and 
$\theta_{\lambda_{\chi},rs\cdot \lambda_{\chi}}$. On the one hand,
they are defined via
\begin{displaymath}
\theta_{\lambda_{\chi},sr\cdot \lambda_{\chi}}L(\lambda_{\chi}) 
\cong L(sr\cdot \lambda_{\chi})\quad\text{ and }\quad
\theta_{\lambda_{\chi},rs\cdot \lambda_{\chi}}L(\lambda_{\chi}) 
\cong L(rs\cdot \lambda_{\chi}).
\end{displaymath}
On the other hand, we have 
\begin{displaymath}
\theta_{\lambda_{\chi},sr\cdot \lambda_{\chi}}\mathbf{L}(\eta,\chi)\cong
\theta_{\lambda_{\chi},rs\cdot \lambda_{\chi}}\mathbf{L}(\eta,\chi)\cong
\mathbf{L}(\eta,\chi),
\end{displaymath}
due to the uniqueness of $\mathbf{L}(\eta,\chi)$.
In particular, $\mathrm{add}(\mathscr{C}\cdot \mathbf{L}(\eta,\chi))$
is semi-simple, which implies that it is also a simple transitive
$\mathscr{C}$-module category. 

To obtain its combinatorics, we can think of 
$\mathrm{add}(\mathscr{C}\cdot \mathbf{L}(\eta,\chi))$ as
a ``quotient'' of $\mathrm{add}(\mathscr{C}\cdot L(\lambda_{\chi}))$
under the action of the stabilizer group $\{e,sr,rs\}$.
Then it is easy to check that Figure~\ref{fig15}
is simply such a quotient of Figure~\ref{fig14}.
\end{proof}

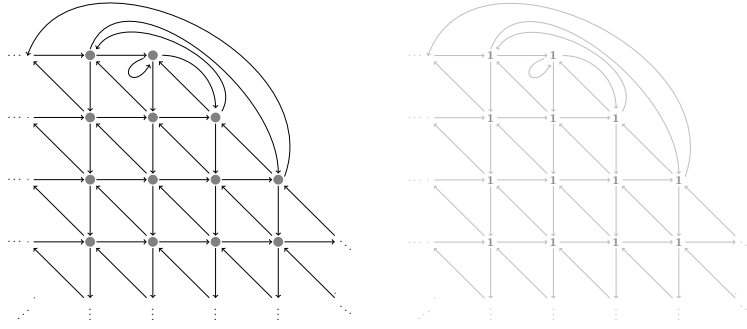
\begin{figure}
\resizebox{10cm}{!}{
\begin{tikzpicture}
\draw[black, thick,  ->] (1.8,0.2) -- (0.2,1.8);
\draw[black, thick,  ->] (2,1.8) -- (2,0.2) node[anchor= north] {\large$\vdots$};
\draw[black, thick,  ->] (3.8,0.2) -- (2.2,1.8);
\draw[black, thick,  ->] (4,1.8) -- (4,0.2) node[anchor= north] {\large$\vdots$};
\draw[black, thick,  ->] (5.8,0.2) -- (4.2,1.8);
\draw[black, thick,  ->] (6,1.8) -- (6,0.2) node[anchor= north] {\large$\vdots$};
\draw[black, thick,  ->] (7.8,0.2) -- (6.2,1.8);
\draw[black, thick,  ->] (8,1.8) -- (8,0.2) node[anchor= north] {\large$\vdots$};
\draw[black, thick,  ->] (0.2,2) -- (1.8,2);
\draw[black, thick,  ->] (1.8,2.2) -- (0.2,3.8);
\draw[black, thick,  ->] (2.2,2) -- (3.8,2);
\draw[black, thick,  ->] (2,3.8) -- (2,2.2);
\draw[black, thick,  ->] (3.8,2.2) -- (2.2,3.8);
\draw[black, thick,  ->] (4.2,2) -- (5.8,2);
\draw[black, thick,  ->] (4,3.8) -- (4,2.2);
\draw[black, thick,  ->] (5.8,2.2) -- (4.2,3.8);
\draw[black, thick,  ->] (6.2,2) -- (7.8,2);
\draw[black, thick,  ->] (6,3.8) -- (6,2.2);
\draw[black, thick,  ->] (7.8,2.2) -- (6.2,3.8);
\draw[black, thick,  ->] (8,3.8) -- (8,2.2);
\draw[black, thick,  ->] (0.2,4) -- (1.8,4);
\draw[black, thick,  ->] (1.8,4.2) -- (0.2,5.8);
\draw[black, thick,  ->] (2.2,4) -- (3.8,4);
\draw[black, thick,  ->] (2,5.8) -- (2,4.2);
\draw[black, thick,  ->] (3.8,4.2) -- (2.2,5.8);
\draw[black, thick,  ->] (4.2,4) -- (5.8,4);
\draw[black, thick,  ->] (4,5.8) -- (4,4.2);
\draw[black, thick,  ->] (5.8,4.2) -- (4.2,5.8);
\draw[black, thick,  ->] (6.2,4) -- (7.8,4);
\draw[black, thick,  ->] (6,5.8) -- (6,4.2);
\draw[black, thick,  ->] (7.8,4.2) -- (6.2,5.8);
\draw[black, thick,  ->] (0.2,6) -- (1.8,6);
\draw[black, thick,  ->] (1.8,6.2) -- (0.2,7.8);
\draw[black, thick,  ->] (2.2,6) -- (3.8,6);
\draw[black, thick,  ->] (2,7.8) -- (2,6.2);
\draw[black, thick,  ->] (3.8,6.2) -- (2.2,7.8);
\draw[black, thick,  ->] (4.2,6) -- (5.8,6);
\draw[black, thick,  ->] (4,7.8) -- (4,6.2);
\draw[black, thick,  ->] (5.8,6.2) -- (4.2,7.8);
\draw[black, thick,  -] (9.8,0.35) -- (9.8,0.35) node[anchor= north west] {\large$\ddots$};
\draw[black, thick,  ->] (8.2,2) -- (9.8,2) node[anchor= west] {\large$\ddots$};
\draw[black, thick,  ->] (9.8,0.2) -- (8.2,1.8);
\draw[black, thick,  ->] (9.8,2.2) -- (8.2,3.8);
\draw[black, thin,  -] (0.2,0.2) -- (0.2,0.2) node[anchor= north east] {\large$\iddots$};
\draw[black, thin,  -] (0,8) -- (0,8) node[anchor= east] {\large$\dots$};
\draw[black, thin,  -] (0,6) -- (0,6) node[anchor= east] {\large$\dots$};
\draw[black, thin,  -] (0,4) -- (0,4) node[anchor= east] {\large$\dots$};
\draw[black, thin,  -] (0,2) -- (0,2) node[anchor= east] {\large$\dots$};
\draw[black, thick,  ->] (0.2,8) -- (1.8,8);
\draw[black, thick,  ->] (2.2,8) -- (3.8,8);
\draw[black, thick,  ->] (4.3,8) to [out=0,in=90]   (6,6.3);
\draw[black, thick,  ->] (3.7,7.9) to [out=200,in=240,looseness=9]   (3.9,7.7);
\draw[black, thick,  ->] (6.2,6.2) to [out=60,in=60]   (2.2,8.2);
\draw[black, thick,  ->] (2,8.2) to [out=70,in=90]   (8,4.2);
\draw[black, thick,  ->] (8.2,4.1) to [out=70,in=70,looseness=1.2] (0,8);
\draw[gray,fill=gray] (2,2) circle (.9ex);
\draw[gray,fill=gray] (4,2) circle (.9ex);
\draw[gray,fill=gray] (6,2) circle (.9ex);
\draw[gray,fill=gray] (8,2) circle (.9ex);
\draw[gray,fill=gray] (2,4) circle (.9ex);
\draw[gray,fill=gray] (4,4) circle (.9ex);
\draw[gray,fill=gray] (6,4) circle (.9ex);
\draw[gray,fill=gray] (8,4) circle (.9ex);
\draw[gray,fill=gray] (2,6) circle (.9ex);
\draw[gray,fill=gray] (4,6) circle (.9ex);
\draw[gray,fill=gray] (6,6) circle (.9ex);
\draw[gray,fill=gray] (2,8) circle (.9ex);
\draw[gray,fill=gray] (4,8) circle (.9ex);
\end{tikzpicture}
\qquad\qquad
\begin{tikzpicture}
\draw[lightgray, thin,  ->] (1.8,0.2) -- (0.2,1.8);
\draw[lightgray, thin,  ->] (2,1.8) -- (2,0.2) node[anchor= north] {\large$\vdots$};
\draw[lightgray, thin,  ->] (3.8,0.2) -- (2.2,1.8);
\draw[lightgray, thin,  ->] (4,1.8) -- (4,0.2) node[anchor= north] {\large$\vdots$};
\draw[lightgray, thin,  ->] (5.8,0.2) -- (4.2,1.8);
\draw[lightgray, thin,  ->] (6,1.8) -- (6,0.2) node[anchor= north] {\large$\vdots$};
\draw[lightgray, thin,  ->] (7.8,0.2) -- (6.2,1.8);
\draw[lightgray, thin,  ->] (8,1.8) -- (8,0.2) node[anchor= north] {\large$\vdots$};
\draw[lightgray, thin,  ->] (0.2,2) -- (1.8,2);
\draw[lightgray, thin,  ->] (1.8,2.2) -- (0.2,3.8);
\draw[lightgray, thin,  ->] (2.2,2) -- (3.8,2);
\draw[lightgray, thin,  ->] (2,3.8) -- (2,2.2);
\draw[lightgray, thin,  ->] (3.8,2.2) -- (2.2,3.8);
\draw[lightgray, thin,  ->] (4.2,2) -- (5.8,2);
\draw[lightgray, thin,  ->] (4,3.8) -- (4,2.2);
\draw[lightgray, thin,  ->] (5.8,2.2) -- (4.2,3.8);
\draw[lightgray, thin,  ->] (6.2,2) -- (7.8,2);
\draw[lightgray, thin,  ->] (6,3.8) -- (6,2.2);
\draw[lightgray, thin,  ->] (7.8,2.2) -- (6.2,3.8);
\draw[lightgray, thin,  ->] (8,3.8) -- (8,2.2);
\draw[lightgray, thin,  ->] (0.2,4) -- (1.8,4);
\draw[lightgray, thin,  ->] (1.8,4.2) -- (0.2,5.8);
\draw[lightgray, thin,  ->] (2.2,4) -- (3.8,4);
\draw[lightgray, thin,  ->] (2,5.8) -- (2,4.2);
\draw[lightgray, thin,  ->] (3.8,4.2) -- (2.2,5.8);
\draw[lightgray, thin,  ->] (4.2,4) -- (5.8,4);
\draw[lightgray, thin,  ->] (4,5.8) -- (4,4.2);
\draw[lightgray, thin,  ->] (5.8,4.2) -- (4.2,5.8);
\draw[lightgray, thin,  ->] (6.2,4) -- (7.8,4);
\draw[lightgray, thin,  ->] (6,5.8) -- (6,4.2);
\draw[lightgray, thin,  ->] (7.8,4.2) -- (6.2,5.8);
\draw[lightgray, thin,  ->] (0.2,6) -- (1.8,6);
\draw[lightgray, thin,  ->] (1.8,6.2) -- (0.2,7.8);
\draw[lightgray, thin,  ->] (2.2,6) -- (3.8,6);
\draw[lightgray, thin,  ->] (2,7.8) -- (2,6.2);
\draw[lightgray, thin,  ->] (3.8,6.2) -- (2.2,7.8);
\draw[lightgray, thin,  ->] (4.2,6) -- (5.8,6);
\draw[lightgray, thin,  ->] (4,7.8) -- (4,6.2);
\draw[lightgray, thin,  ->] (5.8,6.2) -- (4.2,7.8);
\draw[lightgray, thin,  -] (9.8,0.35) -- (9.8,0.35) node[anchor= north west] {\large$\ddots$};
\draw[lightgray, thin,  ->] (8.2,2) -- (9.8,2) node[anchor= west] {\large$\ddots$};
\draw[lightgray, thin,  ->] (9.8,0.2) -- (8.2,1.8);
\draw[lightgray, thin,  ->] (9.8,2.2) -- (8.2,3.8);
\draw[lightgray, thin,  -] (0.2,0.2) -- (0.2,0.2) node[anchor= north east] {\large$\iddots$};
\draw[lightgray, thin,  -] (0,8) -- (0,8) node[anchor= east] {\large$\dots$};
\draw[lightgray, thin,  -] (0,6) -- (0,6) node[anchor= east] {\large$\dots$};
\draw[lightgray, thin,  -] (0,4) -- (0,4) node[anchor= east] {\large$\dots$};
\draw[lightgray, thin,  -] (0,2) -- (0,2) node[anchor= east] {\large$\dots$};
\draw[lightgray, thin,  ->] (0.2,8) -- (1.8,8);
\draw[lightgray, thin,  ->] (2.2,8) -- (3.8,8);
\draw[lightgray, thin,  ->] (4.3,8) to [out=0,in=90]   (6,6.3);
\draw[lightgray, thin,  ->] (3.7,7.9) to [out=200,in=240,looseness=9]   (3.9,7.7);
\draw[lightgray, thin,  ->] (6.2,6.2) to [out=60,in=60]   (2.2,8.2);
\draw[lightgray, thin,  ->] (2,8.2) to [out=70,in=90]   (8,4.2);
\draw[lightgray, thin,  ->] (8.2,4.1) to [out=70,in=70,looseness=1.2] (0,8);
\draw[gray,fill=gray] (2,2) circle (.01ex) node {\large$\mathbf{1}$};
\draw[gray,fill=gray] (4,2) circle (.01ex) node {\large$\mathbf{1}$};
\draw[gray,fill=gray] (6,2) circle (.01ex) node {\large$\mathbf{1}$};
\draw[gray,fill=gray] (8,2) circle (.01ex) node {\large$\mathbf{1}$};
\draw[gray,fill=gray] (2,4) circle (.01ex) node {\large$\mathbf{1}$};
\draw[gray,fill=gray] (4,4) circle (.01ex) node {\large$\mathbf{1}$};
\draw[gray,fill=gray] (6,4) circle (.01ex) node {\large$\mathbf{1}$};
\draw[gray,fill=gray] (8,4) circle (.01ex) node {\large$\mathbf{1}$};
\draw[gray,fill=gray] (2,6) circle (.01ex) node {\large$\mathbf{1}$};
\draw[gray,fill=gray] (4,6) circle (.01ex) node {\large$\mathbf{1}$};
\draw[gray,fill=gray] (6,6) circle (.01ex) node {\large$\mathbf{1}$};
\draw[gray,fill=gray] (2,8) circle (.01ex) node {\large$\mathbf{1}$};
\draw[gray,fill=gray] (4,8) circle (.01ex) node {\large$\mathbf{1}$};

\end{tikzpicture}
}
\caption{Additional combinatorics from the Whittaker setup}\label{fig15}
\end{figure}

\section{Combinatorics of transitive $\mathscr{C}$-module subquotient
of arbitrary $\mathrm{add}(\mathscr{C}\cdot L)$}\label{s7}

\subsection{Setup}\label{s7.1}

Let now $L$ be a simple $\mathfrak{g}$-module. Associated to $L$,
we have the 
$\mathscr{C}$-module category $\mathrm{add}(\mathscr{C}\cdot L)$.
The category $\mathrm{add}(\mathscr{C}\cdot L)$ is locally
finitary, moreover, for any indecomposable object $X$, there are
only finitely many, up to isomorphism, indecomposable objects
$Y$ such that $\mathrm{Hom}(X,Y)\neq 0$ (and, similarly,
in the other direction). We refer to \cite{MMM} for the details
and to \cite{Mac1,Mac2} for more information on finitary module
categories.

\subsection{Main result}\label{s7.2}

We are now ready to formulate the main result of our paper
which describes combinatorics of transitive subquotients
of $\mathrm{add}(\mathscr{C}\cdot L)$.

\begin{theorem}\label{thm-main}
Let $L$ be a simple $\mathfrak{g}$-module and 
   $\mathbf{N}$ a transitive subquotient of 
$\mathrm{add}(\mathscr{C}\cdot L)$. Then the graph
$\Gamma_\mathrm{F}$ for $\mathbf{N}$
is isomorphic to the graph given by one of the following
figures: Figure~\ref{fig3}, Figure~\ref{fig5}, Figure~\ref{fig6},  
Figure~\ref{fig7}, Figure~\ref{fig8}, Figure~\ref{fig9},
Figure~\ref{fig10}, Figure~\ref{fig11}, Figure~\ref{fig12},
Figure~\ref{fig13}, Figure~\ref{fig14}, Figure~\ref{fig15}
(note that some of them are isomorphic).
\end{theorem}

For convenience, we collected all these graphs together
in Figure~\ref{fig16}.

\begin{figure}
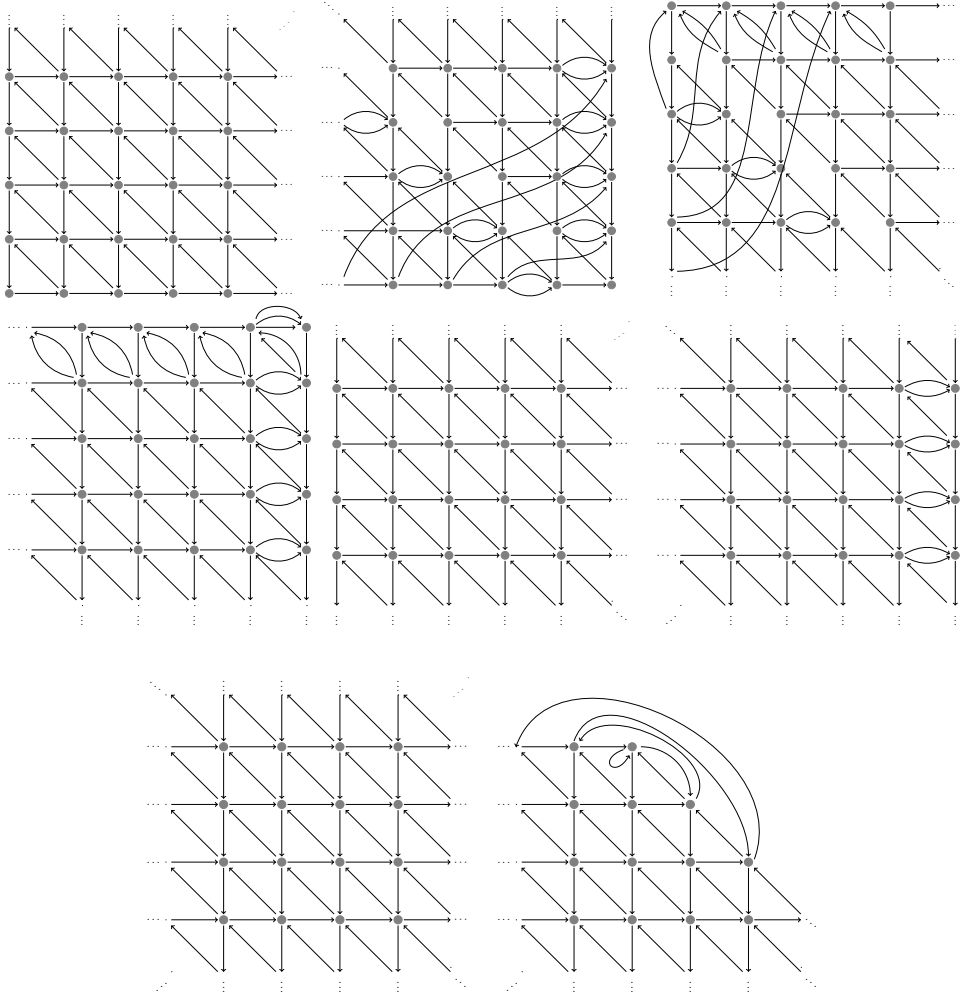

\resizebox{\textwidth}{!}{

}
\caption{All eight graphs mentioned in Theorem~\ref{thm-main}}\label{fig16}
\end{figure}

\subsection{Proof}\label{s7.3}

In order to prove Theorem~\ref{thm-main}, we reinterpret the
results of \cite{MiSo} (which we already mentioned in the
proof of Proposition~\ref{prop-s6.3-1}) in terms of the
approach to representations of monoidal categories via
algebra objects as described in, e.g. \cite[Chapter~7]{EGNO}. 
This will allow us to generalize the necessary results of
\cite{MiSo} to the case of our arbitrary simple module $L$.

Consider the $\mathfrak{g}$-$\mathfrak{g}$-bimodule
$\mathrm{Hom}_\mathbb{C}(L,L)$ of all linear endomorphisms
of $L$. It contains a subbimodule $\mathcal{L}(L,L)$ which
consists of all elements, the adjoint action of $\mathfrak{g}$
on which is locally finite. In fact, the multiplicity of each
simple finite dimensional $\mathfrak{g}$-module in the 
adjoint $\mathfrak{g}$-module $\mathcal{L}(L,L)$ is finite,
that is, $\mathcal{L}(L,L)$ is a Harish-Chandra 
$\mathfrak{g}$-$\mathfrak{g}$-bimodule, see \cite[Kapitel~6]{Ja}.
The monoidal category $\mathscr{H}$ of Harish-Chandra bimodules,
see \cite{BG}, is a proxy for the action of 
$\mathscr{C}$ on $\mathfrak{g}$-modules with locally finite
action of $Z(\mathfrak{g})$.

The space $\mathcal{L}(L,L)$ has the natural structure of an algebra.
By the combination of \cite[Remark~7.9.1]{EGNO} with
\cite[Theorem~7.10.1]{EGNO}, the category of all
$\mathcal{L}(L,L)$-modules in $\mathscr{H}$ is equivalent
to $\mathrm{add}(\mathscr{C}\cdot L)$. Therefore,
in order to prove Theorem~\ref{thm-main}, we just need to show
that there either exists a simple highest weight module
$L(\lambda)$ or a simple Whittaker module
$\mathbf{L}(\eta,\chi)$ such that $\mathcal{L}(L,L)$
is isomorphic (as an algebra) to the corresponding
$\mathcal{L}(L(\lambda),L(\lambda))$ or
$\mathcal{L}(\mathbf{L}(\eta,\chi),\mathbf{L}(\eta,\chi))$,
respectively.

By \cite[Proposition~2.6.8]{Di}, the module $L$ admits a central character,
say $\chi$. Then there exists $\lambda\in\mathfrak{h}^*$ such that
the annihilators of $L$ and $L(\lambda)$ in $U(\mathfrak{g})$
coincide, see \cite{Du}. In particular, $\chi$ is also
the central character of $L(\lambda)$.

Consider first the case $\lambda\not\in (\frac{1}{3},\frac{1}{3})+\Lambda$.
We claim that, in this case, the following holds:
\begin{displaymath}
\mathcal{L}(L(\lambda),L(\lambda))
\cong \mathcal{L}(L,L)\cong U(\mathfrak{g})/\mathrm{Ann}_{U(\mathfrak{g})}(L). 
\end{displaymath}
All this can be found in the literature, see \cite{MS1,MMM,Maz}.
The point of this case is that the part of $\mathcal{O}_{{}_\chi}$
which corresponds to the weights in $\lambda+\Lambda$ is indecomposable
(as a category) and hence the action of  projective
functors on this part is similar to the action on an integral block.
Therefore the isomorphisms above just reflect the fact that 
Kostant's problem has positive solution for all simple 
$\mathfrak{sl}_2$ and $\mathfrak{sl}_3$-modules with integral
central characters, see \cite{Maz} for details.

Now consider the case $\lambda\in (\frac{1}{3},\frac{1}{3})+\Lambda$.
In this case, the block $\mathcal{O}_{{}_\chi}$ is semi-simple, so
$L(\lambda)$ is, in fact, a Verma module. This implies that
\begin{displaymath}
\mathcal{L}(L(\lambda),L(\lambda))
\cong U(\mathfrak{g})/\mathrm{Ann}_{U(\mathfrak{g})}(L),
\end{displaymath}
see  \cite[Section~6.9]{Ja}.

However, the intersection of $W\cdot \lambda$ with 
$\lambda+\Lambda$ contains two other weights, different from $\lambda$.
Let us denote them by $\lambda'$ and $\lambda''$.
This means that there are three pairwise non-isomorphic 
indecomposable projective  endofunctors of $\mathcal{O}_{{}_\chi}$, 
namely $\theta_{\lambda,\lambda}$,
$\theta_{\lambda,\lambda'}$ and
$\theta_{\lambda,\lambda''}$. They all send simple modules
to simple modules and hence are equivalences of categories.
In other words, we have the cyclic group of order $3$ acting
on $\mathcal{O}_{{}_\chi}$ by equivalences.
This leads to two different subcases:

{\bf Subcase 1.} Assume that the modules $L$ and
$\theta_{\lambda,\lambda'} L$ are not isomorphic.
In this subcase we automatically have that the modules
$L$, $\theta_{\lambda,\lambda'} L$ and $\theta_{\lambda,\lambda''} L$
are pairwise non-isomorphic (as the only proper subgroup of
the cyclic group on three elements
is the trivial subgroup). This means that, for an indecomposable
projective functor $\theta$, the inequality
$\mathrm{Hom}(L,\theta L)\neq 0$ implies that
$\theta$ is the identity functor. Exactly the same property holds
if we substitute $L$ by $L(\lambda)$. Therefore, for any 
finite dimensional $\mathfrak{g}$-module $V$, 
using \cite[Section~6.8]{Ja}, we have
\begin{displaymath}
\dim \mathrm{Hom}_{\mathfrak{g}}(V,\mathcal{L}(L,L)^\mathrm{ad})=
\dim \mathrm{Hom}_{\mathfrak{g}}(V,\mathcal{L}(L(\lambda),L(\lambda))^\mathrm{ad}),
\end{displaymath}
where in the second argument we consider the adjoint $\mathfrak{g}$-action,
which implies that the natural inclusion
$U(\mathfrak{g})/\mathrm{Ann}_{U(\mathfrak{g})}(L)\subset 
\mathcal{L}(L,L)$
is an isomorphism. This shows that 
$\mathcal{L}(L,L)\cong
\mathcal{L}(L(\lambda),L(\lambda))$ and we are done.

{\bf Subcase 2.} Assume that the modules $L$ and
$\theta_{\lambda,\lambda'} L$ are isomorphic.
In this subcase we automatically have that the modules
$L$, $\theta_{\lambda,\lambda'} L$ and $\theta_{\lambda,\lambda''} L$
are all isomorphic (as the only non-trivial subgroup of the
cyclic group on three elements is the group itself).
This means that, for any indecomposable
projective endofunctor $\theta$ of $\mathcal{O}_{{}_\chi}$, 
we have $\mathrm{Hom}(L,\theta L)\cong \mathbb{C}$.
Note that the same holds for $\mathbf{L}(\eta,\chi)$.
Therefore, for any finite dimensional $\mathfrak{g}$-module $V$, 
using \cite[Section~6.8]{Ja}, we have
\begin{equation}\label{eq-nnn23}
\dim \mathrm{Hom}_{\mathfrak{g}}(V,\mathcal{L}(L,L)^\mathrm{ad})=
\dim \mathrm{Hom}_{\mathfrak{g}}
(V,\mathcal{L}(\mathbf{L}(\eta,\chi),\mathbf{L}(\eta,\chi))^\mathrm{ad}).
\end{equation}
Note that both $\mathcal{L}(L,L)$
and $\mathcal{L}(\mathbf{L}(\eta,\chi),\mathbf{L}(\eta,\chi))$
are semi-simple by \cite[Theorem~5.9]{BG}. From \eqref{eq-nnn23}
it follows that each simple bimodule appears in
$\mathcal{L}(L,L)$
and $\mathcal{L}(\mathbf{L}(\eta,\chi),\mathbf{L}(\eta,\chi))$
with the same multiplicity. This implies that the bimodules
$\mathcal{L}(L,L)$
and $\mathcal{L}(\mathbf{L}(\eta,\chi),\mathbf{L}(\eta,\chi))$
are isomorphic. The proof
of the main theorem is now complete.

\subsection{The action of $\mathrm{G}$}\label{s7.5}

As the complete combinatorics is uniquely determined
by $[\mathrm{F}]$ and $[\mathrm{G}]$, see Subsection~\ref{s3.4},
for convenience, in Figure~\ref{fig17} we collected the 
graphs $\Gamma_{\mathrm{G}}$ corresponding to the graphs
in Figure~\ref{fig16}. Determination of the graphs 
in Figure~\ref{fig17} is similar
to our determination of the graphs in Figure~\ref{fig16}.

\begin{figure}
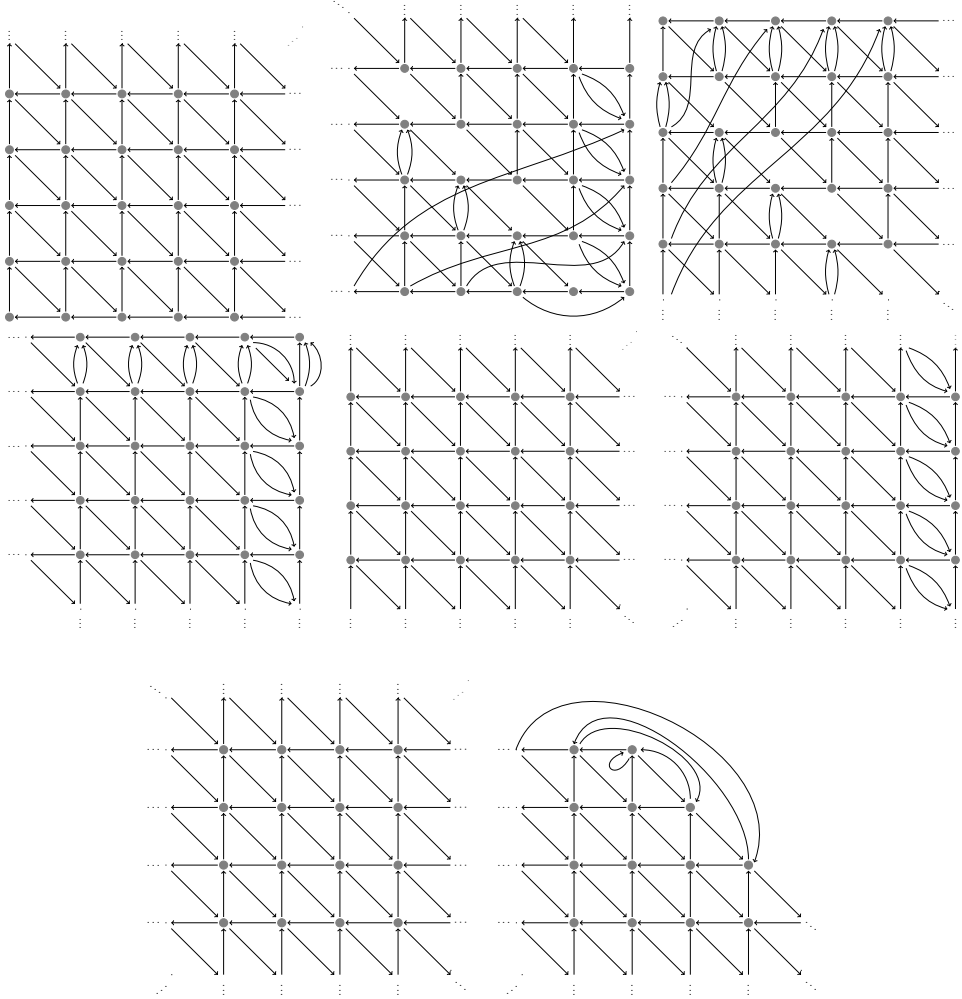

\resizebox{\textwidth}{!}{

}
\caption{The graphs $\Gamma_{\mathrm{G}}$
corresponding to Figure~\ref{fig16}}\label{fig17}
\end{figure}

\subsection{Semi-simplicity of some underlying categories}\label{s7.7}

Similarly to \cite[Proposition~20]{MZ}, we also have the following:

\begin{proposition}\label{prop-s7.7-1}
Let $\mathcal{M}$ be a simple transitive 
admissible $\mathscr{C}$-module category
whose combinatorics is given by one of the four graphs in Figure~\ref{fig16}
that does not have double oriented arrows describing the action of   $\mathrm{F}$,
with the corresponding graph in Figure~\ref{fig17} describing the
action of $\mathrm{G}$. Then the category $\mathcal{M}$
is semi-simple.
\end{proposition}

\begin{proof}
The proof is an adaptation of the proof of \cite[Proposition~20]{MZ}.
We just sketch the main steps below, leaving the details to the reader.
Consider the abelianization $\overline{\mathcal{M}}$.
We need to prove that the radical of $\overline{\mathcal{M}}$
is $\mathscr{C}$-invariant. As $\mathscr{C}$ is generated by
$\mathrm{F}$, we just need to prove that the radical 
of $\overline{\mathcal{M}}$ is $\mathrm{F}$-invariant.

The radical of $\overline{\mathcal{M}}$
is generated by morphisms between non-isomorphic indecomposable 
projectives together with the nilpotent endomorphisms of
indecomposable projectives. The action of $\mathrm{F}$ preserves
nilpotency and hence the outcome, when applied to  
a nilpotent endomorphism of an indecomposable projective, is 
a radical morphism, since, by our assumptions, this action is 
multiplicity free in the basis of projectives.

Under our assumptions, the graphs $\Gamma_\mathrm{F}$ and 
$\Gamma_\mathrm{G}$ are opposite to each other. This implies
that the combinatorics of the action of $\mathrm{F}$ in the basis of
simples in $\overline{\mathcal{M}}$ coincides with the combinatorics of 
of the action of $\mathrm{F}$ in the basis of projectives in 
$\overline{\mathcal{M}}$.
Given a morphism $f:P\to Q$ between two indecomposable projectives,
the image of $f$ belongs to the radical of $Q$. So, we need to show
that the image of $\mathrm{F}(f)$ belongs to the radical of
$\mathrm{F}(Q)$. In fact, $\mathrm{F}(\mathrm{Rad}(Q))=
\mathrm{Rad}(\mathrm{F}(Q))$.
Indeed, applying the exact functor $\mathrm{F}$
to the short exact sequence
\begin{displaymath}
0\to \mathrm{Rad}(Q)\to Q\to Q/ \mathrm{Rad}(Q)\to 0,
\end{displaymath}
we observe that the top of $\mathrm{F}(Q)$ is isomorphic to
$\mathrm{F}(Q/\mathrm{Rad}(Q))$ since the combinatorics of the action of
$\mathrm{F}$ in the bases of projectives and simples agree.
This implies the claim and completes the proof.
\end{proof}

\vspace{2mm}

\noindent
V.~M.: Department of Mathematics, Uppsala University, Box. 480,
SE-75106, Uppsala,\\ SWEDEN, email: {\tt mazor\symbol{64}math.uu.se}

\noindent
X.~Z.: School of Mathematics and Statistics, 
Ningbo University, Ningbo, 
Zhejiang, 315211, China, email: 
{\tt zhuxiaoyu1\symbol{64}nbu.edu.cn}

\end{document}